\documentclass[12pt]{article}
\usepackage{amscd} %20130322
\usepackage{amsmath,amssymb,verbatim}
\usepackage{amsthm}
\usepackage{graphicx}
\usepackage{enumerate}
\usepackage{color}
\usepackage{multind}
\makeindex{sbon}
\numberwithin{equation}{section}
\numberwithin{figure}{section}
\newtheorem{theorem}{Theorem}[section]

\newtheorem{lemma}[theorem]{Lemma}
\newtheorem{proposition}[theorem]{Proposition}
\newtheorem{corollary}[theorem]{Corollary}
\newtheorem{fact}[theorem]{Fact}
\newtheorem{example}[theorem]{Example}
\theoremstyle{definition}
\newtheorem{definition}[theorem]{Definition}
\theoremstyle{remark}
\newtheorem{remark}[theorem]{Remark}
\renewcommand{\setminus}{-}
\newcommand{\ntAln}[2]{{\mathbb{A}}_{{#1},{#2}}}
\newcommand{\A}{\widetilde{\mathbb{A}}}
\newcommand{\tA}[2]{\widetilde{\mathbb{A}}_{{#1},{#2}}}
\newcommand{\AAt}{\Tilde{\Tilde{\mathbb{A}}}}
\newcommand{\B}{\widetilde{\mathbb{B}}}
\newcommand{\tB}[2]{\widetilde{\mathbb{B}}_{{#1},{#2}}}
\newcommand{\BB}{\Tilde{\Tilde{\mathbb{B}}}}
\newcommand{\C}{\widetilde{\mathbb{C}}}
\newcommand{\CC}[2]{\widetilde{\mathbb{C}}_{{#1},{#2}}}
\newcommand{\tC}[2]{\widetilde{\mathbb{C}}_{{#1},{#2}}}
\newcommand{\ska}[2]{{k}_{{#1},{#2}}^{\mathbb{A}}}
\newcommand{\tska}[2]{{\tilde{k}}_{{#1},{#2}}^{\mathbb{A}}}
\newcommand{\ttska}[2]{\tilde{\tilde{k}}_{{#1},{#2}}^{\mathbb{A}}}
\newcommand{\ka}[2]{{K}_{{#1},{#2}}^{\mathbb{A}}}
\newcommand{\KA}[2]{\widetilde{K}_{{#1},{#2}}^{\mathbb{A}}}
\newcommand{\KAA}[2]{\tilde{\tilde{K}}^{\mathbb{A}}_{{#1},{#2}}}
\newcommand{\skb}[2]{{k}_{{#1},{#2}}^{\mathbb{B}}}
\newcommand{\KB}[2]{\widetilde{K}_{{#1},{#2}}^{\mathbb{B}}}
\newcommand{\KBB}[2]
{\Tilde{\Tilde{K}}_{{#1},{#2}}^{\mathbb{B}}}
\newcommand{\KC}[2]{\widetilde{K}_{{#1},{#2}}^{\mathbb{C}}}
\newcommand{\ntT}[2]{{\mathbb{T}}_{{#1}}}
\newcommand{\T}[2]{\widetilde{\mathbb{T}}_{{#1}}}

\newcommand{\nulambda}{(\lambda,\nu)}
\newcommand{\Nbar}{N_-}
\newcommand{\nbar}{{\mathfrak{n}}_-}
\newcommand{\Pbar}{P_-}
\usepackage[dvips,hyperindex]{hyperref}\hypersetup{linkbordercolor=0.2 1 0.5}
\begin{document}
\title
{Symmetry breaking for representations of rank one orthogonal 
groups}
\author{
Toshiyuki Kobayashi
\vspace{0.3mm}
 and Birgit Speh
}
%\address{
%Toshiyuki Kobayashi, 
%Kavli IPMU (WPI), Graduate School of Mathematical Sciences, the University of T%okyo, 3-8-1 Komaba, Tokyo, 153-8914 Japan\\
%\normalsize
%\textit{E-mail address}:
%\texttt{toshi@ms.u-tokyo.ac.jp}
%\\
%Birgit Speh, 
%Department of Mathematics, Cornell University, Ithaca, 14853-4201, NY, USA\\
%\normalsize
%\textit{E-mail address}:
%\texttt{bes12@cornell.edu}
%}
%\endaddress

\maketitle%
\begin{abstract}
We give a complete classification
 of intertwining operators
 ({\it{symmetry breaking operators}})
 between spherical  principal series representations
 of $G=O(n+1,1)$ and $G'=O(n,1)$.  
We construct three meromorphic families
 of the symmetry breaking operators,
 and find their distribution kernels
 and their residues
 at all poles explicitly.  
Symmetry breaking operators
 at exceptional discrete parameters
 are thoroughly studied.

We obtain closed formulae for the functional equations which the composition of the 
  the symmetry breaking operators
with the Knapp--Stein intertwining operators
 of $G$ and $G'$ satisfy, 
 and use them to determine
 the symmetry breaking operators 
 between irreducible composition factors of the spherical principal series representations
 of $G$ and $G'$. 
Some applications
 are included.  
\end{abstract}



\tableofcontents

\section{Introduction} 
\label{sec:intro}
A representation $\pi$ of a group $G$
 defines a representation
 of a subgroup $G'$
 by restriction.  
In general irreducibility 
 is not preserved by the restriction.  
If $G$ is compact
 then the restriction $\pi|_{G'}$ is isomorphic
 to a direct sum of irreducible representations $\pi'$ of $G'$ with multiplicities  $m(\pi,\pi')$. 
These multiplicities are studied by using combinatorial techniques.  
If $G'$ is not compact
 and the representation $\pi$ is infinite-dimensional, 
 then generically the restriction $\pi|_{G'}$
 is not a direct sum of irreducible representations
 \cite{inv98}
 and we have to consider another notion of multiplicity.

For a continuous representation $\pi$ of $G$
 on a complete, 
 locally convex topological vector space $H_\pi$,
 the space $H_\pi^\infty $ of $C^\infty$-vectors of $H_\pi$
 is naturally endowed with a Fr{\'e}chet topology,
 and $(\pi,H_\pi)$ gives rise
 to a continuous representation $\pi^{\infty}$ of $G$
 on $H_\pi^\infty$.  
If ${\mathcal{H}}_{\pi}$ is a Banach space, 
 then the Fr{\'e}chet representation 
 $(\pi^{\infty}, {\mathcal{H}}_{\pi}^{\infty})$
 depends only 
 on the underlying $({\mathfrak {g}}, K)$-module
 $({\mathcal{H}}_{\pi})_K$.  
Given another continuous representation
 $\pi'$ of the subgroup $G'$, 
 we consider the space of continuous $G'$-intertwining operators 
 ({\it{symmetry breaking operators}})
\[ \operatorname{Hom}_{G'} ({\pi}^\infty|_{G'}, ({\pi'})^\infty) .\] 
The dimension $m(\pi,\pi') $ of this space 
 yields important information of the restriction of $\pi $ to $G'$
 and is called the {\it{multiplicity}} of $\pi'$
 occurring in the restriction $\pi|_{G'}$.  
Notice
 that the multiplicity $m(\pi, \pi')$ makes sense
 for non-unitary representations $\pi$ and $\pi'$, 
 too.  
In general,
 $m(\pi,\pi')$ may be infinite.  
For detailed analysis 
 on symmetry breaking operators,
 we are interested in the case
 where $m(\pi,\pi')$ is finite.  
The criterion in \cite{K-O} asserts that the
 multiplicity $m(\pi,\pi')$ is finite
 for all irreducible representations $\pi$ of $G$
 and all irreducible representations $\pi'$ of $G'$
 if and only if the minimal parabolic subgroup $P'$
 of $G'$ has an open orbit 
 on the real flag variety $G/P$, 
 and that the multiplicity is uniformly bounded
 with respect to $\pi$ and $\pi'$
 if and only if a Borel subgroup of $G_{\mathbb{C}}'$
 has an open orbit
 on the complex flag variety of $G_{\mathbb{C}}$.

The classification
 of reductive symmetric pairs
 $({\mathfrak {g}}, {\mathfrak {g}}')$
 satisfying the former condition was recently accomplished
 in \cite{xKMt}.  
On the other hand, 
 the latter condition
 depends only
 on the complexified pairs
 $({\mathfrak {g}}_{\mathbb{C}}, {\mathfrak {g}}_{\mathbb{C}}')$, 
 for which the classification 
 is much simpler
 and was already known in 1970s
 by Kr{\"a}mer 
 \cite{Kr1}.  
In particular,
 the multiplicity $m(\pi,\pi')$ is uniformly bounded
 if the Lie algebras $({\mathfrak {g}}, {\mathfrak {g}}')$ of $(G,G')$ 
 are real forms
 of $({\mathfrak {sl}}(N+1,{\mathbb{C}}),
 {\mathfrak {gl}}(N,{\mathbb{C}}))$
 or $({\mathfrak {o}}(N+1,{\mathbb{C}}),
 {\mathfrak {o}}(N,{\mathbb{C}}))$.

\vskip 1pc
In this article
 we confine ourselves
 to the case 
\begin{equation}
\label{eqn:GG}
     (G,G')=(O(n+1,1),O(n,1)), 
\end{equation}
 and study thoroughly symmetry breaking operators
 between spherical principal series representations
 for the groups $G$
 and $G'$.  
In particular,
 we determine
 the multiplicities for their composition factors.  
Furthermore,
 we give a classification of symmetry breaking operators 
 $I(\lambda)^{\infty} \to J(\nu)^{\infty}$
 for any spherical principal series
 representations $I(\lambda)$ and $J(\nu)$, 
 and find explicit formulae
 of distribution kernels
 of its basis
 for every $(\lambda, \nu)\in {\mathbb{C}}^2$.

\vskip 1pc
The techniques
 of this article
 are actually directed 
 at the more general problems
 of determining symmetry breaking operators
 between (degenerate) principal series representations
 induced from parabolic subgroups $P$ of $G$
 and $P'$ of $G'$
 under the geometric assumption
 that the double coset $P' \backslash G/P$
 is a finite set.  
In the setting \eqref{eqn:GG},
 there are three (nonempty) closed $P'$-invariant 
 subsets in $G/P$.  
Correspondingly,
 we construct a family
 of (generically) {\it{regular}} symmetry
 breaking operators $\A_{\lambda,\nu}$
 and two families of {\it{singular}} symmetry
 ones $\B_{\lambda,\nu}$ and $\C_{\lambda,\nu}$.

The classification of symmetry breaking operators $T$
 is carried out through
 an analysis of their distribution kernels
 $K_T$.  
We consider the system of partial differential equations
 that $K_T$ satisfies,
 and determine 
 when an (obvious) local solution along a $P'$-orbit
 extends to a global solution
 on the whole real flag variety $G/P$.  
The important properties of these symmetry breaking operators 
 are the existence
 of the meromorphic continuation, 
 and the functional equations
 that they satisfy with the Knapp--Stein intertwining operators
 of $G$ and $G'$.  
The residue calculus of $\tA{\lambda}{\nu}$
 provides a third method to obtain
 Juhl's conformally covariant operators $\C_{\lambda,\nu}$
 for the embedding $S^{n-1} \hookrightarrow S^n$
 (see \cite{Juhl}, \cite{K13} for the two earlier proofs, 
 and \cite{xkeastwood60}
 for a heuristic argument
 for the method of this article).

\bigskip

To state our results more precisely,
 we realize $G=O(n+1,1)$ as the automorphism group
 of a quadratic form 
\[
   x^2_0+x^2_1 +\cdots +x^2_{n} -x^2_{n+1} 
\] 
and the subgroup $G'=O(n,1)$ as the stabilizer of the basis vector $e_{n}={}^{t\!}(0,\cdots,0,1,0)$. 

A spherical principal series representation
\index{sbon}{Ilmd@${I(\lambda)}$|textbf}
 $I(\lambda)$ of $G$
 is an (unnormalized) induced representation from a character $\chi_\lambda$
 of a minimal parabolic subgroup $P$
 for $\lambda \in {\mathbb{C}}$. 
In what follows,
 we take the representation space
 of $I(\lambda)$ 
 to be the space
 of $C^{\infty}$-sections
 of the $G$-equivariant line bundle
 $G \times_P (\chi_{\lambda}, {\mathbb{C}}) \to G/P$, 
so that 
 $I(\lambda)^{\infty} \simeq I(\lambda)$
 is the Fr{\'e}chet globalization
 having moderate growth
 in the sense 
 of Casselman--Wallach \cite{W}.  
See Section \ref{subsec:CW}.  
The parametrization is chosen
 so that $I(\lambda)$ is reducible
 if and only if $-\lambda \in {\mathbb{N}}$
 or $\lambda-n \in {\mathbb{N}}$, 
 and that $I(-i)$ ($i \in {\mathbb{N}}$)
 contains a finite-dimensional representation 
 $
\index{sbon}{Fi@$F(i)$|textbf}
F(i)$ as the unique subrepresentation,
 which is isomorphic to the representation
 on the space ${\mathcal{H}}^i({\mathbb{R}}^{n+1,1})$
 of spherical harmonics of degree $i$
 as a representation
 of the identity component $G_0$
 of $G$, 
 see \eqref{eqn:FHi}.  
The irreducible Fr{\'e}chet representation $I(-i)/F(i)$
 of $G$ is denoted by 
$
\index{sbon}{Ti@$T(i)$|textbf}
T(i)$.  
The underlying $({\mathfrak {g}},K)$-module
 $T(i)_K$ is isomorphic to
 a Zuckerman $A_{\mathfrak {q}}(\lambda)$-modules
 where ${\mathfrak {q}}$ is a certain maximal $\theta$-stable
 parabolic subalgebra of ${\mathfrak {g}}_{\mathbb{C}}$
 (see Section \ref{subsec:16.3}).  
In this parametrization,
 $I(\lambda)$ is unitarizable
 if $\lambda \in \frac n 2 + \sqrt{-1}{\mathbb{R}}$
 ({\it{unitary principal series representation}})
 or $0 < \lambda < n$
 ({\it{complementary series representations}}, 
 see Chapter \ref{sec:applbl}).  

Similarly,
 spherical  principal series representations of the subgroup $G'$
 are denoted by 
\index{sbon}{Jnu@${J(\nu)}$|textbf}
$J(\nu)$ and are parametrized
 so that the finite-dimensional representations 
$
\index{sbon}{Fj@$F(j)$|textbf}
F(j)$
 is a subrepresentation of $J(-j)$. 
The irreducible Fr{\'e}chet representation $J(-j)/F(j)$
 of $G'$ is denoted
 by 
$
\index{sbon}{Tj@$T(j)$|textbf}
T(j)$.

\medskip

Consider pairs of nonpositive
 integers  and define 
\index{sbon}{Leven@$L_{\operatorname{even}}$|textbf}
\begin{align*}
L_{\operatorname{even}}:=&\left \{ (-i,-j):
j\leq i \mbox{ and } i\equiv j \mod 2 \right \},
\\
\index{sbon}{Lodd@$L_{\operatorname{odd}}$|textbf}
L_{\operatorname{odd}}:=&\left \{ (-i,-j)
: j\leq i \mbox{ and } i \equiv j +1 \mod 2 \right \}
.
\end{align*}

\medskip
The discrete set $L_{\operatorname{even}}$
 in ${\mathbb{C}}^2$ plays
 a special role
 throughout the article.  
We prove:
\begin{theorem}
[multiplicities for spherical principal series, 
Theorem \ref{thm:dimHom}]
\label{thm:I.1}
We have
\[
m (I(\lambda), J(\nu)) 
= 
\begin{cases}
1
\qquad
&\text{if $(\lambda,\nu) \in {\mathbb{C}}^2 \setminus L_{\operatorname{even}}$, }
\\
2
\qquad
&\text{if $(\lambda,\nu) \in L_{\operatorname{even}}$.}
\end{cases}
\]
\end{theorem}

\medskip

This theorem is new even
 for $G=O(3,1) \approx PGL(2,\mathbb{C})$
 and $G'=O(2,1) \approx PGL(2,\mathbb{R})$.

{}From the viewpoint of differential geometry,
 $G$ is the conformal group of the standard sphere
 $S^n$, 
 and conformally equivariant line bundles
 ${\mathcal{L}}_{\lambda}$ 
 over $S^n$
 are parametrized by $\lambda \in {\mathbb{C}}$
 (we normalize ${\mathcal{L}}_{\lambda}$ 
 such that ${\mathcal{L}}_{0}$ is the trivial line bundle
 and ${\mathcal{L}}_n$ is the bundle 
 of volume densities).  
The subgroup $G'$ is the conformal group 
 of the \lq{great circle}\rq\
 $S^{n-1}$ in $S^n$, 
 and conformally equivariant line bundles
 ${\mathcal{L}}_{\nu}$
 over $S^{n-1}$
 are parametrized by $\nu \in {\mathbb{C}}$.  
Then Theorem \ref{thm:I.1} determines
 the dimension
 of conformally covariant linear maps
 ({\it{i.e.}},  $G'$-equivariant operators)
{} from $C^{\infty}(S^n,{\mathcal{L}}_{\lambda})$
 to $C^{\infty}(S^{n-1},{\mathcal{L}}_{\nu})$.

{}From the representation theoretic viewpoint,
 it was proved recently in Sun and Zhu \cite{S-Z}
 that $m(\pi ,\pi' ) \leq 1$
 for all irreducible admissible representations
 $\pi$ of $G$
 and $\pi'$ of $G'$. 
However,
 it is much more involved 
 to tell whether 
$m (\pi,\pi')=0$
 or $1$
 for given irreducible representations $\pi$ and $\pi'$.

\vskip 1pc
The following theorem determines
 $m(\pi,\pi')$
 for irreducible subquotients
 at reducible points.

\medskip
\begin{theorem} 
[multiplicities for composition factors, 
Theorem \ref{thm:compo}]
\label{thm:compointro}
Let $i,j \in {\mathbb{N}}$.  
\begin{enumerate}
\item[{\rm{(1)}}]
Suppose that $i \ge j$.  

\par\noindent
{\rm{(1-a)}} \enspace
Assume $i \equiv j \mod 2$, 
namely,
$(-i,-j)  \in L_{\operatorname{even}}$. 
Then 
\[  
   m (T(i) ,T(j) ) = 1, 
   \qquad
   m(T(i), F(j) )= 0, 
   \qquad
   m (F(i) ,F(j) )  = 1.  
\]
\par\noindent
{\rm{(1-b)}} \enspace
Assume $i \equiv j+1 \mod 2$, 
namely,
$(-i,-j) \in L_{\operatorname{odd}}$.   
Then
\[  m(T(i) ,T(j) ) = 0, 
\qquad
    m(T(i) ,F(j) ) = 1, 
\qquad
    m (F(i) ,F(j) )  = 0.\]
\item[{\rm{(2)}}]
 Suppose that $i<j$.  
Then 
\[
m(T(i),T(j))=0, 
\quad
m(T(i),F(j))=1, 
\quad
m(F(i),F(j))=0.  
\]
\end{enumerate}
\end{theorem}

Similar results were obtained by  Loke \cite{Lo} for the $(\mathfrak{g},K)$-modules of representations of  $G=GL(2,{\mathbb{C}})$
 and $G'=GL(2,{\mathbb{R}})$.
 
 \medskip
 
We also determine the multiplicity of (possibly, reducible)
 spherical principal series representations
 $I(\lambda)$ of $G$
 and irreducible finite-dimensional representations
 $F(j)$, respectively infinite-dimensional ones $T(j)$
 of the subgroup $G'$ 
 in Theorem \ref{thm:mIF}:
\begin{theorem}
\label{thm:mIFintro}
Suppose $j \in {\mathbb{N}}$. 
\begin{enumerate}
\item[{\rm{1)}}]
$
m(I(\lambda),F(j))
=1
\quad\quad
\text{for all }\,\, \lambda \in {\mathbb{C}}.  
$
\item[{\rm{2)}}]
$m(I(\lambda),T(j))
=
\begin{cases}
1
\quad&\text{if }\,\, \lambda+j \in -2{\mathbb{N}},
\\
0
\quad&\text{if }\,\, \lambda+j \not\in -2{\mathbb{N}}.
\end{cases}
$
\end{enumerate}
\end{theorem}
\vskip 2pc
In the special case $\nu=0$, 
 our results 
 on symmetry breaking operators are closely related to the analysis on
the indefinite hyperbolic space
\[
X(n+1,1) :=
\{ \xi \in \mathbb{R}^{n+2}: 
   \xi_0^2 + \dots + \xi_{n}^2 - \xi_{n+1}^2 = 1 \}
 \simeq
G / G'.
\]
As a hypersurface of the Minkowski space
$$\mathbb{R}^{n+1,1} \equiv
(\mathbb{R}^{n+2}, d\xi_0^2 + \dots + d \xi_{n}^2 - d \xi_{n+1}^2),
$$
$X(n+1,1)$ carries a Lorentz metric for which the sectional curvature 
is constant $-1$,
and thus is a model space of anti-de Sitter manifolds.
The Laplacian $\Delta$ of the Lorentz manifold
 $X(n+1,1)$
 is a hyperbolic operator,
 and for $\lambda \in \mathbb{C}$,
we consider its eigenspace:
\[
\mathcal{S}ol (G/G'; \lambda)
:= \{ f \in \mathcal{C}^\infty (G/G') : \Delta f = -\lambda(\lambda-n) f \}.
\]
The underlying $(\mathfrak{g},K)$-module
 $\mathcal{S}ol (G/G'; \lambda)_K$
 is isomorphic to the underlying $(\mathfrak{g},K)$-module
 of a principal series representation
 \cite{Sch}.  
For $\lambda= -i \in -\mathbb{N}$
 there are two inequivalent reducible principal series representations
 $I(-i)_K$ and $I(n+i)_K$, 
 and our results on the symmetry breaking operators
 $\A_{\lambda,0}$
 for $(\lambda,0) \in L_{\operatorname{even}}$
 give another proof
 of the following $(\mathfrak{g},K)$-isomorphism:
\[
     \mathcal{S}ol (G/G'; \lambda)_K
     \simeq
     \begin{cases}
     I(-i)_K \qquad &\text{if } \lambda =-i \in -2 {\mathbb{N}}, 
     \\
     I(n+i)_K \qquad &\text{if } \lambda =-i \in -2 {\mathbb{N}}-1.  
     \end{cases}
\]

More generally,
 we apply our results
 on symmetry breaking operators
 for $\nu \in -{\mathbb{N}}$
 to the analysis 
 on vector bundles.  
We note that harmonic analysis
 on (general) semisimple symmetric spaces
 has been studied actively by many people
 during the last fifty years,
however,
 not much has been known for vector bundle sections.  
We construct
 in Theorem \ref{thm:lise}
 some irreducible subrepresentations
 in the space of sections
 of the $G$-equivariant vector bundles
 $G \times_{G'} F(j) \to G/G'$
 by using symmetry breaking operators
 $\A_{\lambda,\nu}$.

\bigskip
We also obtain branching laws for unitary complementary series representations  of $I(\lambda)$ $(0< \lambda < n)$, which by abuse of notation we also denote by $I(\lambda)$.
For $\lambda \in {\mathbb{R}}$, 
 we set 
\index{sbon}{Dlmd@$D(\lambda)$}
\[
D(\lambda)
:=
\{\nu\in \lambda-1 + 2 {\mathbb{Z}}:
\frac {n-1}{2}< \nu \le \lambda-1\}.  
\]
Then $D(\lambda)$
 is a finite set,
 and $D(\lambda)$ is non-empty
 if and only if $\lambda> \frac{n+1}{2}$.  
As an application
of differential symmetry breaking operators
 $\C_{\lambda,\nu}$, 
 we have
\begin{theorem}
[branching law of complementary series,
Theorem \ref{thm:rest}]
\label{thm:branch}
Suppose that $\frac {n+1} 2 < \lambda < n$.  
Then $J(\nu)$ is a complementary series representation
 of the subgroup $G'$
 for any $\nu \in D(\lambda)$.  
Moreover,
 the restriction of $I(\lambda)$ to $G'$
 contains the finite sum $\bigoplus_{\nu \in D(\lambda)} J(\nu)$
 as discrete summands.  
\end{theorem}

\bigskip

About 20 years ago Gross and Prasad \cite{GP}
 formulated a conjecture about  the restriction of an irreducible admissible tempered representation
  of an inner form $G$ of the group $O(n)$
 over a local field to a subgroup  $G''$ which is an inner form $G'=O(n-1)$. 
The conjecture in \cite{GP}
 relates the existence of nontrivial homomorphisms
 to the value of an L-function at 1/2 and the value of the epsilon factor. We expect to come back to this in a later paper.

\bigskip
Let us enter the proof
 of Theorem \ref{thm:I.1}
 and its refinement 
(Theorem \ref{thm:II.5} below)
 in a little more details.  
We first construct an analytic family
 of (generically) regular symmetry breaking operators
 and show

\begin{theorem} 
[regular symmetry breaking operators, 
Theorem \ref{thm:poleA}]
\label{thm:II.1}
There exists a family 
 of symmetry breaking operators
$
\index{sbon}{Alnt@$\A_{\lambda,\nu}$|textbf}
   \A_{\lambda,\nu} 
  \in  \operatorname{Hom}_{G'} (I(\lambda ), J(\nu))
$ 
 that depends holomorphically
 for entire $(\lambda, \nu) \in {\mathbb{C}}^2$
 with the distribution kernel
\[
\index{sbon}{Kxt@$\KA{\lambda}{\nu} (x, x_n)$|textbf}
\KA{\lambda}{\nu} (x,x_n)
:= \frac{1}{\Gamma(\frac{\lambda+\nu-n+1}{2})\Gamma(\frac{\lambda-\nu}{2})}
   |x_n|^{\lambda+\nu-n} (|x|^2 + x_n^2)^{-\nu}.
\]
Further,
 $\widetilde{\mathbb{A}}_{\lambda,\nu}$ is nonzero
 if and only if $(\lambda, \nu) \in {\mathbb{C}}^2 \setminus L_{\operatorname{even}}$.  
\end{theorem} 

\medskip

We recall that there exist nonzero Knapp--Stein intertwining operators 
\index{sbon}{ttt4@$\T{\nu}{m-\nu}$}
\begin{center}
$\T{\nu}{n-1-\nu}:J(\nu) \rightarrow J(n-1-\nu)  $
 \  \  and  \ \ 
\index{sbon}{ttt3@$\T{\lambda}{n-\lambda}$|textbf}
$
\T{\lambda}{n-\lambda}: I(\lambda) \rightarrow I(n-\lambda), 
$
\end{center}
with holomorphic parameters $\nu \in {\mathbb C}$
and $\lambda \in \mathbb C$, respectively.
In our normalization 
\index{sbon}{Jnu@${J(\nu)}$}
\begin{alignat*}{2}
\T{n-1-\nu}{\nu} \circ \T{\nu}{n-1-\nu}
=& \frac{\pi^{n-1}}{\Gamma(n-1-\nu)\Gamma(\nu)} \operatorname{id}
\quad
&&\text{on } J(\nu),
\\
\intertext{and }
\T{n-\lambda}{\lambda}
 \circ \T\lambda{n-\lambda}
=& \frac{\pi^{n}}{\Gamma(n-\lambda)\Gamma(\lambda)} \operatorname{id}
\quad
&&\text{on }I(\lambda).  
\end{alignat*}

The following functional identities are crucial
 in the proof of Theorems \ref{thm:II.1}
 and \ref{thm:image}.

\medskip
\begin{theorem} 
[functional identities, 
Theorem \ref{thm:TAAT}]
\label{thm:II.2}
For all $(\lambda, \nu) \in {\mathbb{C}}^2$,  
\begin{align}
 \T{n-1 -\nu}{\nu} 
 \circ 
 \widetilde{\mathbb A}_{\lambda,n-1-\nu}
 =&
 \frac{\pi^{\frac{n-1}{2}}}{\Gamma(n-1-\nu)}
   \widetilde{\mathbb A}_{\lambda,\nu}, 
\label{eqn:TAA0}
\\
 \widetilde{\mathbb A}_{n-\lambda,\nu}
 \circ \T{\lambda}{n-\lambda} 
 =& \frac{\pi^{\frac{n}{2}}}{\Gamma(n-\lambda)} \widetilde{\mathbb A}_{\lambda,\nu}.
\label{eqn:ATA0}
\end{align}

Here $\T{n-1-\nu}{\nu}$
 and $\T{\lambda}{n-\lambda}$
 are the Knapp--Stein intertwining operators
 of $G'$ and $G$, 
 respectively.  
If $\nu -n+1 \in {\mathbb{N}}$
 or $\lambda-n \in {\mathbb{N}}$, 
 then the left-hand side
 of \eqref{eqn:TAA0} or \eqref{eqn:ATA0}
 is zero,  
respectively.  
\end{theorem}

The functional identities 
 in Theorem \ref{thm:II.2}
 are extended
 to other families of singular breaking symmetry operators
 (see Theorem \ref{thm:TABC}, 
 Corollary \ref{cor:DCC}, 
 and Corollary \ref{cor:10.6}).  
\medskip

We construct other families of symmetry breaking operators
 as follows:
We define 
\index{sbon}{XX@${{\mathbb{X}}}$|textbf}
\index{sbon}{Xr@${/\!/}$|textbf}
\index{sbon}{Xl@${\backslash\!\backslash}$|textbf}
\begin{align*}
// :=& \{ (\lambda,\nu)\in {\mathbb{C}}^2 : \lambda -\nu = 0,-2,-4,\dots  \}, 
\\
\backslash \backslash :=&\{ (\lambda,\nu) \in {\mathbb{C}}^2: \lambda +\nu = n-1,n-3, n-5 \dots\}, 
\\
\mathbb{X} :=& \backslash \backslash \cap //.  
\end{align*}
We note
\begin{equation}
\label{eqn:LX}
L_{\operatorname{even}}
\subset 
\begin{cases}
{\mathbb{X}}
&\text{if $n$ is odd,}
\\
/\!/ \setminus \backslash\!\backslash
&\text{if $n$ is even.}
\end{cases}
\end{equation}

\medskip

For $\nu \in -{\mathbb{N}}$, 
 the renormalized operator 
\index{sbon}{Att@$\AAt_{\lambda,\nu}$|textbf}
 $\AAt_{\lambda,\nu}
 :=\Gamma(\frac{\lambda-\nu}2) {\widetilde{\mathbb A}}_{\lambda,\nu}$
 extends to a non-zero $G'$-intertwining map, 
 that depends holomorphically on $\lambda \in {\mathbb{C}}$.

For $(\lambda,\nu) \in \backslash\!\backslash$,
we define a family
 of {\it{singular}} $G'$-intertwining operators
\index{sbon}{Bt@$\B_{\lambda,\nu}$|textbf}
$
\widetilde{{\mathbb{B}}}_{\lambda,\nu} : I(\lambda) \to J(\nu)
$
that depends holomorphically on
$\lambda \in \mathbb{C}$ (or on $\nu \in \mathbb{C}$)
by the distribution kernel
\index{sbon}{KxtB@$\KB{\lambda}{\nu}(x,x_n)$|textbf}
\[
\KB{\lambda}{\nu} (x,x_n)
:= \frac{1}{\Gamma(\frac{\lambda-\nu}{2})}
   (|x|^2 + x_n^2)^{-\nu} \delta^{(2k)} (x_n).
\]

\medskip

For $(\lambda, \nu) \in //$, 
we set $l:=\frac 1 2 (\nu-\lambda)$
 and define a differential operator
\index{sbon}{Ct@$\C_{\lambda,\nu}$|textbf}
\[
\widetilde{\mathbb{C}}_{\lambda,\nu}
= 
\operatorname{rest}\circ\sum_{j=0}^l \frac{2^{2l-2j}}{j!(2l-2j)!}
   \prod_{i=1}^{l-j} \Bigl(\frac{\lambda+\nu-n-1}{2}+i\Bigr)
   \Delta_{\mathbb{R}^{n-1}}^j \Bigl(\frac{\partial}{\partial x_n}\Bigr)^{2l-2j}.
\]
Here $\operatorname{rest}$ denotes the restriction
to the hyperplane $x_n = 0$. 
It gives a differential symmetry breaking operator
$
   \widetilde{{\mathbb{C}}}_{\lambda, \nu} : I(\lambda) \rightarrow J(\nu)
$
of order $2l$, 
and coincides with the conformally covariant differential operator
 for the embedding $S^{n-1} \hookrightarrow S^n$, 
 which was discovered recently by A.~Juhl in \cite{Juhl}.

\medskip 
Using the support of the operators, 
 we prove the following refinement of Theorem \ref{thm:I.1}.  
We show:

\medskip
\begin{proposition}
\label{prop:II.3}
Every operator in $\operatorname{Hom}_{G'} (I(\lambda), J(\nu))$
 is in the ${\mathbb{C}}$-span of the operators $\widetilde{\mathbb A}_{\lambda,\nu}$,  $\widetilde{\widetilde{\mathbb A}}_{\lambda,\nu}$, $\widetilde{{\mathbb{B}}}_{\lambda,\nu} $ and 
 $ \widetilde{{\mathbb{C}}}_{\lambda, \nu}$.
\end{proposition}

\medskip

Examining the linear independence
 of symmetry breaking operators constructed above 
we prove 

\medskip

\begin{theorem} 
[residue formulae, 
Theorem \ref{thm:reduction}]
\label{thm:II.4}
\begin{enumerate}
\renewcommand{\theenumi}{\arabic{enumi}}
\renewcommand{\labelenumi}{\upshape (\theenumi)}
 \  \\
\item
For $(\lambda,\nu) \in \backslash\!\backslash \setminus \mathbb{X}$,
we define $k:=\frac 1 2(n-1-\lambda-\nu) \in \mathbb{N}$.  
Then
\[
\widetilde{{\mathbb{A}}}_{\lambda,\nu}
= \frac{(-1)^k}{2^k(2k-1)!!}
   \widetilde{{\mathbb{B}}}_{\lambda,\nu}.
\]
\item
For
$(\nu,\lambda) \in /\!/$, 
 we define $l:=\frac 1 2(\nu-\lambda)$.  
Then 
\[
\widetilde{{\mathbb{A}}}_{\lambda,\nu}
= \frac{(-1)^{l} l! \pi^{\frac{n-1}{2}}}
          {\Gamma(\nu)2^{2l}}
   \widetilde{{\mathbb{C}}}_{\lambda,\nu}.
\]

\item
Suppose $(\lambda,\nu)\in\mathbb{X}$.
We define
$k,l \in \mathbb{N}$ as above.  
Then 
\[
\widetilde{{\mathbb{B}}}_{\lambda,\nu}
= \frac{(-1)^{l-k} 2^{k-2l}\pi^{\frac{n-1}{2}} l! (2k-1)!!}
          {\Gamma(\nu)}
   \widetilde{{\mathbb{C}}}_{\lambda,\nu}.
\]
\end{enumerate}
\end{theorem}

\bigskip
Theorem \ref{thm:II.4} implies 
 that singular symmetry breaking operators
 $\B_{\lambda,\nu}$ and $\C_{\lambda,\nu}$
 can be obtained
 as the residues
 of the meromorphic family 
 of (generically) regular symmetry breaking operators
 in most cases.  
An exception happens for the differential symmetry 
 breaking operator
 $\C_{\lambda,\nu}$
 for $(\lambda,\nu) \in L_{\operatorname{even}}$
 (see also Remark \ref{rem:CLeven}).  
In fact the dimension of 
 $\operatorname{Hom}_{G'} (I(\lambda), J(\nu))$
 jumps at $(\lambda,\nu) \in L_{\operatorname{even}}$
 as we have seen in Theorem \ref{thm:I.1}.

We prove a stronger form 
 of Theorem \ref{thm:I.1}
 by giving an explicit basis
 of symmetry breaking operators: 
\begin{theorem}
[explicit basis, Theorem \ref{thm:9.5}]
\label{thm:II.5}
For $(\lambda, \nu) \in {\mathbb{C}}^2$, 
 we have
\[
\operatorname{Hom}_{G'}(I(\lambda),J(\nu))
   = \begin{cases}
           \mathbb{C} \widetilde{\widetilde{{\mathbb{A}}}}_{\lambda,\nu} 
  \oplus \mathbb{C} \tilde{{\mathbb{C}}}_{\lambda,\nu}
                &\text{if\/ $(\lambda,\nu) \in L_{\operatorname{even}}$},
     \\
           \mathbb{C} \widetilde{{\mathbb{A}}}_{\lambda,\nu}
                &\text{if\/ $(\lambda,\nu) \in \mathbb{C}^2\setminus L_{\operatorname{even}}$}.
     \end{cases}
\]
\end{theorem}

\bigskip
Denote by 
\index{sbon}{1lmd@${\mathbf{1}}_{\lambda}$|textbf}
$\mathbf{1}_\lambda$
 and
\index{sbon}{1nu@${\mathbf{1}}_{\nu}$|textbf}
 $\mathbf{1}_\nu$
 the normalized spherical vectors in $I(\lambda)$ and  $J(\nu)$,
 respectively.  
The image of spherical vector $\mathbf{1}_{\lambda}$
 under the symmetry breaking operators
$\widetilde{\mathbb A}_{\lambda,\nu}$ and $\widetilde{\mathbb B}_{\lambda,\nu}$ is nonzero
 if and only if  $\lambda \not = 0,-1,-2,-3\dots $,  whereas it is always nonzero under  $\widetilde{\mathbb C}_{\lambda,\nu}$. 
More  precisely
 we prove in Propositions \ref{prop:AminK}, 
 \ref{prop:B1}, 
 and \ref{prop:C11}
 the following:

\begin{theorem}
[transformations of spherical vectors] 
\label{thm:I.4}
  
\begin{enumerate}
\renewcommand{\theenumi}{\arabic{enumi}}
\renewcommand{\labelenumi}{\upshape (\theenumi)}
 \ \\
\item
For $(\lambda, \nu) \in \mathbb{C}^2$, 
\[
\widetilde{{\mathbb A}}_{\lambda,\nu} (\mathbf{1}_\lambda)
= \frac{\pi^{\frac{n-1}{2}}}{\Gamma(\lambda)}
   \mathbf{1}_\nu.  
\]
\item
For $(\lambda, \nu) \in \backslash\!\backslash$,
we set $k:= \frac12 (n-1-\lambda-\nu)$.
Then
\[
\widetilde{{\mathbb{B}}}_{\lambda,\nu} (\mathbf{1}_\lambda)
= \frac{(-1)^k 2^k \pi^{\frac{n-1}{2}} (2k-1)!!}{\Gamma(\lambda)}
   \mathbf{1}_\nu.
\]
\item
For $(\lambda, \nu) \in /\!/$,
we set $l:= \frac12 (\nu-\lambda) \in \mathbb{N}$.
Then
\[
\widetilde{{\mathbb{C}}}_{\lambda,\nu}(\mathbf{1}_\lambda)
= \frac{(-1)^l 2^{2l} (\lambda)_{2l}}{l!} \mathbf{1}_\nu.
\]

\end{enumerate}

\end{theorem}

\bigskip

We also determine
 the image of the underlying $({\mathfrak {g}}, K)$-module
 $I(\lambda)_K$ of $I(\lambda)$
 by the symmetry breaking operators 
 for all the parameters
 $(\lambda, \nu) \in {\mathbb{C}}^2$.  
Using the basis in Theorem \ref{thm:II.5}, 
we have:

\medskip
\begin{theorem}
[image of breaking symmetry operator, 
 see Theorems \ref{thm:ImageF} and \ref{thm:ImageT}]
\label{thm:image}
\ \\
\noindent
{\rm{(1)}}\enspace
Suppose that $(\lambda,\nu)\in L_{\operatorname{even}}$  
and set $j:=-\nu \in {\mathbb{N}}$.  Then
\[
  \operatorname{Image} \AAt_{\lambda,\nu}
  =F(j) \] 
and 
 \[ \operatorname{Image}\C_{{\lambda},{\nu}}=J(\nu)_{K'}.\] 

\par\noindent
{\rm{(2)}}\enspace
Suppose that $(\lambda,\nu) \not\in L_{\operatorname{even}}$. Then
\begin{alignat*}{3}
&{\rm{(2\text{-}a)}}
\qquad
&&\operatorname{Image}
\tA{\lambda}{\nu}
=F(-\nu)
\quad
&&\text{if }\nu \in -{\mathbb{N}}, 
\\
&{\rm{(2\text{-}b)}}
&&\operatorname{Image}\tA{\lambda}{\nu}
=T(\nu+1-n)_{K'}
\quad
&&\text{if }(\lambda,\nu) \in \backslash\backslash
\quad
\text{and } \nu +1-n \in{\mathbb{N}},
\\
&{\rm{(2\text{-}c)}}
&&\operatorname{Image}\tA{\lambda}{\nu}
=J(\nu)_{K'}
\quad
&&\text{otherwise}.  
\end{alignat*}
\end{theorem}

\bigskip
The outline of the article is as follows:   

\medskip \noindent
Before we start with the construction of the intertwining operators between spherical principal series representations
 of $G=O(n+1,1)$ and $G'= O(n,1)$
 we prove in Chapter \ref{sec:2}
 our main results about $G'$-intertwining operators between irreducible composition factors
 of spherical principal series representations
 (Theorem \ref{thm:compointro}). 
In this proof we use the results  about symmetry breaking operators for  spherical principal series representations of $G$ and $G'$
(Theorem \ref{thm:II.5})
 and their functional equations (Theorem \ref{thm:II.2})
 proved later in the article.

Chapter \ref{sec:gen} gives a general method 
 to study symmetry breaking operators 
 for (smooth) induced representations
 by means of their distribution kernels. 
Analyzing their supports
 we obtain a natural filtration
 of the space of symmetry breaking operators
 induced from the closure relation
 on the double coset $P'\backslash G/P$
 in Section \ref{subsec:diffres}, 
 which will be used later to estimate the dimension of  $\operatorname{Hom}_{G'} (I(\lambda), J(\nu))$.

In Chapter \ref{sec:4} 
we give preliminary results 
 on spherical principal series representations
 such as explicit formulae for the realization
 in the noncompact picture
 $C^\infty( \mathbb R ^n)$ using the open Bruhat cell. 
Then we recall the Knapp--Stein intertwining operator 
$\ntT{\lambda}{n-\lambda}$,
 define a normalized operator 
$\T{\lambda}{n-\lambda}$, 
 and show some of its properties. 
Notice that our normalization arises from analytic considerations
 and is not the same as the normalization introduced by Knapp and Stein.

Chapter \ref{sec:double} is a discussion of the double coset decompositions $G'\backslash G/P$ and  $P'\backslash G/P$. We prove in particular that $G=P'N_-P$.

In Chapter \ref{sec:3}
 we derive a system of differential equations on $\mathbb R^n$
 and show in Proposition \ref{prop:HomPDE}
 that its distribution solutions $\mathcal{S}ol (\mathbb{R}^n; \lambda,\nu)$ are isomorphic to  $\operatorname{Hom}_{G'} 
(I(\lambda),J(\nu))$.
An analysis of the solutions shows that generically the multiplicity 
 $m(I(\lambda),J(\nu))$
 of principal series representations is 1
 (see Theorem \ref{thm:I.1}).

In Chapter \ref{sec:kfini} we use the distribution  $K^{\mathbb A}_{\lambda, \nu} \in \mathcal{S}ol (\mathbb{R}^n; \lambda,\nu)$
 to define for $\nulambda$ in an open region
\index{sbon}{Omega0@$\Omega_0$}
 $\Omega_0$
a $(\mathfrak{g}',K')$-homomorphism
$
\A_{\lambda,\nu} : I(\lambda)_K \to J(\nu)_{K'}$.  
Normalizing the distribution kernel
 by a Gamma factor
 we obtain an operator $\widetilde{\mathbb A}_{\lambda,\nu}$ 
and prove that $\A_{\lambda,\nu}(\varphi)$ is holomorphic in
$\nulambda \in \mathbb{C}^2$
for every $\varphi \in I(\lambda)_K$,
and that $\A_{\lambda,\nu}(\varphi) = 0$ for all   $\varphi \in I(\lambda)_K$ if and only if
$\nulambda \in L_{\operatorname{even}}$.

In Chapter \ref{sec:8}
 we prove the existence of the meromorphic continuation
 of ${\mathbb{A}}_{\lambda,\nu}$, 
 initially defined holomorphically
 on the parameter $\nulambda$ in the open region $\Omega_0$, 
 to $\nulambda$ in the entire $\mathbb{C}^2$.  
Besides, we 
determine all the poles of the symmetry breaking operator
$\ntAln{\lambda}{\nu}$
with meromorphic parameter $\lambda$ and $\nu$ 
and show that the  normalized symmetry breaking operators
 $\A_{{\lambda},{\nu}}:I(\lambda) \to  J(\nu)$
 depend holomorphically on ${{\lambda},{\nu}}$. 
Here we use and  prove the functional equations 
 (Theorem \ref{thm:II.2})
 of the symmetry breaking operators.

An analysis on the exceptional discrete set
 $L_{\operatorname{even}}$
 is particularly important.  
We introduce  for $\nu \in -{\mathbb N}$ a different normalization to obtain nonzero operators $\widetilde{\A}_{{\lambda},{\nu}}$ 
 for $(\lambda,\nu) \in L_{\operatorname{even}}$.

In Chapter \ref{sec:B}
 we start the discussion of the singular symmetry breaking operators for $(\lambda,\nu) \in \backslash \backslash$, their analytic continuation and find a necessary and sufficient condition which determines
 if they are not zero.

Chapter \ref{sec:10} is a discussion of the differential symmetry breaking operators, which were first found by Juhl.

Building on these preparations,
 we complete in Chapter \ref{sec:dim}
 the classification
 of symmetry breaking operators
{}from the spherical principal series representation 
 $I(\lambda)$ of $G=O(n+1,1)$
 to the representations $J(\nu)$ of $G'=O(n,1)$
 and prove Theorems \ref{thm:I.1} and \ref{thm:II.5}. 
Here  again the analysis of $\mathcal{S}ol (\mathbb{R}^n; \lambda,\nu)$ for the parameter in $\backslash \backslash$ and $// $ plays  a crucial role.

In Chapter \ref{sec:reduction}  we show
 the relationships
 among the (generically) regular symmetry breaking operators
 $\A_{\lambda,\nu}$, 
 the singular symmetry breaking operators
 $\B_{\lambda,\nu}$
 and the differential symmetry breaking operators
 $\C_{\lambda,\nu}$
 by proving explicitly the residue formulae
 (Theorem \ref{thm:I.4}). 
Furthermore we also extend the functional equations
 to these singular symmetry breaking operators.

Finally, 
 Theorem \ref{thm:I.4} (1), (2), and (3)
 are proved by explicit computations in 
 Chapters \ref{sec:kfini}, 
 \ref{sec:B}, 
 and \ref{sec:10}, 
 respectively, 
 and Theorem \ref{thm:image}
 is proved by using Theorem \ref{thm:I.4}
 in Chapter \ref{sec:Image}.

The last two chapters are applications of our results. 
In Chapter \ref{sec:PGinv}
 we apply our results about symmetry breaking operators 
 $\A_{\lambda,\nu}$
 to the analysis
 on vector bundles over the semisimple symmetric space $O(n+1,1)/O(n,1)$. 
In Chapter \ref{sec:applbl}
 we construct explicitly complementary series representations of the group $G'=O(n,1) $ as discrete summands in the restriction of  the unitary complementary series representations of $O(n+1,1)$
 by using the adjoint of the differential 
 symmetry breaking operators 
 $\C_{\lambda,\nu}$
\vskip 2pc

{\bf{Notation.}}\enspace
${\mathbb{N}}=\{0,1,2,\cdots\}$, 
${\mathbb{N}}_+=\{1,2,3,\cdots\}$, 
 ${\mathbb{R}}^{\times}={\mathbb{R}} \setminus \{0\}$.  
For two subsets $A$ and $B$ of a set,
 we write 

\[
   A \setminus B:=\{a \in A: a \notin B\}
\]
rather than the usual notation $A \backslash B$.  

\section{Symmetry breaking for the spherical principal series representations}
\label{sec:2}

Before we start with the construction
 of the $G'$-intertwining operators
 between spherical principal series representations
 of $G=O(n+1,1)$ and $G'= O(n,1)$
 we want to prove the main results
  (Theorems \ref{thm:compointro} and \ref{thm:mIFintro}, 
 see Theorems \ref{thm:compo} and \ref{thm:mIF} below)
 about $G'$-intertwining operators
 between irreducible composition factors  of spherical principal series representations. 
This is intended
 for the convenience
 of the readers
 who are more interested 
 in representation theoretic results
 rather than geometric analysis
 arising from branching problems 
 in representation theory.  
In the proof we use the results  about symmetry breaking operators for  spherical principal series representations of $G$ and $G'$
 that will be proved later in the article.

\subsection{Notation and review of previous results}
\label{subsec:matrix}
Consider the quadratic form 
\begin{equation}
\label{eqn:quad}
      x_0^2 + x_1^2 +\dots +x_{n}^2-x_{n+1}^2
\end{equation}
 of signature $(n+1,1)$. 
We define $G$
 to be the indefinite orthogonal group $O(n+1,1)$
 that preserves the quadratic form 
 \eqref{eqn:quad}.  
Let $G'$ be the stabilizer
 of the vector $e_{n}={}^t\! (0,0,\cdots,0,1,0)$.  
Then $G' \simeq O(n,1)$.  
We set 
\allowdisplaybreaks
\begin{align}
K:=&O(n+1) \times O(1), 
\\
K':=&K \cap G' 
=\{
\begin{pmatrix}
A & & 
\\
  & 1 &
\\
  & & \varepsilon
\end{pmatrix}
:
A \in O(n), \varepsilon = \pm 1
\}
\simeq O(n) \times O(1).  
\label{eqn:Kprime}
\end{align}
Then $K$ and $K'$ are maximal compact subgroups
 of $G$ and $G'$, 
respectively.

Let ${\mathfrak {g}}={\mathfrak {o}}(n+1,1)$
 and ${\mathfrak {g}}'={\mathfrak {o}}(n,1)$
 be the Lie algebras of $G=O(n+1,1)$
 and $G'=O(n,1)$, 
 respectively.  
We take a hyperbolic element $H$
 as 
\begin{equation}
\index{sbon}{hypelt@$H$|textbf}
\label{eqn:aH}
H :=
  E_{0,n+1} + E_{n+1,0} \in \mathfrak{g}'.  
\end{equation}

Then $H$ is also a hyperbolic element
 in ${\mathfrak {g}}$, 
 and the eigenvalues of 
$\operatorname{ad}(H)
\in \operatorname{End}({\mathfrak {g}})$
 are $\pm 1$ and 0.  
For $1 \le j \le n$, 
 we define nilpotent elements in ${\mathfrak {g}}$
 by 
\begin{align} 
N_j^+:=& -E_{0,j} + E_{j,0} -E_{j, n+1} -E_{n+1, j}, 
\label{eqn:ngen}
\\
N_j^-:=& -E_{0,j} + E_{j,0} +E_{j, n+1} +E_{n+1, j}.  
\label{eqn:nngen}
\end{align}
Then we have maximal nilpotent subalgebras 
 of ${\mathfrak {g}}$:
\[
\index{sbon}{nplus@${\mathfrak {n}}_+$|textbf}
 {\mathfrak {n}}_+:= {\operatorname{Ker}}({\operatorname{ad}}(H)-1)
                   =\sum_{j=1}^n {\mathbb{R}}N_j^+, 
\quad
 {\mathfrak {n}}_-:= {\operatorname{Ker}}({\operatorname{ad}}(H)+1)
                   =\sum_{j=1}^n {\mathbb{R}}N_j^-.  
\]
Since $H$ is contained in the Lie algebra ${\mathfrak {g}}'$
 of split rank one, 
 we can define two maximal nilpotent subalgebras
 of ${\mathfrak {g}}'$ by  
\index{sbon}{nplusprime@${\mathfrak {n}}_+'$|textbf}
\begin{align}
{\mathfrak {n}}_+'
:=&{\mathfrak {n}}_+ \cap {\mathfrak {g}}'
=\sum_{j=1}^{n-1}{\mathbb{R}}N_j^+, 
\label{eqn:nprime}
\\
{\mathfrak {n}}_-'
:=&{\mathfrak {n}}_- \cap {\mathfrak {g}}'
=\sum_{j=1}^{n-1}{\mathbb{R}}N_j^-.  
\notag
\end{align}
Let $N_+$= exp(${\mathfrak{n}}_+)$, 
${N_-}$ =exp$({\mathfrak{n}}_- )$
 and $N'_+ := N _+\cap G'=\exp({\mathfrak {n}}_+')$, 
${N'_-} := {N_-} \cap G'=\exp({\mathfrak {n}}_-')$. 
We define

\index{sbon}{Msubgp@$M$|textbf}
\begin{align}
M :={}
  & Z_K({\mathfrak {a}})
  ={}
 \left\{
    \begin{pmatrix} 
        \varepsilon \\ & A \\ & & \varepsilon
    \end{pmatrix} :
    A \in O(n), \   \varepsilon = \pm 1
    \right\}
    \simeq O(n) \times \mathbb{Z}_2, 
\notag
\\
M' :={}&Z_{K'}({\mathfrak {a}})={}
   \left\{
    \begin{pmatrix} 
        \varepsilon \\ & B \\ & & 1 \\ & & & \varepsilon
    \end{pmatrix} :
    B \in O(n-1):   \varepsilon = \pm 1
    \right\}
\notag
\\
    \simeq & O(n-1) \times {\mathbb{Z}}_2.  
\label{eqn:Mprime}
\end{align}
We set
\index{sbon}{hypelt@$H$}
\begin{equation*}
{\mathfrak {a}}:={\mathbb{R}}H
\qquad
\text{ and }
A:=\exp {\mathfrak {a}}.  
\end{equation*}
Then $P=MAN_+$ is a Langlands decomposition
of a minimal parabolic subgroup $P$
 of $G$.  
Likewise,
 $P'=M'AN_+'$ is a Langlands decomposition
 of a minimal parabolic subgroup $P'$ of $G'$.  
We note 
 that we have chosen $H \in {\mathfrak {g}}'$
 so that we can take a common maximally split abelian subgroup
 $A$ in $P'$ and $P$.  
The Langlands decompositions
 of the Lie algebras of $P$ and $P'$ are given 
 in a compatible way as  
\begin{equation}
\label{eqn:Lang}
{\mathfrak {p}} ={\mathfrak {m}}+{\mathfrak {a}}+{\mathfrak {n}}_+, 
 \quad
{\mathfrak {p}}'={\mathfrak {m}}'+{\mathfrak {a}}+{\mathfrak {n}}_+'
 = ({\mathfrak {m}} \cap {\mathfrak {g}}')
  +({\mathfrak {a}} \cap {\mathfrak {g}}')
  +({\mathfrak {n}}_++{\mathfrak {g}}').  
\end{equation}

\medskip
We assume from now on that the principal series representations
 $I(\lambda)$
 are realized on the Fr{\'e}chet space
 of smooth sections
 of the line bundle
 $G \times_{\lambda}{\mathbb{C}} \to G/P$.  
See Section \ref{subsec:CW} a short discussion
 of the Casselman--Wallach theory
 on Fr{\'e}chet representations 
 having moderate growth
 and the underlying $({\mathfrak{g}},K)$-module.  

Let $G_0=SO_0(n+1,1)$ be the identity component
 of $G=O(n+1,1)$.  
Then the quotient group
 is of order four:
\[
   G/G_0 \simeq \{\pm \} \times \{\pm \}.  
\]
Irreducible representations 
 of the disconnected group $G$ 
 are not necessarily irreducible
 as representations
 of $G_0$.  
We have
\begin{proposition}
\label{prop:connG}
{\rm{1)}}\enspace
Suppose $n \ge 2$.  
Then any irreducible $G$-subquotient $Z$
 of $I(\lambda)$
 remains irreducible as a $G_0$-module.  
\par\noindent
{\rm{2)}}\enspace
Suppose $n =1$.  
\begin{enumerate}
\item[{\rm{2-a)}}]
For $i \in {\mathbb{N}}$, 
 $T(i)$ splits into a direct sum
 of two irreducible $G_0$-modules.  
\item[{\rm{2-b)}}]
Any irreducible $G$-subquotient $Z$
 of $I(\lambda)$
 other than $T(i)$ remains irreducible
 as a $G_0$-modules.  
\end{enumerate}
\end{proposition}

For the proof, 
 we begin with the following observation:
\begin{lemma}
\label{lem:connG}
{\rm{1)}}\enspace
A $({\mathfrak {g}}, K)$-module
 $Z_K$ is irreducible as a ${\mathfrak {g}}$-module
 if every irreducible $K$-module
 occurring in $Z_K$
 is irreducible as a $K_0$-module.  
\par\noindent
{\rm{2)}}\enspace
For $G=O(n+1,1)$, 
 let $P_0:=P \cap G_0$.  
Then $P_0$ is connected, 
 and is a minimal parabolic subgroup of $G$.  
Then we have a natural bijection:
\[
G_0/P_0 \overset \sim \to G/P
(\simeq S^n).  
\]
\end{lemma}

\begin{proof}
[Proof of Proposition \ref{prop:connG}]
Let $Z_K$ be the underlying 
 $({\mathfrak {g}}, K)$-module of $Z$.  
It is sufficient to discuss the irreducibility
 of $Z_K$ as a $({\mathfrak {g}}, K_0)$-module.  
\par\noindent
{1)}\enspace
Any irreducible representation
 of $K \simeq O(n+1) \times O(1)$
 occurring in the spherical principal series representation 
 $I(\lambda)$
 is of the form ${\mathcal{H}}^i(S^n) \boxtimes {\bf{1}}$
 for some $i \in {\mathbb{N}}$, 
 which is still irreducible 
 as a representation
 of $K_0=SO(n+1)$
 if $n \ge 2$.  
Here ${\mathbf{1}}$ denotes the trivial 
 one-dimensional representation of $O(1)$.  
Hence the assumption of Lemma \ref{lem:connG} (1)
 is fulfilled, 
 and the first statement follows.  
\par\noindent
{2)}\enspace
By Lemma \ref{lem:connG} (2), 
 the restriction of $I(\lambda)$
 to $G_0$ is isomorphic to a spherical principal series 
 representation of $G_0=SO_0(2,1)$.  
Comparing the aforementioned composition series
 of representation $I(\lambda)$ 
 of $O(2,1)$
 with a well-known result
 for $G_0=SO_0(2,1)
 \simeq SL(2,{\mathbb{R}})/\{\pm 1\}$, 
 we see
 that $T(i)$ is a direct sum
 of a holomorphic discrete series representation
 and an anti-holomorphic discrete series representation
 of $G_0$
 and that other irreducible subquotients of $G$
 remain irreducible as $G_0$-modules.  
See also Remark \ref{rem:GL}
 for geometric interpretations
 of this decomposition.  

\end{proof}

Proposition \ref{prop:connG} and \cite{JW} imply that 
 the representation $I(\lambda ) $ is reducible
 if and only if
\[  {\lambda } = n+i
 \quad \mbox{or} \quad \lambda
 = -i \quad \mbox{ for }  i \in {\mathbb{N}}.    
\]
A reducible spherical principal series representation
 has two irreducible composition factors.
The Langlands subquotient of $I(n+i)$ is
 a finite-dimensional representation $F(i)$.
We have for $i \in {\mathbb{N}}$
 non-splitting exact sequences
 as Fr{\'e}chet $G$-modules:
\begin{equation}
\label{eqn:FIT}
\index{sbon}{Fi@$F(i)$}
\index{sbon}{Ti@$T(i)$}
 0 \rightarrow
 F(i) \rightarrow I(-i) \rightarrow 
 T(i)\rightarrow 
0, 
\end{equation}
\begin{equation}
\label{eqn:TIF}
 0 \rightarrow T(i) \rightarrow I(n+i) \rightarrow F(i)\rightarrow 0.  
\end{equation}

\medskip

Inducing from the minimal parabolic subgroup $P'$ of $G'$, 
we define the induced representation
$J(\nu)$ and the irreducible representations 
 $F(j)\equiv F^{G'}(j)$
 and $T(j)\equiv T^{G'}(j)$ of $G'$ as we did for $G$. 
We shall simply write 
 $F(j)$ for $F^{G'}(j)$
 and $T(j)$ for $T^{G'}(j)$, 
 respectively,
 if there is no confusion.

\subsection{Finite-dimensional subquotients
of disconnected groups}
\label{subsec:Fi}
Since the group $G=O(n+1,1)$
 has four connected components,  
we need to be careful to identify
 the finite-dimensional subquotient 
 $F(i)$
 with some other 
 (better-understood)
 representations.

First, 
 we consider the space of harmonic polynomials
 of degree $i \in {\mathbb{N}}$
 by 
\[
  {\mathcal{H}}^i({\mathbb{R}}^{n+1,1})
  :=
  \{
   \psi \in {\mathbb{C}}[x_0,\cdots,x_{n+1}]:
   \square \psi=0,
   \,\,
   \text{$\psi$ is homogeneous of degree $i$}
\}, 
\] 
 where $\square
=\frac{\partial^2}{\partial x_0^2}
+\cdots+
\frac{\partial^2}{\partial x_n^2}
-
\frac{\partial^2}{\partial x_{n+1}^2}
$.  
Then $G=O(n+1,1)$ acts irreducibly
 on ${\mathcal{H}}^{i}({\mathbb{R}}^{n+1,1})$
 for any $i \in {\mathbb{N}}$.  
The indefinite signature
 is not the main issue here, 
 because this representation extends
 to a holomorphic representation
of the complexified Lie group $O(n+2,{\mathbb{C}})$.  
Similarly,
 the group $G'=O(n,1)$ acts irreducibly 
 on ${\mathcal{H}}^{j}({\mathbb{R}}^{n,1})$
 for $j \in {\mathbb{N}}$.  
By the classical branching law,
 we have a $G'$-irreducible decomposition:
\begin{equation}
\label{eqn:Hbranch}
   {\mathcal{H}}^{i}({\mathbb{R}}^{n+1,1})|_{G'}
   \simeq
   \bigoplus_{j=0}^{i}
   {\mathcal{H}}^{j}({\mathbb{R}}^{n,1}).  
\end{equation}

Second,
 we notice
 that there are three non-trivial 
 one-dimensional representations
 of the disconnected group $G$.  
For our purpose,
 we consider the following one-dimensional representation
\begin{equation} \label{char:chi}
\chi:O(n+1,1) \to \{\pm 1\}
\end{equation}
 by the composition of the following maps
\[
G \to G/G_0
\simeq
 O(n+1) \times O(1)/SO(n+1) \times SO(1)
\simeq
 \{\pm 1\}\times \{\pm 1\}
 \overset {{\operatorname{pr}}_2} {\to} \{\pm 1\}, 
\]
where $G_0=SO_0(n+1,1)$
 is the identity component
 of $G$, 
 and ${\operatorname{pr}}_2$ denotes 
 the second projection.  
Similarly,
 we define 
$
\chi':O(n,1) \to \{\pm 1\}
$.  
Then by inspecting the action
 of the four disconnected components
 of $G$, 
 we have the following isomorphisms as representations
 of $G$ and $G'$, 
respectively:
\begin{align}
F(i) \simeq & \chi^i \otimes {\mathcal{H}}^i({\mathbb{R}}^{n+1,1}), 
\label{eqn:FHi}
\\
F(j) \simeq & (\chi')^j \otimes {\mathcal{H}}^j({\mathbb{R}}^{n,1}).  
\label{eqn:FHj}
\end{align}

Combining \eqref{eqn:Hbranch}
 with \eqref{eqn:FHi}
 and \eqref{eqn:FHj}, 
 we get the following branching law
 for the restriction $G \downarrow G'$:
\[
  F(i)|_{G'}
  \simeq \bigoplus_{j=0}^i(\chi')^{i-j}
  \otimes F(j).
\]
Thus we have shown the following proposition.  
\begin{proposition}
[branching law of $F(i)$
 for $G \downarrow G'$]
\label{prop:Hbranch}
Suppose $i, j \in {\mathbb{N}}$.  
\begin{enumerate}
\item[{\rm{1)}}]
$\operatorname{Hom}_{G'}(F(i), F(j)) \ne 0$
 if and only if 
 $0 \le j \le i$ and $i \equiv j \mod 2$.  
\item[{\rm{2)}}]
 $\operatorname{Hom}_{G_0'}(F(j), F(i)) \ne 0$
 if and only if $0 \le j \le i$.  
\end{enumerate}
\end{proposition}

\medskip

\subsection{Symmetry breaking operators and  spherical principal series representations.}
\label{subsec:2.2}
We refer to nontrivial homomorphisms in 
\index{sbon}{H1@$H(\lambda, \nu)$|textbf}
\[
 H(\lambda, \nu) := \mbox{Hom}_{G'}(I(\lambda),J(\nu))
\]
as  {\t intertwining restriction maps} or {\it symmetric breaking operators}. In the next chapter
  general properties of symmetry breaking operators
 for principal series representations are discussed. 
In this section
 we will illustrate the functional equations satisfied by the continuous symmetry breaking operators
 (Theorem \ref{thm:TAAT}, 
 see also Theorem \ref{thm:TABC})
 by analyzing their behavior  on $I(\lambda) \times J(\nu )$
 where      both representations $I(\lambda )$ and $J(\nu)$ are reducible, 
{\it{i.e.}}, 
 ($\lambda, \nu $) are in  
\[
     \mathcal{L}
   =\{ (i,j): i,j \in {\mathbb{Z}} \mbox{ and } (i,j) \not \in  (0, n)\times (0,n-1)\}
\]

The Weyl group ${\mathfrak{S}}_2 \times {\mathfrak{S}}_2$
 of $G\times G'$ acts on $\mathcal L$. 
The action is generated by the action of  the generators 
$(\lambda, \nu) \mapsto (-\lambda+n,\nu)$
 and $(\lambda,\nu) \mapsto (\lambda, -\nu +n-1 )$. 
We write  $\mathcal{ L}_{even} \subset \mathcal L$ for the orbit  of 
 \[L=  \{(i,j): i,j \mbox{ nonpositive integers, } i=j\mbox{ mod }2\}\]  under the Weyl group and 
 $\mathcal{L}_{odd} $ its complement in $\mathcal L$.
  We consider case by case  the symmetry breaking operators parametrized by $(\lambda, \nu) $ in the intersection of $\mathcal L_{even}$, respectively $(\lambda, \nu ) \in\mathcal L_{odd}$,  with  the sets
\begin{description} 
\label{octants}
\item[I.A]  $\lambda \leq 0$, $\nu < \lambda $, 
\item[I.B]    $\lambda \leq 0$, $ \lambda\le \nu \leq 0$, 
\item[II.A]   $\lambda \leq 0$, $-\lambda +n- 1< \nu$, 
\item[II.B]  $\lambda \leq 0$, $-\lambda +n-1 \ge \nu \geq n-1$, 
\item[III.A] $\lambda \geq n$,  $  \lambda -1 < \nu $, 
\item[III.B]  $\lambda \geq n$, $ n-1 \leq \nu  \le \lambda -1$, 
 \item[IV.A]  $\lambda \geq n$, $\nu <
-\lambda +n$, 
\item[IV.B] $\lambda \geq n$, $-\lambda +n \leq \nu \leq 0 $.  
\end{description}

\medskip
\begin{figure}[h]
\begin{center}
\includegraphics[scale=1]{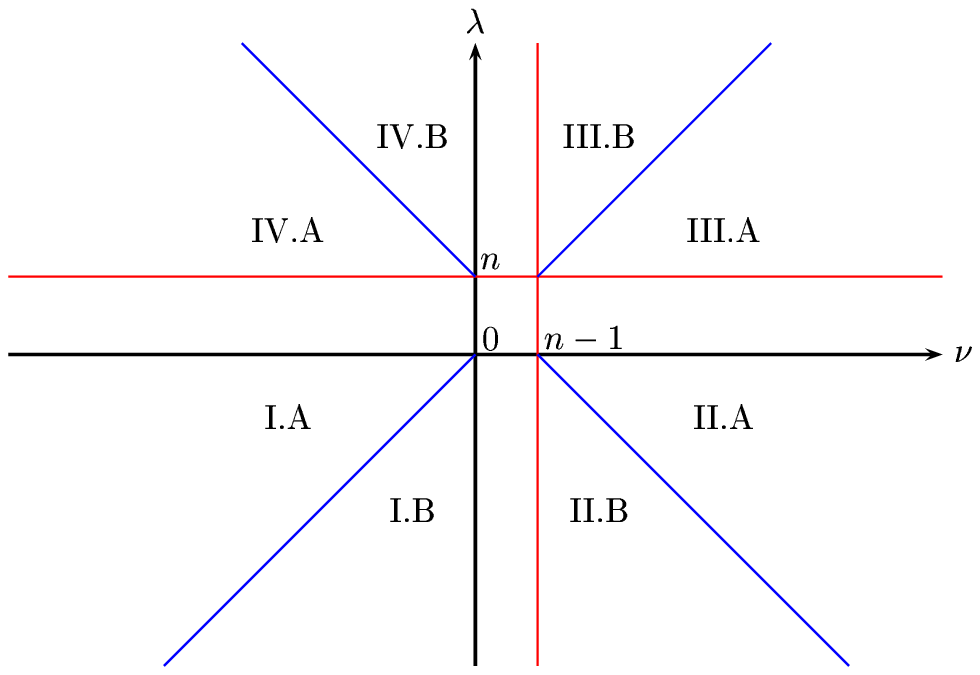}
\end{center}
\caption{Octants of the parameter space}
\end{figure}

The results are graphically represented
 in Figures 2.1--\ref{fig:2.3}.  
The large and the small rectangles stands
 for the reducible principal series representations $I(i),J(j)$ of the large group $G$
 and the small group $G'$ respectively. 
The rectangles are
 located in the octants of the parameter space
 $(\lambda,\nu)$ determined by the conditions on $(i,j)$.

The subrectangles at the bottom represents
 the irreducible subrepresentation;
 a small rectangle represents a finite-dimensional subquotient module,
 a large rectangle an infinite-dimensional subquotient.  

A colored green subrectangle is the subrepresentation
 which is contained
 in the kernel of the operator of the symmetry breaking operator
 $\widetilde{\mathbb A}_{i,j}$, 
 and a white upper rectangle implies the image of the symmetry breaking operator $\widetilde{\mathbb A}_{i,j}$ is contained in the irreducible subrepresentation.

\medskip \noindent
{\bf Suppose first that $(i,j)\in {\mathcal L}$ is contained in I.A}. 
Both representations $I(i)$ and $J(j)$
 have finite-dimensional subrepresentations $F(-i)$ and $F(-j)$
 respectively. 
Since $-j > -i$ 
 the representation  $F(-j)$ is not a summand  $F(-i)_{| G'}$
 by Proposition \ref{prop:Hbranch}, 
 and therefore
 the finite-dimensional subrepresentation $ F(-i)$ is
 in the kernel of the symmetry breaking operator
 $\A_{i,j}$. 
On the other hand, 
 by Theorem \ref{thm:II.2}, 
 we have
\[ \T{n-1 -j}{j} 
 \circ 
 \widetilde{\mathbb A}_{i,n-1-j}
 =
 \frac{\pi^{\frac{n-1}{2}}}{\Gamma(n-1-j)}
   \widetilde{\mathbb A}_{i,j}, 
\]
which implies
 that the image of the nontrivial symmetry breaking operator
 $\A_{i,j}$ is the finite-dimensional
subrepresentation $F(-j)$ or zero.  
Since $\A_{i,j}\ne 0$
 by Theorem \ref{thm:II.1}, 
the image is in fact $F(-j)$.

\medskip \noindent
{\bf Suppose now that $(i,j)\in {\mathcal L}  $ is contained  in II.A.}
The representation $I(i)$ has a finite-dimensional subrepresentation
 $F(-i)$, 
 and $J(j)$ has a finite-dimensional quotient $F(j-n+1)$. 
The image of $F(-i)$ under the symmetry breaking operator 
 $\A_{i,j}$ is finite-dimensional or zero. 
Since $J(j) $ has no finite-dimensional subrepresentation, 
the finite-dimensional subrepresentation $F(-i)$
 must be in the kernel of the symmetry breaking operator
 $\A_{i,j}$. 
By Theorems \ref{thm:II.1} and \ref{thm:II.2}, 
 we have
\[ \T{j}{n-1-j} 
 \circ 
 \widetilde{\mathbb A}_{i,j}
 =
 \frac{\pi^{\frac{n-1}{2}}}{\Gamma(j)}
   \widetilde{\mathbb A}_{i,n-1-j}  \not = 0.\]
Thus the image of $\widetilde{\mathbb A}_{i,n-1-j}$
 is finite-dimensional, 
 and therefore $\widetilde{\mathbb A}_{i,j}$ defines
 a surjective $({\mathfrak {g}},K)$-homomorphism
 $T(i)_K \to J(j)_{K'}$.  

\medskip \noindent
{\bf Suppose that $(i,j)\in {\mathcal L}$ is contained in IV.A.}
The representation $I(i)$ has a finite-dimensional quotient $F(i-n)$, 
 and $J(j)$ has a finite-dimensional subrepresentation $F(-j)$.  
The functional equation in Theorem \ref{thm:II.2}
 and the non-zero condition in Theorem \ref{thm:II.1}
 imply 
\[
 \widetilde{\mathbb A}_{i,j}
 \circ \T{n-i}{i} 
 = \frac{\pi^{\frac{n}{2}}}{\Gamma(i)} \widetilde{\mathbb A}_{n-i,j}\not = 0
\] 
Further, 
 since $(n-i,j)$ is contained in I.A, 
 the image of $T(i)$
 under the symmetry breaking operator $\widetilde{\mathbb A}_{i,j}$
 is the finite-dimensional representation $F(-j)$. 
Since $T(i)$ has a finite-codimension in $I(i)$, 
 the image of $I(i)$ under the symmetry breaking operator
 $\widetilde{\mathbb A}_{i,j}$ is still finite-dimensional, 
 hence is equal to the unique subrepresentation $F(-j)$
 of $J(j)$.

\medskip \noindent
{\bf Suppose that $(i,j)\in {\mathcal L}$ is contained in III.A.} 
The representation $I(i)$ and $J(j)$
 both have finite-dimensional quotients $F(i-n)$ and $F(j-n+1)$.  
Furthermore
 the multiplicity $m(I(i),J(j)))=1$
 by Theorem \ref{thm:I.1}. 
Again the spherical vector is not in the kernel
 of $\widetilde{\mathbb A}_{i,j}$, but its image is a spherical vector
 for $J(j)$
 by Theorem \ref{thm:I.4}, 
 which in turn generates
 the underlying $({\mathfrak {g}}',K')$-module
 of $J(j)_{K'}$. 
Hence the symmetry breaking operator 
 $\A_{i,j}$ is a surjective map from $I(i)_K$ to $J(j)_{K'}$. 

\medskip

\medskip

Figure \ref{fig:2.1} represents the results
 for the operator $\widetilde{\mathbb A}_{i,j}$
 with $(i,j) \in {\mathcal L}$
 in the four octants I.A, 
 II.A, III.A, 
 and IV.A discussed so far.

\medskip
\begin{figure}[h]
\begin{center}
\includegraphics[scale=1]{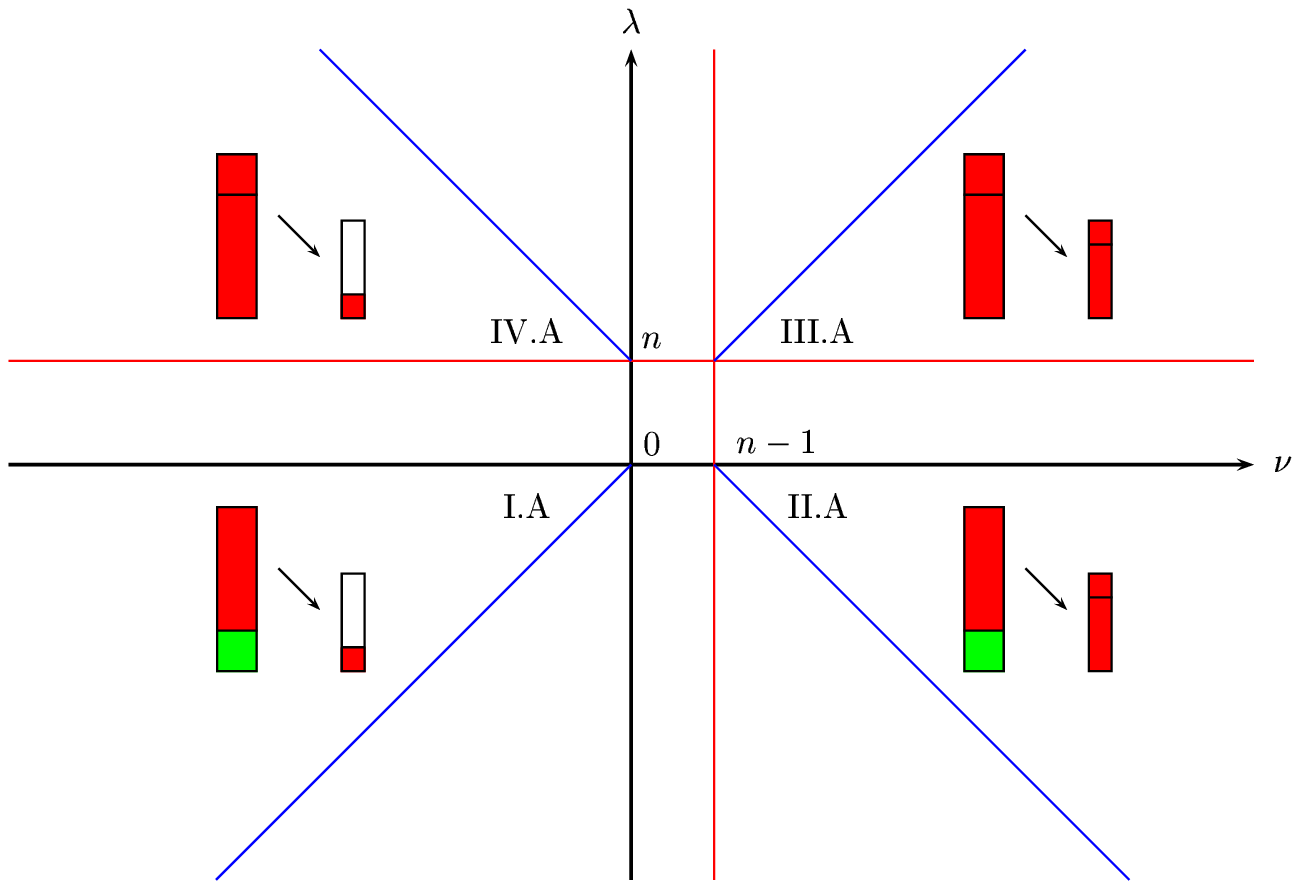}
\end{center}
\caption{Image of $\A_{i,j}$ with $(i,j) \in {\mathcal{L}}$}
\label{fig:2.1}
\end{figure}

\medskip 
\noindent
{\bf Suppose now that $(i,j)\in {\mathcal L}$ is contained in II.B.}
The representation $I(i)$ has a finite-dimensional subrepresentation
 and $J(j)$ has a finite-dimensional quotient. 
Furthermore we have $m(I(i),J(j))=1$
 by Theorem \ref{thm:I.1}.  
The image of a finite-dimensional $G'$-invariant space
 of $I(i)$ under $\widetilde{\mathbb A}_{i,j}$ is finite-dimensional
 or zero. 
Since $J(j)$ has no finite-dimensional subrepresentation, 
 the finite-dimensional subrepresentation $F(i)$ lies 
 in the kernel of the symmetry breaking operator
 $\widetilde{\mathbb A}_{i,j}.$ 
Consider the functional equation from Theorem \ref{thm:II.2}
\[ 
 \T{j}{n-1-j} 
 \circ 
 \widetilde{\mathbb A}_{i,j}
 =
 \frac{\pi^{\frac{n-1}{2}}}{\Gamma(j)}
   \widetilde{\mathbb A}_{i,n-1-j} .
\]
\noindent
\underline{ If $(i,j)\in \mathcal{L}_{odd}$}
 then the right-hand side is non-zero
 by Theorem \ref{thm:II.2}
 and thus $\T{j}{n-1-j} 
 \circ 
 \widetilde{\mathbb A}_{i,j} \not = 0$.  
In particular the image of $\T{j}{n-1-j} 
 \circ 
 \widetilde{\mathbb A}_{i,j}$
 is finite-dimensional.  
Thus the symmetry breaking operator 
$\widetilde{\mathbb A}_{i,j}$ must have a dense image,
 and therefore,
 induces a surjective
 $({\mathfrak {g}},K)$-homomorphism
 $I(i)_K \to J(j)_{K'}$.\\
\noindent
\underline{If $(i,j)\in \mathcal{L}_{even}$}
 then the symmetry breaking operator
 $\A_{i,n-1-j} =0$
 by Theorem \ref{thm:II.1}. 
Hence the image of $\A_{i,j}$
 is contained in the subrepresentation 
 $T(j-n+1)$
 and thus induces a non-zero element
 in $\operatorname{Hom}_{G'}(T(i), T(j-n+1))$.

\medskip
\noindent
{\bf Suppose now that $(i,j)\in {\mathcal L}$ is contained in IV.B.} 
The representation $I(i)$ has a finite-dimensional quotient
 and $J(j)$ has a finite-dimensional subrepresentation. 
Furthermore the multiplicity $m(I(i),J(j))=1$. 
Consider the functional equation
{}from Theorem \ref{thm:II.2}
\[ \widetilde{\mathbb A}_{i,j}
 \circ \T{n-i}i 
 = \frac{\pi^{\frac{n}{2}}}{\Gamma(i)} \widetilde{\mathbb A}_{n-i,j}\]
\noindent
\underline{ If $(i,j)\in \mathcal{L}_{odd}$}
 then 
$
   \frac{\pi^{\frac{n}{2}}}{\Gamma(i)} \widetilde{\mathbb A}_{n-i,j}\ne 0
$
 by Theorem \ref{thm:II.1}.  
Hence the image $T(i-n)$
 of the Knapp--Stein intertwining operator
 $\T{n-i}i$ is not in the kernel of the symmetry breaking operator
 $\widetilde{\mathbb A}_{i,j}$. 
By the same argument as in IV.A, 
 the image of 
the symmetry breaking operator is finite-dimensional, 
 and thus it induces 
 a nontrivial element in $\operatorname{Hom}_{G'}(T(i-n)_K, F(j))$.

\noindent
\underline{If $(i,j)\in \mathcal{L}_{even}$}
 then $ \widetilde{\mathbb A}_{n-i,j} =0$. 
Hence the image $T(i-n)$
 of the Knapp--Stein intertwining operator $\T{n-i}{i}$
 for $G$
is in the kernel of the symmetry breaking operator
 $\widetilde{\mathbb A}_{i,j}$
 and therefore it induces a non-zero operator
 in $\operatorname{Hom}_{G'}(F(i-n), F(j))$.

\medskip

\noindent
{\bf Suppose now that $(i,j)\in {\mathcal L}$ is contained in III.B.} 
The representation $I(i)$  and $J(j)$ have both finite-dimensional quotients. 
Furthermore we have $m(I(i),J(j))=1$
 by Theorem \ref{thm:I.1}. 
Consider the functional equation from Theorem \ref{thm:II.2}
\[
 \widetilde{\mathbb A}_{i,j}
 \circ \T{n-i}i 
 = \frac{\pi^{\frac{n}{2}}}{\Gamma(i)} \widetilde{\mathbb A}_{n-i,j}.
\]
It implies
 that the image $T(i-n)$
 of the Knapp--Stein intertwining operator 
 $\T {n-i}i$
 of $G$ is not in the kernel
 of $\widetilde{\mathbb A}_{i,j}$. 
Furthermore 
 $\A_{i,j}$ acts nontrivially
 on the spherical vector
 by Theorem \ref{thm:I.4},
 and its image is a cyclic vector
 in $J(j)$.  
Hence the symmetry breaking operator
 $\widetilde{\mathbb A}_{i,j}$ induces
 a surjective map from $I(i)_K$ to $J(j)_{K'}$.

\begin{remark}
{If $(i,j)\in \mathcal{L}_{even}$}, 
 then 
 the functional equation also implies 
 that the image of $T(i-n)_{K}$ is $T'(j-n+1)_{K'}$. 
\end{remark}

\medskip

\noindent
{\bf Suppose now that $(i,j)\in {\mathcal L}$ is contained in I.B.} 
The representations $I(i)$ and $J(j)$
 have finite-dimensional subrepresentations
 $F(-i)$ and $F(-j)$, 
 respectively.  

\noindent
\underline{If $(i,j)\in \mathcal{L}_{odd}$}, 
 then $m(I(i),J(j))=1$. The image of 
$\widetilde{\mathbb A}_{i,j}$ is finite-dimensional
 because
\[ 
\T{j}{n-1-j} \circ \widetilde{\mathbb A}_{i,j}
 = \frac{\pi^{\frac{n-1}{2}}}{\Gamma(j)}
   \widetilde{\mathbb A}_{i,n-1-j} = 0.  
\]
Another functional equation
\[ \widetilde{\mathbb A}_{i,j}
 \circ \T{n-i}i 
 = \frac{\pi^{\frac{n}{2}}}{\Gamma(i)} \widetilde{\mathbb A}_{n-i,j}= 0\]
 implies that the finite-dimensional representation $F(-i)$ 
 is contained 
 in the kernel of $\widetilde{\mathbb A}_{i,j}$
 and so it induces a non-zero symmetry breaking operator
 in $\operatorname{Hom}_{G'}(T(-i), F(-j))$.

\noindent
\underline{If $(i,j)\in \mathcal{L}_{even}$},
 then $ \widetilde{\mathbb A}_{i,j}= 0$
 by Theorem \ref{thm:II.1}.

\medskip

Figure \ref{fig:2.2} represents the results
 for $\widetilde{\mathbb A}_{i,j}$
 in the 4 octants
 I.B, II.B, III.B, and IV.B
 for $(i,j) \in {\mathcal L}_{odd}$

\begin{figure}[h]
\begin{center}
\includegraphics[scale=1]{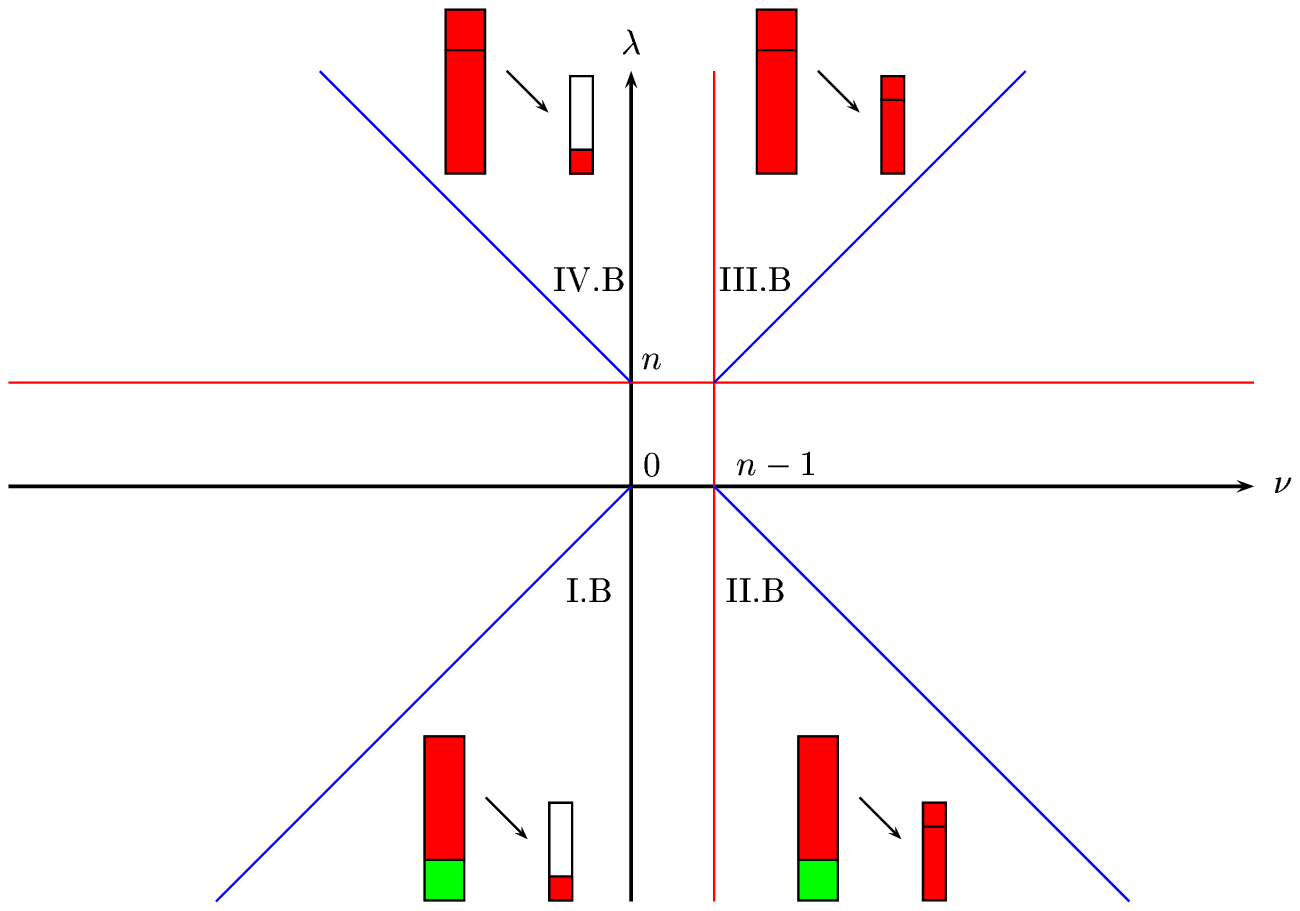}
\end{center}
\caption{Image of $\A_{i,j}$ with $(i,j) 
\in {\mathcal{L}}_{\operatorname{odd}}$
}
\label{fig:2.2}
\end{figure}
\medskip

Similarly, 
Figure \ref{fig:2.3} represents the results
 for $\widetilde{\mathbb A}_{i,j}$
 in the 4 octants
 for $(i,j) \in {\mathcal L}_{even}$
\medskip

\begin{figure}[h]
\begin{center}
\includegraphics[scale=1]{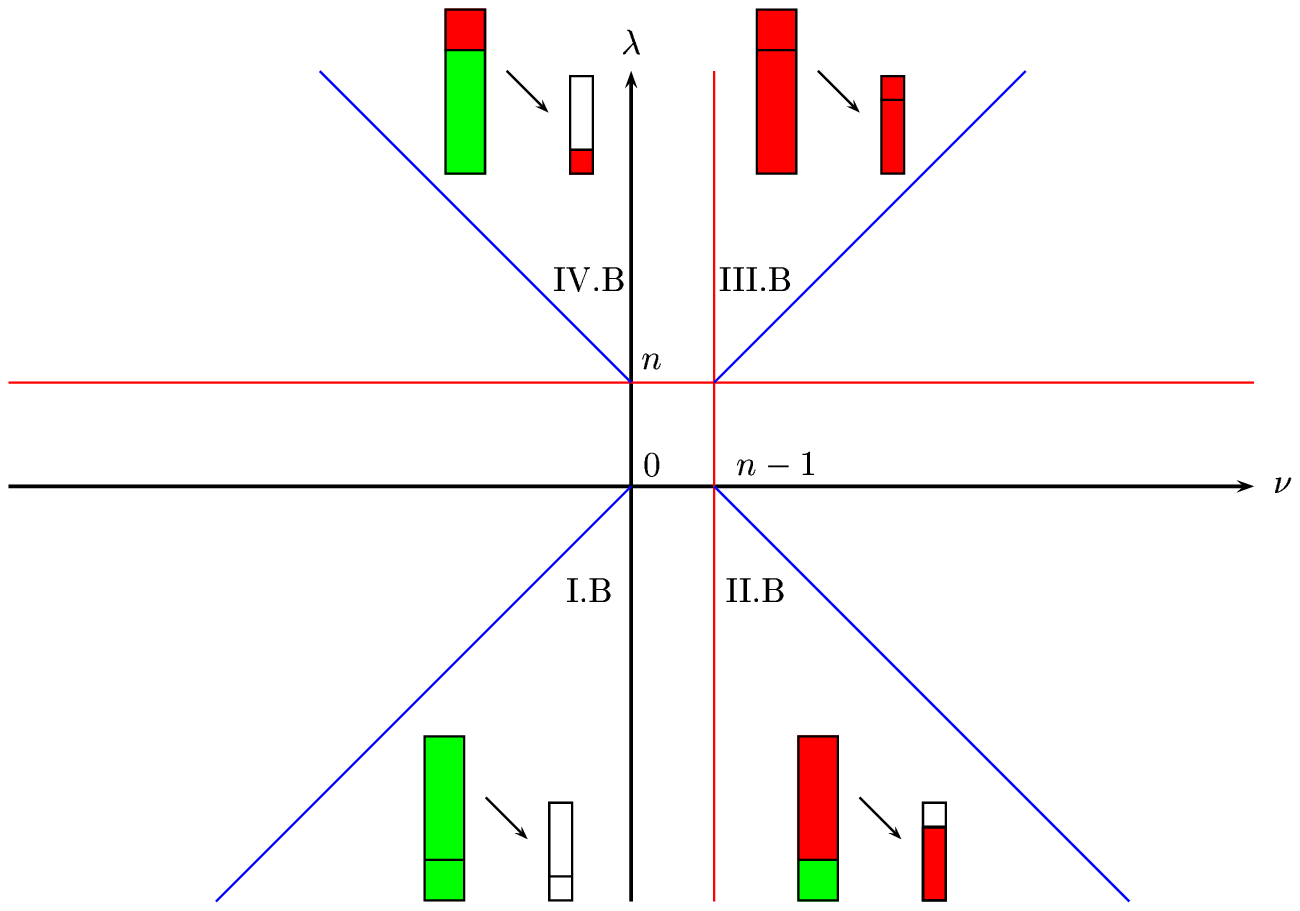}
\end{center}
\caption{Image of $\A_{i,j}$ with $(i,j) 
\in {\mathcal{L}}_{even}$
}
\label{fig:2.3}
\end{figure}
\medskip

If $(i,j) \in L_{\operatorname{even}}$, 
namely,
 if $(i,j) \in {\mathcal{L}}_{even}$
 with $i \le j \le 0$, 
 then the multiplicity $m(I(i),J(j))=2$
 and $H(i,j)$ is spanned by $\widetilde{\widetilde{{\mathbb{A}}}}_{i,j}$ and $ {\tilde{{\mathbb{C}}}}_{i,j} $
 by Theorem \ref{thm:II.5}. 
The image of $\widetilde{\widetilde{{\mathbb{A}}}}_{i,j}$
 is finite-dimensional
 and since the restriction to the finite-dimensional subrepresentation
 is nontrivial it induces an $G'$-equivariant operator
 between the finite-dimensional representations
 $F(-i)$ and $F(-j)$.  
By Theorem \ref{thm:ImageF}
 the image of $I(i)_K$
 under ${\tilde{{\mathbb{C}}}}_{i,j}$ is equal to $J(-j)_{K'}$
 and the finite-dimensional representation is not in the kernel
 by Theorem \ref{thm:kernel} (5).

\medskip
Figure \ref{fig:2.4} represents the results for the operators
 $\AAt_{i,j}$
and $\C_{i,j}$ with $(i,j) \in {L}_{\operatorname{even}}$

\medskip
\begin{figure}[h]
\index{sbon}{Leven@$L_{\operatorname{even}}$}
\begin{center}
\includegraphics[scale=1]{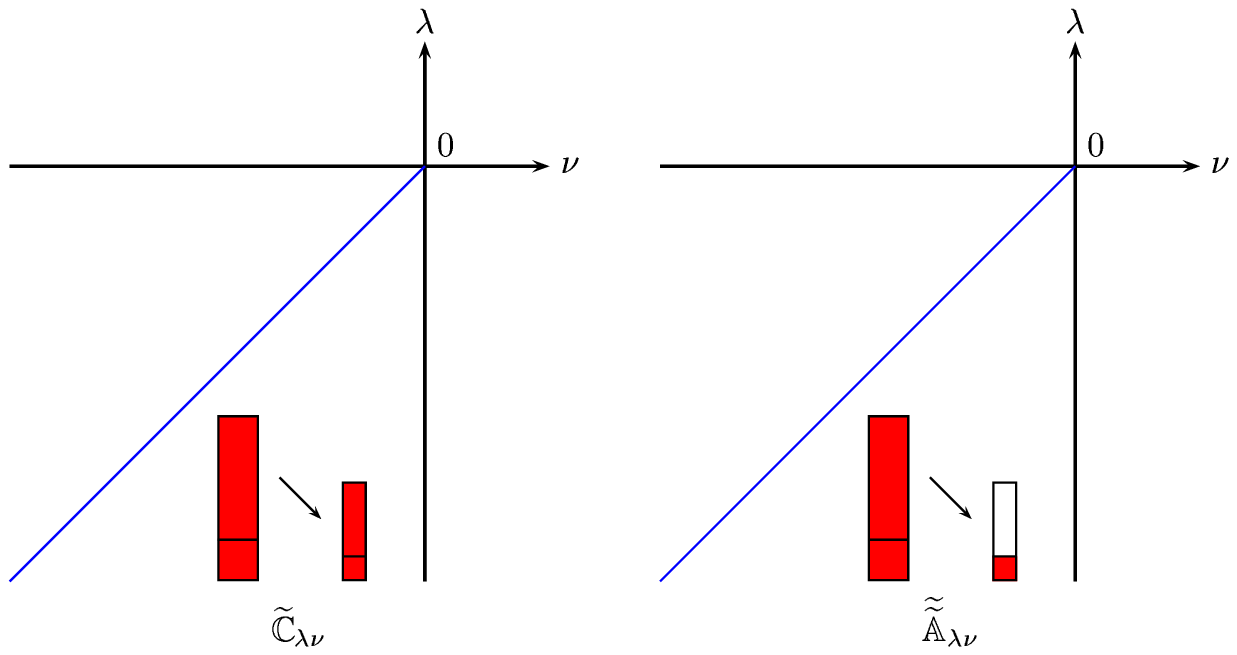}
\end{center}
\caption{Images of $\AAt_{i,j} $ and $\C_{i,j}$
 with $(i,j) \in L_{\operatorname{even}}$}
\label{fig:2.4}
\end{figure}
\medskip

\subsection{Multiplicities for composition factors}
The following theorem  generalizes the results by \cite{Lo} for $G=GL(2,{\mathbb{C}})$
 and $G'=GL(2,{\mathbb{R}})$.
\medskip
\begin{theorem} 
[multiplicities for composition factors] 
\label{thm:compo}
~~~
Let $i,j \in {\mathbb{N}}$.  
\begin{enumerate}
\item[{\rm{(1)}}]
Suppose that $i>j$.  

\par\noindent
{\rm{(1-a)}} \enspace
Assume $i \equiv j \mod 2$, 
namely,
$(-i,-j)  \in L_{\operatorname{even}}$. 
Then 
\[  
   m (T(i) ,T(j) ) = 1, 
   \qquad
   m(T(i), F(j) )= 0, 
   \qquad
   m (F(i) ,F(j) )  = 1.  
\]
\par\noindent
{\rm{(1-b)}} \enspace
Assume $i \equiv j+1 \mod 2$, 
namely,
$(-i,-j) \in L_{\operatorname{odd}} $.   
Then
\[  m(T(i) ,T(j) ) = 0, 
\qquad
    m(T(i) ,F(j) ) = 1, 
\qquad
    m (F(i) ,F(j) )  = 0.\]
\item[{\rm{(2)}}]
 Suppose that $i<j$.  
Then 
\[
m(T(i),T(j))=0, 
\quad
m(T(i),F(j))=1, 
\quad
m(F(i),F(j))=0.  
\]
\end{enumerate}
\end{theorem}

\begin{proof}
The discussion in Section \ref{subsec:2.2}
 (see Figures \ref{fig:2.2} and \ref{fig:2.3}) shows
 that our symmetry breaking operators induce
\[
 m (T(i) ,T(j) ) \not = 0 \ \mbox{ and } m (F(i) ,F(j) )  = 1
\
\mbox{ for }(-i,-j)  \in L_{\operatorname{even}}, 
\]
\[
m(T(i) ,F(j) ) \not = 0
\quad
\text{for }
 (-i,-j) \in L_{\operatorname{odd}}, 
\]
and 
\[
m(T(i),F(j))\ne 0
\quad
\text{for }
i<j.  
\]
Hence by \cite{S-Z}
 the multiplicities are one and it suffices to show that the multiplicities are zero in the remaining cases. 
 
If $(-i,-j)  \in L_{\operatorname{even}}$
 and $m(T(i), F(j) ) \not = 0$, 
 then there would exist a nontrivial symmetry breaking operator
 $I(-i)\rightarrow J(-j)$
 with image in the subrepresentation $F(-j)$
 for which the finite-dimensional representation $F(i)$ is in the kernel. 
Since $F(i) $ is not in the kernel of $\widetilde{\widetilde{\mathbb A}}_{i,j} $ or $\widetilde{\mathbb C}_{i,j}$, 
 this would imply that $m(I(i),J(j))>2$, 
 contradicting Theorem \ref{thm:II.5}.

We have already shown 
 in Proposition \ref{prop:Hbranch}
 that $m(F(i), F(j))=0$
 if $(-i,-j) \in L_{\operatorname{odd}}$
 or $i<j$.  
Alternatively,
 this can be proved as follows:
Suppose that $(-i,-j)  \in L_{\operatorname{odd}}$
 or $i<j$.  
If $m(F(i), F(j) )\not= 0$, 
 then we would obtain an additional symmetry breaking operator
 for $(n+i,-j) $ in the octant IV.A or IV.B, 
 contradicting Theorem \ref{thm:II.5}.

Now suppose that $(-i,-j)\in L_{\operatorname{odd}}$
 or $i<j$.  
Similarly,
 if $m(T(i), T(j) )\not= 0$, 
 then we would obtain an additional symmetry breaking operator
 for $(\lambda,\nu)=(-i,n-1+j) $ in the octant II.A or II.B, 
 contradicting Theorem \ref{thm:II.5}.
\end{proof}

The following theorem determines 
 the multiplicity from principal series representations
 $I(\lambda)$ of $G$
 (not necessarily irreducible)
 to the irreducible representations 
 $F(j)$ and $T(j)$ of $G'$:
\begin{theorem}
\label{thm:mIF}
Suppose $j \in {\mathbb{N}}$. 
\begin{enumerate}
\item[{\rm{1)}}]
$
m(I(\lambda),F(j))
=1
\quad\quad
\text{for all }\,\, \lambda \in {\mathbb{C}}.  
$
\item[{\rm{2)}}]
$m(I(\lambda),T(j))
=
\begin{cases}
1
\quad&\text{if }\,\, \lambda+j \in -2{\mathbb{N}},
\\
0
\quad&\text{if }\,\, \lambda+j \not\in -2{\mathbb{N}}.
\end{cases}
$
\end{enumerate}
\end{theorem}

It is noteworthy 
 that there exist non-trivial symmetry breaking operators
 to the finite-dimensional representations $F(j)$, 
 whereas there do not exist
 to the infinite-dimensional irreducible representations
 $T(j)$
 for generic parameter $\lambda$.  

\begin{proof}
[Proof of Theorem \ref{thm:mIF}]
1) \enspace
For any $\lambda \in {\mathbb{C}}$
 and $\nu=-j \in {\mathbb{N}}$, 
 we have
\index{sbon}{Att@$\AAt_{\lambda,\nu}$}
\[
0 \ne \AAt_{\lambda, \nu}
 \in \operatorname{Hom}_{G'}(I(\lambda), J(\nu))
\]
by Proposition \ref{prop:nuneg}, 
and $\operatorname{Image}\AAt_{\lambda, \nu}=F(j)$
 by Theorem \ref{thm:ImageF} (2)
 (see also Theorem \ref{thm:image}).  
Hence $m(I(\lambda),F(j))\ge 1$.

On the other hand,
 in view of the inclusion relation
 for $\nu =-j \in {\mathbb{N}}$
\[
   \operatorname{Hom}_{G'}(I(\lambda), F(j))
    \subset
   \operatorname{Hom}_{G'}(I(\lambda), J(\nu)), 
\]
we have $m(I(\lambda),F(j))\le 1$
 if $(\lambda, -j) \notin L_{\operatorname{even}}$
 by Theorem \ref{thm:I.1}.

Suppose now
 that $(\lambda, -j) \in L_{\operatorname{even}}$.  
Then $\operatorname{Hom}_{G'}(I(\lambda), J(\nu))$
 is spanned by $\AAt_{\lambda, \nu}$
 and $\C_{\lambda,\nu}$
 with $\nu = - j$
 by Theorem \ref{thm:II.5}, 
 but $\operatorname{Image}\C_{\lambda, \nu} \supsetneqq F(j)$
 by Theorem \ref{thm:image} (1).  
Hence $m(I(\lambda),F(j))\le 1$
 if $(\lambda, -j) \in L_{\operatorname{even}}$, 
 too.  
Thus the first statement is proved.  
\par\noindent
2) For $\nu =m+j$, 
 we have from Theorem \ref{thm:I.1}
 and \eqref{eqn:TIF}, 
\[
   m(I(\lambda),T(j))\le m(I(\lambda),J(\nu))=1
\]
for any $\lambda \in {\mathbb{C}}$.

On the other hand, 
 if $\lambda + j \in -2 {\mathbb{N}}$, 
then $\A_{\lambda,\nu} \ne 0$
by Theorem \ref{thm:II.1}
 and $\operatorname{Image}\A_{\lambda,\nu} \subset T(j)$
 by Theorem \ref{thm:ImageT} (2).  
Hence $m(I(\lambda),T(j))=1$
 for $\lambda +j \in -2{\mathbb{N}}$.

Finally suppose $\lambda +j \not \in -2{\mathbb{N}}$.  
Then $\A_{\lambda,\nu}$ is nonzero
 but $\operatorname{Image}\A_{\lambda,\nu}$ is not 
contained in $T(j)$
 by Theorem \ref{thm:ImageT} (2).  
If $m(I(\lambda),T(j)) \ne 0$, 
 then we would obtain 
 an additional symmetry breaking operator from 
 $I(\lambda)$ to $J(\nu)$ 
 for $\nu=m+j$, 
 contradicting Theorem \ref{thm:I.1}.  
Thus Theorem \ref{thm:mIF} is proved.  
\end{proof}
\bigskip

\section{Symmetry breaking operators}
\label{sec:gen}

Although our main object
 is the pair
 of groups $(G, G')=(O(n+1,1),O(n,1))$, 
 the techniques 
 of this article are actually directed
 at the more general problems
 of determining symmetry breaking operators.  
In this chapter we study the distribution
 kernels
 of symmetry breaking operators
 between induced representations of a Lie group $G$
 and its subgroup $G'$ from their subgroups
 $H$ and $H'$, 
 respectively,
 in the general setting,
 and introduce the notion of 
 {\textit{regular}} ({\textit{singular}}, or {\textit{differential}}) 
 symmetry breaking operators
 in terms of the double coset $H' \backslash G/H$.  
When these representations are (possibly, degenerate)
 principal series representations
 of reductive groups, 
 we discuss a reduction
 to the analysis on an open Bruhat cell 
 under some mild condition.  

\subsection{Restriction of representations
 and symmetry breaking operators}
\label{subsec:contres}

Let $H$ be a closed subgroup of $G$.
Given a finite-dimensional representation 
$\lambda:H \to GL_{\mathbb{C}}(V)$,
we define the homogeneous vector bundle
\[
\mathcal{V}_X \equiv \mathcal{V} := G \times_H V
\]
over the homogeneous space $X := G/H$.
The group $G$ acts continuously on the space
$C^\infty(X,\mathcal{V})$
of smooth sections endowed with the natural Fr\'{e}chet topology.

Suppose that $G'$ is a subgroup of $G$, 
 and $H'$ is a closed subgroup of $G'$.  
Similarly,
 given a finite-dimensional representation
 $\nu:H' \to GL_{\mathbb{C}}(W)$,
 we have a continuous representation of $G'$
 on the Fr{\'e}chet space $C^{\infty}(Y,{\mathcal{W}})$, 
where
$\mathcal{W} := G' \times_{H'} W$
is the homogeneous vector bundle
 over $Y:=G'/H'$.

We denote by
\index{sbon}{H1@$H(\lambda, \nu)$|textbf}
\[
H(\lambda,\nu):=\operatorname{Hom}_{G'}(C^\infty(X,\mathcal{V}),C^\infty(Y,\mathcal{W}))
\]
the space of continuous $G'$-homomorphisms, 
{\it{i.e.}}, 
 {\it{symmetry breaking operators}}.  

\subsection{Distribution kernels of
  symmetry breaking operators}
\label{subsec:KT}

By the Schwartz kernel theorem a continuous linear operator 
$
   T:C^{\infty}(X, {\mathcal{V}}) \to C^{\infty}(Y, {\mathcal{W}})
$ 
 is given by a distribution kernel.
 In this section,
 we analyze the kernels of the symmetry breaking operators.

Let $\mathbb{C}_{2\rho}$ be the one-dimensional representation of
$H$ defined by
\[
h \mapsto |\det(\operatorname{Ad}_{G/H}(h):
     \mathfrak{g}/\mathfrak{h} \to \mathfrak{g}/\mathfrak{h})|^{-1}.
\]
The bundle of volume densities $\Omega_X$ of $X=G/H$ is given as a
$G$-homogeneous line bundle
$\Omega_X \simeq G \times_H \mathbb{C}_{2\rho}$.
Then the dualizing bundle of $\mathcal{V}$ is given,
as a homogeneous vector bundle,
by
\[
\index{sbon}{Vcast@${{\mathcal{V}}^{\ast}}$|textbf}
\mathcal{V}^* := (G \times_H V^\vee) \otimes \Omega_X
   \simeq G \times_H (V^\vee \otimes \mathbb{C}_{2\rho}), 
\]
where $V^\vee$ denotes the contragredient representation of $V$.

In what follows
$\mathcal{D}'(X,\mathcal{V}^*)$
denotes the space of $\mathcal{V}^*$-valued distributions.
\begin{remark}
\label{rem:distr}
We shall regard distributions as generalized functions 
\`{a} la Gelfand
\cite{GS}
(or a special case of hyperfunctions \`{a} la Sato)
rather than continuous linear forms on
$C_c^\infty(X,\mathcal{V})$.
The advantage of this convention 
is that the formula of the $G$-action
 (and of the infinitesimal action of the Lie algebra ${\mathfrak{g}}$)
 on ${\mathcal{D}}'(X,{\mathcal{V}}^{\ast})$
 is the same with that of $C^{\infty}(X,{\mathcal{V}}^{\ast})$.  
\end{remark}
\begin{proposition}\label{prop:Distker}
Suppose that $G'$ and $H$ are closed  subgroups of $G$ and that
 $H'$ is a closed subgroup of $G'$.  

\begin{enumerate}[\upshape 1)]
\item  
There is a natural injective map:
\begin{equation}\label{eqn:HinvD}
\operatorname{Hom}_{G'} (C^\infty(X,\mathcal{V}),C^\infty(Y,\mathcal{W}))
 \hookrightarrow (\mathcal{D}'(X,\mathcal{V}^*) \otimes W)^{\Delta(H')}, 
 \,\,
 T \mapsto K_T.  
\end{equation}
Here $H'$ acts diagonally via the action of $G \times H'$
on $\mathcal{D}'(X,\mathcal{V}^*) \otimes W$.
\item  
If $H$ is cocompact in $G$ (e.g.\ a parabolic subgroup of $G$),
then \eqref{eqn:HinvD} is a bijection.
\end{enumerate}
\end{proposition}

\begin{proof}
1) \enspace
Any continuous operator
$T: C^\infty(X,\mathcal{V}) \to C^\infty(Y,\mathcal{W})$
is given uniquely by a distribution kernel
$K_T \in \mathcal{D}' (X \times Y, \mathcal{V}^* \boxtimes \mathcal{W})$
owing to the Schwartz kernel theorem.
If $T$ intertwines with the $G'$-action,
then the distribution $K_T$ is invariant under the diagonal action of $G'$, 
 namely,
 $K_T(g' \cdot, g' \cdot)=K_T(\cdot, \cdot)$
 for any $g' \in G'$.
In turn,
the multiplication map
\[
m: G \times G' \to G, \ (g,g') \mapsto (g')^{-1} g
\]
induces a natural bijection
\begin{equation}\label{eqn:invaD}
m^{\ast}:\mathcal{D}' (X \times Y, \mathcal{V}^* \boxtimes \mathcal{W})^{\Delta(G')}
\overset{\sim}{\leftarrow}
(\mathcal{D}' (X,\mathcal{V}^*) \otimes W)^{\Delta(H')}.
\end{equation}
Thus we have proved the first statement.
\par\noindent
2)\enspace
Conversely,
any distribution
$K \in \mathcal{D}' (X \times Y, \mathcal{V}^* \boxtimes \mathcal{W})$
induces a linear map
\[
T: C^\infty (X,\mathcal{V}) \to \mathcal{D}'(Y,\mathcal{W})
\]
if $X$ is compact.
Further,
if $K$ is $G'$-invariant via the diagonal action,
then it follows from \eqref{eqn:invaD} that $Tf$ is
 a smooth section
 of the bundle ${\mathcal{W}} \to Y$
 for any
$f \in C^\infty(X,\mathcal{V})$ and
$T: C^\infty(X,\mathcal{V}) \to C^\infty(Y,\mathcal{W})$
is a continuous $G'$-homomorphism.
Therefore the injective morphism \eqref{eqn:HinvD} is also surjective.
\end{proof}

\medskip
In Proposition \ref{prop:Distker}, 
 the support of $K_T$
 is an $H'$--invariant closed
 subset in $G/H$.  
Thus the closed $H'$--invariant sets define
  a coarse invariant
 of a symmetry breaking operator:
\begin{equation}
\label{eqn:Suppinv}
\operatorname{Hom}_{G'}
 (C^\infty(X,\mathcal{V}), C^\infty(Y,\mathcal{W}))
\to 
\{\text{$H'$-invariant closed subsets
 in $G/H$}\},
\end{equation} 
\[
 T \mapsto \operatorname{Supp}K_T.  
\]

\medskip
In rest of the chapter we assume that:
\begin{align}
&\text{$H'$ has an open orbit on $G/H$.}
\label{eqn:GHopen}
\end{align}

\begin{definition}
\label{def:regular}
Let $U_i$ ($i=1,2,\cdots$) be the totality
 of $H'$-open orbits on $X=G/H$.  
A non-zero $G'$-intertwining operator
 $T:C^{\infty}(X,{\mathcal{V}}) \to C^{\infty}(Y,{\mathcal{W}})$
 is {\it{regular}}
 if $\operatorname{Supp}K_{T}$ contains
 at least one open orbit $U_i$.  
We say $T$ is {\it{singular}}
if $T$ is not regular,
 namely,
 if $\operatorname{Supp}K_{T} \subset X-\cup_i U_i$.  
\end{definition}

\medskip
We write  
$
\index{sbon}{H2@$H(\lambda, \nu)_{\operatorname{sing}}$|textbf}
H(\lambda,\nu)_{\operatorname{sing}}$
for  the space of singular symmetry breaking operators.  
\begin{example}
\label{ex:regular}
\begin{enumerate}
\item[{\rm{1)}}]
{\rm{(Knapp--Stein intertwining operators).}}
Here $G=G'$ and 
$H=H'$ is a minimal parabolic subgroup $P$
 of $G$ and  
$W$ and $V$ are irreducible representations of $P$.
The Bruhat decomposition determines the orbits of $P$ on $G/P$.  
Hence we have exactly one open orbit corresponding to the longest element
 in the Weyl group. 
Thus the intertwining operator
 corresponding to the longest element of the Weyl group is regular
 for generic parameter.    
See \cite{KS}.  
\item[{\rm{2)}}]
{\rm{(Poisson transforms for symmetric spaces).}}
Here $G=G'$, 
$H$ is a minimal parabolic subgroup $P$ of $G$, 
 $H'$ is a maximal compact subgroup $K$ of $G$, 
 $V$ is a one-dimensional representations of $P$ and $W$ is the trivial one-dimensional representation, 
 see \cite{Hel}.  
Since $KP=G$, 
 the assumption \eqref{eqn:GHopen} is satisfied
 and the {\rm{Poisson transform}} is a regular symmetry breaking operator.
More generally,
 if $(G,H')$ is a reductive symmetric pair,
 then $H'$ has finitely many open orbits
 on $G/P$
 and a similar integral transform 
(Poisson transform for the reductive 
 symmetric space $G/H'$)
 can be defined
 and has a meromorphic continuation
 with respect to the parameter
 of the one-dimensional representations
 of $P$
 for each $H'$-open orbit, 
 see \cite{Os}.  
\item[{\rm{3)}}]
{\rm{(Fourier transform for symmetric spaces).}}
If we switch the role of $H$ and $H'$ in 2), 
 the integral transforms
 can be defined as the adjoint of the Poisson transforms,
 and are said to be the {\rm{Fourier transforms}}
 for the Riemannian symmetric space $G/K$ \cite{Hel}
 and the reductive symmetric space $G/H'$.  
\item[{\rm{4)}}]
{\rm{(invariant trilinear form).}}
Let $P_1$ be a parabolic subgroup
 of a reductive group $G_1$, 
 $G = G_1 \times G_1$, 
 $H = P_1 \times P_1$, 
 $G'=\operatorname{diag}(G_1)$, 
 and $H'=\operatorname{diag}(G')$.  
The study of symmetry breaking operators
 is equivalent to that of invariant trilinear forms
 on $\pi_{\lambda_1}\otimes \pi_{\lambda_2}\otimes \pi_{\lambda_3}$
 where $\pi_{\lambda_i}=\operatorname{Ind}_{P_1}^{G_1} (\lambda_i)$
 $(i=1,2,3)$.  
See \cite{CKOP} for the construction
 of invariant trilinear forms
 and explicit formula of generalized Bernstein--Reznikov integrals
 in some examples
 where \eqref{eqn:GHopen} is satisfied.  
\item[{\rm{5)}}]
{\rm{(Jantzen--Zuckerman translation functor).}}
\enspace
Here $G=G' \times G'$ and $G'$ 
 is a diagonally embedded subgroup of $G$.  
We start with a brief review
 of (abstract) translation functors.  
Let $Z({\mathfrak {g}})$ be the center 
 of the enveloping algebra
 $U({\mathfrak {g}})$.  
Given a smooth admissible, 
 irreducible representation $\pi$
 and a finite-dimensional representation $F$ of $G'$, 
 we consider the restriction
 of the outer tensor product representation 
 of $G$ to $G'$:
\[
  \pi \otimes F=(\pi \boxtimes F)|_{G'}.  
\]
Since $\pi \otimes F$ is an admissible representation
 of finite length,
 the projection $\operatorname{pr}_{\chi}$
 to the component of a generalized infinitesimal character
 $\chi$ is well-defined.  
The functor $\pi \rightsquigarrow \operatorname{pr}_{\chi}(\pi \otimes F)$
 is called a translation functor.  
Geometrically 
 if $\pi$ is an induced representation
 $C^{\infty}(G/P, {\mathcal{V}})$
{}from a finite-dimensional representation
 $V$  of $P$, 
then the translation functor
 is a symmetry breaking operator
\[
     C^{\infty}(G/P, {\mathcal{V}}) \otimes F
     \to 
     C^{\infty}(G/P, {\mathcal{W}}), 
\]
where ${\mathcal{W}}=G \times_P W$
 and $W$ is a certain $P$-subquotient 
 (determined by $\chi$)
 of the finite-dimensional representation
 $(V \otimes F)|_P$.  
\end{enumerate}
\end{example}
If $T$ is a symmetry breaking operator, 
 the restriction of the distribution kernel 
 $K_T$
 to open $H'$-orbits $U_i$
 is a regular function.  
By using this, 
 we obtain an upper bound
 of linearly independent, 
 regular symmetry breaking operators
as follows:
We take $x_i \in U_i$, 
 and denote by $M_i'$ 
 the stabilizer of $H'$ at $x_i$, 
 and thus we have $H'/M_i' \simeq U_i$.  
\begin{proposition}
\label{prop:dimgen}
Assume \eqref{eqn:GHopen}.  Then
\[ 
  \dim H(\lambda,\nu)/H(\lambda,\nu)_{\operatorname{sing}}
  \le 
 \sum_i \dim \operatorname{Hom}_{M_i'}
 (\lambda|_{M_i'}, \nu|_{M_i'}).  
\]
\end{proposition}
\begin{proof}
Since the distribution kernel $K_T$
 is $H'$-invariant,
 the restriction of $K_T$
 to each open $H'$-orbit $U_i$ 
is a regular function which is determined uniquely by its value at a single point, e.g.,
$K_T ( x_i )$ at $x_i$.  
Further,
 $K_T(x_i)\in \operatorname{Hom}_{\mathbb{C}}
 (V,W)$
  inherits the $M_i'$-invariance from the $H'$-invariance
 of $K_T$, 
namely, 
we have $K_T (x_i)  \in
 \operatorname{Hom}_{M'_i} ( \lambda|_{M'_i}, \nu|_{M'_i} )$.  
\end{proof}

\begin{corollary}\label{prop:Msing}
If\/
$\operatorname{Hom}_{M'_i} (\lambda |_{M'_i}, \nu |_{M'_i}) = 0$
for all $i$, 
then any $G'$-intertwining operator
$T \in H(\lambda,\nu)$
is singular.
\end{corollary}

\begin{remark}
Even if
$\operatorname{Hom}_{M'_i} (\lambda |_{M'_i}, \nu |_{M'_i} ) = 0$,
there may exist a nonzero 
\textit{singular symmetry breaking} operator
$I (\lambda) \to J (\nu)$
 for specific parameters $\lambda$ and $\nu$. 
\end{remark}

\begin{corollary}
\label{cor:Regbd}
Assume both $V$ and $W$ are one-dimensional.  
If there are $N$ 
open $H'$-orbits on $G/H$, 
then 
\[
\dim H(\lambda,\nu)/H(\lambda,\nu)_{\operatorname{sing}} \le N.  
\]
In particular,
 if there exists a unique $H'$-orbit on $G/H$, 
then 
\[
\dim H(\lambda,\nu)/H(\lambda,\nu)_{\operatorname{sing}} \le 1.  
\]
\end{corollary}

\begin{remark}
If $H$ is a minimal parabolic subgroup of $G$
 then $N$ does not exceed the cardinality
 of the little Weyl group of $G$
 for any subgroup $H'$ of $G$
 (\cite[Corollary E]{K-O}).  
\end{remark}
\medskip
The role of the assumption \eqref{eqn:GHopen}
 is illuminated
 by the following:
\begin{proposition}
\label{prop:multinfty}
Suppose $G$ is an algebraic group,
 and the subgroups $H$ and $H'$ are defined algebraically.  
If \eqref{eqn:GHopen} is not fulfilled,
 then for any algebraic finite-dimensional representation
 $V$ of $H$, 
 there exists a finite-dimensional representation
 $W$ of $H'$
 such that 
\[
  \dim H(\lambda, \nu)=\infty. 
\]
\end{proposition}
\begin{proof}
The argument is similar to that of 
\cite[Theorem 3.1]{K-O},
 and we omit the proof.  
\end{proof}

\medskip
In the case $(G,G')$ is a reductive symmetric pair,
 and $H$, $H'$ are minimal parabolic subgroups
 of $G$, $G'$, 
respectively,
 the pairs $(G,G')$ satisfying 
 \eqref{eqn:GHopen} were classified recently
 \cite{xKMt}.  

\subsection{Differential intertwining operators}
\label{subsec:diffres}

Suppose further that 
\begin{equation}
\label{eqn:HHG}
H' \subset H \cap G'.  
\end{equation}
Then we have a natural homomorphism
$\iota: Y \rightarrow X$, 
 which is $G'$-equivariant.
Using the morphism $\iota$, 
 we can define the notion of differential operators in a wider sense
 than the usual:
\begin{definition}
\label{def:diff}
We say a continuous linear operator
 $T:C^{\infty}(X, {\mathcal{V}}) \to C^{\infty}(Y, {\mathcal{W}})$
 is a {\it{differential operator}}
 if
\[
\iota({\operatorname{Supp}}(Tf))
\subset {\operatorname{Supp}}f
\quad\text{for any }f \in C^{\infty}(X, {\mathcal{V}}).  
\]
\end{definition}

In the case $H' = H \cap G'$, 
 the morphism $\iota$ is injective, 
 and a differential operator $T$
 in the sense of Definition \ref{def:diff} is locally
 of the form
\[
  T=\sum_{(\alpha, \beta) \in {\mathbb{N}}^{\dim X}}
    g_{\alpha\beta}(y)
    \frac{\partial^{|\alpha|+|\beta|}}{\partial y^\alpha\partial z^\beta}
\]
where $g_{\alpha\beta}(y)$ are
 $\operatorname{Hom}(V,W)$-valued smooth functions on $Y$, 
 and the local coordinates
 $\{(y_i, z_i)\}$ on $X$
 are chosen
 in such a way
 that $\{y_i\}$
 form an atlas on $Y$.

We denote by 
\index{sbon}{H3@$H(\lambda,\nu)_{\rm{diff}}$|textbf}
\begin{equation*}
\index{sbon}{H3@$H(\lambda,\nu)_{\rm{diff}}$|textbf}
H(\lambda,\nu)_{\rm{diff}}
\equiv
\operatorname{Diff}_{G'}(\mathcal{V}_X,\mathcal{W}_Y), 
\end{equation*}
a subspace 
 of $H(\lambda,\nu)\equiv\operatorname{Hom}_{G'}(C^\infty(X,\mathcal{V}),C^\infty(Y,\mathcal{W}))$
 consisting of $G'$-intertwining differential operators.

It is widely known that, 
when $G=G'$ and $X=Y$,  
holomorphic $G$-equivariant differential operators
between holomorphic homogeneous bundles over the complex
flag variety of a complex reductive Lie group $G$
are dual to $\mathfrak{g}$-homomorphisms between
generalized Verma modules ({\it{e.g.}}, \cite{HJ}).  
The duality can be extended to our more general situation 
where two homogeneous bundles are defined over
two different base spaces $\iota:Y \rightarrow X$.
In order to give a precise statement,
 we need to take the disconnectedness of $H$
 into account
 (see Remark \ref{rem:Hdiff}).  
For this, 
 we regard the generalized Verma module
\[
   \operatorname{ind}_{\mathfrak{h}_{\mathbb{C}}}^{\mathfrak{g}_{\mathbb{C}}}(V) 
= U(\mathfrak{g}_{\mathbb{C}}) \otimes_{U(\mathfrak{h}_{\mathbb{C}})} V
\]
as a $(\mathfrak{g},H)$-module
where $V$ is an $H$-module.  
Here
the $H$-action
 on 
$\operatorname{ind}_{\mathfrak{h}_{\mathbb{C}}}^{\mathfrak{g}_{\mathbb{C}}}(V)$
is induced from the diagonal action of $H$ on
$U(\mathfrak{g}_{\mathbb{C}}) \otimes_{\mathbb{C}} V$.  
Likewise,
$\operatorname{ind}_{\mathfrak{h}'_{\mathbb{C}}}^{\mathfrak{g}'_{\mathbb{C}}}(W^\vee)$
is a $(\mathfrak{g}',H')$-module if $W$ is an $H'$-module.
We then have:
\begin{fact}[see {\cite[Theorem 2.7]{KP}}]\label{fact:Diff}
Suppose $H' \subset H \cap G'$.  
\begin{enumerate}
\item[{\rm{1)}}]
$T \in \operatorname{Hom}_{G'}(C^{\infty}(X, {\mathcal{V}}), C^{\infty}(Y, {\mathcal{W}}))$
 is a differential operator
 if and only if 
 $\operatorname{Supp}K_T = \{e H\}$
 in $G/H$.  
\item[{\rm{2)}}]
There is a natural bijection:
\[
\operatorname{Hom}_{(\mathfrak{g}',H')}
     (\operatorname{ind}_{\mathfrak{h}'_{\mathbb{C}}}^{\mathfrak{g}'_{\mathbb{C}}}(W^\vee),
       \operatorname{ind}_{\mathfrak{h}_{\mathbb{C}}}^{\mathfrak{g}_{\mathbb{C}}}(V^\vee))
   \overset{\sim}{\to} \operatorname{Diff}_{G'} (\mathcal{V}_X, \mathcal{W}_Y).
\]
\end{enumerate}
\end{fact}

\medskip
\begin{remark}
\label{rem:Hdiff}
The disconnectedness of $H'$ affects
 the dimension
 of the space of covariant differential operators. We will give an example of this in the case of Juhl's operators in Section \ref{subsec:Juhl}.   
\end{remark}

\medskip
When \eqref{eqn:GHopen}
 and \eqref{eqn:HHG} are satisfied, 
 we have the following inclusion relation:
\begin{equation}
\label{eqn:grH}
  H(\lambda,\nu) \supset H(\lambda,\nu)_{\operatorname{sing}}
  \supset H(\lambda,\nu)_{\operatorname{diff}}.  
\end{equation}
This filtration will be used
 in the classification 
 of the symmetry breaking operators
 in the later chapters
 (see Section \ref{subsec:grdim}, 
 for instance).  

\subsection{Smooth representations
 and intertwining operators}
\label{subsec:CW}

Suppose
 that $G$ is a real reductive linear Lie group.  
In this section we review quickly 
 the notion of Harish-Chandra modules
 and the Casselman--Wallach theory of Fr{\'e}chet globalization,
 and show a closed range property
 of intertwining operators
 (Proposition \ref{prop:CW} below).  

We fix a maximal compact subgroup $K$ of $G$.  
We may and do realize $G$
 as a closed subgroup 
 of $GL(n,{\mathbb{R}})$
 for some $n$
 such that $g \in G$
 if and only if ${}^tg^{-1} \in G$
 and $K=O(n) \cap G$.  
For $g \in G$ we define a map $\| \cdot \|
: G \to {\mathbb{R}}$
 by 
\[
\|g\|
:=\| g \oplus {}^t g^{-1}\|_{\operatorname{op}}
\]
where $\| \, \|_{\operatorname{op}}$
 is the operator norm of $M(2n,{\mathbb{R}})$.  

Let ${\mathcal{HC}}$ denote the category
 of Harish-Chandra modules
 where the objects are $({\mathfrak {g}}, K)$-modules
 of finite length,
 and the morphisms are $({\mathfrak {g}}, K)$-homomorphisms.

A continuous representation $\pi$ of $G$ 
on a Fr{\'e}chet space $U$
 is said to be of {\it{moderate growth}}
 if for each continuous semi-norm
 $|\cdot |$ on $U$ 
 there exist a continuous semi-norm
 $|\cdot |'$ on $U$ 
 and $d \in {\mathbb{R}}$
 such that
\[
   |\pi(g)u| \le \| g \|^d |u|'
\quad
\text{for }
g \in G, u \in U.  
\]

Suppose $(\pi, {\mathcal{H}})$ is a continuous representation 
 of $G$ on a Banach space ${\mathcal{H}}$.  
A vector $v \in {\mathcal{H}}$ is said
 to be {\it{smooth}}
 if the map
 $G \to {\mathcal{H}}$, 
 $g \mapsto \pi(g)v$ is of $C^{\infty}$-class.  
Let ${\mathcal{H}}^{\infty}$ denote
 the space
 of smooth vectors
 of the representation $(\pi,{\mathcal{H}})$.  
Then ${\mathcal{H}}^{\infty}$ carries 
 a Fr{\'e}chet topology
 with a family of semi-norms
$\|v\|_{i_1\cdots i_k}:=\|d\pi(X_{i_1}) \cdots d\pi(X_{i_k})v\|$.  
Here $\{X_1, \cdots, X_n\}$ is a basis
 of ${\mathfrak {g}}$.  
Then ${\mathcal{H}}^{\infty}$ is a $G$-invariant subspace
 of ${\mathcal{H}}$, 
 and $(\pi,{\mathcal{H}}^{\infty})$ 
 is a continuous Fr{\'e}chet representation of $G$.  

We collect some basic properties
 (\cite[Lemma 11.5.1, Theorem 11.6.7]{W}):
\begin{fact}
\label{fact:CW}
{\rm{1)}}\enspace
If $(\pi, {\mathcal{H}})$ is a Banach representation,
 then $(\pi,{\mathcal{H}}^{\infty})$
 has  moderate growth.

{\rm{2)}}\enspace
Let $U_1, U_2$ be continuous representations
 of $G$ 
 having moderate growth
such that the underlying $({\mathfrak {g}},K)$-modules
 $(U_1)_K, (U_2)_K \in {\mathcal{HC}}$.  
If $T:(U_1)_K \to (U_2)_K$
 is a $({\mathfrak {g}},K)$-homomorphism,
 then $T$ extends to a continuous $G$-intertwining operator
 $\overline T:U_1 \to U_2$,
 with closed image
 that is a topological summand of $U_2$.  
\end{fact}

Let $P$, $P'$ be parabolic subgroups of $G$, 
 and ${\mathcal{V}}$, ${\mathcal{W}}$
 the $G$-equivariant vector bundles
 over the real flag varieties 
 $X=G/P$, $Y=G/P'$ 
 associated to finite-dimensional representations
 of $V$, $W$
 of $P$, $P'$, 
 respectively. 
(In the setting here, 
 $G'=G$
 in the notation 
 of the previous sections.)

\begin{proposition}
[closed range property]
\label{prop:CW}
If $T:C^{\infty}(X, {\mathcal{V}})_K
 \to C^{\infty}(Y, {\mathcal{W}})_K$
 is a $({\mathfrak {g}},K)$-homomorphism,
 then $T$ lifts uniquely 
 to a continuous $G$-intertwining operator 
 $\overline T$ from 
 $C^{\infty}(X, {\mathcal{V}})$
 to $C^{\infty}(Y, {\mathcal{W}})$, 
 and the image of $\overline T$ is closed 
 in the Fr{\'e}chet topology of $C^{\infty}(Y, {\mathcal{W}})$.  
\end{proposition}
\begin{proof}
[Proof of Proposition \ref{prop:CW}]
We fix a Hermitian inner product 
 on every fiber ${\mathcal{V}}_x$
 that depends smoothly on $x \in X$, 
and denote by $L^2(X, {\mathcal{V}})$
 the Hilbert space
 of square integrable sections
 to the Hermitian vector bundle 
 ${\mathcal{V}} \to X$.  
Then we have a continuous representation
 of $G$ 
 on the Hilbert space $L^2(X, {\mathcal{V}})$, 
 and the space $(L^2(X, {\mathcal{V}}))^{\infty}$
 of smooth vectors
 is isomorphic to $C^{\infty}(X, {\mathcal{V}})$
 as Fr{\'e}chet spaces
 because $X$ is compact 
 and $V$ is finite-dimensional.  
Therefore $C^{\infty}(X, {\mathcal{V}})$ has moderate growth
 by Fact \ref{fact:CW} (1).  
Furthermore,
 the underlying $({\mathfrak {g}},K)$-module
 $C^{\infty}(X, {\mathcal{V}})_K$ is admissible
 and finitely generated
 because $P$ is parabolic subgroup
 of $G$
 and $V$ is of finite length 
as a $P$-module.  
Likewise 
 $C^{\infty}(Y, {\mathcal{W}})$ has moderate growth
 with $C^{\infty}(Y, {\mathcal{W}})_K \in {\mathcal{HC}}$.  
Now Proposition follows from Fact \ref{fact:CW} (2).  
\end{proof}

\subsection{Symmetry breaking operators
 for principal series representations}
\label{subsec:GPmap}

{}From now on we assume that  $X$ and $Y$ are 
real flag varieties of real reductive groups $G$ and $G'$.
Let $P = MAN$ and $P' = M'A'N'$
be Langlands decompositions of parabolic subgroups
$G$ and $G'$, respectively,
satisfying 
\begin{equation}\label{eqn:compati}
M' = M \cap G',  \ 
A' = A \cap G',   \
N' = N \cap G',   \
P' = P \cap G'.   
\end{equation}
This condition is fulfilled, for example,
if the two parabolic subgroups $P$ and $P'$ are defined
by the same hyperbolic element of $G'$, 
 as in the setting of Section \ref{subsec:matrix}
 that we shall work with in this article.

\medskip
Let $\Pbar = MA\Nbar$ by the opposite parabolic subgroup,
and $\nbar$ the Lie algebra of $\Nbar$.
The composition $\nbar \hookrightarrow G \to G/P$,
$Z \mapsto \exp Z \mapsto (\exp Z)P/P$
gives a parametrization of the  open Bruhat cell of the real flag variety
$X = G/P$.
We shall regard $\nbar$ simply as an open subset of
$X = G/P$.
We trivialize the dualizing vector bundle
\index{sbon}{Vcast@${{\mathcal{V}}^{\ast}}$}
$\mathcal{V}^* \to X$ on the open Bruhat cell,
and consider the restriction map
\begin{equation}\label{eqn:restDn}
\mathcal{D}' (X,\mathcal{V}^*) \to \mathcal{D}'(\nbar) \otimes
   (V^\vee \otimes \mathbb{C}_{2\rho}).
\end{equation}
The natural $G$-action on
$\mathcal{D}'(X,\mathcal{V}^*)$
defines an infinitesimal action of the Lie algebra $\mathfrak{g}$
on $\mathcal{D}'(U,\mathcal{V}^*|_U)$
for any open subset $U$ of $X$.
This induces a $\mathfrak{g}$-action on
$\mathcal{D}'(\nbar) \otimes (V^\vee \otimes \mathbb{C}_{2\rho})$
so that \eqref{eqn:restDn} is a $\mathfrak{g}$-homomorphism.
Likewise we let $MA$ act on
$\mathcal{D}'(\nbar) \otimes (V^\vee \otimes \mathbb{C}_{2\rho})$
so that \eqref{eqn:restDn} is a
$(\mathfrak{g},M')$-map.

\medskip
From now on we assume that 
\begin{equation}\label{eqn:PNPG}
P' \Nbar P = G.
\end{equation}

\medskip
\begin{theorem}\label{thm:KT}
Suppose \eqref{eqn:compati} and \eqref{eqn:PNPG} are satisfied.
Then we have a natural bijection:
\[
\operatorname{Hom}_{G'}(C^\infty(G/P,\mathcal{V}),C^\infty(G'/P',\mathcal{W}))
  \simeq \mathcal{D}'(\nbar,\operatorname{Hom}_{\mathbb{C}}
             (V \otimes \mathbb{C}_{-2\rho},W))^{M'A',\mathfrak{n}}.
\]
\end{theorem}

\begin{remark}\label{rem:KT}
The right-hand side may be regarded as 
$\operatorname{Hom}(V,W)$-valued distributions on
$\nbar$ satisfying certain
$(M'A',\mathfrak{n}')$-invariance conditions
because $\dim \mathbb{C}_{-2\rho} = 1$.
In the main part of this article
we treat the case 
\[\dim V = \dim W = 1\]
and we  identify the distribution kernel $K_T$
with a distribution on $\nbar$.  
We shall see that \eqref{eqn:PNPG} is fulfilled
 for $(G,G')=(O(n+1,1),O(n,1))$
 in Corollary \ref{cor:PNPG}.  
\end{remark}

\begin{proof}[Proof of Theorem \ref{thm:KT}]
Since $P'\Nbar P = G$,
the restriction map \eqref{eqn:restDn} induces an injective morphism 
on $\Delta(P')$-invariant distributions:
\begin{equation}\label{eqn:Drestinj}
(\mathcal{D}'(X,\mathcal{V}^*)\otimes W)^{\Delta(P')}
 \hookrightarrow
(\mathcal{D}'(\nbar) \otimes 
   (V^\vee \otimes \mathbb{C}_{2\rho}) \otimes W)^{M'A',\mathfrak{n}'}.
\end{equation}
Conversely,
given a 
$K \in \mathcal{D}'(\nbar,\mathcal{V}^*|_{\nbar}) \otimes W$
we define for each $p' \in P'$
\[
p' \cdot K \in \mathcal{D}' (p' \cdot \nbar, \mathcal{V}^*|_{p'\cdot\nbar}) \otimes W.
\]
Then $p' \cdot K = K$ on the intersection
$\nbar \cap p' \cdot \nbar$ if
$T \in (\mathcal{D}'(\nbar) \otimes
   \operatorname{Hom}_{\mathbb{C}} (V \otimes \mathbb{C}_{-2\rho},W))^{M'A',\mathfrak{n}'}$.
This shows the surjectivity of the restriction map \eqref{eqn:Drestinj}.
Hence Theorem \ref{thm:KT} follows from Proposition \ref{prop:Distker}.
\end{proof}

\medskip

\subsection{Meromorphic continuation of symmetry breaking operators}
\label{subsec:MeroTK}

In order to discuss the meromorphic continuation
 of symmetry breaking operators for principal series representations,
we fix finite-dimensional representations 
\[\sigma:M \rightarrow GL_{\mathbb{C }}(V) \] and 
\[\tau: M' \rightarrow GL_{\mathbb{C}}(W).\]For
 $\lambda \in \mathfrak{a}_{\mathbb{C}}^*$ 
and $\nu \in (\mathfrak{a}'_{\mathbb{C}})^*$,
we define representations   $V_\lambda$ of $P$ and
$W_\nu$ of $P'$  by
\begin{alignat*}{4}
P &= MAN &&\ni m e^H n &&\mapsto \sigma(m) e^{\lambda(H)} &&\in GL_{\mathbb{C}}(V)
\\
P' &= M'A'N' &&\ni m' e^{H'} n' &&\mapsto \tau(m') e^{\nu(H')} &&\in GL_{\mathbb{C}}(W),
\end{alignat*}
respectively.
Here by abuse of notation, 
 we use $\lambda$ and $\nu$ also as
  parameters
 of representation spaces of $P$ and $P'$, 
respectively.  
(Later,
 we shall work with the case
 where $\sigma ={\bf{1}}$
 and $\tau={\bf{1}}$.)
We then have  homogeneous bundles
\[\mathcal{V}_\lambda := G \times_P V_\lambda\]
and
\[\mathcal{W}_\nu := G' \times_{P'} W_\nu \]
over the real flag varieties $G/P$ and $G'/P'$, 
 respectively.  
Observe that the pull-back 
 of the vector bundle ${\mathcal{V}}_{\lambda}$
 via the $K$-diffeomorphism
 $K/M \overset \sim \to G/P$
 is a $K$-homogeneous
 vector bundle
 ${\mathcal{V}}=K \times_M(\sigma, V)$.  
Similarly,
 the pull-back of the vector bundle ${\mathcal{W}}_{\nu}$
 yields a $K'$-homogeneous vector bundle
 ${\mathcal{W}}=K' \times_{M'}(\tau, W)$.  
Thus,
 even though the $G$-module
$C^\infty(G/P,\mathcal{V}_\lambda)$
and the $G'$-module
$C^\infty(G'/P',\mathcal{W}_\nu)$
have complex parameters $\lambda\in\mathfrak{a}_{\mathbb{C}}^*$
and $\nu\in(\mathfrak{a}'_{\mathbb{C}})^*$,
respectively,
but the spaces themselves
 can be defined independently of $\lambda$ and $\nu$
as Fr\'{e}chet spaces via the isomorphisms
\begin{equation}\label{eqn:Kpict}
C^\infty(G/P,\mathcal{V}_\lambda) \simeq C^\infty(K/M,{\mathcal{V}}),
\
C^\infty(G'/P',\mathcal{W}_\nu) \simeq C^\infty(K'/M',{\mathcal{W}}).
\end{equation}
Thus, 
 for a family of
continuous $G'$-homomorphisms
$T_{\lambda,\nu} : C^\infty(G/P,\mathcal{V}_\lambda) \to C^\infty(G'/P',\mathcal{W}_\nu)$, 
 we can define the holomorphic (or meromorphic)
dependence on the complex parameter $\nulambda$ 
via \eqref{eqn:Kpict}, 
  namely,
 we call a continuous linear map
 $T_{\lambda,\nu}$ depends holomorphically/meromorphically
 on $(\lambda,\nu)$
 if $T_{\lambda,\nu}(\varphi)$ depends 
 holomorphically/meromorphically
 on $(\lambda,\nu)$
 for any $\varphi\in C^{\infty}(K/M, {\mathcal{V}})$.

\medskip
For a family of distributions
 $D_{\lambda,\nu} \in \mathcal{D}'(\nbar,\operatorname{Hom}(V, W))$, 
 we say $D_{\lambda,\nu}$ depends holomorphically/meromorphically
 on the parameter $(\lambda,\nu)\in \Omega $
 if for every test function $F \in C^{\infty}({\mathfrak{n}}_-, V)$
 the function $D_{\lambda,\nu}(F)$ depends 
 holomorphically/meromorphically on $(\lambda,\nu)$.

\medskip
The following proposition asserts
 that the existence
 of holomorphic/meromorphic continuation
 of symmetry breaking operators
 is determined
 only by that of distribution kernels
 on the open Bruhat cell under a mild assumption.  

\begin{proposition}\label{prop:meroK}
Assume \eqref{eqn:compati} and \eqref{eqn:PNPG} are satisfied.
Let $\Omega' \subset \Omega$ be two open domains in
$\mathfrak{a}_{\mathbb{C}}^* \times (\mathfrak{a}'_{\mathbb{C}})^*$.
Suppose we are given a family of continuous $G'$-homomorphisms
\[T_{\lambda,\nu}: C^\infty(G/P,\mathcal{V}_\lambda)
   \to C^\infty(G'/P',\mathcal{W}_\nu)\]
for $\nulambda \in \Omega'$.
Denote by $K_{\lambda,\nu}$ the restriction
 of the distribution kernel
 to the open Bruhat cell and suppose that $K_{\lambda,\nu}$ depends holomorphically on $(\lambda,\nu) \in \Omega'$. 
We assume
\begin{equation}\label{eqn:holoOmega}
\text{$K_{\lambda,\nu}$ extends holomorphically to $\Omega$ as a
$\operatorname{Hom}(V,W)$-valued distribution on $\nbar$}.
\end{equation}
Then the family $T_{\lambda,\nu}$ of symmetry breaking operators also extends 
to a family of continuous $G'$-homomorphisms 
$T_{\lambda,\nu}:C^\infty(G/P,\mathcal{V}_\lambda)\rightarrow 
C^\infty(G'/P',\mathcal{W}_\nu)$ which depends holomorphically on $(\lambda, \nu)\in \Omega$.
\end{proposition}

\begin{proof}
By Theorem \ref{thm:KT},
it is sufficient to prove that if
\begin{equation}\label{eqn:Kman}
K_{\lambda,\nu} \in \mathcal{D}'(\nbar,\operatorname{Hom}(V_{\lambda-2\rho},
         W_\nu))^{M'A',\mathfrak{n}'}
\end{equation}
for all $\nulambda \in \Omega'$
then \eqref{eqn:Kman} holds for all
$\nulambda \in \Omega$.
This statement holds
 because the equation 
 for the $(M'A', {\mathfrak {n}})$-invariance
 in \eqref{eqn:Kman} depends
holomorphically on
$\nulambda \in \mathfrak{a}_{\mathbb{C}}^* \times (\mathfrak{a}'_{\mathbb{C}})^*$
and because $K_{\lambda,\nu}$
is a $\operatorname{Hom}(V,W)$-valued distribution on
$\nbar$ which depends holomorphically
on $\nulambda \in \Omega$.
\end{proof}

We can strengthen Proposition \ref{prop:meroK} by relaxing
the assumption \eqref{eqn:holoOmega}
 as in the following proposition,
which we will use in later chapters.
\begin{proposition}\label{prop:meroKK}
Retain the setting of Proposition \ref{prop:meroK}.
Then the same conclusion still holds if we replace the assumption \eqref{eqn:holoOmega}
by the following two assumptions:
\begin{itemize}
\item
$K_{\lambda,\nu}$ extends meromorphically to\/ $\Omega$
as a\/ $\operatorname{Hom}(V,W)$-valued distribution on\/ $\nbar$.
\item
There exists a dense subspace $Z$ of
$C^\infty(K/M, {\mathcal{V}})$
such that $T_{\lambda,\nu} \varphi$
is holomorphic on\/ $\Omega$ for any
$\varphi \in Z$.
\end{itemize}
\end{proposition}
\begin{proof}
This can be shown similarly
 to Proposition \ref{prop:meroK}.  
Thus we omit the proof.  
\end{proof}

\section{More about principal series representations}
\label{sec:4}
This chapter collects some basic facts on the principal
series representation of $G=O(n+1,1)$
 in a way
 that we shall use them later.  
Most of the material here is well-known.  
\bigskip

\subsection{Models of principal series representations}
\label{subsec:KNpicture}
To obtain a formula for the symmetry breaking operator and to obtain its analytic continuation  we work
 on models of the representations
\index{sbon}{Ilmd@${I(\lambda)}$}
 $I(\lambda)$ and $J(\lambda)$.

\bigskip

The isotropic cone
\[
\index{sbon}{Xi@$\Xi$|textbf}
     \Xi\equiv \Xi({\mathbb{R}}^{n+1,1})=
     \{(\xi_0, \cdots, \xi_{n+1}) \in {\mathbb{R}}^{n+2}:
        \xi_0^2+ \cdots + \xi_{n}^2 - \xi_{n+1}^2=0\}
     \setminus
     \{0\}.    
\]
is a homogeneous $G$-space.

For $\lambda \in \mathbb C$,
let
\begin{equation}\label{eq:2.9.1}
  C_{\lambda}^{\infty}(\Xi) := \{h \in C^\infty(\Xi):
                  h(t \widetilde{\xi}) = t^\lambda h(\widetilde{\xi}), \ 
 \text{ for any } \widetilde{\xi} \in \Xi, t \in {\mathbb{R}}^{\times}\},
\end{equation}
be the space  of smooth functions on $\Xi$ homogeneous of degree $\lambda$. 
Likewise we define
\[
{\mathcal{A}}_{\lambda}(\Xi)
\subset
C_{\lambda}^{\infty}(\Xi)
\subset
{\mathcal{D}}_{\lambda}'(\Xi)
\subset
{\mathcal{B}}_{\lambda}(\Xi)
\]
for the sheaves
 ${\mathcal{A}}$ (analytic functions), 
 ${\mathcal{D}}'$ (distributions), 
 and ${\mathcal{B}}$ (hyperfunctions).

We endow $C_{\lambda}^{\infty}(\Xi)$
 with the Fr{\'e}chet topology
 of uniform convergence
 of derivatives of finite order on compact sets.  
The group $G$ acts on $C_\lambda^\infty(\Xi)$ by left translations $l(g)$ and thus we obtain a representation
 $(l_{\Xi},C_{\lambda}^{\infty}(\Xi))$.  Then 
\begin{equation}
\label{eq:2.9.3}
 I(\lambda) 
 \simeq (l_{\Xi},C_{-\lambda}^{\infty}(\Xi)).  
\end{equation}
and we will from now
 on identify  $I(\lambda)$
 and its homogeneous model $(l_{\Xi},C_{-\lambda}^{\infty}(\Xi))$.

Let denote by 
\index{sbon}{Ilmdinf@${I(\lambda)^{-\infty}}$|textbf}
 $I(\lambda)^{-\infty}$
 the space of distribution vectors.  
Then as the dual of the isomorphism
 \eqref{eq:2.9.3}, 
 we have a natural isomorphism
\begin{equation}
\label{eqn:IDist}
I(\lambda)^{-\infty} \simeq {\mathcal{D}}_{-\lambda}'(\Xi), 
\end{equation}
see Remark \ref{rem:distr}
 for our convention.  
{

\medskip

The isotropic cone $\Xi$ covers the sphere $S^n =G/P$
\begin{alignat*}{6}
   &G/O(n) N \,\,&& \simeq\,\,  && \Xi
   && \quad g O(n)N && \mapsto && \ g p_+
\\
&  {\mathbb{R}}^{\times}\downarrow && && \downarrow {\mathbb{R}}^{\times}
   && \qquad \rotatebox[origin=c]{-90}{$\mapsto$}  && && \  \rotatebox[origin=c]{-90}{$\mapsto$}
\\
     &\,G/P && \simeq \,\, && S^n,
     &&
     \quad\  gP &&\mapsto && \  g p_+
\end{alignat*}
where 
\index{sbon}{pplus@${p_+}$|textbf}
$p_+ :={} {}^{t\!} (1, 0, \cdots, 0, 1) \in \Xi (\mathbb{R}^{n+1,1})$.

\medskip
For $b={}^t(b_1, \cdots, b_n) \in {\mathbb{R}}^n$, 
 we define unipotent matrices in $O(n+1,1)$ by
\index{sbon}{nminusb@$n_-(b)$|textbf}
\begin{align}
n(b):=& \exp(\sum_{j=1}^n b_j N_j^+)=I_{n+2} 
       + 
\begin{pmatrix} 
     - \frac 1 2 Q(b) &-{}^{t\!}b & \frac 1 2 Q(b) 
\\
       b & 0 & -b
\\
     - \frac 1 2 Q(b) &-{}^{t\!}b & \frac 1 2 Q(b) 
\end{pmatrix},
\nonumber
\\
n_-(b):=&\exp(\sum_{j=1}^n b_j N_j^-)= I_{n+2} 
       + 
\begin{pmatrix} 
     - \frac 1 2 Q(b) &-{}^{t\!}b & -\frac 1 2 Q(b) 
\\
       b & 0 & b
\\
     \frac 1 2 Q(b) &{}^{t\!}b & \frac 1 2 Q(b) 
\end{pmatrix}, 
\label{eqn:nbar}
\end{align}
where $N_j^+$ and $N_j^-$
 ($1 \le j \le n$) are 
 the basis elements of the nilpotent Lie algebras ${\mathfrak {n}}_+$
 and ${\mathfrak {n}}_-$
 defined in \eqref{eqn:ngen} and \eqref{eqn:nngen}
 respectively.  
We collect some basic formulae:
\index{sbon}{hypelt@$H$}
\begin{align}
\label{eqn:Nfixed}
    n(b) \begin{pmatrix} 1 \\ 0 \\ 1 \end{pmatrix}
 =&
\begin{pmatrix} 1 \\ 0 \\ 1 \end{pmatrix},
\quad
    n_-(b) \begin{pmatrix} 1 \\ 0 \\ -1 \end{pmatrix}
 =
\begin{pmatrix} 1 \\ 0 \\ -1 \end{pmatrix}
\\
{n_-} (Ab)
   =& \begin{pmatrix} 1 \\ & A \\ && 1 \end{pmatrix}
      n_- (b)
      \begin{pmatrix} 1 \\ & A^{-1} \\ && 1 \end{pmatrix}
\quad
\text{for $A \in O(n)$, }
\nonumber
\\
 {n_-} (-b)
   =& m_- {n_-} (b) m_-^{-1},
\nonumber
\\
 {n_-} (e^{-t} b)
   =& e^{tH} {n_-} (b) e^{-tH},
\label{eqn:invm}
\end{align}
where we set 
\index{sbon}{mminus@$m_-$|textbf}
\begin{equation}\label{eqn:m-}
m_- \ := \ 
  \left(
   \begin{array}{cc|c}
      -1  &        &   \\
        & I_{n}&      \\
           \hline
        &        &      -1
    \end{array}
    \right)\in K'.  
   \end{equation}
We note
 that $m_-$ does not belong to the identity component
 of $G'$
 (cf. Lemma \ref{lem:5.1}).

The $N_+$-action on the isotropic cone
\index{sbon}{Xi@$\Xi$}
 $\Xi$ is given in the coordinates as 
\begin{equation}
\label{eqn:nbxi}
  n(b) \begin{pmatrix} \xi_0 \\ \xi \\ \xi_{n+1} \end{pmatrix}
 =
\begin{pmatrix}
 \xi_0 -(b,\xi)
\\
 \xi 
\\ 
 \xi_{n+1}-(b,\xi)
\end{pmatrix}
+
\frac{\xi_{n+1} -\xi_0}{2}
\begin{pmatrix} Q(b)\\ -2 b\\ Q(b)\end{pmatrix}.  
\end{equation}
where
 $b \in {\mathbb{R}}^n$, 
 $\xi \in {\mathbb{R}}^n$ and $\xi_0, \xi_{n+1}\in\mathbb{R}$.

The intersections
 of the isotropic cone $\Xi$
 with the hyper planes
 $\xi_0 + \xi_{n+1}=2$
 or $\xi_{n+1}=1$
 can be identified with ${\mathbb{R}}^n$ or $S^n$, 
 respectively.  
We write down the embeddings
 $\iota_N:{\mathbb{R}}^n \hookrightarrow \Xi$
 and $\iota_K:S^n \hookrightarrow \Xi$
 in the coordinates as follows:
\begin{align}
\label{eqn:NXi}
\iota_N:&
\mathbb{R}^n \hookrightarrow \Xi,
\
{}^t(x,x_n) \mapsto 
{n_-} (x,x_n) p_+
   = \left(\begin{array}{l}
           1 - |x|^2 - x_n^2  \\
           2x                     \\
           2x_n                  \\
           1 + |x|^2 + x_n^2
     \end{array}\right), 
\\
\label{eqn:KXi}
\iota_K:&
S^n \to \Xi, 
\eta \mapsto (\eta,1).  
\end{align}

The composition of $\iota_N$ and the projection
\begin{equation}
\label{eqn:XiSn}
   \Xi \to \Xi/\mathbb{R}^\times \overset{\sim}{\to} S^n,
\quad
\xi \mapsto \frac{1}{\xi_{n+1}} (\xi_0,\dots,\xi_n)
\end{equation}
yields the conformal compactification of ${\mathbb{R}}^n$:
\begin{equation*}
\mathbb{R}^n \hookrightarrow S^n,
\quad
r \omega \mapsto \eta = (s, \sqrt{1-s^2} \, \omega ) 
= \Bigl( \frac{1-r^2}{1+r^2}, \frac{2r}{1+r^2} \omega \Bigr).  
\end{equation*}
Here $\omega \in S^{n-1}$
 and the inverse map is given by 
$
r = \sqrt{\frac{1-s}{1+s}}
$
 for $s \ne -1$.

The composition of $\iota_K$ and the projection \eqref{eqn:XiSn}
 is clearly the identity map on $S^n$.

We thus obtain two models
 of $I(\lambda)$:
\begin{definition}\label{def:model}
{\rm 1)}
{\rm{(compact model of $I(\lambda)$, $K$-picture)}}
\enspace
The restriction 
\index{sbon}{iotaKast@$\iota_K^*$|textbf}
$\iota_K^*:C_{-\lambda}^{\infty}(\Xi) \to C^\infty(S^n),
 h \mapsto h|_{S^n}$
 induces  for  $\lambda \in \mathbb C$
 an isomorphism of $G$-modules between $C_{-\lambda}^{\infty}(\Xi)$ 
 and a representation 
$\pi_{\lambda,K}$ on $C^\infty(S^n))$  .
\newline
{\rm 2)}\enspace
{\rm{(noncompact model of $I(\lambda)$, $N$-picture)}}\enspace
The restriction 
\index{sbon}{iotaNast@$\iota_N^*$|textbf}
$\iota^*_N: C_{-\lambda}^{\infty}(\Xi) \to C^\infty(\mathbb R^n),
 h \mapsto h|_{\mathbb R^n}$
 induces  for $ \lambda \in \mathbb C$ 
 an isomorphism of $G$-modules between $C_{-\lambda}^{\infty}(\Xi)$ 
 and a representation 
$\pi_{\lambda}$ on 
$ \iota_N^*(C_{-\lambda}^\infty(\Xi)))$.
\end{definition}

\medskip
In order to connect the two models directly,
 we define a linear map
 for each $\lambda \in {\mathbb{C}}$:
\index{sbon}{iotalmdast@$\iota_{\lambda}^*$|textbf}
\[
\iota_\lambda^* : 
                   C^\infty (S^n) 
\rightarrow 
         C^\infty (\mathbb{R}^n),
\  f \mapsto F
\]
 by
\begin{equation}
\label{eqn:KtoN}
F(r \omega) 
:= 
( 1+r^2)^{-\lambda}
 f  \Bigl(\frac{1-r^2}{1+r^2} , 
 \frac{2r}{1+r^2} \omega
\Bigr).
\end{equation}
Then
 the inverse of $\iota_{\lambda}^{\ast}$
 is given by
\[
(\iota_\lambda^*)^{-1} F (u_0, u)
= \Bigl| \frac{1+u_0}{2} \Bigr|^{-\lambda}
   F \Bigl( \frac{u}{1+u_0}, \frac{1}{1+u_0} \Bigr).
\]
We note that the parity condition 
$f (-\eta) = \pm f (\eta)$
($\eta \in S^n$) holds
if and only if
\begin{equation}
\label{eqn:Rparity}
F (-\frac{\omega}{r} ) = \pm r^{2\lambda} F(r \omega).
\end{equation}
Since $\iota_K^*$ is bijective
and
 $\iota_N^*$ is injective, 
 we have the commutative diagram 

\begin{alignat}{4}
&&&\hphantom{Mo}  C_{-\lambda}^{\infty}(\Xi) &&
\nonumber
\\
&&&\iota_K^*\swarrow &\searrow \iota_N^*&
\label{eq:2.9.4}
\\
&&\hphantom{M}C^\infty(S^n)
&\ \quad\quad \ \overset{\iota_\lambda^*}{\longrightarrow}
&& \iota_{\lambda}^{\ast}(C^\infty(S^n))
\subset C^\infty({\mathbb{R}}^n)
\nonumber
\end{alignat}
of representations of $G$.

\bigskip

We shall use \eqref{eq:2.9.4}, 
 in particular,
 for the description
 of the regular symmetry breaking operators
 $\A_{\lambda,\nu}$
 in \eqref{eqn:iNk} and \eqref{eqn:iKk}, 
 and the singular symmetry breaking operators 
 $\B_{\lambda,\nu}$
 in \eqref{eqn:iNkB} and \eqref{eqn:iKkB}.  

\medskip
We define a natural bilinear form
$
\langle \ , \ \rangle :
C_{-\lambda}^{\infty}(\Xi) 
\times 
C_{\lambda-n}^{\infty}(\Xi)
 \to \mathbb C
$
 by
\begin{align}
  \langle h_1, h_2 \rangle 
  &:=\int_{S^n} h_1(\iota_K(b)) h_2(\iota_K(b)) \ d b
\label{eqn:Gpair}
\\
  &=  \int_{\mathbb R^n} h_1(\iota_N(z)) h_2(\iota_N(z)) \ d z.
\notag
\end{align}
Here, $d b$ is the Riemannian measure on $S^{n}$.
The bilinear form is $G$-invariant,
 namely,
\begin{equation}
\label{eqn:Ginvpair}
\langle h_1(g^{-1}\cdot) , h_2(g^{-1}\cdot) \rangle
=
\langle h_1 , h_2 \rangle
\quad
\text{for }
g \in G, 
\end{equation}
and extends
 to the space of distributions:
\begin{equation}
\label{eqn:IDC}
I(\lambda) \times {\mathcal{D}}_{\lambda-n}'(\Xi)
\to 
{\mathbb{C}}.  
\end{equation}

For later purpose,
 we write down explicit formulae for the action of elements in the Lie algebras \index{sbon}{nplus@${\mathfrak {n}}_+$}
${\mathfrak{n}}_+$ and ${\mathfrak{n}}_-$ in the noncompact model.

\begin{lemma}
\label{lem:4.5}
Let $N_j^+$ and $N_j^-$
 be the basis elements 
 for ${\mathfrak {n}}_+$ and 
 ${\mathfrak {n}}_-$
 defined in \eqref{eqn:ngen} and \eqref{eqn:nngen}, 
 respectively. 
For $(x_1, \cdots, x_n) \in {\mathbb{R}}^n$
\[d\pi_\lambda(N_j^-) = \frac{\partial}{\partial x_j}\]
\[d\pi_\lambda(N_j^+) = -2\lambda x_j -2x_j \sum_{i= i} ^n x_i \frac{\partial}{\partial x_i }
+\sum_{i=1}^{n} x_i^2 \frac{\partial}{\partial x_j}\
\]
for $1 \leq j \leq n$.
\end{lemma}
\begin{proof}
In light of the formulae \eqref{eqn:nbxi}
 and \eqref{eqn:NXi}, 
 the action of the unipotent groups $N_{\pm}$ on $\Xi$
 is given as follows:
\[
   n(b) \iota_N(x) = c(b) \iota_N(\frac{x-|x|^2b }{c(b)})
\]
\index{sbon}{nminusb@$n_-(b)$}
\[n_-(b) \iota_N(x) = \iota_N(x+b),\]
where 
$
   c(b):= 1-2(b,x) +Q(b)|x|^2$. Hence for $F \in C_{-\lambda}^\infty(\Xi)
$
 we have
\[ 
 (\pi_\lambda(n_-(b)^{-1})F)(x) = F(x-b).  
\]
Now the results follow by differentiation.
\end{proof}

\subsection{Explicit $K$-finite functions
 in the noncompact model}
\label{subsec:Kfinite}
We give explicit formulae
 for $K$-finite functions
 in  $(\pi_{\lambda}, \iota_\lambda^*(C^\infty(S^n)))$
 in a way
 that we can take a control of the $M$-action as well.

The eigenvalues
 of the Laplacian $\Delta_{S^n}$ on the standard sphere $S^n$
 are of the form 
 $-L(L+n-1)$ for some $L\in {\mathbb{N}}$, 
 and we denote
\[
{\mathcal{H}}^L(S^n)
:=
\{\phi \in C^{\infty}(S^n)
:
\Delta_{S^n}\phi = -L(L+n-1) \phi\}.  
\]
Then the orthogonal group $O(n+1)$ acts irreducibly 
 on ${\mathcal{H}}^L(S^n)$ for every $L \in {\mathbb{N}}$, 
 and the direct sum $\bigoplus_{L=0}^{\infty}{\mathcal{H}}^L(S^n)$
 is a dense subspace in $C^{\infty}(S^n)$.  
The branching law
 with respect to the restriction 
 $O(n+1) \downarrow O(n)$
\[
\mathcal{H}^L (S^n)
\simeq \bigoplus_{N=0}^L \mathcal{H}^N (S^{n-1})
\]
is explicitly constructed by using the Gegenbauer polynomial 
(see \cite[Appendix]{KO3})
\[
I_{N\to L} :
\mathcal{H}^N (S^{n-1}) \to \mathcal{H}^L (S^n),
\]
\begin{equation}
\label{eqn:Ikl}
(I_{N\to L} \phi) (\eta_0, \eta)
:= |\eta|^N\phi \Bigl(\frac{\eta}{|\eta|}\Bigr)
   \tilde{\tilde{C}}_{L-N}^{\frac{n-1}{2}+N} (\eta_0).
\end{equation}
Here 
\index{sbon}{Ctt@$\tilde{\tilde{C}}_N^\mu (t)$}
$\Tilde{\Tilde{C}}_{N}^{\mu}(t)$ is the renormalized
 Gegenbauer polynomial 
 (see \eqref{eqn:nGeg} for the definition).  
Then
\index{sbon}{iotalmdast@$\iota_{\lambda}^*$}
\[
\iota_\lambda^* (I_{N\to L} \phi)
= (1 + r^2)^{-\lambda} \Bigl( \frac{2r}{1+r^2}\Bigr)^N
   \phi (\omega) \tilde{\tilde{C}}_{L-N}^{\frac{n-1}{2}+N} \Bigl(\frac{1-r^2}{1+r^2}\Bigr).
\]
In particular, 
$(1+r^2)^{-\lambda}$ is a spherical vector 
 in the noncompact model.  
We note that\/
$I_{N \to L}$ is $(1 \times O(n))$-equivariant 
 but  not $K'$-equivariant.

\medskip

For $\psi \in {\mathcal{H}}^N(S^{n-1})$
 and $h \in {\mathbb{C}}[s]$, 
 we set 
\begin{equation}\label{eqn:F12}
\index{sbon}{Fl@$F_{\lambda}[\psi,h]$|textbf}
  F_{\lambda}[\psi,h](r\omega)
  :=
  (1+r^2)^{-\lambda}
  (\frac {2r}{1+r^2})^N
  \psi(\omega)
  h(\frac{1-r^2}{1+r^2}).  
\end{equation}
The following proposition describes
 all $K$-finite functions
 in the noncompact model of $I(\lambda)$.

\begin{proposition}\label{prop:IKfinite}
\[
\iota_K^{\ast}(C^{\infty}(S^n)_K)
=
{\mathbb{C}}\operatorname{-span}
\{F_{\lambda}[\psi,h]
:N \in {\mathbb{N}}, 
 \psi \in {\mathcal{H}}^N(S^{n-1}), 
 h \in {\mathbb{C}}[s]
\}.  
\]
\end{proposition}

\begin{remark}
As abstract groups, 
 $K'$ and $M$ are isomorphic to the group
$O(n) \times O(1)$.
Proposition \ref{prop:IKfinite}
 respects the restrictions
 of the chain of subgroups  
$G \supset K \supset M$,
but not the chains $G \supset K \supset K'$.
\end{remark}
\begin{proof}[Proof of Proposition \ref{prop:IKfinite}]
Since $C^{\infty}(S^n)_K \simeq \bigoplus_{L=0}^{\infty}
{\mathcal{H}}^L(S^n)$, 
 we have 
\begin{align*}
\iota_{\lambda}^{\ast}(C^{\infty}(S^n)_K)
=& \bigoplus_{L=0}^{\infty}\iota_{\lambda}^{\ast}({\mathcal{H}}^L(S^n))
\\
=& \bigoplus_{L=0}^{\infty} \bigoplus_{N=0}^{L}
   \iota_{\lambda}^{\ast}(I_{N \to L}({\mathcal{H}}^N(S^{n-1})))
\\
=& \bigoplus_{N=0}^{\infty} \bigoplus_{L=N}^{\infty}
   \iota_{\lambda}^{\ast}(I_{N \to L}({\mathcal{H}}^N(S^{n-1}))).  
\end{align*}
The formula \eqref{eqn:Ikl} relates the values
 at the north/south poles
 of the spherical harmonics
{} to the initial data of a differential equation 
 on the equator.  
Since the Gegenbauer polynomials
 $ \tilde{\tilde C}_{m}^{\nu}(s)$ are polynomials
 of degree $m$
 and since their highest order term
 does not vanish 
 if $\nu \notin -{\mathbb{N}}_+$, 
we have  
\[
  {\mathbb{C}}\operatorname{-span}
  \{\tilde{\tilde{C}}_{L-N}^{\nu}(s):L \ge N\}
  =
  {\mathbb{C}}[s]
\]
 if $\nu \notin -{\mathbb{N}}_+$.  
This completes the proof.  
\end{proof}

\bigskip

\subsection{Normalized Knapp--Stein intertwining operator}
\label{subsec:Knapp}
We now review the Knapp--Stein intertwining operators
 in the noncompact picture
 for the group $G'=O(n,1)$.  

\bigskip
The Riesz distribution
\[
r^\nu := (x_1^2 +\dots+ x_m^2)^{\frac{\nu}{2}}
\]
is a locally integrable function on ${\mathbb{R}}^m$
 if $\operatorname{Re}\nu> -m$, 
 and satisfies the following identity:

\begin{equation}
\label{eqn:bfunction}
     \Delta_{\mathbb{R}^m}(r^{\nu+2})=(\nu+2)(\nu+m)r^{\nu}, 
\end{equation}
{}from which we see 
 that $r^{\nu}$, 
 initially holomorphic in $\operatorname{Re} \nu > -m$, 
 extends to a tempered distribution 
with meromorphic parameter $\nu$
 in $\operatorname{Re}\nu >-m-2$.  
By an iterated argument,
 we see
 that $r^{\nu}$ extends meromorphically
 in the entire complex plane.  
The poles are located 
 at $\nu \in \{-m, -m-2, -m-4, \cdots\}$, 
 and all the poles are simple.  
Therefore, 
if we normalize it by 
\begin{equation}\label{eqn:rtilde}
\index{sbon}{rt@$\widetilde{r}^\nu$|textbf}
\widetilde{r}^\nu
  := \frac{1}{\Gamma(\frac{\nu+m}{2})} r^\nu, 
\end{equation}
then $\widetilde{r}^\nu$ is a tempered distribution on
$\mathbb{R}^m$ with holomorphic parameter
$\nu\in\mathbb{C}$.  
The residue of $r^\nu$ at
$\nu \in -m-2\mathbb{N}$ is given by 
\begin{align}
\label{eqn:resr}
  \widetilde r^{\nu}|_{\nu = - m -2k}
&  =
  \frac{(-1)^k \operatorname{vol}(S^{m-1}) \Delta_{\mathbb{R}^m}^k \delta (x)}
       {2^{k+1} m \cdots (m+2k-2)}
\\
& = \frac{(-1)^k \pi^{\frac{m}{2}}}
             {2^{2k} \Gamma(\frac{m}{2}+k)}
     \Delta_{\mathbb{R}^m}^k \delta(x), 
\nonumber
\end{align}
see \cite[Ch. I, \S 3.9]{GS}.  
Here,
 $\delta(x)$ is the Dirac delta function
 in ${\mathbb{R}}^m$
 and 
 $\operatorname{vol}(S^{m-1})$
 is the volume
 of the standard sphere $S^{m-1}$ in ${\mathbb{R}}^m$, 
 which is equal to $\frac{2\pi^{\frac{m}{2}}}{\Gamma(\frac{m}{2})}$.  
In particular, 
\[
\widetilde r^{\nu}|_{\nu=-m}
=\frac{\pi^{\frac m 2}}{\Gamma(\frac m 2)}\delta(x).  
\]

We define the Fourier transform 
 ${\mathcal{F}}_{{\mathbb{R}}^m}$
 on the space ${\mathcal{S}}'({\mathbb{R}}^m)$
 of tempered distributions
 by 
\[
   {\mathcal{F}}_{{\mathbb{R}}^m}:
   {\mathcal{S}}'({\mathbb{R}}^m)\to {\mathcal{S}}'({\mathbb{R}}^m), 
   \quad
   f(x) \mapsto ({\mathcal{F}}_{{\mathbb{R}}^m}f)(\xi)
   =\int_{-\infty}^{\infty}f(x) e^{-\langle x, \xi \rangle} dx.  
\]
\begin{lemma} Let $\rho:=(\xi_1^2+\cdots+\xi_m^2)^{\frac 12}$.  
With the normalization \eqref{eqn:rtilde},
we have
\begin{equation}\label{eqn:Fxpower}
  {\mathcal{F}}_{\mathbb{R}^m}
  (\widetilde {r}^{2(\nu -m)})
  = 2^{2 \nu-m}\pi^{\frac m 2}\widetilde{\rho}^{m-2\nu}, 
\end{equation}
\begin{equation}
\label{eqn:rstarr}
\widetilde r^{2(\nu-m)} \ast \widetilde r^{-2\nu}
=
\frac{\pi^m}{\Gamma(m-\nu)\Gamma(\nu)} \delta(x).  
\end{equation}
\end{lemma}

\begin{proof}
The first formula follows from
\begin{equation}
\label{eqn:Fr}
  {\mathcal{F}}_{\mathbb{R}^m}(r^{\nu})(\rho)
  =
  2 ^{\nu+ m} \pi^{\frac m 2}
  \frac{\Gamma(\frac{\nu+m}{2})}{\Gamma(-\frac{\nu}{2})}
  \rho^{-\nu-m}
\end{equation}
and its analytic continuation
 \cite[Ch.2, \S 3.3]{GS}.

The second formula is the inverse Fourier transform of
the following identity:
\begin{align*}
  {\mathcal{F}}_{{\mathbb{R}}^m}
(\widetilde r^{2(\nu-m)} \ast \widetilde r^{-2\nu})
=&{\mathcal{F}}(\widetilde r^{2(\nu-m)})
  {\mathcal{F}}(\widetilde r^{-2\nu})
\\
=&(2^{2\nu-m}\pi^{\frac m 2}\widetilde \rho^{m-2\nu})
(2^{m-2\nu}\pi^{\frac m 2}\widetilde \rho^{2\nu-m})
\\
=& \frac{\pi^m}{\Gamma(m-\nu)\Gamma(\nu)} {\bf{1}}.  
\end{align*}
\end{proof}

We review  now the Knapp--Stein intertwining operators for
$G' = O(n,1)$ in the noncompact model.
We set $n=m+1$ as before.
The finite-dimensional representation
$F(k)$ ($k=0,1,2,\cdots$)
 of $G'$ occurs as the unique submodule 
 of $J(-k)$
 and as the unique quotient of $J(k+n-1)$.  

In the noncompact model of $J(\nu)$, 
 the normalized Knapp--Stein intertwining operator
\begin{equation*}
\index{sbon}{ttt4@$\T{\nu}{m-\nu}$|textbf}
  \T{\nu}{m-\nu}
  : J(\nu) \to J(-\nu+m),
  \end{equation*}
 is the convolution operator
 with $\widetilde{|x|}^{2(\nu-m)}$, 
{\it{i.e.}}, 
\begin{equation}\label{eqn:KSdef}
   (\T{\nu}{m-\nu}f)(y)
  = \frac{1}{\Gamma(\nu-\frac m 2)}
     \int_{\mathbb{R}^m} |x-y|^{2(\nu-m)}f(x) dx.  
\end{equation}

 By \eqref{eqn:rstarr}, 
we recover a well-known formula:
\begin{equation}\label{eqn:TTKS}
  \T{\nu}{m-\nu}
  \circ 
   \  \T{m-\nu}{\nu}
  = \ 
  \frac{\pi^m}{\Gamma(m-\nu)\Gamma(\nu)}
  \operatorname{id}.
\end{equation}

The following formula is more informative,
and also it gives an alternative proof of \eqref{eqn:TTKS}.
\begin{proposition}\label{prop:T1}
The Knapp--Stein intertwining operator
\[  \T{\nu}{m-\nu}: J(\nu) \to J(-\nu+m)\]
acts on spherical vectors as follows:
\begin{equation*}
\index{sbon}{1nu@${\mathbf{1}}_{\nu}$}
\T{\nu}{m-\nu} ({\mathbf{1}}_\nu)
=
\frac{\pi^{\frac{m}{2}}}{\Gamma(\nu)}
\mathbf{1}_{-\nu+m}.
\end{equation*}
\end{proposition}

\begin{proof}
Since $K'$-fixed vectors 
 in the principal series representation
 are unique 
up to scalar,
there exists a constant $c$ depending on $\nu$ such that
\begin{equation}\label{eqn:T1c}
\T{\nu}{m-\nu} (\mathbf{1}_\nu)
=
c\mathbf{1}_{-\nu+m}.
\end{equation}

Let us find the constant $c$.
We recall from Section \ref{subsec:Kfinite}
 that the normalized spherical vector
 in the noncompact model
 of $J(\nu)$
 is given by $(1+|x|^2)^{-\nu}$.  
Therefore, 
 the identity \eqref{eqn:T1c} amounts to
\[
|\tilde{x}|^{2(\nu-m)} * (1+|x|^2)^{-\nu}
=
c (1+|x|^2)^{m-\nu}.
\]
Taking the Fourier transform,
we get
\[
2^{2\nu-m} \pi^{\frac m 2} |\tilde{\xi}|^{m-2\nu}
\cdot \frac{2\pi^{\frac m 2}}{\Gamma(\nu)}
\widetilde{K}_{\frac m 2 - \nu} (|\xi|)
=
c \frac{2\pi^{\frac m 2}}{\Gamma(m-\nu)}
\widetilde{K}_{\nu-\frac m 2} (|\xi|)
\]
by the integration formulae \eqref{eqn:Fxpower} and \eqref{eqn:Fone}. 
Here $\widetilde{K}_{\frac m 2 - \nu}$
 is a renormalized K-Bessel function, 
 see \eqref{eqn:KBessel}
 for the normalization.

By  definition
$|\tilde{\xi}|^{m-2\nu}
 = \frac 1{\Gamma(m-\nu)} |\xi|^{m-2\nu}$
and by the duality
$\widetilde{K}_{\nu-\frac m 2} (|\xi|)
 = ( \frac{|\xi|} 2 )^{m-2\nu} \widetilde{K}_{-\nu+\frac m 2} (|\xi|)$
(see \eqref{eqn:Kdual}),
we get
$c = \frac{\pi^{\frac m 2}}{\Gamma(\nu)}$.
\end{proof}

\medskip 
\begin{remark} 
\label{rem:4.8} 
By the residue formula \eqref{eqn:resr}
 of the Riesz distribution, 
 we see
 that the normalized Knapp--Stein intertwining operator
 $\T{\frac{m}{2}-j}{\frac{m}{2}+j}:
J(\frac{m}{2}-j) \to J(\frac{m}{2}+j)$ 
 reduces to a differential operator
 of order $2j$
 if $j \in {\mathbb{N}}$,
 which amounts to 
\begin{equation}\label{eqn:KSres}
\T{\frac{m}{2}-j}{\frac{m}{2}+j}
= \frac{(-1)^j \pi^{\frac{m}{2}}}
          {2^{2j} \Gamma(\frac{m}{2}+j)}
   \Delta_{\mathbb{R}^m}^j.  
\end{equation}
Combining with Proposition \ref{prop:T1}
 that for $\nu=\frac m 2 -j$
 ($j \in {\mathbb{N}}$), 
we get 
\begin{equation}
\label{eqn:D1}
  \Delta_{{\mathbb{R}}^m}^j {\bf{1}}_{\nu}
  =
  \frac{(-1)^j 2^{2j}\Gamma(\frac m 2+j)}{\Gamma(\frac m 2-j)}
  {\bf{1}}_{m-\nu}.  
\end{equation}
This formula \eqref{eqn:D1} is also derived
{} from the computation in \eqref{eqn:Lap0}.  

Conversely,
 any differential $G'$-intertwining operator
 between spherical principal series representations
 $J(\nu)$ and $J(\nu')$
 of $G'$ is 
 are of the form $\Delta^j:J(\frac m 2-j) \to J(\frac m 2+j)$
 for some $j\in {\mathbb{N}}$
 up to scalar multiple
 (see Lemma \ref{lem:Vermagg}).  
\end{remark}
 
\begin{remark}
For $\nu=m+j$
 ($j=0,1,\cdots$)
\[
   \T{m+j}{-j}:J(m+j) \to J(-j)
\]
 is given by the convolution of a polynomial of degree
 at most $2j$.  
In particular,
 the image of $\T{m+j}{-j}$
 is the finite-dimensional representation 
\index{sbon}{Fj@$F(j)$}
 $F(j)$.  
\end{remark}

\begin{remark}
Our normalization arises from analytic considerations
 and is not the same as the normalization introduced by Knapp and Stein
 in \cite{KS}.
\end{remark}
\bigskip

\section{Double coset decomposition $P' \backslash G/P$  }
\label{sec:double}

We have shown in Section \ref{subsec:KT}
 that the double coset $P' \backslash G/P$
 plays a fundamental role 
 in the analysis of symmetry breaking operators
{}from the principal series representation 
 $I(\lambda)$ of $G$
 to $J(\nu)$ of the subgroup $G'$.  
In general
 if a subgroup $H$ of a reductive Lie group $G$
 has an open orbit 
 on the real flag variety $G/P$
 then the number of $H$-orbits on $G/P$
 is finite
 (\cite[Remark 2.5 (4)]{K-O}).  
If $(G,H)$ is a symmetric pair,
then $H$ has an open orbit on $G/P$
 and the combinatorial description of
 the double coset space
 $H \backslash G/P$
 was studied in details by T. Matsuki.
For the symmetric pair 
 $(G,G')=(O(n+1,1),O(n,1))$, 
 we shall see
 that not only $G'$ but also a minimal parabolic subgroup $P'$
 of $G'$
 has an open orbit on $G/P$, 
and thus both $G' \backslash G/P$
 and $P' \backslash G/P$ are finite sets.  
In this chapter, 
 we give an explicit description
 of the double coset decomposition $G' \backslash G/P$ and  $P' \backslash G/P$.

\medskip

We recall from Section \ref{subsec:KNpicture}
 that the isotropic cone 
\index{sbon}{Xi@$\Xi$}
$\Xi\equiv \Xi({\mathbb{R}}^{n+1,1})$
 is a $G$-homogeneous space.  
We shall consider the $G'$-action
 (or the action of subgroups of $G'$)
 on $\Xi$, 
 and then transfer the orbit decomposition on $\Xi$
 to that on $G/P$ 
 by the natural projection:
\[
  \Xi \to \Xi/{\mathbb{R}}^{\times}
  \simeq S^n \simeq G/P.  
\]
We rewrite the defining equation of $\Xi$ 
 as 
\[
\xi_{0}^2+ \cdots +\xi_{n-1}^2-\xi_{n+1}^2=-\xi_{n}^2.
\]
Since $G'$ leaves the $(n+1)$-th coordinate $\xi_n$ invariant,
 the $G'$-orbit decomposition 
 of $\Xi({\mathbb{R}}^{n+1,1})$ is given 
 as 
\begin{equation}
\label{eqn:Xitwo}
\Xi({\mathbb{R}}^{n+1,1})
=
 \coprod_{\xi_{n} \in {\mathbb{R}} \setminus \{0\}}
 G' \cdot (0,\cdots, 0, \xi_{n}, |\xi_{n}|)
\amalg
 \Xi({\mathbb{R}}^{n,1}).  
\end{equation}
We set
\begin{align*}
  p_{\pm} :={}& {}^{t\!} (\pm 1, 0, \cdots, 0, 1) \in \Xi (\mathbb{R}^{n+1,1}), \\
  q_{\pm} :={}& 
            {}^{t\!}(0,\cdots, 0,\pm 1,1)\in \Xi({\mathbb{R}}^{n+1,1}).     
\end{align*}
Let $[p_{\pm}]$ and $[q_{\pm}]$
 denote the image of the points $p_\pm$ and $q_\pm$
 by the projection 
 $\Xi \to \Xi/{\mathbb{R}}^{\times} \simeq S^n \simeq G/P$.  
We begin with the following double coset decompositions
 $G' \backslash G/P$ and $G_0' \backslash G/P$:
\begin{lemma} 
\label{lem:5.1}
\
\begin{enumerate}
\item[{\rm{1)}}]  $G/P$ is a union of two disjoint $G'$-orbits.
We have
\index{sbon}{pplus@${p_+}$}
\begin{equation}
\label{eqn:GpGP}
  G/P = G'[q_+] \cup G'[p_+]
      \simeq  G' /O(n) \cup G'/P'.  
\end{equation} 
\item[{\rm{2)}}]
 Let $G_0'$, 
 $K_0'$ and $P_0'$ denote the identity components
 of $G'$, $K'$, 
 and $P'$, 
 respectively.  
(Thus $G_0' \simeq SO_0(n,1)$, 
 $K_0' \simeq SO(n)$, 
 and $P'$ is a minimal parabolic subgroup of $G'$.)
Then 
$G/P$ is a union of three disjoint orbits
 of $G_0'$.  
We have
\[
G/P =G_0'[q_+] \cup G_0'[q_-] \cup G_0'[p_+]
    \simeq G_0'/K_0' \cup G_0'/K_0'\cup G_0'/P_0'.  
\]
\end{enumerate}
\end{lemma}
\begin{proof}
1)\enspace
The first statement is immediate from \eqref{eqn:Xitwo}.  
Indeed, 
 the isotropy subgroup of $G'$
 at $[q_{\pm}]\in \Xi/{\mathbb{R}}^{\times} \simeq G/P$
 is $O(n) \times 1$.  
(We note
 that this subgroup is of index two
 in $K'$.)
The other orbit 
 $\Xi({\mathbb{R}}^{n,1})/{\mathbb{R}}^{\times}
 \simeq S^{n-1}$
 is closed 
 and passes through $[p_+]$.  
In view of \eqref{eqn:Nfixed}, 
the isotropy subgroup 
 at $[p_+]$ is $P'$.  
Thus the first statement is proved.  
\par\noindent
2)\enspace
By \eqref{eqn:GpGP}, 
 it suffices to consider the $G_0'$-orbit decomposition 
 on $G'[q_+] \simeq G'/O(n)$
 and $G'[p_+] \simeq G'/P'$.

We begin with the open $G'$-orbit
 $G'[q_+] \simeq G'/O(n)=O(n,1)/O(n)$, 
 which has two connected components.  
The connected group $G_0'$ has two orbits
 containing $[q_+]$ and $[q_-]=[m_-q_+]$, 
respectively
 where $m_- \in G' \setminus G_0'$
 was defined in \eqref{eqn:m-}.  
On the other hand,
 $G_0'$ acts transitively
 on the closed $G'$-orbit:
\[
  G_0'[p_+] \simeq
  G_0'/P_0'
  \simeq
  G'/P'
 \simeq 
  G'[p_+]
  \simeq
  S^{n-1}.  
\]
Thus the second claim is shown.  
\end{proof}

\begin{remark}
For $n=2$, 
 the action of $SO_0(2,1)$ on $S^2$
 is identified with the action of 
 $SL(2,{\mathbb{C}})$
 on ${\mathbb{P}}^1{\mathbb{C}}\simeq {\mathbb{C}} \cup \{\infty\}$.  
Then $p_+$, $p_-$, $q_+$
 and $q_-$
 correspond to $0$, $\infty$, $i$, 
 and $-i$, 
respectively.  
\end{remark}

If we set 
\[w_{\varepsilon} \  := \ 
  \left(
   \begin{array}{ccc|c}
        &        & -\varepsilon  & \\
        & I_{n-1}&     & \\
     \varepsilon  &        &     & \\
    \hline
        &        &     & 1
    \end{array}
    \right)\in K
    \quad
    \text{for } \varepsilon =\pm 1,   
\]
then we have    
\[
  q_\pm = w_{\pm}p_+ = w_{\mp} p_- .
\]

\medskip

We define subgroups
 of
\index{sbon}{Msubgp@$M$}
 $M$ by 
\begin{align}
M^{w}
 :={}
  &\{g \in M:g w_- = w_- g\}
  =\{g \in M:g w_+ = w_+ g\}
\notag
\\
 ={}
  & \left\{
    \begin{pmatrix} 
        \varepsilon \\ & B \\ & & \varepsilon \\ & & & \varepsilon
    \end{pmatrix} :
    B \in O(n-1), \   \varepsilon = \pm 1
    \right\}
    \simeq O(n-1) \times \mathbb{Z}_2, 
\notag
\\
M'_+ :={}& 
M^w \cap M' 
\simeq
\left\{ \begin{pmatrix} 1 \\ & B \\ && 1 \\ &&& 1 \end{pmatrix}
           : B \in O(n-1) \right\}.  
\label{eqn:m+}
\end{align}

\begin{remark}
We recall from \eqref{eqn:Mprime}
 that $M'=Z_{K}({\mathfrak {a}})$
 is isomorphic to $O(n-1) \times \mathbb{Z}_2$.  
The group $M^{w}$ is also isomorphic
 to $O(n-1) \times \mathbb{Z}_2$, 
 however,
 $M^w \ne M'$.  
In fact, 
 $M_+' =M^w \cap M'$
 is a subgroup of $M^w$
 of index two,
 and also is of index two in $M'$.  
\end{remark}

\medskip

Now the following proposition and corollary are
 derived directly from the description in Lemma~\ref{lem:5.1}.
\begin{proposition}\label{lem:PGP}
\

{\upshape 1)}\enspace
$G/P$ is a union of three disjoint $P'$-orbits. We have
\begin{equation}
  G/P = P'[q_+] \cup P'[p_-] \cup P'[p_+]
\label{eqn:PGP}
\end{equation}

{\upshape 2)}\enspace
The isotropy subgroups at\/ $[q_+]$, $[p_-]$, and $[p_+]$
are given by
\begin{alignat*}{3}
& S^n \setminus S^{n-1}
&& = P' [q_+] 
&& \simeq P' / M'_+
\\
& S^{n-1} \setminus \{ [p_+] \}
&& = P' [p_-]
&& \simeq P' / M' A'
\\
& \{ [p_+ ] \}
&& = P' [p_+]
&& = P' / P' .
\end{alignat*}
\end{proposition}

Thus the assumption \eqref{eqn:PNPG}
 of Theorem \ref{thm:KT}
 is fulfilled for $(G,G')=(O(n+1,1),O(n,1))$:
\medskip
\begin{corollary}
\label{cor:PNPG}
We have 
$P' N_- P = G.$
\end{corollary}
\bigskip

\section{Differential equations satisfied by the distribution kernels of  symmetry breaking operators}
\label{sec:3}

In this chapter
 we characterize the distribution kernel $K_T$
 of symmetry breaking operators
  for $(G,G')=(O(n+1,1),O(n,1))$.  
We derive a systems of differential equations on $\mathbb R^n$
 and prove that its distribution solutions 
$
   \mathcal{S}ol (\mathbb{R}^n; \lambda,\nu)
$
 are isomorphic to 
\index{sbon}{H1@$H(\lambda, \nu)$}
$
   H(\lambda,\nu)\equiv\operatorname{Hom}_{G'}(I(\lambda),J(\nu))$.
An analysis of the solutions shows that generically the multiplicity 
 $m(I(\lambda),J(\nu))$
 of principal series representations is 1.
  
\subsection{A system of differential equations
 for symmetry breaking operators}
For future reference,
we begin with a formulation in the vector-bundle case.

We have seen in \eqref{eqn:Suppinv}
 that the support of the distribution kernel 
 $K_T$ of a symmetry breaking operator
 $T:C^{\infty}(X,{\mathcal{V}}) \to C^{\infty}(Y,{\mathcal{W}})$
 is a $P'$-invariant closed subset 
 of $G/P$
 if ${\mathcal{V}}$ is a $G$-equivariant vector 
bundle over $X=G/P$
 and ${\mathcal{W}}$ is a $G'$-equivariant vector 
bundle over $Y=G'/P'$.  
By the description
 of the double coset space $P' \backslash G/P$
 for $(G,G')=(O(n+1,1),O(n,1))$
 in Proposition \ref{lem:PGP}, 
we get
\begin{lemma}\label{lem:SuppInv}
If
$T : C^{\infty}(X,{\mathcal{V}}) \to C^{\infty}(Y,{\mathcal{W}})$
is a nonzero continuous $G'$-homomorphism,
then the support of the distribution kernel $K_T$
 is one of 
\index{sbon}{pplus@${p_+}$}
$\{[p_+]\}$, 
 $\overline{P' [p_-]} = P' [p_-] \cup \{ [p_+] \}
 (\simeq S^{n-1})$,
or $G/P \simeq S^n$.
\end{lemma}

We recall from \eqref{eqn:nbar}
 that the open Bruhat cell of $G/P$ is given 
 in the coordinates
 by ${\mathbb{R}}^n \hookrightarrow G/P$, 
 $(x,x_n) \mapsto n_- (x,x_n)P$.

Then we have:
\begin{lemma}\label{lem:resinj}
There is a natural bijection:
\begin{equation}
\operatorname{Hom}_{G'} (C^{\infty}(G/P, {\mathcal{V}}), 
                    C^{\infty}(G'/P', {\mathcal{W}}))
\overset{\sim}{\to} \mathcal{D}' ( \mathbb{R}^n, \operatorname{Hom}
     (V \otimes {\mathbb{C}}_{-2\rho}
, W))^{M'A,\mathfrak{n}_+'}.
\label{eqn:Rinv}
\end{equation}
\end{lemma}

\begin{proof}
The assumption $P'\Nbar P = G$ 
 of Theorem \ref{thm:KT} is satisfied by Corollary \ref{cor:PNPG}.  
Thus Lemma follows.  
\end{proof}

\begin{remark}\label{rem:singR}
Recall from Definition \ref{def:regular}
that a non-zero symmetry breaking operator 
$T$ is
\textit{singular} if
$\operatorname{Supp} K_T \ne G/P$, equivalently,
$\operatorname{Supp} K_T \subset S^{n-1}$
 by Lemma \ref{lem:SuppInv}.  
Further, 
 $T$ is a differential operator if and only if
$\operatorname{Supp} K_T = \{ [p_+] \}$.
By Lemma \ref{lem:resinj},
 we do not lose any information
if we restrict $K_T$ to ${\mathbb{R}}^n$.  
Therefore, 
 $T$ is singular if and only if
$\operatorname{Supp} (K_T |_{\mathbb{R}^n}) \subset \mathbb{R}^{n-1}$.
$T$ is a differential operator if and only if
$\operatorname{Supp}(K_T|_{\mathbb{R}^n}) = \{0\}$.
\end{remark}

In \eqref{eqn:Rinv},
the invariance under $M' A$ for
$F \in \mathcal{D}' (\mathbb{R}^n$,
$\operatorname{Hom} (V \otimes {\mathbb{C}}_{-2\rho}, W))$
is written as 
\begin{alignat}{7}
& \tau(m)  && \,\circ\, && F(m^{-1} \cdot )  && \,\circ\,  && \sigma(m^{-1})  && = F
   \qquad&&\text{for $m \in M'_+ \simeq O(n-1)$},
\label{eqn:m+F}
\\
& \tau (m_-)  && \,\circ\,  &&  F ((-1) \cdot )  && \,\circ\,   && \sigma (m_-^{-1})   && = F, &&
\label{eqn:panty}
\\
& e^{t\nu}   &&   && F (e^t \cdot )   &&  && e^{(n-\lambda)t}   && = F
   \qquad&&\text{for any $t \in {\mathbb{R}}$}.  
\nonumber
\end{alignat}
Here, the identity \eqref{eqn:panty}
 for 
\index{sbon}{mminus@$m_-$}
$m_- \in M'$ (see \eqref{eqn:m-})
is derived from \eqref{eqn:invm}.

\vskip 1pc
Returning to the line bundle setting 
 as before, 
we obtain:
\begin{proposition}\label{prop:PDE}
Let
$T: I(\lambda) \to J(\nu)$ be
 any $G'$-intertwining operator.
Then the restriction
$K_T|_{\mathbb{R}^n}$
of the distribution kernel
 satisfies
the following system of differential equations:
\begin{align}
&(E - (\lambda - \nu - n)) F = 0,
\label{eqn:Fa}
\\
&( (\lambda-n) x_j - x_j E + \frac12 ( |x|^2 + x_n^2 ) \frac{\partial}{\partial x_j})
        F = 0  \quad (1 \le j \le n-1),
\label{eqn:Fn}
\end{align}
and the $M'$-invariance condition:
\begin{align}
\label{eqn:invM}
F(mx,x_n) =& F(x,x_n)
\quad\text{for any $m \in O(n-1)$},
\\
\label{eqn:parity}
F (-\tilde{x}) =& F(\tilde{x}).
\end{align}
Here $\tilde{x} = (x, x_n) \in \mathbb{R}^n$ and
$E = \sum_{j=1}^n x_j \frac{\partial}{\partial x_j}$.
\end{proposition}

\begin{proof}
We recall from Lemma \ref{lem:4.5}
 that the Lie algebra 
\index{sbon}{nplus@${\mathfrak {n}}_+$}
$\mathfrak{n}_+$ acts on
$\mathcal{B} (G/P, \sigma_{2\rho}^\vee)$
by 
\[
N_j \mapsto
   2(\lambda-n) x_j - 2x_j E + ( |x|^2 + x_n^2 ) \frac{\partial}{\partial x_j}
  \quad  (1 \le j \le n).  
\]
Hence
\eqref{eqn:Fn} is the invariance of 
\index{sbon}{nplusprime@${\mathfrak {n}}_+'$}
$\mathfrak{n}_+'$.
The remaining conditions \eqref{eqn:Fa}, \eqref{eqn:invM} and \eqref{eqn:parity} is the invariance of
$\mathfrak{a}$ and $m_- \in M'$ as above.
\end{proof}

For an open subset $U$ of $\mathbb{R}^n$ which is
$(O(n-1) \times O(1))$-invariant,
we define
\begin{equation}\label{eqn:KTeqU}
\index{sbon}{Sol@$\mathcal{S}ol (U; \lambda,\nu)$|textbf}
\mathcal{S}ol (U; \lambda,\nu)
:= \{ F \in \mathcal{D}'(U) :
       \text{$F$ satisfies \eqref{eqn:Fa}, \eqref{eqn:Fn},
                \eqref{eqn:invM}, and \eqref{eqn:parity}}\}
\end{equation}
Then by Lemma \ref{lem:resinj},
we have
\begin{proposition}\label{prop:HomPDE}
The correspondence $T \mapsto K_T$ gives a bijection:
\index{sbon}{Ilmd@${I(\lambda)}$}
\index{sbon}{Jnu@${J(\nu)}$}
\begin{equation}\label{eqn:HomSol}
\operatorname{Hom}_{G'} 
(I(\lambda),J(\nu)) \overset{\sim}{\to}
\mathcal{S}ol (\mathbb{R}^n; \lambda,\nu).
\end{equation}
\end{proposition}

\subsection{ The solutions  $\mathcal{S}ol (\mathbb{R}^n \setminus \{0\}; \lambda,\nu)$}

For a closed subset $S$ of $U$,
we define a subspace of $\mathcal{S}ol (U; \lambda,\nu)$
 by 
\index{sbon}{Sol@$\mathcal{S}ol (U; \lambda,\nu)$|textbf}
\index{sbon}{Sole@$\mathcal{S}ol_S (U; \lambda,\nu)$|textbf}
\[
\mathcal{S}ol_S (U; \lambda,\nu)
:= \{ F \in \mathcal{S}ol(U;\lambda,\nu): \operatorname{Supp} F \subset S\}.
\]
Then we have an exact sequence
\begin{equation}\label{eqn:Solexact}
0 \to \mathcal{S}ol_S(U;\lambda,\nu) \to \mathcal{S}ol(U;\lambda,\nu)
   \to \mathcal{S}ol(U \setminus S;\lambda,\nu).
\end{equation}
Applying \eqref{eqn:Solexact} to $U = \mathbb{R}^n$ and $S = \{0\}$,
we get
\begin{proposition}\label{prop:DHexact}
There is an exact sequence
\[
0 \to \operatorname{Diff}_{G'} (I(\lambda),J(\nu))
   \to \operatorname{Hom}_{G'} (I(\lambda),U(\nu))
   \to \mathcal{S}ol(\mathbb{R}^n \setminus \{0\};\lambda,\nu).
\]
Here 
\index{sbon}{H1@$H(\lambda, \nu)$}
$\operatorname{Diff}_{G'}(I(\lambda),J(\nu))
\equiv H(\lambda,\nu)_{\operatorname{diff}}$
denotes the space
 of differential symmetry breaking operators.  
\end{proposition}

\begin{proof}
As subspaces of \eqref{eqn:HomSol},
we have from Fact \ref{fact:Diff} (1)
 the following natural bijection:
\[
\operatorname{Diff}_{G'} (I(\lambda),J(\nu))
\overset{\sim}{\to} \mathcal{S}ol_{\{0\}} (\mathbb{R}^n;\lambda,\nu).
\]
Hence Proposition is immediate from \eqref{eqn:Solexact}.
\end{proof}

In order to analyze 
 $\mathcal{S}ol(\mathbb{R}^n;\lambda,\nu) $, 
 we begin with an explicit structural result on 
 $\mathcal{S}ol(\mathbb{R}^n \setminus \{0\};\lambda,\nu)$:
\begin{lemma}
\label{lem:Sol}
$\dim \mathcal{S}ol (\mathbb{R}^n \setminus \{0\};\lambda,\nu)=1$
 for all $\nulambda \in {\mathbb{C}}^2$.  
More precisely, 
\begin{equation*}
\mathcal{S}ol(\mathbb{R}^n \setminus \{0\};\lambda,\nu)
= \begin{cases}
  \mathbb{C} |x_n|^{\lambda+\nu-n}(|x|^2+x_n^2)^{-\nu}
  &\text{ if\/ $\nulambda \notin\backslash\!\backslash$}, 
  \\
  \mathbb{C} \delta^{(-\lambda-\nu+n-1)}(x_n)(|x|^2+x_n^2)^{-\nu}
  &\text{ if\/ $\nulambda \in \backslash\!\backslash$}.  
  \end{cases}
\end{equation*}
\end{lemma}
\begin{proof}
Substituting \eqref{eqn:Fa} into \eqref{eqn:Fn},
we have
\[
( (|x|^2 + x_n^2) \frac{\partial}{\partial x_j} + 2 \nu x_j ) F = 0,
\]
or equivalently,
\[
\frac{\partial}{\partial x_j} ( ( |x|^2 + x_n^2 )^\nu F ) = 0
\quad  (1 \le j \le n-1).
\]
For $n \ge 3$, 
 the level set $\{x_n=c\}\setminus \{0\}$
 is connected for all $c \in {\mathbb{R}}$, 
 and therefore 
the restriction  $F|_{\mathbb{R}^n\setminus\{0\}}$ must be of the form
\[
F(x) = ( |x|^2 + x_n^2 )^{-\nu} g (x_n)
\]
for some $g \in \mathcal{D}' (\mathbb{R})$.
In turn,
\eqref{eqn:Fa} and \eqref{eqn:parity} force $g$
 to be even and homogeneous
 of  degree
$\lambda + \nu - n$.

For $n \ge 2$, 
using in addition that $F(-x,-x_n)=F(x,x_n)$, 
 we get the same conclusion.  

Since any even and homogeneous distribution on ${\mathbb{R}}$
 of degree $a$
 is of the form 
\begin{equation*}
  g(t)=
  \begin{cases}
  |t|^a &\text{ if $a \ne -1,-3,-5, \cdots$}
  \\
  \delta^{(-a-1)}(t) & \text{ if $a=-1,-3,-5,\cdots$}
  \end{cases}
\end{equation*}
up to a scalar multiple,
 we obtain the Lemma.  
\end{proof}

Lemma \ref{lem:Sol} explains why
 and how (generically) regular symmetry 
 breaking operators
\index{sbon}{Alnt@$\A_{\lambda,\nu}$}
 $\A_{\lambda,\nu}$
 (Chapter \ref{sec:kfini})
 and singular symmetry 
 breaking operators
\index{sbon}{Bt@$\B_{\lambda,\nu}$}
 $\B_{\lambda,\nu}$
 (Chapter \ref{sec:8})
 appear.

In order to find
$\operatorname{Hom}_{G'}(I(\lambda),J(\nu))$
by using Proposition \ref{prop:DHexact},
we need to find
$\operatorname{Diff}_{G'}(I(\lambda),J(\nu))$.

The dimension is known as follows,
 see Fact \ref{fact:5.2}:
\begin{align*}
\dim\operatorname{Diff}_{G'}(I(\lambda),J(\nu))
&= \dim\operatorname{Hom}_{(\mathfrak{g}',P')}
           (\operatorname{ind}_{\mathfrak{p}'_{\mathbb{C}}}^{\mathfrak{g}'_{\mathbb{C}}} (\mathbb{C}_{-\nu}),
            \operatorname{ind}_{\mathfrak{p}_{\mathbb{C}}}^{\mathfrak{g}_{\mathbb{C}}} (\mathbb{C}_{-\lambda}))
\\
&= \begin{cases}
         1   &\text{if $\nulambda \in /\!/$},   \\
         0   &\text{if $\nulambda \notin /\!/$}.
     \end{cases}
\end{align*}

Combining the above mentioned dimension formula with 
Proposition \ref{prop:DHexact} and Lemma \ref{lem:Sol},
we obtain
\begin{proposition}\label{prop:upperdim}
\[
\dim\operatorname{Hom}_{G'}
   (I(\lambda),J(\nu)) \le 1
\quad\text{for any $\nulambda \in \mathbb{C}^2 \setminus /\!/$}.
\]
\end{proposition}
This proposition will be used in the proof of the meromorphic
continuation of the operator $\A_{\lambda,\nu}$
and its functional equations.
We shall determine the precise dimension of
$\operatorname{Hom}_{G'}(I(\lambda),J(\nu))$
for all $\nulambda \in\mathbb{C}^2$
in Theorem \ref{thm:dimHom}.

\begin{remark}\label{rem:extGP}
In Proposition \ref{prop:DHexact},
$\operatorname{Hom}_{G'}(I(\lambda),J(\nu)) \to
   \mathcal{S}ol(\mathbb{R}^n \setminus \{0\};\lambda,\nu)$
is not necessarily surjective.
See Proposition \ref{prop:surj}.  
\end{remark}

\section{$K$-finite vectors
 and regular symmetry breaking operators $\tA{\lambda}{\nu}$}
\label{sec:kfini}
The goal of this chapter
 is to introduce a $(\mathfrak{g}'_{\mathbb{C}},K')$-homomorphism
\[
\A_{\lambda,\nu} : I(\lambda)_K \to J(\nu)_{K'}.  
\]
We see that $\A_{\lambda,\nu}(\varphi)$ is holomorphic in
$\nulambda \in \mathbb{C}^2$
for any $\varphi \in I(\lambda)_K$,
and that $\A_{\lambda,\nu}$ vanishes if and only if
$\nulambda \in L_{\operatorname{even}}$.
In the next chapter
 we shall discuss an analytic continuation
of the operator $\A_{\lambda,\nu}$ acting on the space
$I(\lambda)$ of smooth vectors.  

\subsection
{Distribution kernel $\ka{\lambda}{\nu}$ and its normalization}
\label{subsec:5.1}
For $(x,x_n) \in \mathbb{R}^{n-1} \oplus \mathbb{R}$,
we define
\begin{equation}\label{eqn:KAdef}
\index{sbon}{Kxnt@$\ka{\lambda}{\nu} (x,x_n)$|textbf}
\ka{\lambda}{\nu} (x,x_n)
:= |x_n|^{\lambda+\nu-n} (|x|^2 + x_n^2)^{-\nu}.
\end{equation}
We write $d \omega$
 for the volume form 
 on the standard sphere $S^{n-1}$.  
Using the polar coordinates
$(x,x_n) = r\omega$, $r > 0$, $\omega \in S^{n-1}$,
we see
\begin{equation}\label{eqn:Kpolar}
\ka{\lambda}{\nu} (x , x_n) dx \, dx_n
= r^{\lambda-\nu-n} |\omega_n|^{\lambda+\nu-n} r^{n-1} dr \, d\omega
\end{equation}
is locally integrable on $\mathbb{R}^n$ if
$\nulambda$ belongs to 
\begin{equation}
\label{eqn:Omega0}
\Omega_0
:=\{\nulambda \in {\mathbb{C}}^2:
\operatorname{Re}(\lambda-\nu) > 0
\text{ and }
\operatorname{Re}(\lambda+\nu) > n-1
\}.  
\end{equation}
In order to see the $P'$-covariance
 of $\ka{\lambda}{\nu}$, 
 it is convenient 
 to use homogeneous coordinates.  
Namely,
 for $\xi = (\xi_0,\dots,\xi_{n+1}) \in \Xi$,
we set
\begin{equation}\label{eqn:klmdmu}
\ska{\lambda}{\nu} (\xi)
:= 2^{-\lambda+n} |\xi_n|^{\lambda+\nu-n} (\xi_{n+1} - \xi_0)^{-\nu}
 \in \mathcal{D}'_{\lambda-n} (\Xi) \simeq I(n - \lambda)^{-\infty}.
\end{equation}
In view of the formula \eqref{eqn:NXi} of the embedding
$i_N:\mathbb{R}^n \hookrightarrow \Xi$
given by
\[
(\xi_0,\dots,\xi_{n+1})
= (1 - \sum_{i=1}^n x_i^2, 2x_1, \dots, 2x_n, 1 + \sum_{i=1}^n x_i^2),
\]
we have
\index{sbon}{iotaNast@$\iota_N^*$}
\index{sbon}{Kxa@$k_{\lambda,\nu}^{\mathbb{A}}$|textbf}
\index{sbon}{iotaKast@$\iota_K^*$}
\begin{alignat}{2}
& \iota_N^* \ska{\lambda}{\nu} &&= \ka{\lambda}{\nu}
\label{eqn:iNk}
\\
& (\iota_K^* \ska{\lambda}{\nu}) (\eta)
&& = 2^{-\lambda+n}
|\eta_n|^{\lambda+\nu-n} (1 - \eta_0)^{-\nu} ,
\label{eqn:iKk}
\end{alignat}
where
$\eta = (\eta_0, \eta_1, \dots, \eta_n) \in S^n$.
Then 
\begin{equation}
\label{eqn:PkA}
\ska{\lambda}{\nu}(me^{-t H}n\xi)
=e^{\nu t} \ska{\lambda}{\nu}(\xi), 
\end{equation}
 for any $m \in M'$, 
 $t \in {\mathbb{R}}$
 and $n \in N_+$
 (see \eqref{eqn:nbxi}), 
 and therefore we have the following lemma:
\begin{lemma}
\label{lem:KAOmega}
For $\nulambda \in \Omega_0$, 
$
\ka{\lambda}{\nu} \in \mathcal{S}ol(\mathbb{R}^n;\lambda,\nu).  
$
\end{lemma}
Thus we get a continuous $G'$-homomorphism
\index{sbon}{aln@${\mathbb{A}}_{\lambda,\nu}$|textbf}
\[
\ntAln{\lambda}{\nu} : I(\lambda) \to J(\nu)
\]
and a $(\mathfrak{g}',K')$-homomorphism
\[
\ntAln{\lambda}{\nu} : I(\lambda)_K \to J(\nu)_{K'}
\]
for $\nulambda \in \Omega_0$.  

For the meromorphic continuation of $\ka{\lambda}{\nu}$, 
 we note that the singularities
 of $\ka{\lambda}{\nu}$ arise from the equations
 $x_n = 0$ and
$|x|^2 + x_n^2 = 0$.  
Since the corresponding varieties 
$\mathbb{R}^{n-1}$ and $\{0\}$, respectively
(or $S^{n-1}$
 and 
\index{sbon}{pplus@${p_+}$}
$[p_+]$ in $G/P$, respectively)
 are
not transversal to each other, 
 the proof of the meromorphic distribution
$\ka{\lambda}{\nu} (x-y, x_n) dx \, dx_n$
is more involved.
We shall study $\ka{\lambda}{\nu}$ algebraically
 in this chapter,
 and analytically in the next chapter 
(see Theorem \ref{thm:poleA}).
Our idea is to look carefully
 at the two variable case by using special functions in accordance to a 
`desingularization' of the (real) algebraic variety.

We normalize the distribution $\ka{\lambda}{\nu}$
 by 
\index{sbon}{Kxt@$\KA{\lambda}{\nu} (x, x_n)$|textbf}
\begin{align}
\KA{\lambda}{\nu} (x, x_n)
:=& \frac{1}{\Gamma (\frac{\lambda+\nu-n+1}{2}) \Gamma (\frac{\lambda-\nu}{2})}
       \ka{\lambda}{\nu} (x,x_n)
\notag
\\
=&\frac{1}{\Gamma (\frac{\lambda+\nu-n+1}{2}) \Gamma (\frac{\lambda-\nu}{2})}
|x_n|^{\lambda+\nu-n}(|x|^2+x_n^2)^{-\nu}
\label{eqn:KAA}
\end{align}
and
\[
\index{sbon}{Alnt@$\A_{\lambda,\nu}$|textbf}
\A_{\lambda,\nu}
:= \frac{1}{\Gamma(\frac{\lambda+\nu-n+1}{2})\Gamma(\frac{\lambda-\nu}{2})}
    \mathbb{A}_{\lambda,\nu}.
\]

\begin{remark}\label{rem:basicpole}
The denominator of the distribution kernel
$\KA{\lambda}{\nu}$ has poles at
$\backslash\!\backslash \cup /\!/$ as follows:
\begin{alignat*}{2}
&\Gamma (\frac{\lambda+\nu-n+1}{2})
  &&\text{ has a simple pole } \Leftrightarrow \nulambda \in \backslash\!\backslash,
\\
&\Gamma (\frac{\lambda-\nu}{2})
  &&\text{ has a simple pole } \Leftrightarrow \nulambda \in /\!/.
\end{alignat*}.
\end{remark}

We recall from Definition \ref{def:model}
 and \eqref{eq:2.9.4}
 the following isomorphism:
\index{sbon}{iotalmdast@$\iota_{\lambda}^*$}
\[
  I(\lambda)_K \simeq \iota_{\lambda}^{\ast}(C^{\infty}(S^n)_K).  
\]
\begin{proposition}\label{prop:4.2}
\begin{enumerate}
\item[{\rm{1)}}]
For any $f \in C^{\infty}(S^n)_K$, 
$\langle \KA{\lambda}{\nu},  \iota_{\lambda}^{\ast} f\rangle$
is holomorphic in $\nulambda \in {\mathbb{C}}^2$.  
\item[{\rm{2)}}]
$\langle \KA{\lambda}{\nu},  F\rangle=0$
 for any $F\in I(\lambda)_K$
 if and only if\/ $\nulambda \in 
\index{sbon}{Leven@$L_{\operatorname{even}}$}
L_{\operatorname{even}}$. 
\end{enumerate}
\end{proposition}

As the proof requires a number
 of preliminary results,
 we show the proposition in Section \ref{subsec:profpr7.3}
In the course of the proof,
we also obtain the following result
(see Lemma \ref{lem:KtildeFC} for a more general statement):
\begin{proposition}\label{prop:AminK}
Let\/
$\mathbf{1}_\lambda := \iota_\lambda^* (\mathbf{1})$,
$\mathbf{1}_\nu := \iota_\nu^* (\mathbf{1})$
be the normalized spherical vectors in
$I(\lambda)$, $J(\nu)$, respectively.
Then
\[
\A_{\lambda,\nu} (\mathbf{1}_\lambda)
= \frac{\pi^{\frac{n-1}{2}}}{\Gamma(\lambda)} \mathbf{1}_\nu.
\]
\end{proposition}

\begin{proof}
Applying Lemma \ref{lem:KtildeFC}
with $k = 0$ and $h = \mathbf{1}$,
we get
\[
\langle \KA{\lambda}{\nu}, \Gamma(\frac{n}{2}) \mathbf{1}_\lambda \rangle
=
\frac{\pi^{\frac{n-1}{2}} \Gamma(\frac{n}{2})}{\Gamma(\lambda)} .
\]
Here we have used the duplication formula \eqref{eqn:dupl} of the Gamma function.
Hence we get the proposition.
\end{proof}

Proposition \ref{prop:AminK} will be used
 in finding the constants 
 appearing in various functional equations
 (see Chapter \ref{sec:reduction}).  
We shall discuss the meaning of Proposition \ref{prop:AminK}
 also in Chapter \ref{sec:PGinv}
 in relation with analysis
 on the semisimple symmetric space $G/G'$.

\subsection{Preliminary results}
We prepare two elementary lemmas that will be used in the proof of Proposition \ref{prop:4.2}.

The first lemma illustrates
 that the zero set of an operator with holomorphic
parameters is not necessarily of codimension one in the parameter space,
as Proposition \ref{prop:4.2} states.

For a polynomial $g(s)$ of one variable,
we define
\begin{equation}\label{eqn:Pab}
P_{a,b}(g) := \frac{1}{\Gamma(a)\Gamma(b)}
   \int_{-1}^1 (1-s)^{a-1} (1+s)^{b-1} g(s) ds.
\end{equation}
Our normalization is given as
\[
P_{a,b} (1) = \frac{1}{\Gamma(a+b)}
\]

\begin{lemma}\label{lem:Beta}
For any $g \in \mathbb{C}[s]$,
 $P_{a,b} (g)$ is holomorphic as a function of two variables $(a,b) \in \mathbb{C}^2$.
Further,
$P_{a,b} \equiv 0$
if and only if $(-a, -b) \in \mathbb{N} \times \mathbb{N}$.
\end{lemma}

\begin{proof}
For $l_1, l_2 \in {\mathbb{N}}$, 
 we set 
\[
  g_{l_1,l_2}(s):=(1-s)^{l_1}(1+s)^{l_2}.  
\]
Then the polynomials 
$g_{l_1,l_2}(s)$ ($l_1,l_2 \in \mathbb{N}$)
span the vector space $\mathbb{C} [s]$.
We then compute 
\begin{align*}
P_{a,b} (g_{l_1,l_2})
&= \frac{2^{a+b+l_1+l_2-1}B(a+l_1,b+l_2)}{\Gamma(a)\Gamma(b)}
\\
&= \frac{2^{a+b+l_1+l_2-1}}{\Gamma(a+b+l_1+l_2)}
   \prod_{i=0}^{l_1-1} (a+i)
   \prod_{j=0}^{l_2-1} (b+j), 
\end{align*}
where $B(\,,\,)$ is the Beta function.   
Thus $P_{a,b} (g_{l_1,l_2})$ is holomorphic for any $l_1$
 and $l_2$, 
 and the first statement is proved.

The zero set of $P_{a,b}(g_{l_1,l_2})$ is given by
\begin{align*}
\mathcal{N}_{l_1,l_2} :={}&
 \{(a,b) \in \mathbb{C}^2: P_{a,b}(g_{l_1,l_2}) = 0 \}
\\
  ={}&
 \bigcup_{i=0}^{l_1-1} \{ a = -i \} 
\cup 
\bigcup_{j=0}^{l_2-1} \{ b = -j \}  
\cup \{ a + b = -2l \}.
\end{align*}
Taking the intersection of all $\mathcal{N}_{l_1,l_2}$,
we get
\[
\bigcap_{l_1=0}^\infty \bigcap_{l_2=0}^{\infty}\, 
\mathcal{N}_{l_1,l_2}
= \{ (a,b) \in \mathbb{C}^2 : a \in -\mathbb{N},  \  b \in -\mathbb{N} \}.
\]
Thus Lemma \ref{lem:Beta} is proved.
\end{proof}

The orthogonal group $O(n)$ acts irreducibly on the space
$\mathcal{H}^N(S^{n-1})$ of spherical harmonics,
and we let
 $O(n-1)$ act on ${\mathbb{R}}^n$
 in the first $(n-1)$-coordinates.  
We denote by ${\mathcal{H}}^N(S^{n-1})^{O(n-1)}$
 the subspace consisting of $O(n-1)$-invariant
 spherical harmonics
 of degree $N$, 
 and by $({\mathcal{H}}^N(S^{n-1})^{O(n-1)})^{\perp}$
 the orthogonal complementary subspace
 with respect the $L^2$-inner product on $S^{n-1}$.  
Then we have a direct sum decomposition:
\[
  {\mathcal{H}}^N(S^{n-1})
  ={\mathcal{H}}^N(S^{n-1})^{O(n-1)}
   \oplus
   ({\mathcal{H}}^N(S^{n-1})^{O(n-1)})^{\perp}.  
\]
 
Let 
\index{sbon}{Ctt@$\tilde{\tilde{C}}_N^\mu (t)$}
$\Tilde{\Tilde{C}}_N^{\mu}(t)$ be
 the renormalized Gegenbauer polynomial, 
 see \eqref{eqn:nGeg}.  
The next lemma is classical.  
\begin{lemma}\label{lem:Bnint}
\begin{enumerate}
\item[{\rm{1)}}]
We regard $\Tilde{\Tilde{C}}_N^{\mu} (\omega_n)$
 as a function on $S^{n-1}$
 in the coordinates $(\omega_1,\dots,\omega_n)$
 of the ambient space ${\mathbb{R}}^n$.  
Then we have:
\[
{\mathcal{H}}^N(S^{n-1})^{O(n-1)}
={\mathbb{C}}\operatorname{-span}
\Tilde{\Tilde{C}}_N^{\frac{n}{2}-1} (\omega_n).  
\]
\item[{\rm{2)}}]
Let $\psi \in {\mathcal{H}}^N(S^{n-1})$.  
If $N$ is odd or $\psi \perp {\mathcal{H}}^N(S^{n-1})^{O(n-1)}$, 
 then 
\[
   \int_{S^{n-1}} |\omega_n|^{\lambda+\nu-n}\psi(\omega)d\omega=0.  
\]
If $N$ is even,
 then 
\[
   \int_{S^{n-1}} |\omega_n|^{\lambda+\nu-n} \Tilde{\Tilde{C}}_N^{\frac{n}{2}-1} (\omega_n)d\omega=
   d_{n,N}(\lambda,\nu) \, g (\lambda,\nu), 
\]
where 
\begin{align}
d_{n,N}(\lambda,\nu)
:=&{} \frac{2^{2-\lambda-\nu} \pi^{\frac{n+1}{2}} \Gamma(n+N-1)}
             {\Gamma(\frac{n-1}{2}) \Gamma(N+1)}, 
\label{eqn:intGegen}
\\
   g(\lambda, \nu)
  :=&{} \frac{\Gamma (\lambda+\nu-n+1)}
         {\Gamma (\frac{\lambda+\nu-n-N+2}{2}) 
          \Gamma (\frac{\lambda+\nu+N}{2})}.  
\notag
\end{align}
\end{enumerate}
\end{lemma}

\begin{proof}
1) \enspace
The result is well-known.
See {\it{e.g.}}, 
 \cite[Lemma 5.2]{KO3}.  
\begin{enumerate}
\item[2)]
For $\phi(\omega_n)$ regarded as an $O(n-1)$-invariant function,
we have
\[
\int_{S^{n-1}} \phi (\omega_n) d\omega
= \operatorname{vol} (S^{n-2}) \int_{-1}^1 \phi (t) (1-t^2)^{\frac{n-3}{2}}dt.
\]
Therefore,
\begin{align*}
\int_{S^{n-1} }|\omega_n|^{\lambda+\nu-n} \Tilde{\Tilde{C}}_N^{\frac{n}{2}-1} (\omega_n) d\omega
&= \operatorname{vol} (S^{n-2}) \int_{-1}^1 |t|^{\lambda+\nu-n} (1-t^2)^{\frac{n-3}{2}} \Tilde{\Tilde{C}}_N^{\frac{n}{2}-1} (t) dt
\\
&= \begin{cases}
             0                                                                       &\text{for $N$ odd, }
     \\
             d_{n,N}(\lambda,\nu) g(\lambda,\nu)                                     &\text{for $N$ even.}
     \end{cases}
\end{align*}
The last equality follows from 
 $\operatorname{vol}(S^{n-2}) = \frac{2\pi^{\frac{n-1}{2}}}{\Gamma(\frac{n-1}{2})}$
 and from the integration formula \eqref{eqn:Ct int2}
 of the Gegenbauer polynomial.
\end{enumerate}
\end{proof}

\subsection{Proof of Proposition \ref{prop:4.2}}
\label{subsec:profpr7.3}
In light of Proposition \ref{prop:IKfinite}, 
any element in $I(\lambda)_K$
 is a linear combination of functions
 of the form \eqref{eqn:F12}, namely,
\index{sbon}{Fl@$F_{\lambda}[\psi,h]$}
\[
F_\lambda [\psi, h] (r\omega) =
(1 + r^2)^{-\lambda} \left( \frac{2r}{1+r^2} \right)^N
   \psi (\omega) h \left( \frac{1-r^2}{1+r^2} \right)
\]
for $\psi \in \mathcal{H}^N (S^{n-1})$
 and $h \in {\mathbb{C}}[s]$ in the polar coordinates
$(x, x_n) = r\omega$ ($r > 0$, $\omega \in S^{n-1}$).

Then the first statement of Proposition \ref{prop:4.2} follows immediately
{}from the next lemma:
\begin{lemma}\label{lem:KtildeFC}
\begin{enumerate}
\item[{\rm{1)}}]
Suppose $N \in {\mathbb{N}}$, 
 $\psi \in {\mathcal{H}}^N(S^{n-1})$
 and $h \in \mathbb{C}[s]$.  
If $N$ is odd or
 $\psi \perp {\mathcal{H}}^N(S^{n-1})^{O(n-1)}$, 
then 
\[
\langle
\KA{\lambda}{\nu},F_{\lambda}[\psi,h]
\rangle
=0.  
\]
\item[{\rm{2)}}]
For $N \in 2\mathbb{N}$ and $h \in \mathbb{C}[s]$,
 we have
\[
\langle \KA{\lambda}{\nu}, F_\lambda [ \Tilde{\Tilde{C}}_N^{\frac{n}{2}-1} (\omega_n), h] \rangle
= c P_{\frac{\lambda-\nu+N}{2},\frac{\lambda+\nu+N}{2}} (h)
   \prod_{j=0}^{\frac{N}{2}-1}
   \Bigl( \frac{\lambda-\nu}{2} + j \Bigr)
   \Bigl( \frac{\lambda+\nu-n}{2} - j \Bigr), 
\]
where $P_{a,b}(h)$ was defined in \eqref{eqn:Pab}, 
 and the non-zero constant $c$ is given by 
\[
c=2^{\nu-n+1} d_{n,N}(\lambda,\nu) \pi^{-\frac12} 
=
\frac{2^{3-\lambda-n} \pi^{\frac{n}{2}} \Gamma(n+N-1)}
       {\Gamma(\frac{n-1}{2}) \Gamma(N+1)} .
\]
\end{enumerate}
\end{lemma}

\begin{proof}
By using the expression \eqref{eqn:Kpolar}
 of $\ka{\lambda}{\nu}$
 in the polar coordinates,
 we have
\begin{align}
\langle \ka{\lambda}{\nu}, F_\lambda [\psi, h] \rangle
=& 2^{-\lambda} R S,
\label{eqn:KFRS}
\\
\langle
\KA{\lambda}{\nu},F_{\lambda}[\psi,h]
\rangle
=&
\frac{2^{-\lambda}}
     {\Gamma(\frac{\lambda+\nu-n+1}{2})
      \Gamma(\frac{\lambda-\nu}{2})}RS,  
\notag
\end{align}
where
\begin{align*}
R :={}
& 2^{N+\lambda} \int_0^\infty r^{\lambda-\nu+N-1} (1+r^2)^{-\lambda-N}
   h \left( \frac{1-r^2}{1+r^2} \right) dr
\\
={}
& \int_{-1}^1 (1-s)^{\frac{\lambda-\nu+N-2}{2}}
   (1+s)^{\frac{\lambda+\nu+N-2}{2}} h(s) ds
\\
={}
& \Gamma(\frac{\lambda-\nu+N}{2})
  \Gamma(\frac{\lambda+\nu+N}{2})
  P_{\frac{\lambda-\nu+N}{2}, \frac{\lambda+\nu+N}{2}}(h), 
\\
S :={}
& \int_{S^{n-1}} |\omega_n|^{\lambda+\nu-n} \psi (\omega) d\omega.  
\end{align*}
\begin{enumerate}
\item[1)]
It follows from Lemma \ref{lem:Bnint}
 that $S$ vanishes
 if $N$ is odd or 
 $\psi \perp ({\mathcal{H}}^N(S^{n-1})^{O(n-1)})^{\perp}$.  
\item[2)]
By \eqref{eqn:KFRS} and Lemmas \ref{lem:Beta} and \ref{lem:Bnint},
\[
\langle \KA{\lambda}{\nu}, F_\lambda [ \Tilde{\Tilde{C}}_N^{\frac{n}{2}-1} (\omega_n), h] \rangle
= c P_{\frac{\lambda-\nu+N}{2},\frac{\lambda+\nu+N}{2}} (h) V,
\]
where
\begin{align*}
V :={}&
   2^{-\lambda-\nu+n-1} \pi^{\frac{1}{2}}
   \frac{\Gamma(\frac{\lambda-\nu+N}{2})\Gamma(\frac{\lambda+\nu+N}{2})}
          {\Gamma(\frac{\lambda+\nu-n+1}{2})\Gamma(\frac{\lambda-\nu}{2})}
   g(\lambda, \nu)
\\
 ={}&
\frac{\Gamma (\frac{\lambda-\nu+N}{2}) \Gamma (\frac{\lambda+\nu-n+2}{2})}
       {\Gamma (\frac{\lambda-\nu}{2}) \Gamma (\frac{\lambda+\nu-n-N+2}{2})}
\\
 ={}&
\prod_{j=0}^{\frac{N}{2}-1}
   \Bigl( \frac{\lambda-\nu}{2} + j \Bigr)
   \Bigl( \frac{\lambda+\nu-n}{2} - j \Bigr).
\end{align*}
In the second equality we have used the duplication formula \eqref{eqn:dupl} of the Gamma function.
\end{enumerate}
\end{proof}

\begin{proof}[Proof of Proposition \ref{prop:4.2} (2)]
For $N \in {\mathbb{N}}$, 
 Let $\psi \in {\mathcal{H}}^N(S^{n-1})$
 and $h \in {\mathbb{C}}[s]$, 
 we set
\begin{align*}
\mathcal{N} [\psi, h]
:=& \{ \nulambda \in \mathbb{C}^2 :
      \langle \KA{\lambda}{\nu}, F_\lambda [\psi, h] \rangle = 0 \}, 
\\
{\mathcal{Z}}_N:=&\bigcap_{\psi\in\mathcal{H}^N(S^{n-1})}  \,
\bigcap_{h\in\mathbb{C}[s]}
\mathcal{N} [\psi, h].  
\end{align*}
Then, 
we have
\begin{align*}
&\{\nulambda \in \mathbb{C}^2 :
  \langle \KA{\lambda}{\nu}, F \rangle = 0
  \text{ for all } F \in I(\lambda)_K\}
\\
=& \bigcap_{N \in {\mathbb{N}}}
  \{ \nulambda \in \mathbb{C}^2 :
   \langle \KA{\lambda}{\nu}, F_\lambda [\psi, h] \rangle = 0
   \  \text{for any $\psi \in \mathcal{H}^N (S^{n-1})$, $h \in \mathbb{C} [s]$
} \}
\\
=& \bigcap_{N\in \mathbb{N}} {\mathcal{Z}}_N.  
\end{align*}

Let us compute ${\mathcal{Z}}_N$ explicitly.  
For this, 
 we set
\begin{align*}
\Lambda_{\mathbb{Z}^2}
:={}&
\{ \nulambda \in \mathbb{Z}^2 : \lambda + |\nu| \le 0,
    \lambda \equiv \nu \bmod 2 \}
\\
={}&
\{ \nulambda \in \mathbb{Z}^2 : \lambda - \nu \in -2\mathbb{N} \text{ and }
   \lambda + \nu \in -2\mathbb{N} \}.
\end{align*}
For $N \in \mathbb{N}$,
we define the parallel translation of $\Lambda_{\mathbb{Z}^2}$ by
$(-N,0)$:
\[
\Lambda_{\mathbb{Z}^2} [N] 
:= \{ \nulambda \in \mathbb{Z}^2 :
      (\lambda+ N, \nu) \in \Lambda_{\mathbb{Z}^2} \}.
\]
Then it follows from Lemma \ref{lem:Bnint} that
$
P_{ \frac{\lambda-\nu+N}{2}, \frac{\lambda+\nu+N}{2}} (h) = 0
$
for all $h \in \mathbb{C}[s]$
if and only if 
$\nulambda \in \Lambda_{\mathbb{Z}^2} [N]$.

In turn,
 it follows from Lemma \ref{lem:KtildeFC}
that 
${\mathcal{Z}}_N={\mathbb{C}}^2$ for $N\in 2 {\mathbb{N}}+1$, 
 and for $n\in 2 {\mathbb{N}}$, 
\begin{align*}
{\mathcal{Z}}_N&= \Lambda_{\mathbb{Z}^2} [N] \cup
   \{ \nulambda \in \mathbb{C}^2:
       \prod_{j=0}^{\frac{N}{2}-1}
       \Bigl( \frac{\lambda-\nu}{2} + j \Bigr)
       \Bigl( \frac{\lambda+\nu-n}{2} - j \Bigr) = 0 \}
\\
&=
    \Lambda_{\mathbb{Z}^2} [N] \cup
    \bigcup_{j=0}^{\frac{N}{2}-1}
    \{ \nulambda : \lambda - \nu = -2j \}
   \cup \bigcup_{j=0}^{\frac{N}{2}-1}
   \{ \nulambda : \lambda + \nu = n + 2j \}.  
\end{align*}
In Figure \ref{fig:ZNLeven} below,
 ${\mathcal {Z}}_N$
 ($N=8$)
 consists of black dots and $4+4$ lines;
 $L_{\operatorname{even}}$ consists of red circles.  

Taking the intersection of all ${\mathcal{Z}}_N$, 
 we get 
\begin{equation*}
\bigcap_{N \in {\mathbb{N}}} {\mathcal{Z}}_N
= \{ \nulambda : \text{$\lambda, \nu \in -{\mathbb{N}}$, 
 $\lambda \equiv \nu \mod 2$, and $\lambda \le \nu$}  \} .
\end{equation*}

\medskip
\begin{figure}[h]
\begin{center}
\includegraphics[scale=1]{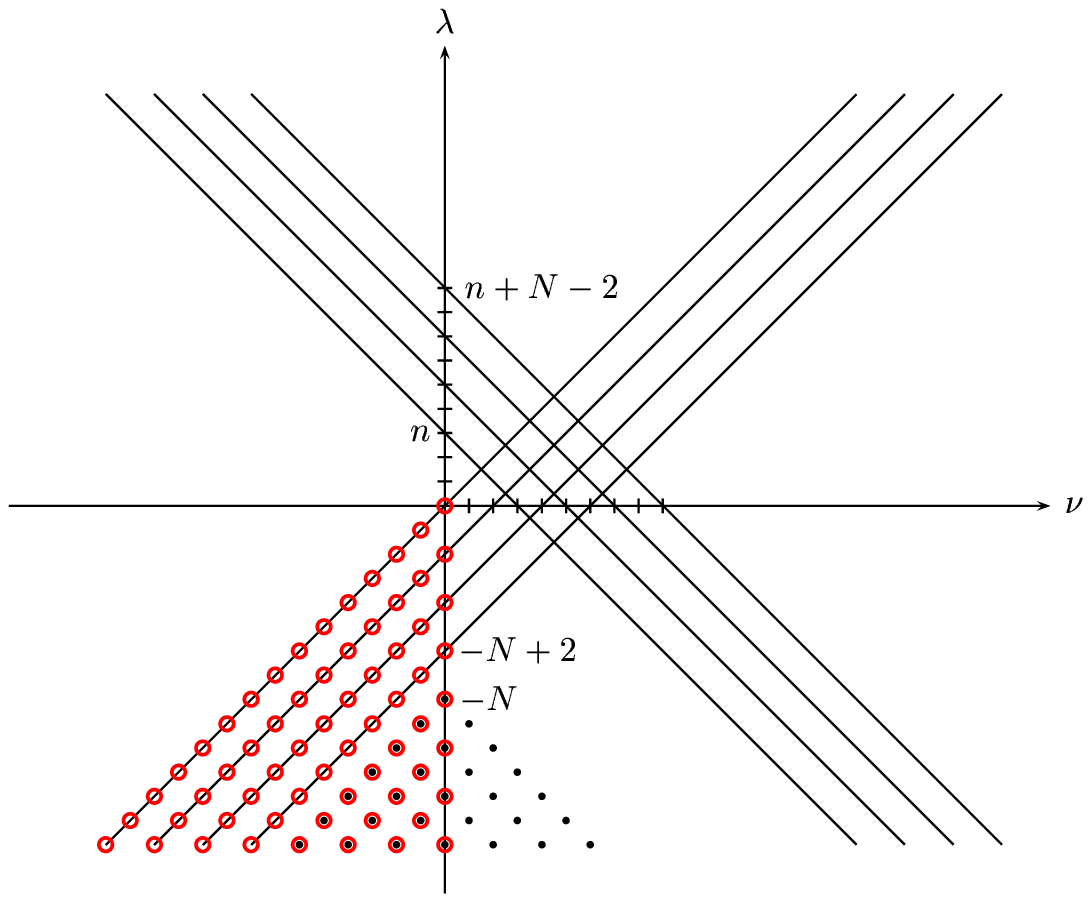}
\end{center}
\caption{${\mathcal{Z}}_N$ ($N=8$) and $L_{\operatorname{even}}$}
\label{fig:ZNLeven}
\end{figure}
\medskip

Hence Proposition \ref{prop:4.2} is proved.
\end{proof}

\section{Meromorphic continuation
 of regular symmetry breaking operators $\ka{\lambda}{\nu}$}
\label{sec:8}

The goal of this chapter
 is to prove the existence of the meromorphic
continuation of our symmetry breaking operator
$\ntAln{\lambda}{\nu}$,
 initially holomorphic in an open set $\Omega_0$, 
 to $\nulambda \in \mathbb{C}^2$.
Besides, we 
determine all the poles of the symmetry breaking operator
$\ntAln{\lambda}{\nu}$
with meromorphic parameter $\lambda$ and $\nu$.  
The normalized symmetry breaking operators
 $\A_{{\lambda},{\nu}}$
 depends holomorphically on $\nulambda$
 in the entire space ${\mathbb{C}}^2$.  
Surprisingly,
 there exist  countably many points 
 in the complex set in ${\mathbb{C}}^2$
 such that $\A_{\lambda,\nu}$ vanishes,
 namely, 
 $\A_{\lambda,\nu}$ is zero
 on the set $L_{\operatorname{even}}$ of codimension two 
in ${\mathbb{C}}^2$.  
We shall prove 
\begin{theorem} 
\label{thm:poleA}
{\rm{1)}}\enspace
$\KA{\lambda}{\nu}$
is a distribution on $\mathbb{R}^n$ that depends holomorphically on
 parameters $\lambda$
 and $\nu$ in the entire plane ${\mathbb{C}}^2$.
\begin{enumerate}
\item[{\rm{2)}}]
$\KA{\lambda}{\nu} \in \mathcal{S}ol(\mathbb{R}^n;\lambda,\nu)$
for all $\nulambda \in\mathbb{C}^2$,
and thus defines a continuous $G'$-homomorphism
\begin{equation}
\label{eqn:Adef}
\A_{\lambda,\nu} : I(\lambda) \to J(\nu).
\end{equation}
This operator $\A_{\lambda,\nu}$ vanishes
 if and only if\/  
\index{sbon}{Leven@$L_{\operatorname{even}}$}
$\nulambda \in 
L_{\operatorname{even}}$, namely,
\begin{equation}
\label{eqn:Avanish}
\lambda , \nu \in - {\mathbb{N}}, 
\,
\lambda \equiv \nu \mod 2, 
\text { and }
\lambda \le \nu.  
\end{equation}
\end{enumerate}
\end{theorem}

In what follows,
we shall consider the following open subsets in $\mathbb{C}^2$:
\begin{alignat}{2}
&\Omega_1&&:=\{\nulambda \in {\mathbb{C}}^2
             :\operatorname{Re}(\lambda-\nu)>0\},
\label{eqn:Omega1}
\\
&D_{n-1}^+
&&:= \{\nulambda \in \mathbb{C}^2:
        \operatorname{Re}(\lambda+\nu) > n-1\},
\nonumber
\\
&D_{n}^-
&&:= \{\nulambda \in \mathbb{C}^2:
        \operatorname{Re}(\lambda+\nu) < n\}.
\nonumber
\end{alignat}
Obviously,
$D_{n-1}^+ \cup D_n^- = \mathbb{C}^2$.
We recall from \eqref{eqn:Omega0}
\begin{equation}
\index{sbon}{Omega0@$\Omega_0$|textbf}
\Omega_0 = \Omega_1 \cap D_{n-1}^+.
\end{equation}

The proof of  Theorem \ref{thm:poleA}
 consists of the following two steps:
\par\noindent
{\bf{Step 1.}}\enspace
$\Omega_0 \Rightarrow \Omega_1$.  
Use differential equations, 
 see \eqref{eqn:mero1} and \eqref{eqn:mero2}.  
\par\noindent
{\bf{Step 2.}}\enspace
$\Omega_1 \Rightarrow {\mathbb{C}}^2$.  
Use functional equations, 
 see \eqref{eqn:TAAnew} and \eqref{eqn:ATTnew}.  

\subsection{Recurrence relations
 of the distribution kernels $\ka{\lambda}{\nu}$}
\label{subsec:meroPDE}

As the first step,
we shall use recurrence relations of
$\ka{\lambda}{\nu}(x,x_n)$.
We set 
\[
  K_{\lambda,\nu}^{\pm}(x,x_n)
  :=
  (x_n)_\pm^{\lambda+\nu-n}(|x|^2+x_n^2)^{-\nu}.  
\]
Then $K_{\lambda,\nu}^{\pm}(x,x_n)$ is locally integrable
 if $\nulambda \in\Omega_0$, 
 and thus gives a distribution on ${\mathbb{R}}^n$
 with holomorphic parameter $\nulambda \in \Omega_0$.  
\begin{lemma}
\label{lem:mero1}
$K_{\lambda,\nu}^{\pm}(x,x_n)$ extends
 meromorphically to $\Omega_1$
as distributions on ${\mathbb{R}}^n$.  
\end{lemma}
\begin{proof}
We only give 
 a proof for $K^+(x,x_n)$;
the case for $K(x,x_n)$ can be shown similarly.  
First observe
 that the distribution 
$K_{\lambda,\nu}^+ \in \mathcal{D}'(\mathbb{R}^n)$
satisfies the following differential equations when
$\operatorname{Re}(\lambda-\nu) \gg 0$
and $\operatorname{Re}(\lambda+\nu) \gg 0$:
\begin{align}
\label{eqn:mero1}
\frac{\partial}{\partial x_n}K_{\lambda+1, \nu}^{+}
=&
(\lambda+\nu-n+1)K_{\lambda, \nu}^{+}
-2\nu K_{\lambda, \nu+1}^{+}
\\
\label{eqn:mero2}
\Delta_{{\mathbb{R}}^{n-1}}
K_{\lambda+1, \nu-1}^{+}
=&
2(\nu-1)(2\nu-n+1)K_{\lambda, \nu}^{+}
-4(\nu-1)\nu K_{\lambda+1, \nu+1}^{+}
\end{align}
We show the lemma 
 by iterating meromorphic continuations
 based on the two steps
 $\Omega \rightsquigarrow \Omega \cup \Omega^+$
 and 
$\Omega \rightsquigarrow \Omega \cup \Omega_+$
 below
 using \eqref{eqn:mero1} and \eqref{eqn:mero2}, 
respectively.  
Suppose 
 $K_{\lambda, \nu}^{+}$ is proved
 to extend meromorphically on a certain domain $\Omega$ in $\mathbb{C}^2$ as
  distributions on ${\mathbb{R}}^{n}$.  
Then the equation \eqref{eqn:mero1} shows 
 that $K_{\lambda, \nu}^{+}$ extends meromorphically
 to the following open subset 
\begin{align*}
  \Omega^+:={}&
   \{\nulambda \in \mathbb{C}^2 :
     (\lambda+1,\nu) \in \Omega \ \text{and}\ 
     (\lambda,\nu+1) \in \Omega \} 
\\
={}&
   (\Omega+ (-1,0)) \cap (\Omega+(0,-1)).  
\end{align*}
Here we have used the following notation:
\[
  \Omega+(a,b)
  :=
  \{\nulambda \in {\mathbb{C}}^2
   : (\lambda-a, \nu-b)\in \Omega\}.  
\]
Likewise, the equation \eqref{eqn:mero2} shows that
 $K_{\lambda, \nu}^{+}$ extends meromorphically
 to 
\[
  \Omega_+:=(\Omega+(-1,1)) \cap (\Omega+(-1,-1)).  
\]

Now, 
 first, 
 we set $\Omega=\Omega_0$.  
By iterating the meromorphic continuation process
 $\Omega \rightsquigarrow \Omega \cup \Omega^+$, 
 the distribution $K_{\lambda,\nu}^+$ extends
 meromorphically to the domain
$\bigcup_{k=0}^\infty (\Omega_0 + (0,-k))$,
which contains
\[
\Omega'_0 :=
  \{\nulambda \in {\mathbb{C}}^2
    :\operatorname{Re}(\lambda-\nu)>0, 
     \operatorname{Re}\lambda>\frac n 2\}, 
\]
see Figure \ref{fig:8.1}.  
\medskip
\begin{figure}[h]
\includegraphics{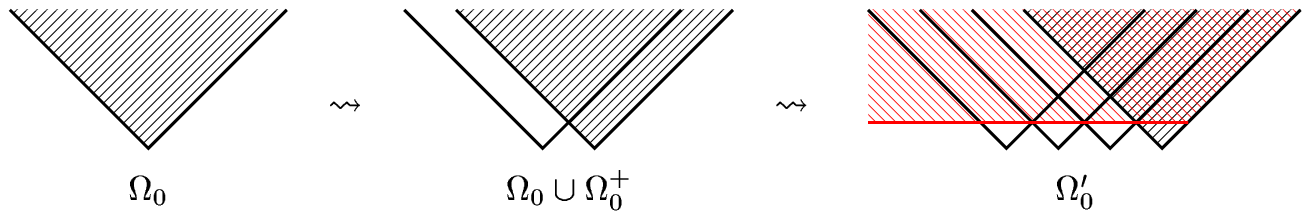}
\caption{Analytic continuation from $\Omega_0$
 to $\Omega_0'$}
\label{fig:8.1}
\end{figure}
\medskip

Second, we begin with $\Omega'_0$ and 
 iterate the process $\Omega \rightsquigarrow\Omega \cup \Omega_+$. 
Then we see that $K_{\lambda,\nu}^+$ extends meromorphically
 to the domain $\Omega_1=\{\nulambda \in {\mathbb{C}}^2
    :\operatorname{Re}(\lambda-\nu)>0\}$, 
 see Figure \ref{fig:8.2}.      
\medskip
\begin{figure}[h]
\includegraphics{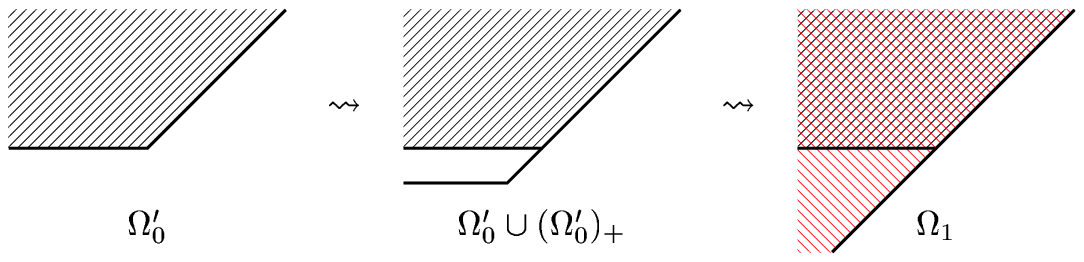}
\caption{Analytic continuation from $\Omega_0'$
 to $\Omega_1$}
\label{fig:8.2}
\end{figure}
\end{proof}

\begin{lemma}\label{lem:mero2}
If\/ $\operatorname{Re}(\lambda-\nu) > 0$,
then
$\KA{\lambda}{\nu} \in \mathcal{S}ol(\mathbb{R}^n;\lambda,\nu)$
and defines a nonzero $G'$-intertwining operator
$I(\lambda) \to J(\nu)$,
to be denoted by the same symbol $\A_{\lambda,\nu}$.
Then $\A_{\lambda,\nu}$ depends holomorphically on\/ $\nulambda$
in the domain\/ $\Omega_1$.
\end{lemma}

\begin{proof}
By Lemma \ref{lem:mero1}, 
 $\ka{\lambda}{\nu}= K_{\lambda,\nu}^+ + K_{\lambda,\nu}^-$
is a distribution on $\mathbb{R}^n$ that depends meromorphically
on $\Omega_1$.
On the other hand,
 Proposition \ref{prop:4.2} shows that $\A_{\lambda,\nu}$ is nowhere
vanishing and that $\A_{\lambda,\nu}(\varphi)$ does not have
a pole for any $K$-finite vector
$\varphi \in I(\lambda)_K$ in the domain
$\operatorname{Re}(\lambda-\nu) > 0$.
Now Lemma \ref{lem:mero2} follows from Proposition \ref{prop:meroKK}.
\end{proof}

\subsection{Functional equations}\label{subsec:TAAT}

Let $\T{m-\nu}{\nu}:J(m-\nu)\to J(\nu)$
 be the Knapp--Stein intertwining operator
 for $G'$, 
 and $\T{\lambda}{n-\lambda}:I(\lambda)\to I(n-\lambda)$
 for $G$
 with $m=n-1$.  
The second step of the proof of Theorem \ref{thm:poleA}
 is to prove the following functional equations
 (see Theorem \ref{thm:TAAT} below):
\index{sbon}{ttt3@$\T{\lambda}{n-\lambda}$}
\begin{alignat}{2}
&
\T{m-\nu}{\nu} \circ \A_{\lambda,m-\nu}
&&= \frac{\pi^{\frac{m}{2}}}{\Gamma(m-\nu)} \A_{\lambda,\nu}.
\label{eqn:TAAnew}
\\
&\A_{n-\lambda,\nu} \circ \T{\lambda}{n-\lambda}
&&= \frac{\pi^{\frac{n}{2}}}{\Gamma(n-\lambda)} \A_{\lambda,\nu}.
\label{eqn:ATTnew}
\end{alignat}

We begin with 
\begin{lemma}
\label{lem:TAAT}
{\rm{1)}}\enspace
The identity \eqref{eqn:TAAnew} holds in the domain
$\Omega_0 = \Omega_1 \cap D_{n-1}^+$ (see \eqref{eqn:Omega0}).  
\par\noindent
{\rm{2)}}\enspace
The identity \eqref{eqn:ATTnew} holds in
the domain
\begin{equation}\label{eqn:Omega2}
\Omega_2 := \Omega_1 \cap D_n^-.
\end{equation}
\end{lemma}
\begin{proof}
Since the (renormalized) Knapp--Stein intertwining operator
$\T{m-\nu}{\nu}$
depends holomorphically on $\nu \in \mathbb{C}$,
the composition
\[
\T{m-\nu}{\nu} \circ \A_{\lambda,m-\nu} :
I(\lambda) \to J(m-\nu) \to J(\nu)
\]
is a continuous $G'$-homomorphism that depends holomorphically
on $\nulambda$ by Lemma \ref{lem:mero2}
if $\operatorname{Re}(\lambda-(m-\nu)) > 0$,
namely,
if  $\nulambda \in D_{n-1}^+$.
Thus $\A_{\lambda,\nu}$
and $\T{m-\nu}{\nu} \circ \A_{\lambda,m-\nu}$
 are in $\operatorname{Hom}_{G'}(I(\lambda),J(\nu))$
if $\nulambda \in \Omega_0$.

We recall from Lemma \ref{lem:mero2} that
$\A_{\lambda,\nu} \ne 0$
if $\operatorname{Re}(\lambda-\nu) > 0$
and from Proposition \ref{prop:upperdim} that
$\dim\operatorname{Hom}_{G'}(I(\lambda),J(\nu)) \le 1$
if
$\nulambda \in \Omega_0$ ($\subset \mathbb{C}^2 \setminus /\!/$).
Therefore there exists
$b(\lambda,\nu) \in \mathbb{C}$ such that
\[
\T{m-\nu}{\nu} \circ \A_{\lambda,m-\nu}
= b(\lambda,\nu) \A_{\lambda,\nu}
\]
for $\nulambda \in \Omega_0$.
Applying these operators to the trivial one-dimensional
$K$-type $\mathbf{1}_\lambda \in I(\lambda)_K$,
we get from Proposition \ref{prop:AminK} and Proposition \ref{prop:T1}
\[
\frac{\pi^{\frac{m}{2}}}{\Gamma(m-\nu)}
\frac{\pi^{\frac{m}{2}}}{2^{\lambda-1}\Gamma(\lambda)}
\mathbf{1}_\nu
=
b(\lambda,\nu)
\frac{\pi^{\frac{m}{2}}}{2^{\lambda-1}\Gamma(\lambda)}
\mathbf{1}_\nu,
\]
and therefore
$b(\lambda,\nu) = \frac{\pi^{\frac{m}{2}}}{\Gamma(m-\nu)}$.
Thus we have proved
 \eqref{eqn:TAAnew}
 in the domain $\Omega_0 = \Omega_1 \cap D_{n-1}^+$.

Similarly the composition
\[
\A_{n-\lambda,\nu} \circ \T{\lambda}{n-\lambda} :
I(\lambda) \to I(n-\lambda) \to J(\nu)
\]
is a continuous $G'$-homomorphism if
$\operatorname{Re}((n-\lambda)-\nu) > 0$,
namely, if $\nulambda \in D_n^-$.
Therefore there exists
$c(\lambda,\nu) \in \mathbb{C}$
such that
\[
\A_{n-\lambda,\nu} \circ \T{\lambda}{n-\lambda}
= c(\lambda,\nu) \A_{\lambda,\nu}
\]
if $\operatorname{Re}(\lambda-\nu) > 0$
and $\operatorname{Re}(\lambda+\nu) < n$.
Applying these operators to $\mathbf{1}_\lambda$,
we get
\[
\frac{\pi^{\frac{n-1}{2}}}{\Gamma(n-\lambda)}
\frac{\pi^{\frac{n}{2}}}{\Gamma(\lambda)}
\mathbf{1}_\nu
=
c(\lambda,\nu)
\frac{\pi^{\frac{n-1}{2}}}{\Gamma(\lambda)}
\mathbf{1}_\nu,
\]
whence
$c(\lambda,\nu) = \frac{\pi^{\frac{n}{2}}}{\Gamma(n-\lambda)}$.
Thus we have proved \eqref{eqn:ATTnew}
 in the domain 
 $\Omega_2=\Omega_1 \cap D_n^-$.  
\end{proof}

Now we are ready to complete the proof of Theorem \ref{thm:poleA}.

\begin{proof}
[Proof of Theorem \ref{thm:poleA}]
In order to extend $\ka{\lambda}{\nu}$ meromorphically,
 it is sufficient to prove it
 for $\ka{\lambda}{\nu}|_{{\mathbb{R}}^n}$
 by Proposition \ref{prop:meroK}.  
Then it follows from Lemma \ref{lem:mero1}
 that $\ka{\lambda}{\nu}$ extends meromorphically
 from $\Omega_0$ to $\Omega_1$
 as distributions on ${\mathbb{R}}^n$.  
In turn,
 $\ka{\lambda}{\nu}$ extends meromorphically
to the domain
$\Omega_1 \cup D_{n-1}^+$
as a distribution on $\mathbb{R}^n$
 by Lemma \ref{lem:TAAT} (1), 
 and it extends meromorphically
to the domain $\Omega_1 \cup D_n^-$
as a distribution on $\mathbb{R}^n$
 by Lemma \ref{lem:TAAT} (2).  
Hence $\ka{\lambda}{\nu}$ extends meromorphically to ${\mathbb{C}}^2$.  
\end{proof}

Now,
 by Theorem \ref{thm:poleA}, 
 Lemma \ref{lem:TAAT}
 can be strengthened
 as follows:
\begin{theorem}
\label{thm:TAAT}
The functional equations \eqref{eqn:TAAnew}
 and \eqref{eqn:ATTnew} hold
 for entire $(\lambda,\nu)\in {\mathbb{C}}^2$, 
 namely,
\begin{alignat}{3}
&
\T{m-\nu}{\nu} \circ \A_{\lambda,m-\nu}
&&= \frac{\pi^{\frac{m}{2}}}{\Gamma(m-\nu)} \A_{\lambda,\nu}
\qquad
&&\text{for $(\lambda,\nu)\in {\mathbb{C}}^2$}, 
\\
&\A_{n-\lambda,\nu} \circ \T{\lambda}{n-\lambda}
&&= \frac{\pi^{\frac{n}{2}}}{\Gamma(n-\lambda)} \A_{\lambda,\nu}
\qquad
&&\text{for $(\lambda,\nu)\in {\mathbb{C}}^2$}.
\end{alignat}
\end{theorem}

\subsection{Support of $\KA{\lambda}{\nu}$}
We determine the support of the distribution kernel
\index{sbon}{Kxt@$\KA{\lambda}{\nu} (x, x_n)$}
 $\KA{\lambda}{\nu}$
 when it is nonzero, 
 equivalently,
 by Theorem \ref{thm:poleA} (2), 
 namely,
when 
\index{sbon}{Leven@$L_{\operatorname{even}}$}
$\nulambda \notin L_{\operatorname{even}}$.
\begin{proposition}\label{prop:Ksupp}
Suppose $\nulambda \notin L_{\operatorname{even}}$,
equivalently, 
 $\A_{\lambda,\nu} \ne 0$.  
Then
\index{sbon}{pplus@${p_+}$}
\[
\operatorname{Supp} \KA{\lambda}{\nu}
= \begin{cases}
         S^{n-1}   &\text{if\/ $\nulambda \in \backslash\!\backslash \setminus \mathbb{X}$},   \\
         \{[p_+]\}  &\text{if\/ $\nulambda \in /\!/ - L_{\operatorname{even}}$},     \\
         G/P       &\text{otherwise}.
   \end{cases}
\]
\end{proposition}

\begin{proof}
We recall from \eqref{eqn:Kpolar} that
as a distribution on $\mathbb{R}^n \setminus \{0\}$,
\[
\ka{\lambda}{\nu} (x, x_n)|_{{\mathbb{R}}^n \setminus \{0\}}
= r^{\lambda-\nu-1} |\omega_n|^{\lambda+\nu-n}
\]
in the polar coordinates
$(x, x_n) = r\omega$.
Since the function
\[
\omega_n:S^{n-1} \to {\mathbb{R}}, 
\quad
\omega =(\omega_1, \cdots, \omega_n) \mapsto \omega_n
\]
 is regular at $\omega_n=0$, 
the distribution 
$|\omega_n|^{\lambda+\nu-n}$
 on $S^{n-1}$ 
has a simple pole at
$\lambda+\nu = n-1, n-3, \dotsc$,
and the support of its residue is equal to
$S^{n-1} = \{ \omega_{n} = 0 \}$.
In light of Remark \ref{rem:basicpole},
we have thus 
\begin{alignat*}{2}
&
\operatorname{Supp} (\KA{\lambda}{\nu}|_{\mathbb{R}^n\setminus\{0\}})
  = \mathbb{R}^{n-1} \setminus \{0\}
&&\text{ if } \nulambda \in \backslash\!\backslash \setminus \mathbb{X},
\\
&
\operatorname{Supp} (\KA{\lambda}{\nu}|_{\mathbb{R}^n\setminus\{0\}})
  = \emptyset
&&\text{ if } \nulambda \in /\!/ \setminus \mathbb{X}.
\end{alignat*}
For
$\nulambda \in \mathbb{X}$,
$\Gamma(\frac{\lambda+\nu-n+1}{2}) \Gamma(\frac{\lambda-\nu}{2})$
has a pole of order two,
and therefore
$\KA{\lambda}{\nu}|_{\mathbb{R}^n \setminus \{0\}} = 0$.
Hence
\[
\operatorname{Supp}
(\KA \lambda \nu |_{{\mathbb{R}}^n \setminus \{0\}})=\emptyset
\quad
\text{if }
(\lambda,\nu) \in {\mathbb{X}}.  
\]
Now
we get Proposition \ref{prop:Ksupp} by
 Lemma \ref{lem:SuppInv}.
\end{proof}

\subsection{Renormalization $\AAt_{\lambda,\nu}$ for $\nu \in -\mathbb{N}$}
\label{subsec:5.5}
We have seen in Theorem \ref{thm:poleA} (2) that the distribution
$\A_{\lambda,\nu}$ with holomorphic parameter $\nulambda$
vanishes in the discrete subset of $\mathbb{C}^2$, 
{\it{i.e.}},   
if $\nulambda \in L_{\operatorname{even}}$.  
In this section
 we renormalize
$\A_{\lambda,\nu}$ 
as a function of a single variable $\lambda$
by fixing
$\nu \in -\mathbb{N}$
in order to obtain nonzero symmetry breaking operators.

Suppose $\nu \in -{\mathbb{N}}$.  
Then the factor $(\xi_{n+1}-\xi_0)^{-\nu}$
of the distribution kernel 
\index{sbon}{Kxa@$k_{\lambda,\nu}^{\mathbb{A}}$}
$\ska{\lambda}{\nu}(\xi)$
in \eqref{eqn:klmdmu} is a polynomial,
and thus the distribution kernel $\ska{\lambda}{\nu}$
has a better regularity.

\begin{proposition}\label{prop:nuneg}
Suppose $\nu \in -\mathbb{N}$.
Then
\begin{equation}\label{eqn:K2lmdnu}
\index{sbon}{Kxta@$\KAA{\lambda}{\nu}(x,x_n)$|textbf}
\index{sbon}{Kxnt@$\ka{\lambda}{\nu} (x,x_n)$}
\KAA{\lambda}{\nu}(x,x_n)
:= \frac{1}{\Gamma(\frac{\lambda+\nu-n+1}{2})} \ka{\lambda}{\nu}(x,x_n)
= \Gamma (\frac{\lambda-\nu}{2}) \KA{\lambda}{\nu}(x,x_n)
\end{equation}
extends to a distribution on
$K/M \simeq G/P$
which depends holomorphically
 in $\lambda$ in the whole complex plane.
Then there exists
 a nonzero $G'$-intertwining operator
\begin{equation}
\label{eqn:AAdef}
\index{sbon}{Att@$\AAt_{\lambda,\nu}$|textbf}
\AAt_{\lambda,\nu} : I(\lambda) \to J(\nu), 
\end{equation}
whose distribution kernel 
 is $\KAA{\lambda}{\nu}$.  
Further,
\index{sbon}{1lmd@${\mathbf{1}}_{\lambda}$}
\begin{equation}
\label{eqn:AA1}
\AAt_{\lambda,\nu}(\mathbf{1}_\lambda)
= \frac{\pi^{\frac{n-1}{2}}\Gamma(\frac{\lambda-\nu}{2})}{\Gamma(\lambda)}
   \mathbf{1}_\nu.  
\end{equation}
\end{proposition}

\begin{remark}\label{rem:nuneg}
For a fixed $\nu \in -\mathbb{N}$,
$\nulambda \in L_{\operatorname{even}}$
 if and only if
$\lambda$ is a (simple) pole of
$\Gamma (\frac{\lambda-\nu}{2})$.
In this case,
 the formula \eqref{eqn:AA1}
 amounts to 
\[
\AAt_{\lambda,\nu}(\mathbf{1}_\lambda)
= \frac{\pi^{\frac{n-1}{2}}(-\lambda)! (-1)^{\lambda+l}}{l!} \mathbf{1}_\nu, 
\]
where $l\in\mathbb{N}$ is defined by the relation
$\nu-\lambda = 2l$.
\end{remark}

\begin{proof}[Proof of Proposition \ref{prop:nuneg}]
In the coordinates
$\eta = (\eta_0,\dots,\eta_n) \in S^n$,
\[
\frac{1}{\Gamma(\frac{\lambda+\nu-n+1}{2})}
|\eta_n|^{\lambda+\nu-n}
\]
is a nonzero distribution on $S^n$ which is holomorphic in 
$\lambda$ in the whole complex plane
 because $\eta_n:S^n \to {\mathbb{R}}$
 is regular at $\eta_n=0$.  
On the other hand,
 since $\nu \in -\mathbb{N}$,
$(1-\eta_0)^{-\nu}$
is a polynomial in $\eta_0$,
\[
\frac{1}{\Gamma(\frac{\lambda+\nu-n+1}{2})}
|\eta_n|^{\lambda+\nu-n} (1 - \eta_0)^{-\nu}
\]
is well-defined and gives a distribution on
$S^n \simeq G/P$
with holomorphic parameter $\lambda \in \mathbb{C}$.
Moreover it is nonzero for any $\lambda$ in any neighbourhood of $\eta$ with
$\eta_0 \ne 1$ and $\eta_n = 0$.
Now Proposition \ref{prop:nuneg} follows from \eqref{eqn:iKk}.
\end{proof}

The following proposition shows
 that the renormalized symmetry breaking operator
 $\AAt_{\lambda,\nu}$ is 
 generically regular
 in the sense of Definition \ref{def:regular}.  

\begin{proposition}
[Support of $\KAA{\lambda}{\nu}$]
\label{prop:SuppAtwo}
\index{sbon}{Kxta@$\KAA{\lambda}{\nu}(x,x_n)$}
Suppose $\nu \in -\mathbb{N}$.
Then the distribution kernel $\KAA{\lambda}{\nu}$
of the symmetry breaking operator
\index{sbon}{Att@$\AAt_{\lambda,\nu}$|textbf}
$\AAt_{\lambda,\nu}$
has the following support:
\[
\operatorname{Supp} \KAA{\lambda}{\nu}
= \begin{cases}
                    G/P     &\text{if\/ $\nulambda \notin \backslash\!\backslash$},   \\
                    S^{n-1} &\text{if\/ $\nulambda \in \backslash\!\backslash$}.
   \end{cases}
\]
\end{proposition}

\begin{proof}
As a distribution on $G/P$,
we have
\[
\operatorname{Supp} \frac{1}{\Gamma(\frac{\lambda+\nu-n+1}{2})}
   |\eta_n|^{\lambda+\nu-n}
= \begin{cases}
                    G/P     &\text{if\/ $\nulambda \notin \backslash\!\backslash$},   \\
                    S^{n-1} &\text{if\/ $\nulambda \in \backslash\!\backslash$}.
   \end{cases}
\]
The multiplication of
$(1-\eta_0)^{-\nu} \in C^\infty(S^n)$
is well-defined,
and does not affect the support
 because the equation 
$1-\eta_0 = 0$ holds 
 only if
$\eta = (1,0,\dots,0) \in S^n$.
Thus we proved the Proposition.
\end{proof}

\section{Singular symmetry breaking operator $\B_{\lambda,\nu}$}
\label{sec:B}

We have seen in Lemma \ref{lem:Sol}
 that
 singular symmetry breaking operators exist
 only if 
\index{sbon}{Xl@${\backslash\!\backslash}$}
$\nulambda \in \backslash\!\backslash$.
In this chapter
 we construct a family of singular symmetry breaking operators
\begin{equation}
\label{eqn:Bdef}
\B_{\lambda,\nu}: I(\lambda) \to J(\nu)
\quad\text{for}\quad \nulambda \in \backslash\!\backslash
\end{equation}
 by giving an explicit formula
 of the distribution kernel,
 see \eqref{eqn:KB}.  
The operator $\tB{\lambda}{\nu}$ depends 
 holomorphically on $\lambda \in {\mathbb{C}}$
 (or on $\nu \in {\mathbb{C}}$)
under the constraints 
 that $\nulambda \in \backslash \!\backslash$.  
We find a necessary and sufficient condition
that $\tB{\lambda}{\nu} \ne 0$.  
Other singular symmetry breaking operators
 are only the differential operators
 $\tC{\lambda}{\nu}$
 that will be discussed in the next chapter.  
 
The classification of singular symmetry breaking operators
 will be
given in Proposition \ref{prop:sing}.

\subsection{Singular symmetry breaking operator $\B_{\lambda,\nu}$}
\label{subsec:Blmd}

For $\nulambda \in \backslash\!\backslash$,
we define $k \in \mathbb{N}$ by the relation
\begin{equation}\label{eqn:defk}
\lambda + \nu = n - 1 - 2k.
\end{equation}
In what follows,
we shall fix $k \in \mathbb{N}$
 and discuss the meromorphic continuation by varying 
$\nu \in \mathbb{C}$ (or $\lambda \in \mathbb{C}$)
under the constraints \eqref{eqn:defk}.

For $\xi= (\xi_0,\dots,\xi_{n+1})\in\Xi$,
we set
\begin{equation*}
k_{\lambda,\nu}^{\mathbb{B}}(\xi)
:= 2^{2k+1+\nu} \delta^{(2k)} (\xi_n) (\xi_{n+1}-\xi_0)^{-\nu}.  
\end{equation*}
Then $k_{\lambda,\nu}^{\mathbb{B}}(\xi)$ is 
 a distribution on $\Xi$, 
 when $\operatorname{Re}\nu \ll 0$.  
Further 
$
   k_{\lambda,\nu}^{\mathbb{B}}(\xi)
   \in \mathcal{D}'_{\lambda-n} (\Xi) \simeq I(n-\lambda)^{-\infty}
$
 and, 
 as in \eqref{eqn:PkA}, 
it satisfies a $P'$-covariance
\begin{equation}
\label{eqn:PkB}
k_{\lambda,\nu}^{\mathbb{B}}(m e^{-t H} n \xi)
= e^{\nu t} k_{\lambda,\nu}^{\mathbb{B}}(\xi)
\end{equation}
for any $m e^{-t H} n \in M'A N_+'=P'$.  
By \eqref{eqn:KBnorm} we have:
\index{sbon}{iotaNast@$\iota_N^*$}
\index{sbon}{Kxb@$k_{\lambda,\nu}^{\mathbb{B}}$|textbf}
\index{sbon}{iotaKast@$\iota_K^*$}
\begin{alignat}{2}
&\iota_N^* k_{\lambda,\nu}^{\mathbb{B}} 
&&= (|x|^2+x_n^2)^{-\nu}\delta^{(2k)}(x_n).
\label{eqn:iNkB}
\\
&(\iota_K^* k_{\lambda,\nu}^{\mathbb{B}})(\eta)
&&= 2^{2k+1+\nu} \delta^{(2k)} (\eta_n) (1-\eta_0)^{-\nu}.
\label{eqn:iKkB}
\end{alignat}
In order to give the meromorphic continuation
 of the distribution kernel,
 which is initially holomorphic 
 when $\operatorname{Re}\nu \ll 0$,
 we normalize \eqref{eqn:iNkB} as 
\index{sbon}{KxtB@$\KB{\lambda}{\nu}(x,x_n)$}
\begin{align}\label{eqn:KB}
\KB \lambda\nu(x,x_n) :={}&
\frac{1}{\Gamma(\frac{\lambda-\nu}{2})}
   (|x|^2+x_n^2)^{-\nu} \delta^{(2k)} (x_n)
\\
={}&
\frac{1}{\Gamma(\frac{n-1}{2}-\nu-k)}
   (|x|^2+x_n^2)^{-\nu} \delta^{(2k)} (x_n).
\nonumber
\end{align}

The main properties of $\tB{\lambda}{\nu}$
 are summarized as follows.  
\begin{theorem}\label{thm:BK}
Suppose $\nulambda \in \backslash\!\backslash$.
\begin{enumerate}[\upshape 1)]
\item
For $\nulambda \in\backslash\!\backslash$ with
$\operatorname{Re}(\lambda-\nu) > 0$,
$\KB{\lambda}{\nu}$ is well-defined as
a distribution on $\mathbb{R}^n$,
and satisfies:
\begin{equation}\label{eqn:Bexp}
\KB{\lambda}{\nu} (x,x_n)
= \frac{1}{\Gamma(\frac{n-1}{2}-\nu-k)} \sum_{i=0}^k
     \frac{(-1)^i(2k)!(\nu)_i}{(2k-2i)! \, i!}
\,  |x|^{-2\nu-2i} \delta^{(2k-2i)} (x_n).
\end{equation}
\item
Fix $k\in\mathbb{N}$.
Then $\KB \lambda \nu$
extends to a distribution on $\mathbb{R}^n$
that depends holomorphically on
$\nu$ in the entire plane $\mathbb{C}$ (or $\lambda\in\mathbb{C}$).
\item
$\KB{\lambda}{\nu} \in \mathcal{S}ol(\mathbb{R}^n; \lambda,\nu)$
(see \eqref{eqn:KTeqU} for the definition) for all
$\nulambda \in \backslash\!\backslash$,
and induces a continuous $G'$-intertwining operator
\[
\index{sbon}{Bt@$\B_{\lambda,\nu}$|textbf}
\B_{\lambda,\nu} : I(\lambda) \to J(\nu).
\]
\item
$\B_{\lambda,\nu} = 0$
if and only if $n$ is odd and
$\nulambda \in L_{\operatorname{even}}$.
\end{enumerate}
\end{theorem}

For the proof of Theorem \ref{thm:BK},
we use:
\begin{lemma}\label{lem:4.15}
\quad
\begin{enumerate}[\upshape 1)]
\item
For ${\operatorname{Re}} \nu \ll 0$, 
\begin{equation*}
\Bigl( \frac{\partial}{\partial x_n} \Bigr)^i (|x|^2+x_n^2)^{-\nu} \Bigm|_{x_n=0}
= \begin{cases}
                  0                                     &\text{if\/ $i=2j+1$},   \\
                  \frac{(-1)^j(2j)! \, \Gamma(\nu+j)}{j! \, \Gamma(\nu)} |x|^{-2\nu-2j}
                                                         &\text{if\/ $i=2j$}.
    \end{cases}
\end{equation*}
\item
Suppose\/ $\operatorname{Re} \nu < \frac{n-1}{2} - N$.
Then\/ $|x|^{-2\nu-2j} \in L_{\operatorname{loc}}^1 (\mathbb{R}^{n-1})$
for all\/ $0 \le j \le [\frac{N}{2}]$,
and we have the following identity of distributions on\/ $\mathbb{R}^n$:
\begin{equation*}
(|x|^2+x_n^2)^{-\nu} \delta^{(N)}(x_n)
= \sum_{j=0}^{[\frac{N}{2}]}
     \frac{(-1)^j N! \, \Gamma(\nu+j)}{(N-2j)! \, j! \, \Gamma(\nu)}
     |x|^{-2\nu-2j} \delta^{(N-2j)} (x_n).
\end{equation*}
\end{enumerate}
\end{lemma}

\begin{proof}
1)\enspace
The expansion
\[
(A+y^2)^{\mu} = \sum_{i=0}^\infty
              \frac{\Gamma(\mu+1)}{i! \, \Gamma(\mu+1-i)}
              A^{\mu-i} y^{2i}
\]
implies
\begin{equation}\label{eqn:Aynu}
\Bigl(\frac{\partial}{\partial y}\Bigr)^{2i}  (A+y^2)^\mu \Bigm|_{y=0}
= \frac{(2i)! \, \Gamma(\mu+1)}{i! \, \Gamma(\mu+1-i)} A^{\mu-i}.
\end{equation}
Now the statement is clear.  

2)\enspace
For a test function $\varphi \in C_0^\infty(\mathbb{R}^n)$,
\begin{align*}
& \Bigl( \frac{\partial}{\partial x_n} \Bigr)^N \bigm|_{x_n=0}
   ((|x|^2+x_n^2)^{-\nu} \varphi(x,x_n))
\\
&=
    \sum_{i=0}^{N} \begin{pmatrix} N \\ i \end{pmatrix}
    \left( \Bigl( \frac{\partial}{\partial x_n} \Bigr)^{i} (|x|^2+x_n^2)^{-\nu} \right)
    \left( \Bigl( \frac{\partial}{\partial x_n} \Bigr)^{N-i} \varphi(x,x_n) \right) \biggm|_{x_n=0} .
\end{align*}
Substituting the formula of (1) into the right-hand side, we get (2).
\end{proof} 
\subsection{$K$-finite vectors
 and singular operators $\tB{\lambda}{\nu}$}
\label{subsec:KtypeB}
\begin{proposition}
\label{prop:ABres}
Suppose $(\lambda,\nu) \in \backslash\!\backslash$.  
We define $k\in\mathbb{N}$ by the relation
\eqref{eqn:defk}.  
Then 
$
\langle \KB{\lambda}{\nu}, F\rangle
 =(-1)^k 2^k (2k-1)!!
 \langle \KA{\lambda}{\nu}, F\rangle
$
 for any $F \in I(\lambda)_K$.  
\end{proposition}
We give a proof 
 of Proposition \ref{prop:ABres}
 in parallel to the argument
 of Chapter \ref{sec:kfini}.  
A new ingredient is the following:
\begin{lemma}
\label{lem:7.3}
Suppose 
$\nulambda \in \backslash\!\backslash$.
\begin{enumerate}
\item[{\rm{1)}}]
If $N$ is odd or $\psi \perp {\mathcal{H}}^N(S^{n-1})^{O(n-1)}$, 
 then 
\[
\int_{S^{n-1}}\psi(\omega) \delta^{(2k)}(\omega_n)=0.  
\]
 
\item[{\rm{2)}}]
If $N$ is even, 
 then 
\[\int_{S^{n-1}} \widetilde C_{N}^{\frac n 2 -1}(\omega_n)
       \delta^{(2k)}(\omega_n)=
(-1)^k 2^k(2k-1)!! \frac{d_{n,N}(\lambda,\nu)g(\lambda,\nu)}
{\Gamma(\frac{\lambda+\nu-n+1}{2})}.  
\]
\end{enumerate}
\end{lemma}

\begin{proof}

In light of the residue formula
\[
\frac{1}{\Gamma (\frac{\mu+1}{2})} |t|^\mu \Bigm|_{\mu=-1-2k}
= \frac{(-1)^k}{2^k (2k-1)!!} \delta^{(2k)} (t)
\quad\text{in $\mathcal{D}' (\mathbb{R})$},
\]
 the statements follow from Lemma \ref{lem:Bnint}.  
\end{proof}
Owing to Lemma \ref{lem:7.3}, 
 the following lemma is derived as in Lemma \ref{lem:KtildeFC}.  
\begin{lemma}\label{lem:KBtildeFC}
Let $\nulambda \in \backslash\!\backslash$.  
Let $k \in {\mathbb{N}}$ be as in \eqref{eqn:defk}.  
\begin{enumerate}
\item[{\rm{1)}}]
Suppose $N \in \mathbb{N}$, 
 $\psi \in {\mathcal{H}}^N(S^{n-1})$, 
and $h \in \mathbb{C}[s]$.  
Then we have:
\index{sbon}{Fl@$F_{\lambda}[\psi,h]$}
\[
\langle \KB{\lambda}{\nu}, F_\lambda [ \psi, h] \rangle
=0
\quad
\text{if $N$ is odd or }
\psi \perp {\mathcal{H}}^N(S^{n-1}).  
\]
\item[{\rm{2)}}]
If $N \in 2 {\mathbb{N}}$, 
then 
\[
\langle \KB{\lambda}{\nu}, F_\lambda [ \widetilde{C}_N^{\frac{n}{2}-1} (\omega_n), h] \rangle
= c P_{\frac{\lambda-\nu+N}{2},\frac{\lambda+\nu+N}{2}} (h)
   \prod_{j=0}^{\frac{N}{2}-1}
   \Bigl( \frac{\lambda-\nu}{2} + j \Bigr)
   \Bigl( \frac{\lambda+\nu-n}{2} - j \Bigr), 
\]
where the non-zero constant $c$ is given by 
\[
(-1)^k 2^k (2k-1)!!
\frac{2^{3-\lambda-n} \pi^{\frac{n}{2}} \Gamma(n+N-1)}
       {\Gamma(\frac{n-1}{2}) \Gamma(N+1)} .
\]
\end{enumerate}
\end{lemma}

\begin{proof}
[Proof of Proposition \ref{prop:ABres}]
Clear from the comparison
 of Lemma \ref{lem:KtildeFC} and Lemma \ref{lem:KBtildeFC}.  
\end{proof}
As a special case of Proposition \ref{prop:ABres}, 
 we obtain 
\begin{proposition}\label{prop:B1}
For $\nulambda \in \backslash\!\backslash$,
we set $k \in \mathbb{N}$ by
$\lambda+\nu = n-1-2k$.
Then we have

\begin{equation}
\label{eqn:B1}
\index{sbon}{1lmd@${\mathbf{1}}_{\lambda}$}
\B_{\lambda,\nu} (\mathbf{1}_\lambda)
= \frac{(-1)^k 2^k \pi^{\frac{n-1}{2}} (2k-1)!!}
          {\Gamma(\lambda)}
   \mathbf{1}_\nu.
\end{equation}
\end{proposition}
\begin{proof}
We recall from Proposition \ref{prop:AminK}
 that 
\[
  \tA{\lambda}{\nu}({\bf{1}}_{\lambda})
  =
  \frac{\pi^{\frac{n-1}{2}}}{\Gamma(\lambda)}{\bf{1}}_{\nu}.  
\]
Now the statement is immediate from Proposition \ref{prop:ABres}.  
\end{proof}
As another consequence of Proposition \ref{prop:ABres}, 
 we have 
\begin{proposition}\label{prop:4.2B}
$\langle \KB{\lambda}{\nu},  F\rangle=0$
 for any $F\in I(\lambda)_K$
 if and only if\/ $\nulambda \in 
\index{sbon}{Leven@$L_{\operatorname{even}}$}
L_{\operatorname{even}}$. 
\end{proposition}

\begin{proof}
The proof is the same as
 that of Proposition \ref{prop:4.2}. 
\end{proof}

\subsection{Proof of Theorem \ref{thm:BK}}
\label{subsec:7.3}
\begin{proof}[Proof of Theorem \ref{thm:BK}]
1)\enspace
Let $k$ be defined as in \eqref{eqn:defk}.  
Then
$|x|^{-2\nu-2i} \in L^1_{\operatorname{loc}}(\mathbb{R}^{n-1})$
for all $0 \le i \le k$
if and only if
$\operatorname{Re}(\lambda-\nu) > 0$.
By Lemma \ref{lem:4.15}
we get the expansion formula \eqref{eqn:Bexp},   
which shows that
$\KB{n-1-2k-\nu}{\nu}$
extends to a distribution on $\mathbb{R}^n$ 
depending holomorphically on $\nu$ in the domain
$\{\nu\in\mathbb{C}: \operatorname{Re}\nu < \frac{n-1}{2}-k\}$.  

2)\enspace
By the expression \eqref{eqn:Bexp},
the assertion is clear.

3)\enspace
It is easy to see
$\KB{\lambda}{\nu} \in \mathcal{S}ol(\mathbb{R}^n; \lambda,\nu)$
for $\operatorname{Re}\nu \ll 0$.
Since
$\KB{\lambda}{\nu}$ extends holomorphically to
$\nulambda \in\backslash\!\backslash$
as a distribution on $\mathbb{R}^n$,
the third statement follows from Proposition \ref{prop:meroK}.

4)
\enspace
Clear from Proposition \ref{prop:4.2B}.  
\end{proof}

\subsection{Support of the distribution kernel
 of $\B_{\lambda,\nu}$}
\label{subsec:Bsupp}
We have seen in Theorem \ref{thm:BK}
 that $\B_{\lambda,\nu} \ne 0$
 if and only if 
 $(\lambda,\nu) \in \backslash\!\backslash
 \setminus  L_{\operatorname{even}}$. 
In this section,
 we find the support
 of the distribution kernel of $\KB{\lambda}{\nu}$.  
\begin{proposition}\label{prop:Bsupp}
Assume
\[
\nulambda \in \backslash\!\backslash - L_{\operatorname{even}}
=\begin{cases}
\backslash\!\backslash - L_{\operatorname{even}}
\quad
&\text{($n:odd$),}
\\
\backslash\!\backslash
&\text{($n:even$).}
\end{cases}
\]
Then the kernel of the non-zero singular symmetry 
 breaking operator $\B_{\lambda,\nu}$
 has the following support:
\[
\operatorname{Supp}\KB{\lambda}{\nu}
= \begin{cases}
     \{[p_+]\}   
     &\text{if\/ $\nulambda
                  \in {\mathbb{X}}- L_{\operatorname{even}}$}, 
\\
     S^{n-1}    &\text
                 {if\/ $\nulambda
                       \in \backslash\!\backslash -{\mathbb{X}}$}. 
   \end{cases}
\]
\end{proposition}
\begin{proof}
Suppose $(\lambda,\nu) \in \backslash\!\backslash \setminus L_{\operatorname{even}}$.  
By Lemma \ref{lem:SuppInv}
 and the definition \eqref{eqn:KB} 
 of $\KB{\lambda}{\nu}$, 
 $\operatorname{Supp}\KB{\lambda}{\nu}$
 is either $\{[p_+]\}$ or $S^{n-1}$.  
Since $|x|^2+x_n^2$ is non-zero 
 on ${\mathbb{R}}^n \setminus \{0\}$, 
 $(|x|^2+x_n^2)^{-\nu}$ is a nowhere vanishing smooth
 function on ${\mathbb{R}}^n \setminus \{0\}$.  
Therefore,
 the restriction of $\KB{\lambda}{\nu}$
 to the open set ${\mathbb{R}}^n \setminus \{0\}$
 vanishes 
 if and only if 
 $\Gamma(\frac{\lambda-\nu}{2})=\infty$, 
namely, 
 $(\lambda, \nu) \in /\!/$.  
Thus Proposition is proved.   
\end{proof}

\subsection{Renormalization $\BB_{\lambda,\nu}$
 for $\nulambda \in L_{\operatorname{even}}$
 with $n$ odd}
\label{subsec:B2}

For $n$ odd,
 the singular symmetry breaking operator 
 $\B_{\lambda,\nu}$ vanishes 
 when $\nulambda \in L_{\operatorname{even}}$, 
 see Theorem \ref{thm:BK}.  
As we renormalized 
 the (generically) regular symmetry breaking operator
 $\A_{\lambda,\nu}$
for $\nu \in -\mathbb{N}$ in Section \ref{subsec:5.5},
we will renormalize $\B_{\lambda,\nu}$
for $\nu \in -\mathbb{N}$
 as follows.  
For $\nulambda \in L_{\operatorname{even}}$
with $n$ odd,
we define $k\in\mathbb{N}$ by
$\lambda+\nu=n-1-2k$ and set
\index{sbon}{KxttB@$\KBB{\lambda}{\nu}(x,x_n)$|textbf}
\begin{align}
\KBB{\lambda}{\nu}(x,x_n)
:={}& (|x|^2 + x_n^2)^{-\nu} \delta^{(2k)} (x_n)
\nonumber
\\
\label{eqn:KBnorm}
={}& \sum_{i=0}^k
     \frac{(-1)^i(2k)!(\nu)_i}{(2k-2i)! \, i!}
\,  |x|^{-2\nu-2i} \delta^{(2k-2i)} (x_n), 
\end{align}
see \eqref{eqn:Bexp}
 for $\KB \lambda \nu$. 
Then $\KBB \lambda \nu \in \mathcal{S}ol(\mathbb{R}^n;\lambda,\nu)$
 and 
 we have a $G'$-intertwining operator
\begin{equation}
\index{sbon}{Btt@$\BB_{\lambda,\nu}$|textbf}
\label{eqn:BBdef}
\BB_{{\lambda},{\nu}}
:
I(\lambda) \to 
J(\lambda)
\end{equation}
with $\KBB{\lambda}{\nu}$ 
 its distribution kernel.  
Since $L_{\operatorname{even}}$ is a discrete set
 in ${\mathbb{C}}^2$, 
there is no continuous parameter for $\BB_{{\lambda},{\nu}}$.  
We note that $L_{\operatorname{even}} \subset \mathbb{X}$
 for $n$ odd.  
We shall prove in Theorem \ref{thm:reduction} (4)
 that $\BB_{\lambda,\nu}$ is a scalar multiple of
 $\AAt_{\lambda,\nu}$ 
for any $\nulambda\in L_{\operatorname{even}}$
if $n$ is odd.  
The following proposition is clear from \eqref{eqn:KBnorm}.  
\begin{proposition}\label{prop:BB}
Suppose $n$ is odd and
$\nulambda \in L_{\operatorname{even}}$.
Then $\BB_{\lambda,\nu} \ne0$
and
\begin{equation*}
\operatorname{Supp} \KBB{\lambda}{\nu} = S^{n-1}.
\end{equation*}
\end{proposition}

\section{Differential symmetry breaking operators}
\label{sec:10}

In this chapter
 we give a brief review
 on differential symmetry breaking operators.  
Nontrivial such operators from $I(\lambda)$ to $J(\nu)$ exist
 if and only if 
\index{sbon}{Xr@${/\!/}$}
$\nu - \lambda \in /\!/$, 
 and explicit formulae
 of all such operators were recently found
 in \cite{Juhl, KOSS}.  
The new ingredient is Proposition \ref{prop:C11}, 
 which gives an explicit action
 of the normalized differential symmetry breaking operators
 $\C_{\lambda,\nu}$
 on the spherical vectors.  
In Chapter \ref{sec:reduction}, 
 we shall see 
 that these differential symmetry breaking operators arise
 as the residues
 of the (generically) regular symmetry breaking operators
 ${\mathbb{A}}_{\lambda,\nu}$
 except for the discrete set $L_{\operatorname{even}}$
 (Theorem \ref{thm:reduction}).

\subsection{Power of the Laplacian}
\label{subsec:powerL}
We begin with the classical results
 on conformally covariant differential operators
 acting on line bundles 
 on the sphere (\lq\lq{$G=G'=O(n+1,1)$ case}\rq\rq\
 in the general setting of Chapter \ref{sec:gen}).

Let ${\mathfrak {g}}_{\mathbb{C}}$
 be the complexified Lie algebra
 of ${\mathfrak {g}}={\mathfrak {o}}(n+1,1)$,
 and 
$
   {\mathfrak {p}}_{\mathbb{C}} 
  ={\mathfrak {m}}_{\mathbb{C}} + {\mathfrak {a}}_{\mathbb{C}}
   +({\mathfrak {n}}_+)_{\mathbb{C}}
$
 the complexification
of the Levi decomposition
of the minimal parabolic subalgebra
$
   {\mathfrak {p}} 
  ={\mathfrak {m}} + {\mathfrak {a}}+{\mathfrak {n}}_+
$ 
 (see \eqref{eqn:Lang}).  
We note
 that $
   {\mathfrak {p}}_{\mathbb{C}} $
 is a maximal parabolic subalgebra
 of $
   {\mathfrak {g}}_{\mathbb{C}}\simeq {\mathfrak {o}}(n+2,{\mathbb{C}})$.  
Let $H$ be the generator
 of ${\mathfrak {a}}$
 defined in \eqref{eqn:aH}
\index{sbon}{hypelt@$H$}
 and denote by ${\mathbb{C}}_{\lambda}$ 
 the ${\mathfrak {p}_{\mathbb{C}}}$-module given by
\[
 {\mathfrak {p}_{\mathbb{C}}}
 \to
 {\mathfrak {p}_{\mathbb{C}}}/({\mathfrak {m}}_{\mathbb{C}}+{\mathfrak {n}}_{\mathbb{C}}^+)
     \simeq {\mathfrak{a}}_{\mathbb{C}} ={\mathbb{C}}H
     \to {\mathbb{C}},
     \quad
     t H \mapsto t \lambda.  
\]
We define a generalized Verma module by
\[
 \operatorname{ind}_{{\mathfrak {p}_{\mathbb{C}}}}^{{\mathfrak {g}_{\mathbb{C}}}}(\lambda)
 \equiv
 \operatorname{ind}_{{\mathfrak {p}_{\mathbb{C}}}}^{{\mathfrak {g}_{\mathbb{C}}}}({\mathbb{C}}_{\lambda})
 :=
 U({\mathfrak {g}_{\mathbb{C}}}) \otimes_{U({\mathfrak {p}_{\mathbb{C}}})}{\mathbb{C}}_{\lambda}.  
\]
Then the following result holds
 (e.g. \cite{EG, GGV}):

\begin{lemma}
\label{lem:Vermagg}
Let $n \ge 2$.  
Let $G=O(n+1,1)$ with $n \ge 2$.  
Then the following three conditions
 on $(\lambda, \nu) \in {\mathbb{C}}^2$
 are equivalent:
\begin{enumerate}
\item[{\rm{(i)}}]
$
\operatorname{Hom}_{\mathfrak {g}}
(\operatorname{ind}_{{\mathfrak {p}_{\mathbb{C}}}}^{{\mathfrak {g}_{\mathbb{C}}}}(-\nu),
\operatorname{ind}_{{\mathfrak {p}_{\mathbb{C}}}}^{{\mathfrak {g}_{\mathbb{C}}}}(-\lambda)
)\ne 0.  
$
\item[{\rm{(ii)}}] 
$\operatorname{Hom}_{\mathfrak {g}, P}
      (\operatorname{ind}_{{\mathfrak {p}}_{\mathbb{C}}}^
       {{\mathfrak {g}}_{\mathbb{C}}}
       (-\nu),
       \operatorname{ind}_{{\mathfrak {p}}_{\mathbb{C}}}^
       {{\mathfrak {g}}_{\mathbb{C}}}
       (-\lambda)
) \ne 0$.  

\item[{\rm{(iii)}}] $(\lambda,\nu)=(\frac n 2-l, \frac n 2 + l)$
for some $l = 0,1,2,\cdots$.  
\end{enumerate}
In this case,
 the space of homomorphisms 
 in (i) (also in (ii))
 is one-dimensional
 and the resulting $G$-intertwining differential operator
 ($G=G'$ and $H=H'$ case in Fact \ref{fact:Diff})
 is given as the power of the Laplacian 
 in the noncompact model:
\[
\Delta^{l}:I(\frac{n}{2}-l) \to I(\frac{n}{2}+l).  
\]
\end{lemma}

\begin{remark}
\label{rem:Yamabe}
In the compact model
 (see Definition \ref{def:model}), 
 the $G$-differential intertwining operator
 for $l=1$ is given 
 by the {\it{Yamabe operator}}
 in conformal geometry, 
 which is define to be
\[
\Delta_{S^n}
-\frac{n-2}{4(n-1)}
\kappa_{S^n}
=\Delta_{S^n}
-\frac 1  4 n(n-2), 
\]
where $\Delta_{S^n}$ is the Laplacian
 and $\kappa_{S^n}$ is the scalar curvature
 of $S^n$
 with standard metric
 (e.g. \cite[Example 2.2]{KO1}).  
\end{remark}

\begin{remark}
\label{rem:Vermagg}
As we have seen in Section \ref{subsec:Knapp}
 the Knapp--Stein intertwining operator
\[
\index{sbon}{ttt3@$\T{\lambda}{n-\lambda}$}
\T{\lambda}{n-\lambda}:I(\lambda) \to I(n-\lambda)
\]
is singular
 in the sense of Definition \ref{def:regular}
 (with $G=G'$)
 if and only if $l:=\frac{n}{2}-\lambda \in {\mathbb{N}}$, 
 and reduces to a scalar multiple
 of the $l$-th power 
 $\Delta^l$
 of the Laplacian,   
 see \eqref{eqn:KSres}.  
Thus any differential $G$-intertwining operator
 between spherical principal series representations
 of $G$
 is obtained as the residue of the (generically) 
 regular intertwining operators
 \cite{EG}.  
A similar result does not hold
 for symmetry breaking operators
 ($G \ne G'$ case)
 as we shall see in Remark \ref{rem:CLeven}.  
\end{remark}

\subsection{Juhl's family of differential operators}
\label{subsec:Juhl}

For 
$\nulambda \in /\!/$,
we define $l \in \mathbb{N}$ by
\[
\nu - \lambda = 2l.
\]
We recall from \eqref{eqn:ajlmu} and \eqref{eqn:Cmul2}
 that 
\index{sbon}{Ctst@$\widetilde C_{2l}^\mu (s,t)$}
$\widetilde C_{2l}^{\mu}(s,t)
=\sum_{j=0}^l a_j(l;\mu)s^j t^{2l-2j}$
 is a polynomial 
of $s$ and $t$
 built on the Gegenbauer polynomial, 
 and this definition makes sense
 if $s$, $t$ are elements
 in any commutative algebra $R$.  
In particular,
 taking
 $s=-\Delta_{{\mathbb{R}}^{n-1}}$, 
 $t=\frac{\partial}{\partial x_n}$
 in $R={\mathbb{C}}[\frac{\partial}{\partial x_1}, \cdots, \frac{\partial}{\partial x_n}]$, 
 we obtain a differential operator
 $C^{\infty}({\mathbb{R}}^{n})
\to 
 C^{\infty}({\mathbb{R}}^{n-1})$
 by
\index{sbon}{Ct@$\C_{\lambda,\nu}$|textbf}
\begin{align}
\C_{\lambda,\nu} :={}&
\operatorname{rest}\circ\widetilde{C}_{\nu-\lambda}^{\lambda-\frac{n-1}{2}}
   \Bigl( -\Delta_{\mathbb{R}^{n-1}}, \frac{\partial}{\partial x_n} \Bigr)
\label{eqn:Clmdnu}
\\
={}&
\operatorname{rest}\circ\sum_{j=0}^{l} a_j
   \Bigl( \frac{\nu-\lambda}{2} ; \lambda - \frac{n-1}{2} \Bigr)
   \Bigr( -\Delta_{\mathbb{R}^{n-1}} \Bigr)^j
   \Bigl( \frac{\partial}{\partial x_n} \Bigr)^{2l-2j}
\nonumber
\\
={}&
\operatorname{rest}\circ
\sum_{j=0}^{l}
   \frac{2^{2l-2j}
\prod_{i=1}^{l-j}
   \Bigl( \frac{\lambda+\nu-n-1}{2} + i \Bigr)
}
{j! (2l-2j)!}
   \Delta_{\mathbb{R}^{n-1}}^j \Bigr( \frac{\partial}{\partial x_n} \Bigr)^{2l-2j},
\nonumber
\end{align}
where $\operatorname{rest}$ denotes the restriction to
$x_n=0$.

Then $\tC {\lambda}{\nu}$ coincides with Juhl's family 
 of conformally covariant differential operators
 \cite{Juhl}, 
that is, 
 $\C_{\lambda,\nu}$ induces a $G'$-intertwining operator
 ({\it{differential symmetry breaking operator}})
\begin{equation}
\label{eqn:Cdef}
\tC {\lambda}{\nu}
:
I(\lambda) \to 
J(\nu), 
\end{equation}
 in the noncompact model.  
See \cite{Juhl, xkeastwood60, KOSS}
 for three different proofs.

It is immediate from the definition 
\eqref{eqn:Clmdnu}
 that the distribution kernel
 of the differential symmetry breaking operators
 $\C_{\lambda,\nu}$
 is given by 
\index{sbon}{KxttCnu@$\KC{\lambda}{\nu}$|textbf}
\begin{align}
\KC\lambda\nu
=&
\sum_{j=0}^l
\frac{2^{2l-2j} \prod_{i=1}^{l-j}
      (\frac{\lambda+\nu-n-1}{2}+i)}
     {j! (2l-2j)!}
(\Delta_{{\mathbb{R}}^{n-1}}^j\delta(x_1,\cdots,x_{n-1}))
\delta^{(2l-2j)}(x_n)
\notag
\\
=&
\widetilde C_{2l}^{\lambda-\frac{n-1}{2}}
(-\Delta_{\mathbb{R}^{n-1}}, \frac{\partial}{\partial x_n})
\delta(x_1, \cdots, x_{n-1})
\delta(x_n).  
\label{eqn:KC}
\end{align}

The differential operator $\C_{\lambda,\nu}$
is homogeneous 
 of degree $\nu-\lambda$, 
Here are examples
 of low degrees.  
\[
\C_{\lambda,\nu}
= \begin{cases}
               \operatorname{rest}     & \text{if $l = 0$},
     \\
               \operatorname{rest} \circ
                      \Bigl( (\lambda+\nu-n+1) \dfrac{\partial^2}{\partial x_n^2}
                               + \Delta_{\mathbb{R}^{n-1}} \Bigr)
              & \text{if $l = 1$},
     \\
              \frac1 2\operatorname{rest} \circ
              \Bigl( \frac 1 3 (\lambda+\nu-n+1)(\lambda+\nu-n+3) \dfrac{\partial^4}{\partial x_n^4}
  &                  
\\
\quad+ 2(\lambda+\nu-n+1) \Delta_{\mathbb{R}^{n-1}}\dfrac{\partial^2}{\partial x_n^2}
+\Delta_{\mathbb{R}^{n-1}}^2
\Bigr)
& \text{if $l = 2$}.  
   \end{cases}
\]

We define a maximal parabolic subalgebra
 ${\mathfrak{p}}_{\mathbb{C}}'
={\mathfrak{m}}_{\mathbb{C}}'+{\mathfrak{a}}_{\mathbb{C}}
+({\mathfrak{n}}_+')_{\mathbb{C}}$
 of ${\mathfrak{g}}_{\mathbb{C}}$
 as the complexification
 of ${\mathfrak{p}}'
={\mathfrak{m}}'+{\mathfrak{a}}+({\mathfrak{n}}_+')$
 (see \eqref{eqn:Lang}), 
 so that
$
 {\mathfrak {p}}_{\mathbb{C}}'= {\mathfrak {p}}_{\mathbb{C}}
  \cap {\mathfrak {g}}_{\mathbb{C}}'.   
$

\begin{fact}[\cite{Juhl,KOSS}]
\label{fact:5.2}
Suppose $\lambda,\nu \in {\mathbb{C}}$.  
\begin{enumerate}
\item[{\rm{1)}}]
The following three conditions are equivalent:
\begin{alignat*}{2}
&\text{{\rm{(i)}}}\enspace
&&\operatorname{Hom}_{\mathfrak {g}'}
(\operatorname{ind}_
{{\mathfrak {p}}'_{\mathbb{C}}}^{{\mathfrak {g}}'_{\mathbb{C}}}(-\nu),
\operatorname{ind}_{{\mathfrak {p}}_{\mathbb{C}}}^{{\mathfrak {g}}_{\mathbb{C}}}
(-\lambda)
)\ne 0.  
\\
&\text{{\rm{(ii)}}}\enspace
&&\dim \operatorname{Hom}_{\mathfrak {g}'}
(\operatorname{ind}_{{\mathfrak {p}}'_{\mathbb{C}}}^{{\mathfrak {g}}'_{\mathbb{C}}}(-\nu),
\operatorname{ind}_{{\mathfrak {p}}_{\mathbb{C}}}^{{\mathfrak {g}}_{\mathbb{C}}}(-\lambda)
)=1.  
\\
&\text{{\rm{(iii)}}}\enspace
&&
\lambda- \nu =0,-1,-2,\cdots.  
\end{alignat*}
\item[{\rm{2)}}]
The following three conditions are equivalent:
\begin{alignat*}{2}
&\text{{\rm{(i)}}}\enspace
&&
\operatorname{Hom}_{\mathfrak {g}',P'}
(\operatorname{ind}_
{{\mathfrak {p}}'_{\mathbb{C}}}
^{{\mathfrak {g}}'_{\mathbb{C}}}(-\nu),
\operatorname{ind}_{{\mathfrak {p}}_{\mathbb{C}}}
^{{\mathfrak {g}}_{\mathbb{C}}}(-\lambda)
)\ne 0.  
\\
&\text{{\rm{(ii)}}}\enspace
&&
\dim \operatorname{Hom}_{\mathfrak {g}',P'}
(\operatorname{ind}_{{\mathfrak {p}}'_{\mathbb{C}}}^{{\mathfrak {g}}'_{\mathbb{C}}}(-\nu),
\operatorname{ind}_{{\mathfrak {p}}_{\mathbb{C}}}^{{\mathfrak {g}}_{\mathbb{C}}}(-\lambda)
)=1.  
\\
&\text{{\rm{(iii)}}}\enspace
&&
\lambda- \nu =0,-2,-4,\cdots, 
\text{ namely, }
\nulambda \in /\!/.  
\end{alignat*}
\item[{\rm{3)}}]
Assume one of (hence, all of) 
 the equivalent conditions
 in (1)
 is satisfied.  
Then the resulting ${\mathfrak {g}}'$-intertwining
 differential operator
 $I(\lambda) \to J(\nu)$
(see Fact \ref{fact:Diff})
 is homogeneous of degree
$\nu-\lambda \in \mathbb{N}$.  
Further,
 it induces a $G'$-intertwining operator
\[
  C_{\lambda,\nu}:
  I(\lambda) \oplus \chi I(\lambda) \to J(\nu) \oplus \chi' J(\nu), 
\] 
 where $\chi$ and  $\chi'$ are defined in \eqref{char:chi}.
Furthermore, 
${\mathbb{C}}C_{\lambda,\nu}$ transforms
 under $G'/G_0'$ by the character $(\chi')^{\nu-\lambda}$.
\item[{\rm{4)}}]
For $\nu-\lambda \in 2\mathbb{N}$,
$C_{\lambda,\nu}$ is a scalar multiple of\/ 
$\C_{\lambda,\nu}$.
\end{enumerate}
\end{fact}

\medskip
\begin{remark}
\label{rem:Cres}
We shall see
 in Theorem \ref{thm:reduction} (2)
 that the differential symmetry breaking operators
 $\C_{\lambda,\nu}$
 are obtained 
as the residues
 of the (generically) regular symmetry breaking operators
\index{sbon}{aln@${\mathbb{A}}_{\lambda,\nu}$}
 ${\mathbb{A}}_{\lambda,\nu}$
 for $(\lambda,\nu) \in \mathbb{C}^2 \setminus L_{\operatorname{even}}$,
 however,
 $\C_{\lambda, \nu}$ cannot be obtained
 as a residue
 if $(\lambda,\nu) \in L_{\operatorname{even}}$.  
This phenomenon is reflected
 by the jump in the dimension
 of $\operatorname{Hom}_{G'}(I(\lambda),J(\nu))$
  at $(\lambda,\nu) \in L_{\operatorname{even}}$
 (see Theorem \ref{thm:dimHom} (1)).  
It should be compared with the classical fact
 that all conformally covariant differential operators
 for densities
 on the standard sphere $S^n$
 ($=G/P$)
 are given by residues of the Knapp--Stein intertwining operators,
 as we saw in Remark \ref{rem:Vermagg}.   
\end{remark}
\subsection
{The kernel of the differential symmetry breaking operator $\C_{\lambda,\nu}$}
\label{subsec:KerC}
We recall
 that ${\mathbf{1}}_{\lambda}$ is 
 the normalized spherical vector
 in the principal series representation $I(\lambda)$.  
Since
$\C_{\lambda,\nu} : I(\lambda) \to J(\nu)$
is a $G'$-homomorphism,
$\C_{\lambda,\nu} (\mathbf{1}_\lambda)$
is a scalar multiple of
$\mathbf{1}_\nu \in J(\nu)$.
In this section
 we find this scalar explicitly.  
In particular,
 we obtain 
 a necessary and sufficient condition
on $\nulambda \in /\!/$ such that
$\mathbf{1}_\lambda \in \operatorname{Ker}\C_{\lambda,\nu}$.
We begin with the following:
\begin{lemma}\label{lem:nLap}
For
$(x,x_n) \in \mathbb{R}^{n-1} \oplus \mathbb{R}$
 and $\lambda \in {\mathbb{C}}$, 
\begin{align*}
& \Delta_{\mathbb{R}^{n-1}}^j \Bigl( \frac{\partial}{\partial x_n} \Bigr)^{2l-2j}
    (1+|x|^2+x_n^2)^{-\lambda} \Bigm|_{x=0,x_n=0}
\\
={}&
\frac{2^{2j}(2l-2j)!}{(l-j)!} \,
  \frac{\Gamma(-\lambda+1)\Gamma(\frac{n-1}{2}+j)}{\Gamma(-\lambda-l+1)\Gamma(\frac{n-1}{2})}
\\
={}&
\frac{2^{2j}(2l-2j)!}{(l-j)!}
  (-\lambda-l)_l \Bigl(\frac{n-1}{2}\Bigr)_j .
\end{align*}
Here 
\index{sbon}{x()poch@$(x)_j$|textbf}
 $(x)_j=x(x+1) \cdots (x+j-1)=\frac{\Gamma(x+j)}{\Gamma(x)}$
 is the Pochhammer symbol.  
\end{lemma}

\begin{proof}
It follows from \eqref{eqn:Aynu}
 that 
\begin{equation}\label{eqn:n0}
\Bigl( \frac{\partial}{\partial x_n} \Bigr)^{2i}
(1+|x|^2+x_n^2)^\mu \Bigm|_{x_n=0}
= \frac{(2i)! \, \Gamma(\mu+1)}{i! \, \Gamma(\mu+1-i)}
   (1+|x|^2)^{\mu-i}.
\end{equation}
By using \eqref{eqn:bfunction} iteratedly, 
we get
\[
\Delta_{\mathbb{R}^{n-1}}^j |x|^{2j}
= 2^{2j} j! \, \frac{\Gamma(\frac{n-1}{2}+j)}{\Gamma(\frac{n-1}{2})}.
\]
Hence
\begin{align}\label{eqn:Lap0}
 \Delta_{\mathbb{R}^{n-1}}^j (1+|x|^2)^{\mu} |_{x=0}
&= \begin{pmatrix} \mu \\ j \end{pmatrix} \Delta_{\mathbb{R}^{n-1}}^j |x|^{2j}
\nonumber
\\
&= \frac{2^{2j}\Gamma(n+1)\Gamma(\frac{\mu-1}{2}+j)}
            {\Gamma(\mu+1-j)\Gamma(\frac{n-1}{2})}.  
\end{align}
Combining \eqref{eqn:n0} and \eqref{eqn:Lap0},
completes the proof of the Lemma.
\end{proof}

\begin{proposition}\label{prop:C11}
Suppose
$\nulambda \in /\!/$
and we define $l \in \mathbb{N}$ by
$\nu-\lambda = 2l$.
Then
\[
\index{sbon}{Ct@$\C_{\lambda,\nu}$}
\index{sbon}{1lmd@${\mathbf{1}}_{\lambda}$}
\C_{\lambda,\nu} (\mathbf{1}_\lambda)
= \frac{(-1)^l 2^{2l} (\lambda)_{2l}}{l!} \mathbf{1}_\nu.
\]
In particular,
$\mathbf{1}_\lambda \in I(\lambda)$
lies in the kernel of the symmetry breaking operator
$\C_{\lambda,\nu} : I(\lambda) \to J(\nu)$
if and only if
$\lambda \in \{ 0,-1,-2,\dotsc, -2l+1 \}$.
\end{proposition}

\begin{proof}
Let 
\index{sbon}{ajlmu@$a_j (l; \mu)$}
$a_j(l;\mu)$ be defined
 by \eqref{eqn:ajlmu}, 
and we set
\[
D(l,\mu)
:= \sum_{j=0}^l a_j (l;\mu) (-\Delta)^j \Bigl(\frac{\partial}{\partial x_n}\Bigr)^{2l-2j} .
\]
By applying Lemma \ref{lem:nLap},
 we have
\begin{align*}
& D(l,\mu) (1+|x|^2+x_n^2)^{-\lambda} |_{x=0,x_n=0}
\\
&=2^{2l} (-\lambda-l+1)_l \sum_{j=0}^l \frac{1}{j!(l-j)!} (\mu+l)_{l-j} \Bigl(\frac{n-1}{2}\Bigr)_j
\\
&=
\frac{2^{2l} (-\lambda-l+1)_l (\mu+l+\frac{n-1}{2})_l}{l!}.
\end{align*}
In the last equation,
we have used the identity
\[
\sum_{j=0}^l \frac{l!}{j! \, (l-j)!} (p)_j (q)_{l-j}
= (p+q)_l.
\]

Since
$\C_{\lambda,\nu} = D(l,\lambda-\frac{n-1}{2})$,
we get
\begin{align*}
\C_{\lambda,\nu} (\mathbf{1}_\lambda)
&= \frac{2^{2l} (-\lambda-l+1)_l (\lambda+l)_l}{l!} \mathbf{1}_\nu
\\
&= \frac{(-1)^l 2^{2l} (\lambda)_{2l}}{l!} \mathbf{1}_\nu.
\end{align*}
Thus we have proved the proposition.
\end{proof}

\begin{remark}
\label{rem:Delta1}
We gave in \eqref{eqn:D1}
an explicit formula
 of the action of the differential $G'$-intertwining operator
 $\Delta_{{\mathbb{R}}^m}^j:J(\frac m 2-j) \to J(\frac m 2+j)$
 ($j \in {\mathbb{N}}$)
 on the spherical vector ${\bf{1}}_{\nu}$
 as 
\[
\Delta_{{\mathbb{R}}^m}^j({\bf{1}}_{\nu})
=
\frac{(-1)^j 2^{2j} \Gamma(\frac m 2+j)}
     {\Gamma(\frac m 2-j)}
{\bf{1}}_{m-\nu}
\]
by using the residue formula
 of the Knapp--Stein intertwining operator.  
The formula \eqref{eqn:Lap0} gives 
 another elementary proof for this.  
\end{remark}

\section{Classification of symmetry breaking operators}
\label{sec:dim}
In this chapter
 we give a complete classification
 of symmetry breaking operators
{}from the spherical principal series representation 
 $I(\lambda)$ of $G=O(n+1,1)$
 to the one $J(\nu)$ of $G'=O(n,1)$.  

\subsection{Classification of symmetry 
 breaking operators}
\label{subsec:class}

So far
 we have constructed
 symmetry breaking operators
\index{sbon}{Alnt@$\A_{\lambda,\nu}$}
 $\tA{\lambda}{\nu}$,
\index{sbon}{Att@$\AAt_{\lambda,\nu}$}
 $\AAt_{{\lambda},{\nu}}$,
\index{sbon}{Bt@$\B_{\lambda,\nu}$}
 $\tB{\lambda}{\nu}$,
\index{sbon}{Btt@$\BB_{\lambda,\nu}$}
 $\BB_{{\lambda},{\nu}}$,
 and 
\index{sbon}{Ct@$\C_{\lambda,\nu}$}
 $\CC{\lambda}{\nu}$,
see \eqref{eqn:Adef},
 \eqref{eqn:AAdef},
 \eqref{eqn:Bdef},
 \eqref{eqn:BBdef},
 and \eqref{eqn:Cdef},
 respectively.    
We shall prove
 that any element 
 in 
\index{sbon}{H1@$H(\lambda, \nu)$}
$H(\lambda,\nu)=\operatorname{Hom}_{G'}(I(\lambda),J(\nu))$
 is a linear combination
 of these operators,
 and complete a classification
 of symmetry breaking operators
for all $\lambda$ and $\nu$
 (Theorem \ref{thm:Hombasis} below).  
The above operators are not necessarily linearly independent.  
In the next chapter,
 we shall list linear relations
 among these operators
 as {\it{residue formulae}}.

\begin{theorem}
[classification of symmetry breaking operators]
\label{thm:Hombasis}
If $n$ is odd
\index{sbon}{XX@${{\mathbb{X}}}$}
\[
H(\lambda,\nu) =
   \begin{cases}
       \mathbb{C} \BB_{\lambda,\nu} \oplus \mathbb{C} \C_{\lambda,\nu}
       & \text{if\/ $\nulambda \in L_{\operatorname{even}}$},
   \\
       \mathbb{C} \C_{\lambda,\nu}
       & \text{if\/ $\nulambda \in /\!/ \setminus L_{\operatorname{even}}$},
   \\
       \mathbb{C} \B_{\lambda,\nu} 
       & \text{if\/ $\nulambda \in \backslash\!\backslash \setminus \mathbb{X}$},
   \\
       \mathbb{C} \A_{\lambda,\nu}
       & \text{if\/ $\nulambda \in \mathbb{C}^2 \setminus (\backslash\!\backslash \cup /\!/)$}.
   \end{cases}
\]
If $n$ is even
\[
H(\lambda,\nu) =
   \begin{cases}
       \mathbb{C} \AAt_{\lambda,\nu} \oplus\mathbb{C} \C_{\lambda,\nu}
       & \text{if\/ $\nulambda \in L_{\operatorname{even}}$},
   \\
       \mathbb{C} \C_{\lambda,\nu}
       & \text{if\/ $\nulambda \in /\!/ \setminus L_{\operatorname{even}}$},
   \\
       \mathbb{C} \B_{\lambda,\nu} 
       & \text{if\/ $\nulambda \in \backslash\!\backslash \setminus \mathbb{X}$},
   \\
       \mathbb{C} \A_{\lambda,\nu}
       & \text{if\/ $\nulambda \in \mathbb{C}^2 \setminus (\backslash\!\backslash \cup /\!/)$}.
   \end{cases}
\]
\end{theorem}

\begin{remark}
1)\enspace
We shall see in Proposition \ref{prop:ABC}
 that
$\BB_{\lambda,\nu}$
is a nonzero multiple of
$\AAt_{\lambda,\nu}$
if $n$ is odd and $\nulambda \in L_{\operatorname{even}}$.

2)\enspace
$\C_{\lambda,\nu}$ is a nonzero multiple of
$\A_{\lambda,\nu}$ if
$\nulambda \in /\!/ \setminus L_{\operatorname{even}}$
by Theorem \ref{thm:reduction} (2).

3)\enspace
$\B_{\lambda,\nu}$ is a nonzero multiple of
$\A_{\lambda,\nu}$ if
$\nulambda \in \backslash\!\backslash \setminus \mathbb{X}$
by Theorem \ref{thm:reduction} (1).

Using the residue formulae
 (Theorem \ref{thm:reduction})
 in the next chapter, 
 we can restate Theorem \ref{thm:Hombasis} 
as follows:
\end{remark}
\begin{theorem}\label{thm:9.5}
\[
H(\lambda,\nu)
   = \begin{cases}
           \mathbb{C} \AAt_{\lambda,\nu} \oplus \mathbb{C} \C_{\lambda,\nu}
                &\text{if\/ $\nulambda \in L_{\operatorname{even}}$},
     \\
           \mathbb{C} \A_{\lambda,\nu}
                &\text{if\/ $\nulambda \in \mathbb{C}^2 \setminus L_{\operatorname{even}}$}.
     \end{cases}
\]
\end{theorem}

\subsection{Strategy of the proof of Theorem \ref{thm:Hombasis}}
\label{subsec:grdim}
We divide the proof of Theorem \ref{thm:Hombasis}
 into the following three steps.  
\vskip 2pc
\par\noindent
{\bf{Step 1.}}\enspace
Dimension formula of graded modules.

According to the natural filtration
 (see \eqref{eqn:grH})
\index{sbon}{H2@$H(\lambda, \nu)_{\operatorname{sing}}$}
\index{sbon}{H3@$H(\lambda,\nu)_{\rm{diff}}$}
\[
H(\lambda,\nu) \supset H(\lambda,\nu)_{\operatorname{sing}}
   \supset H(\lambda,\nu)_{\operatorname{diff}},
\]
we give the dimension formula
 of the graded modules,
 summarized in the following table:
\newline\vspace*{2ex}
{
\renewcommand{\arraystretch}{1.5}
\begin{tabular}{c|c|cl}
   dimension   &  $0$   &   $1$
\\
\hline
  $H(\lambda,\nu)/H(\lambda,\nu)_{\operatorname{sing}}$
  & $\backslash\!\backslash \cup /\!/$
  & $\mathbb{C}^2 \setminus (\backslash\!\backslash \cup /\!/ )$
  &$n$: odd
\\
  & $\backslash\!\backslash \cup (/\!/ - L_{\operatorname{even}})$
  & $L_{\operatorname{even}} \cup (\mathbb{C}^2 \setminus (\backslash\!\backslash \cup /\!/))$
  &$n$: even
\\
\hline
  $H(\lambda,\nu)_{\operatorname{sing}}/H(\lambda,\nu)_{\operatorname{diff}}$
  &$(\mathbb{X} - L_{\operatorname{even}}) \cup (\mathbb{C}^2 \setminus \backslash\!\backslash)$
  &$L_{\operatorname{even}} \cup (\backslash\!\backslash \setminus \mathbb{X})$
  &$n$: odd
\\
  & $\mathbb{X} \cup (\mathbb{C}^2 - \backslash\!\backslash)$
  & $\backslash\!\backslash - \mathbb{X}$
  & $n$: even
\\[1ex]
\hline
  $H(\lambda,\nu)_{\operatorname{diff}}$
  & $\mathbb{C}^2 - /\!/$
  &  $/\!/$
\end{tabular}

The first row is proved in Proposition \ref{prop:regSupp},
the second is in Proposition \ref{prop:3.10}.
The third row was already stated 
in Fact \ref{fact:5.2}.

\vskip 2pc
\par\noindent
{\bf{Step 2.}}\enspace
Dimension formula of $H(\lambda,\nu)$.  

By using 
the following obvious relations:
\begin{alignat*}{2}
&\dim H(\lambda,\nu)_{\operatorname{sing}}
&&= \dim H(\lambda,\nu)_{\operatorname{sing}} /
   H(\lambda,\nu)_{\operatorname{diff}} 
   + \dim H(\lambda,\nu)_{\operatorname{diff}},
\\
&\dim H(\lambda,\nu)
&&= \dim H(\lambda,\nu) / H(\lambda,\nu)_{\operatorname{sing}}
   + \dim H(\lambda,\nu)_{\operatorname{sing}}, 
\end{alignat*} 
we obtain the dimension formula
 of $H(\lambda,\nu)$ and $H(\lambda,\nu)_{\operatorname{sing}}$
in addition to the known formula of 
 $H(\lambda,\nu)_{\operatorname{diff}}$:
\begin{theorem}\label{thm:dimHom}
{\upshape 1)}\enspace
{\rm{(symmetry breaking operators)}}
\begin{equation*}
\dim H (\lambda,\nu)
= \begin{cases}
         2   & \text{if\/ $\nulambda \in L_{\operatorname{even}}$},
      \\
         1   &\text{otherwise}.
   \end{cases}
\end{equation*}

{\upshape 2)}\enspace
{\rm{(singular symmetry breaking operators)}}\enspace
Suppose $n$ is odd.
\begin{equation*}
 \dim H (\lambda,\nu)_{\operatorname{sing}}
   = \begin{cases} 2 &\text{if $\nulambda \in L_{\operatorname{even}}$}, \\
                        1 &\text{if $\nulambda \in (/\!/ \cup \backslash\!\backslash) - L_{\operatorname{even}}$},   \\
                        0 &\text{otherwise}.  
     \end{cases}
\end{equation*}
Suppose $n$ is even.
Then
\[
\dim H(\lambda,\nu)_{\operatorname{sing}}
= \begin{cases} 1   &\text{if $\nulambda \in\backslash\!\backslash \cup /\!/$},
   \\                0  &\text{otherwise}.
   \end{cases}
\]

{\upshape 3)}\enspace
{\rm{(differential symmetry breaking operators)}}\enspace
\[
 \dim H(\lambda,\nu)_{\operatorname{diff}}
   = \begin{cases}  1 &\text{if $\nulambda  \in /\!/$},  \\
                         0 &\text{otherwise.}
      \end{cases}
\]
\end{theorem}
\vskip 2pc
\par\noindent
{\bf{Step 3.}}\enspace
Explicit basis
 of symmetry breaking operators.  

As is clear from the table in Step 1,
 we obtain:
\begin{proposition}
\label{prop:graded}
The dimensions of the graded module
 are either 0 or 1.  
\end{proposition}
We then give an explicit basis
 of $H(\lambda,\nu)$
 by taking representatives 
 for the generators
 of the graded modules.  
This yields Theorem \ref{thm:Hombasis}.

\subsection{Lower bounds of the multiplicities}
\label{subsec:ABC}

In the previous chapters,
 we found explicitly 
 the condition
 for the non-vanishing of the operators,
  $\tA{\lambda}{\nu}$,
 $\AAt_{{\lambda},{\nu}}$,
 $\tB{\lambda}{\nu}$,
 $\BB_{{\lambda},{\nu}}$,
 and 
 $\CC{\lambda}{\nu}$, 
 and 
 determined the support
 of their distribution kernels.  
With respect to the natural filtration
\[
H(\lambda,\nu) \supset H(\lambda,\nu)_{\operatorname{sing}}
  \supset H(\lambda,\nu)_{\operatorname{diff}}, 
\] 
we summarize the properties
 of these operators as follows:

\begin{proposition}
\label{prop:ABC}
\begin{itemize}
\item[{\rm{1)}}]
{\rm{(regular symmetry breaking operators).}}

The following operators
\begin{alignat*}{3}
&\tA{\lambda}{\nu}
\quad
\text{for}
\quad
&&(\lambda,\nu) \in D_{\operatorname{sing}}(A_1)
&&:={\mathbb{C}}^2 \setminus (\backslash\!\backslash \cup /\!/), 
\\
&\AAt_{{\lambda},{\nu}}
\quad
\text{for}
\quad
&&(\lambda,\nu) \in D_{\operatorname{reg}}(A_2)
&&:=\{\nu \in -{\mathbb{N}}\} \cap 
({\mathbb{C}}^2 \setminus \backslash\!\backslash)
\end{alignat*}
are non-zero and belong to 
 $H(\lambda,\nu) \setminus H(\lambda,\nu)_{\operatorname{sing}}$.  

In particular, 
$
\dim H(\lambda,\nu)/H(\lambda,\nu)_{\operatorname{sing}}
\ge 1
$
if
\[ (\lambda,\nu)
 \in 
\begin{cases}
{\mathbb{C}}^2 \setminus(\backslash\!\backslash \cup /\!/)
&\text{$n$: odd,}
\\
L_{\operatorname{even}} \cup({\mathbb{C}}^2 \setminus(\backslash\!\backslash \cup /\!/))
\quad
 &\text{$n$: even.}
\end{cases}
\]

\item[{\rm{2)}}]
{\rm{(singular symmetry breaking operators I:
 non-differential operators).}}

The following operators 
\begin{alignat*}{3}
&\tA{\lambda}{\nu}
\quad
\text{for}
\quad
&&(\lambda,\nu) \in D_{\operatorname{sing}}(A_1)
&&:=\backslash\!\backslash \setminus {\mathbb{X}}, 
\\
&\AAt_{{\lambda},{\nu}}
\quad
\text{for}
\quad
&&(\lambda,\nu) \in D_{\operatorname{sing}}(A_2)
&&:=\{\nu \in -{\mathbb{N}}\} \cap \backslash\!\backslash, 
\\
&\tB{\lambda}{\nu}
\quad
\text{for}
\quad
&&(\lambda,\nu) \in D_{\operatorname{sing}}(B_1)
&&:=\backslash\!\backslash \setminus {\mathbb{X}}, 
\\
&\BB_{{\lambda},{\nu}}
\quad
\text{for}
\quad
&&(\lambda,\nu) \in D_{\operatorname{sing}}(B_2)
&&:=\begin{cases}
L_{\operatorname{even}}\quad &\text{$n$: odd}, 
\\
\emptyset &\text{$n$: even} 
\end{cases}
\end{alignat*}
are non-zero and belong to 
$H(\lambda,\nu)_{\operatorname{sing}}
  \setminus H(\lambda,\nu)_{\operatorname{diff}}$.  

In particular,
$
\dim
 H(\lambda,\nu)_{\operatorname{sing}}
/
 H(\lambda,\nu)_{\operatorname{diff}}
\ge 1
$
if 
\[
(\lambda,\nu)\in
\begin{cases}
{\mathbb{C}}^2 \setminus(\backslash\!\backslash \cup /\!/)
&\text{$n$: odd,}
\\
L_{\operatorname{even}} \cup({\mathbb{C}}^2 \setminus(\backslash\!\backslash \cup /\!/))
\quad
 &\text{$n$: even.}
\end{cases}
\]

\item[{\rm{3)}}]
{\rm{(singular symmetry breaking operators II:
differential operators).}}

The operators
\[
\CC{\lambda}{\nu}
\quad
\text{for}
\quad
(\lambda,\nu) \in /\!/
\hphantom{MMMMMMMMMMMMM}
\]
are non-zero 
and belong to $H(\lambda,\nu)_{\operatorname{diff}}$.  
\end{itemize}
\end{proposition}

\begin{proof}
1) The statements for $\tA{\lambda}{\nu}$
 and $\AAt_{{\lambda},{\nu}}$
 follow from Proposition \ref{prop:Ksupp} and Proposition \ref{prop:SuppAtwo}, 
respectively.  
The lower bound for the dimension holds
 because 
\[
D_{\operatorname{sing}}(A_1) \cup D_{\operatorname{sing}}(A_2)
=
\begin{cases}
{\mathbb{C}}^2 \setminus(\backslash\!\backslash \cup /\!/)
&\text{$n$: odd,}
\\
L_{\operatorname{even}} \cup({\mathbb{C}}^2 \setminus(\backslash\!\backslash \cup /\!/))
\quad
 &\text{$n$: even.}
\end{cases}
\]

2)\enspace
The statements for 
 $\tA{\lambda}{\nu}$,
 $\AAt_{{\lambda},{\nu}}$,
 $\tB{\lambda}{\nu}$,
 and $\BB_{{\lambda},{\nu}}$
 follow from Proposition \ref{prop:Ksupp}, 
 Proposition \ref{prop:SuppAtwo}, 
 Proposition \ref{prop:Bsupp},
 and Proposition \ref{prop:BB}, 
 respectively.  
The lower bound
 for the dimension holds because 
\[
 D_{\operatorname{sing}}(A_1) \cup
 D_{\operatorname{sing}}(A_2) \cup
 D_{\operatorname{sing}}(B_1) \cup
 D_{\operatorname{sing}}(B_2)
=
\begin{cases}
L_{\operatorname{even}} \cup (\backslash\!\backslash \setminus {\mathbb{X}})
&\text{$n$: odd,}
\\
\backslash\!\backslash \setminus{\mathbb{X}}
\quad
 &\text{$n$: even.}
\end{cases}
\]

3)\enspace
See Fact \ref{fact:5.2}.  
\end{proof}

\subsection
{Extension of solutions from $\mathbb{R}^n \setminus \{0\}$
 to $\mathbb{R}^n$}
\label{subsec:ext}

We consider the following exact sequence
(see \eqref{eqn:Solexact}):
\begin{equation}\label{ex:0exact}
0 \to \mathcal{S}ol_{\{0\}} (\mathbb{R}^n; \lambda,\nu)
   \to \mathcal{S}ol(\mathbb{R}^n; \lambda,\nu)
   \to \mathcal{S}ol(\mathbb{R}^n \setminus \{0\}; \lambda,\nu).
\end{equation}
According to Lemma \ref{lem:Sol},
$\mathcal{S}ol(\mathbb{R}^n \setminus \{0\}; \lambda,\nu)$
is one-dimensional for any
$\nulambda \in \mathbb{C}^2$.
However, 
 the following proposition
 asserts that 
for specific $\nulambda$,
 we cannot extend any non-zero element 
$F \in \mathcal{S}ol(\mathbb{R}^n \setminus \{0\}; \lambda,\nu)$
 to an element in
$\mathcal{S}ol(\mathbb{R}^n; \lambda,\nu)$.  
\begin{proposition}\label{prop:surj}
Assume $\nulambda \in /\!/ \setminus L_{\operatorname{even}}$.
\begin{enumerate}[\upshape 1)]
\item  
The restriction map
\begin{equation}\label{eqn:Solorigin}
\mathcal{S}ol(\mathbb{R}^n; \lambda,\nu)
\to \mathcal{S}ol(\mathbb{R}^n \setminus \{0\};\lambda,\nu)
\end{equation}
is identically zero.
\item   
$\mathcal{S}ol (\mathbb{R}^n; \lambda,\nu)
= \mathcal{S}ol_{\{0\}} (\mathbb{R}^n; \lambda,\nu)$.
\end{enumerate}
\end{proposition}
The proof of Proposition \ref{prop:surj} is divided
into the following two lemmas:
\begin{lemma}\label{lem:restA0}
Proposition \ref{prop:surj} holds if
\[
\nulambda \in /\!/ \setminus (\mathbb{X} \cup L_{\operatorname{even}})
= \begin{cases}
           /\!/ \setminus \mathbb{X}   &(n: \text{odd}),
     \\
           /\!/ \setminus (\mathbb{X} \cup L_{\operatorname{even}})  &(n: \text{even}).
  \end{cases}
\]
\end{lemma}
Owing to Lemma \ref{lem:Sol}, 
 the distribution solution
\[
\mathbb{C} |x_n|^{\lambda+\nu-n} (|x|^2 + x_n^2)^{-\nu} |_{\mathbb{R}^n\setminus\{0\}}
\in 
\mathcal{S}ol(\mathbb{R}^n \setminus \{0\}; \lambda, \nu)
\]
does not extend to an element
 of $\mathcal{S}ol(\mathbb{R}^n; \lambda, \nu)$
 if $(\lambda, \nu)$ satisfies
 the assumption
 of Lemma \ref{lem:restA0}.  

\begin{lemma}\label{lem:restB0}
Proposition \ref{prop:surj} holds if
\[
\nulambda \in \mathbb{X} \setminus L_{\operatorname{even}}
= \begin{cases}
           \mathbb{X} \setminus L_{\operatorname{even}}   &(n: \text{odd}),
      \\
           \mathbb{X}   &(n: \text{even}).
   \end{cases}
\]
\end{lemma}
Similarly,
 the distribution solution 
\[
  \mathbb{C} \delta^{(-\lambda-\nu+n-1)} (x_n)
     (|x|^2 + x_n^2)^{-\nu} |_{\mathbb{R}^n\setminus \{0\}}
  \in 
  \mathcal{S}ol(\mathbb{R}^n \setminus \{0\}; \lambda,\nu)
\]
does not extend to an element
 of $\mathcal{S}ol(\mathbb{R}^n; \lambda,\nu)$
 if $(\lambda, \nu)$ satisfies
 the assumption
 of Lemma \ref{lem:restB0}.

For the proofs of Lemmas \ref{lem:restA0} and \ref{lem:restB0}, 
we need  the following two general results:
\begin{lemma}\label{lem:DEmero}
Suppose $D_\mu$ is a differential operator with holomorphic parameter $\mu$,
and $F_\mu$ is a distribution on $\mathbb{R}^n$
that depends meromorphically on $\mu$
having the following expansions:
\begin{align*}
&D_\mu = D_0 + \mu D_1 + \mu^2 D_2 + \dotsb ,
\\
&F_\mu = \frac1\mu F_{-1} + F_0 + \mu  F_1 + \dotsb, 
\end{align*}
where $D_j$ and $F_i$ are distributions on ${\mathbb{R}}^n$.  
Assume that there exists $\varepsilon>0$
 such that $D_\mu F_\mu = 0$ for any
complex number $\mu$ with $0 < |\mu| < \varepsilon$.  
Then the distributions $F_0$ and $F_{-1}$ satisfy
the following differential equations:
\[
D_0 F_{-1} = 0  \quad\text{and}\quad 
D_0 F_0 + D_1 F_{-1} = 0.
\]
\end{lemma}

\begin{proof}
Clear from the Laurent expansion 
\[
 D_{\mu} F_{\mu}
=\frac 1 {\mu} D_0 F_{-1}
+(D_0F_0+D_1F_{-1})+ \cdots.  
\]
\end{proof}

We recall that
$E = \sum_{j=1}^n x_j \frac{\partial}{\partial x_j}$
is the Euler differential operator.

\begin{lemma}\label{lem:deltass}
Suppose $h \in \mathcal{D}' (\mathbb{R}^n)$
is supported at the origin.
If
$(E + A)^2  h = 0$
for some $A \in \mathbb{Z}$ then
$(E + A) h = 0$.
\end{lemma}

\begin{proof}[Proof of Lemma \ref{lem:deltass}]
By the structural theory of distributions,
$h$ is a finite linear combination of the
derivations of the Dirac delta function of $n$-variables:
\[
h = \sum_{\alpha\in\mathbb{N}^n} b_\alpha \delta^{(\alpha)}
\quad\text{(finite sum) for some $b_{\alpha} \in {\mathbb{C}}$}.
\]
For a multi-index $\alpha =(\alpha_1, \cdots, \alpha_n) \in {\mathbb{N}}^n$, 
 we set $|\alpha| =\alpha_1+ \cdots + \alpha_n$.  
In view that $\delta^{(\alpha)}$ is homogeneous of degree
$-n-|\alpha|$,
we have
\begin{align*}
&(E+A)h =
 \sum_{\alpha\in\mathbb{N}^n} b_\alpha
  (A - n - |\alpha|) \delta^{(\alpha)},
\\
&(E+A)^2 h =
  \sum_{\alpha\in\mathbb{N}^n} b_\alpha
  (A - n - |\alpha|)^2 \delta^{(\alpha)}.
\end{align*}
Hence
$(E+A)h = 0$
if and only if
$(E+A)^2 h = 0$
because $\delta^\alpha$ ($\alpha\in\mathbb{N}^n)$ are linearly independent.
\end{proof}

We are ready to prove
 Lemma \ref{lem:restA0} and \ref{lem:restB0}.  
\begin{proof}[Proof of Lemma \ref{lem:restA0}]
Suppose
$(\lambda_0, \nu_0) \in /\!/ \setminus (\mathbb{X} \cup 
L_{\operatorname{even}})$.
We set 
$l := \frac12(\nu_0 - \lambda_0) \in \mathbb{N}$.
Consider
$\nulambda \in \mathbb{C}^2$
with constraints
$\lambda +\nu= \lambda_0+\nu_0$,
so that $\nulambda$ stays in $\backslash\!\backslash$
with a complex parameter
\[
\mu:=\lambda-\nu+2l.
\]
We note
 that $\nulambda \notin \backslash\!\backslash$
 because $(\lambda_0, \nu_0)\notin \backslash\!\backslash$.
Then $\ka{\lambda}{\nu}$ is a distribution on $\mathbb{R}^n$
that depends meromorphically on $\mu$
 by Theorem \ref{thm:poleA}(1).  
Since the normalizing factor 
 of $\KA\lambda\nu$, 
namely,
$\Gamma(\frac{\lambda-\nu}{2}) \Gamma(\frac{\lambda+\nu-n+1}{2})$
has a simple pole
 at $\mu=0$, 
 we can expand $\ka{\lambda}{\nu}$
near $(\lambda_0, \nu_0)$ as
\[
\ka{\lambda}{\nu}
= \frac{1}{\mu} F_{-1} + F_0 + \mu F_1 + \dotsb
\]
with some distributions $F_{-1}$, $F_0$, $F_1$, 
$\cdots$.  
By Theorem \ref{thm:poleA}(2),
we see that the distribution $F_{-1}$ is non-zero
because $(\lambda_0, \nu_0) \notin L_{\operatorname{even}}$.

Applying Lemma \ref{lem:DEmero} to 
 the differential equation
\[
(E - (\lambda-\nu-n)) \ka{\lambda}{\nu}
= ((E + n + 2l) - \mu \cdot 1)\ka{\lambda}{\nu}=0,
\]
we have
\begin{equation}\label{eqn:K0eq}
(E+n+2l)F_{-1} = 0
\quad\text{and}\quad
(E+n+2l) F_0 - F_{-1} = 0.
\end{equation}
Suppose that
$F \in \mathcal{S}ol(\mathbb{R}^n; \lambda_0,\nu_0)$.
It follows from Lemma \ref{lem:Sol} that
$F |_{\mathbb{R}^n\setminus\{0\}} = cF_0 |_{\mathbb{R}^n\setminus\{0\}}$
for some $c\in\mathbb{C}$.
We set
\[
h := F - c F_0 \in \mathcal{D}' (\mathbb{R}^n).
\]
Then
$\operatorname{Supp} h \subset \{0\}$ and
\[
(E+n+2l)^2 h = (E+n+2l)^2 F - c(E+n+2l)^2 F_0 = 0
\]
by \eqref{eqn:Fa} and \eqref{eqn:K0eq}.
Applying Lemma \ref{lem:deltass},
we have $(E+n+2l)h = 0$.
In turn,
$c F_{-1} = 0$ by \eqref{eqn:Fa} and \eqref{eqn:K0eq}.
Since
$F_{-1} \ne 0$,
we get $c=0$.
Hence
$\operatorname{Supp} F = \operatorname{Supp} h \subset \{0\}$.
Thus we have proved
 that \eqref{eqn:Solorigin} is
 a zero-map
 and  
 $\mathcal{S}ol(\mathbb{R}^n; \lambda_0,\nu_0)
=
\mathcal{S}ol_{\{0\}}(\mathbb{R}^n; \lambda_0,\nu_0)$.
\end{proof}

\begin{proof}[Proof of Lemma \ref{lem:restB0}]
Suppose $(\lambda_0, \nu_0) \in \mathbb{X}$,
and we define $k, l \in \mathbb{N}$
by $\lambda_0+\nu_0 = n-1-2k$
and $\nu_0-\lambda_0 = 2l$.
Consider $\nulambda \in \mathbb{C}^2$
with constraints
$\nu+\lambda = \nu_0+\lambda_0$
so that $\nulambda$ stays in $\backslash\!\backslash$
with a complex parameter
\[
\mu:=\lambda-\nu+2l.
\]
Then 
$\KBB{\lambda}{\nu} (x,x_n)$
is a distribution on $\mathbb{R}^n$ that depends meromorphically on
$\mu \in \mathbb{C}$ by Theorem \ref{thm:BK},
and we have an expansion
\[
\KBB{\lambda}{\nu} =
\frac{1}{\mu} F_{-1} + F_0 + \mu F_1 + \dotsb ,
\]
where $F_{-1}, F_0, F_1, \dotsc$ are distributions on $\mathbb{R}^n$.
Since $(|x|^2+x_n^2)^{-\nu}$ is a smooth function on
$\mathbb{R}^n \setminus \{0\}$
for all $\nu \in \mathbb{C}$,
the restriction
$(|x|^2+x_n^2)^{-\nu} \delta^{(2k)} (x_n) |_{\mathbb{R}^n\setminus\{0\}}$
is a distribution on $\mathbb{R}^n\setminus\{0\}$
that depends \textit{holomorphically} on $\nu$.
Hence we have
\[
F_{-1} |_{\mathbb{R}^n\setminus\{0\}} = 0
\quad\text{and}\quad
F_0 |_{\mathbb{R}^n\setminus\{0\}}
= (|x|^2+x_n^2)^{-\nu} \delta^{(2k)} (x_n) |_{\mathbb{R}^n\setminus\{0\}}.
\]
Applying Lemma \ref{lem:DEmero} to the differential equation
\[
(E - (\lambda-\nu-n))\KBB{\lambda}{\nu}
 = ((E+n+2l) - \mu \cdot 1)\KBB{\lambda}{\nu}=0,
\]
we have
\begin{equation}\label{eqn:EEF}
(E+n+2l)F_{-1} = 0, \
(E+n-2l)F_0 - F_{-1} = 0.
\end{equation}
Take any $F \in \mathcal{S}ol(\mathbb{R}^n; \lambda_0,\nu_0)$.  
By Lemma \ref{lem:Sol}, 
there exists $c \in {\mathbb{C}}$
such that
\[
F|_{\mathbb{R}^n\setminus\{0\}}
=c (|x|^2+x_n^2)^{-\nu} \delta^{(2k)} (x_n) |_{\mathbb{R}^n\setminus\{0\}}.
\]
Then $h := F - c F_0 \in\mathcal{D}'(\mathbb{R}^n)$
is supported at the origin,
and satisfies
\[
(E+n+2l)^2 h = 0
\]
because $F$ satisfies \eqref{eqn:Fa}
 and $F_0$ satisfies \eqref{eqn:EEF}.
By Lemma \ref{lem:deltass},
we get $(E+n+2l)h = 0$.
Using again \eqref{eqn:Fa} and \eqref{eqn:EEF},
we see $c F_{-1} = 0$.
On the other hand,
 by Theorem \ref{thm:BK}, 
 $F_{-1} \ne 0$
if $(\lambda_0,\nu_0) \in L_{\operatorname{even}}$
and $n$ is odd.
Thus if $(\lambda_0,\nu_0) \in \mathbb{X} \setminus L_{\operatorname{even}}$
then $c=0$, 
and therefore the restriction map \eqref{eqn:Solorigin} is identically zero.
\end{proof}

\subsection{Regular symmetry breaking operators}
In this section
 we find the dimension of the space
$H(\lambda,\nu)/H(\lambda,\nu)_{\operatorname{sing}}$.

\begin{proposition}\label{prop:regSupp}
For any $(\lambda, \nu) \in \mathbb{C}^2$,
\begin{equation}\label{eqn:HHsing}
\dim H(\lambda,\nu) / H(\lambda,\nu)_{\operatorname{sing}} \le 1.
\end{equation}
Moreover,
 there exist non-zero regular symmetry breaking operators
 (see Definition \ref{def:regular})
 if and only if
\begin{equation*}
\nulambda
\in 
\begin{cases}
{\mathbb{C}} \setminus (\backslash\!\backslash \cup /\!/)
\quad
&\text{($n$:odd)},
\\
({\mathbb{C}}^2 \setminus (\backslash\!\backslash \cup /\!/))
\cup
L_{\operatorname{even}}
\quad
&\text{($n$:even)}.  
\end{cases}
\end{equation*}
\end{proposition}

\begin{proof}
In light of the isomorphism
\[
H(\lambda,\nu) / H(\lambda,\nu)_{\operatorname{sing}}
\simeq
  \mathcal{S}ol(\mathbb{R}^n;\lambda,\nu)
  /
  \mathcal{S}ol_{\mathbb{R}^{n-1}}(\mathbb{R}^n;\lambda,\nu), 
\]
we consider the following exact sequences:
\[
\begin{matrix}
0  & \to  & \mathcal{S}ol_{\{0\}}(\mathbb{R}^n;\lambda,\nu)
   & \to  & \mathcal{S}ol(\mathbb{R}^n;\lambda,\nu)
   & \to  & \mathcal{S}ol(\mathbb{R}^n \setminus \{0\}; \lambda,\nu)
\\
   && \|   && \cup   &&  \cup
\\
0  & \to  & \mathcal{S}ol_{\{0\}}(\mathbb{R}^n;\lambda,\nu)
   & \to  & \mathcal{S}ol_{\mathbb{R}^{n-1}}(\mathbb{R}^n;\lambda,\nu)
   & \to  & \mathcal{S}ol_{\mathbb{R}^{n-1}\setminus\{0\}}(\mathbb{R}^n \setminus \{0\}; \lambda,\nu).
\end{matrix}
\]
If $F \in \mathcal{S}ol(\mathbb{R}^n;\lambda,\nu)$
satisfies $\operatorname{Supp}(F|_{\mathbb{R}^n\setminus\{0\}})
\subset \mathbb{R}^{n-1}\setminus\{0\}$,
then clearly
$\operatorname{Supp} F \subset \mathbb{R}^{n-1}$.
Hence the following natural homomorphism is injective:
\begin{equation}\label{eqn:Solhyper}
\mathcal{S}ol(\mathbb{R}^n;\lambda,\nu)/
\mathcal{S}ol_{\mathbb{R}^{n-1}}(\mathbb{R}^n;\lambda,\nu)
\to
\mathcal{S}ol(\mathbb{R}^n \setminus \{0\}; \lambda,\nu)/
\mathcal{S}ol_{\mathbb{R}^{n-1}\setminus\{0\}}(\mathbb{R}^n\setminus\{0\};\lambda,\nu).
\end{equation}
Since $\dim\mathcal{S}ol(\mathbb{R}^n\setminus\{0\};\lambda,\nu) = 1$
for any $\nulambda\in \mathbb{C}^2$
by Lemma \ref{lem:Sol},
we get the inequality \eqref{eqn:HHsing}.
By the same lemma,
the right-hand side of \eqref{eqn:Solhyper} is zero
if $\nulambda \in \backslash\!\backslash$,
and thus
$\mathcal{S}ol(\mathbb{R}^n;\lambda,\nu)/
 \mathcal{S}ol_{\mathbb{R}^{n-1}}(\mathbb{R}^n;\lambda,\nu)
 = \{0\}$.

On the other hand,
Proposition \ref{prop:surj} tells that if
$\nulambda \in /\!/ \setminus L_{\operatorname{even}}$
then
$\mathcal{S}ol(\mathbb{R}^n;\lambda,\nu)
 = \mathcal{S}ol_{\{0\}}(\mathbb{R}^n;\lambda,\nu)$,
and therefore
$\mathcal{S}ol(\mathbb{R}^n;\lambda,\nu)/
\mathcal{S}ol_{\mathbb{R}^{n-1}}(\mathbb{R}^n;\lambda,\nu) = \{0\}$.
In summary we have shown that if
$\nulambda \in \backslash\!\backslash \cup (/\!/ \setminus L_{\operatorname{even}})$
then
$H(\lambda,\nu)/H(\lambda,\nu)_{\operatorname{sing}} = \{0\}$.

Conversely,
 if $\lambda \not\in \backslash\!\backslash \cup 
(/\!/\setminus L_{\operatorname{even}})$
 then 
$\dim H(\lambda,\nu) / H(\lambda,\nu)_{\operatorname{sing}}=1$
 by Proposition \ref{prop:ABC} (1) 
and \eqref{eqn:LX}.  
Thus the proof of Proposition \ref{prop:regSupp} is completed.
\end{proof}

\subsection{Singular symmetry breaking operators}
\label{subsec:dimsing}
We have seen in Fact \ref{fact:5.2}
\[
H(\lambda,\nu)_{\operatorname{diff}}
= \begin{cases} \mathbb{C} \C_{\lambda,\nu}
                     &\text{if $\nulambda \in /\!/$},
      \\
                     \{0\}   &\text{otherwise}.
  \end{cases}
\]
In this section
 we find the dimension of the space
$H(\lambda,\nu)_{\operatorname{sing}}
 / H(\lambda,\nu)_{\operatorname{diff}}$.
\begin{proposition}\label{prop:3.10}
For any $(\lambda, \nu) \in \mathbb{C}^2$,
\begin{equation}\label{eqn:dimsing}
\dim H(\lambda, \nu)_{\operatorname{sing}} / H(\lambda, \nu)_{\operatorname{diff}} \le 1.
\end{equation}
The equality holds if and only if\/ 
\[
\nulambda \in 
\backslash\!\backslash \setminus (\mathbb{X} \setminus L_{\operatorname{even}})
= \begin{cases}
        L_{\operatorname{even}} \cup (\backslash\!\backslash \setminus \mathbb{X})
        & (n: \text{odd}\/),
     \\
        \backslash\!\backslash \setminus \mathbb{X}
        &(n: \text{even}).
   \end{cases}
\]
\end{proposition}

\begin{proof}
We analyze the right-hand side
 of the following isomorphism: 
\[ 
H(\lambda,\nu)_{\operatorname{sing}}/H(\lambda,\nu)_{\operatorname{diff}}
\simeq
\mathcal{S}ol_{\mathbb{R}^{n-1}}(\mathbb{R}^n;\lambda,\nu)/
\mathcal{S}ol_{\{0\}}(\mathbb{R}^n;\lambda,\nu).
\]  
In view of the exact sequence
\[
0 \to \mathcal{S}ol_{\{0\}}(\mathbb{R}^n; \lambda,\nu)
   \to \mathcal{S}ol_{\mathbb{R}^{n-1}}(\mathbb{R}^n;\lambda,\nu)
   \to \mathcal{S}ol_{\mathbb{R}^{n-1}\setminus\{0\}}(\mathbb{R}^n\setminus\{0\};\lambda,\nu),
\]
we have an inclusive relation.  
\[
\mathcal{S}ol_{\mathbb{R}^{n-1}}(\mathbb{R}^n;\lambda,\nu)/
\mathcal{S}ol_{\{0\}}(\mathbb{R}^n;\lambda,\nu)
\subset
\mathcal{S}ol_{\mathbb{R}^{n-1}\setminus\{0\}}
   (\mathbb{R}^n\setminus\{0\}; \lambda,\nu).
\]
Therefore it follows from Lemma \ref{lem:Sol}
that the inequality \eqref{eqn:dimsing} holds for any
$\nulambda \in \mathbb{C}^2$
and that \eqref{eqn:dimsing} becomes the equality
only if $\nulambda \in \backslash\!\backslash$.

On the other hand,
if $\nulambda \in /\!/ \setminus L_{\operatorname{even}}$
then
$\mathcal{S}ol_S(\mathbb{R}^n;\lambda,\nu) = \mathcal{S}ol_{\{0\}}(\mathbb{R}^n;\lambda,\nu)$
for any $S$ containing $0$
by Proposition \ref{prop:surj},
and consequently
$H(\lambda,\nu)_{\operatorname{sing}}/H(\lambda,\nu)_{\operatorname{diff}}
= \{0\}$.

In summary, 
we have shown that the equality holds in \eqref{eqn:dimsing}
only if
\[
\nulambda \in \backslash\!\backslash \setminus (/\!/ \setminus L_{\operatorname{even}})
= \backslash\!\backslash \setminus (\mathbb{X} \setminus L_{\operatorname{even}}).
\]
Conversely,
the equality in \eqref{eqn:dimsing} holds
if $(\lambda,\nu) 
\in \backslash\!\backslash \setminus ({\mathbb{X}}
 \setminus L_{\operatorname{even}})$
 by Proposition \ref{prop:ABC}(2).  

Thus the proof of Proposition \ref{prop:3.10} is completed.
\end{proof}
Combining the above proof and Theorem \ref{thm:reduction} (3),
we obtain:
\begin{proposition}
[classification of singular symmetry breaking operators]
\label{prop:sing}
If $n$ is odd,
then $L_{\operatorname{even}} \subset \mathbb{X}$
and we have
\[
H(\lambda,\nu)_{\operatorname{sing}} =
   \begin{cases}
       \mathbb{C} \BB_{\lambda,\nu} \oplus \mathbb{C} \C_{\lambda,\nu}
       &\text{if\/ $\nulambda \in L_{\operatorname{even}}$},
   \\
       \mathbb{C} \C_{\lambda,\nu}
       &\text{if\/ $\nulambda \in /\!/ \setminus L_{\operatorname{even}}$},
   \\
       \mathbb{C} \B_{\lambda,\nu} 
       &\text{if\/ $\nulambda \in \backslash\!\backslash
              \setminus \mathbb{X}$},
   \\
       \{0\}
       &\text{if\/ $\nulambda \in \mathbb{C}^2 \setminus (\backslash\!\backslash \cup /\!/)$}.
   \end{cases}
\]
If $n$ is even,
we have
\[
H(\lambda,\nu)_{\operatorname{sing}} =
   \begin{cases}
       \mathbb{C} \C_{\lambda,\nu}
       &\text{if\/ $\nulambda \in /\!/$},
   \\
       \mathbb{C} \B_{\lambda,\nu} 
       &\text{if\/ $\nulambda \in \backslash\!\backslash \setminus \mathbb{X}$}.
   \\
       \{0\}   
       &\text{if\/ $\mathbb{C}^2 \setminus (\backslash\!\backslash\cup/\!/)$}.
   \end{cases}
\]
\end{proposition}

\section{Residue formulae
 and functional identities}
\label{sec:reduction}
The (generically) regular symmetry breaking operators $\A_{\lambda,\nu}$
 have two complex parameters
 $(\lambda,\nu)\in {\mathbb{C}}^2$, 
 whereas singular operators
 $\B_{\lambda, \nu}$ and $\C_{\lambda, \nu}$
 are defined for 
\index{sbon}{Xl@${\backslash\!\backslash}$}
$(\lambda,\nu)\in \backslash\!\backslash$
 and 
\index{sbon}{Xr@${/\!/}$}
 $/\!/$, 
 respectively,
 and thus having only one complex parameter 
 and one integral parameter.  
Further,
 the renormalized operators $\AAt_{\lambda,\nu}$
 are defined for $\nu \in -{\mathbb{N}}$
 and $\lambda \in {\mathbb{C}}$, 
 whereas $\BB_{\lambda,\nu}$
 are defined only for discrete parameter,
 namely,
\index{sbon}{Leven@$L_{\operatorname{even}}$}
 $(\lambda,\nu)\in L_{\operatorname{even}}$
 when $n$ is odd.  
The goal of this chapter
 is to find the relationship
 among these operators
 as explicit residue formulae
 according to the following hierarchy.  
The main results are given in Theorem \ref{thm:reduction}.  
\begin{figure}[h]
\begin{alignat*}{4}
&\,\,\,\A_{\lambda,\nu}&& && &&\AAt_{\lambda,\nu}
\\
\,\,{}^{\backslash\!\backslash}\swarrow 
& 
&&\searrow 
&&{}^{/\!/}
&&\downarrow \,\, {}^{L_{\operatorname{even}}, \text{$n$:odd}}
\\
\B_{\lambda,\nu} 
&\,\,\underset{\mathbb{X}}\longrightarrow 
&&
&&\CC{\lambda}{\nu} 
\qquad\quad
&&\BB_{\lambda,\nu}
\end{alignat*}
\caption{Reduction of operators }
\label{fig:ABChier}
\end{figure}

As a corollary,
 we extend the functional identities
 (Theorem \ref{thm:TAAT})
 for the (generically) regular symmetry breaking operators
 $\A_{\lambda,\nu}$
 with the Knapp--Stein intertwining operators
 of $G$ and $G'$
 to those for singular ones $\B_{\lambda,\nu}$ and $\CC{\lambda}{\nu}$
 (see Theorem \ref{thm:TABC}, 
 Corollaries \ref{cor:DCC} and \ref{cor:10.6}).  
The {\it{factorization identities}}
 for conformally equivariant operators
 by Juhl \cite[Chapter 6]{Juhl}
 are obtained 
 as a special case
 of Corollaries \ref{cor:DCC} and \ref{cor:10.6}.  

\vskip 1pc

We retain the following convention
\begin{alignat}{2}
   2l = & \nu -\lambda 
\quad
 &&\text{for } (\lambda,\nu) \in /\!/, 
\label{eqn:ldef}
\\
   2k = & n-1-\lambda-\nu 
\quad
 &&\text{for } (\lambda,\nu) \in \backslash\!\backslash, 
\label{eqn:kdef}
\\
   m=&n-1&&
\notag
\end{alignat}
 throughout this chapter.  

\subsection{Residues of symmetry breaking operators}
\label{subsec:red}

In this section,
 we prove that the special values
 of the operators $\A_{\lambda,\nu}$
 are proportional to $\B_{\lambda,\nu}$
 or $\CC \lambda \nu$
 up to scalar multiples
 according to the hierarchy illustrated in Figure \ref{fig:ABChier}:

Let $l\in {\mathbb{N}}$ be defined
 by \eqref{eqn:ldef}
 for $(\lambda,\nu) \in /\!/$, 
 and $k\in {\mathbb{N}}$ be \eqref{eqn:kdef}
 for $(\lambda,\nu) \in \backslash\!\backslash$.  
We set
\index{sbon}{qBA@$q_B^A$|textbf}
\index{sbon}{qCA@$q_C^A$|textbf}
\index{sbon}{qCB@$q_C^B$|textbf}
\begin{alignat*}{2}
q_B^A(\lambda,\nu)&:= \frac{(-1)^k}{2^k (2k-1)!!}
\quad
&&\text{for }\nulambda \in \backslash\!\backslash,
\\
q_C^A(\lambda,\nu)&:= \frac{(-1)^{l}l!\pi^{\frac m 2}}
                      {2^{2l}\Gamma(\nu)}
\quad
&&\text{for }\nulambda \in /\!/,
\\
q_C^B(\lambda,\nu)&:= \frac{(-1)^{l-k} \pi^{\frac m 2}l! (2k-1)!!}
                     {2^{2l-k}\Gamma(\nu)}
\quad
&&\text{for }\nulambda \in \mathbb{X}.  
\end{alignat*}   

Then the following lemma is immediate from 
the definition:
\begin{lemma}
\label{lem:qACzero}
{\rm{1)}}\enspace
For $(\lambda,\nu) \in /\!/$, 
$q_C^A(\lambda,\nu)=0$
 if and only if $\nulambda \in L_{\operatorname{even}}$.  
\par\noindent
{\rm{2)}}\enspace
For $(\lambda,\nu) \in {\mathbb{X}}$, 
 $q_C^B(\lambda,\nu)=0$
 if and only if $(\lambda,\nu)\in L_{\operatorname{even}}$.  
\par\noindent
{\rm{3)}}\enspace
$
  q_B^A(\lambda,\nu) q_C^B(\lambda,\nu)
  =
  q_C^A(\lambda,\nu)
  \quad
  \text{for }\nulambda \in \mathbb{X}.  
$
\end{lemma}
Here is the main result
 of this chapter.

\begin{theorem}
[residue formulae]
\label{thm:reduction}
~~~
\begin{enumerate}
\item[{\rm{1)}}]
\index{sbon}{Alnt@$\A_{\lambda,\nu}$}
\index{sbon}{Bt@$\B_{\lambda,\nu}$}
$\tA {\lambda}{\nu}
=q_B^A(\lambda,\nu)
\tB{\lambda}{\nu}
\quad
\text{for }
\nulambda \in \backslash\!\backslash$.  
\item[{\rm{2)}}]
\index{sbon}{Ct@$\C_{\lambda,\nu}$}
$\tA {\lambda}{\nu}
=q_C^A(\lambda,\nu)
\CC{\lambda}{\nu}
\quad
\text{for }
\nulambda \in /\!/$.  
\item[{\rm{3)}}]
$\tB {\lambda}{\nu}
=q_C^B(\lambda,\nu)
\CC{\lambda}{\nu}
\quad
\text{for }
\nulambda \in 
\mathbb{X}$.  
\item[{\rm{4)}}]
\index{sbon}{Att@$\AAt_{\lambda,\nu}$}
\index{sbon}{Btt@$\BB_{\lambda,\nu}$}
$\AAt_{{\lambda},{\nu}}
=q_B^A(\lambda,\nu)
\BB_{{\lambda},{\nu}}
\quad
\text{for }
\nulambda \in 
L_{\operatorname{even}}$
 if $n$ is odd.  
\end{enumerate}
\end{theorem}

\begin{remark}
$L_{\operatorname{even}} \subset {\mathbb{X}}$
 if $n$ is odd. 
$L_{\operatorname{even}} \cap {\mathbb{X}} = \emptyset$
 if $n$ is even.  
\end{remark}
\begin{remark}
\label{rem:CLeven}
For $(\lambda,\nu)\in L_{\operatorname{even}}$, 
 the differential symmetry breaking operators
 $\C_{\lambda,\nu}$ cannot be obtained
 as  residues
 of the (generically) regular symmetry breaking operators
 ${\mathbb{A}}_{\lambda,\nu}$
 because the coefficient $q_C^A(\lambda,\nu)$
 in Theorem \ref{thm:reduction} (2) vanishes
 if $(\lambda,\nu)\in L_{\operatorname{even}}$.  
See also Remark \ref{rem:Cres}.  
\end{remark}

\begin{proof}[Proof of Theorem \ref{thm:reduction}]
1)\enspace
By Proposition \ref{prop:ABres}, 
 the identity 
 $\tA \lambda \nu=
 q_B^A(\lambda, \nu) \tB \lambda\nu$
 holds 
 on $I(\lambda)_K$
 for any $(\lambda, \nu) \in \backslash\!\backslash$.  
Since $I(\lambda)_K$ is dense
 in the Fr{\'e}chet space
 $I(\lambda)$, 
 the identity holds on $I(\lambda)$.  
\par\noindent
2)\enspace
Both sides are zero if $\nulambda \in L_{\operatorname{even}}$
by Theorem \ref{thm:poleA}.

Suppose $\nulambda \in /\!/ \setminus L_{\operatorname{even}}$.
Then 
\index{sbon}{pplus@${p_+}$}
$\operatorname{Supp} \widetilde{K}_{\lambda,\nu} = \{[p_+]\}$
by Proposition \ref{prop:Ksupp},
and therefore $\A_{\lambda,\nu}$ is a differential operator
 by Fact \ref{fact:Diff}.
Since the dimension of symmetry breaking operators 
is at most one by Fact \ref{fact:5.2},
$\A_{\lambda,\nu}$ must be proportional to $\C_{\lambda,\nu}$.
The proportionality constant is found by using
 Proposition \ref{prop:AminK}
and \ref{prop:C11}.
\par\noindent
3)\enspace
Both sides are zero
 if $(\lambda, \nu) \in L_{\operatorname{even}} \cap {\mathbb{X}}$.
If $(\lambda, \nu) \in {\mathbb{X}} \setminus L_{\operatorname{even}}$
 then $\operatorname{Supp} \KB \lambda\nu
 =\{[p_+]\}$
 by Proposition \ref{prop:Bsupp}, 
 and therefore $\tB\lambda\nu$ is a differential operator.  
Since $\dim H(\lambda, \nu)_{\operatorname{diff}}=1$
 by Theorem \ref{thm:dimHom} (3), 
 $\tB \lambda\nu$
 and $\tC\lambda\nu$ must be proportional.  
The proportionality constant is computed by using
  Proposition \ref{prop:B1}
 and \ref{prop:C11}.  

\par\noindent
4)\enspace
The residue formula
 of a distribution of one-variable
\[
   \frac{|t|^{\lambda+\nu-n}}{\Gamma(\frac{\lambda+\nu-n+1}{2})}
  =\frac{(-1)^k}{2^k(2k-1)!!} \delta^{(2k)} (t)
\quad
 \text{for }\,\,2k=n-1-\lambda-\nu
\]
 implies the following identity
 of distributions on $S^n$:  
\[
\frac{|\eta_n|^{\lambda+\nu-n}}{\Gamma(\frac{\lambda+\nu-n+1}{2})}
=
\frac{(-1)^k}{2^k(2k-1)!!} \delta^{(2k)} (\eta_n), 
\]
because the coordinate function
 $\eta_n:S^n\to {\mathbb{R}}$
 is regular at $\eta_n=0$.  
Since $(1-\eta_0)^{-\nu}$
is a smooth function on $S^n$ if
$\nu \in -\mathbb{N}$,
we can multiply the above identity by
$2^{-\lambda+n} (1-\eta_0)^{-\nu}$,
and then get
\index{sbon}{iotaKast@$\iota_K^*$}
\index{sbon}{Kxb@$k_{\lambda,\nu}^{\mathbb{B}}$}
\[
(\iota_K^* \ttska{\lambda}{\nu})(\eta)
=
\frac{(-1)^k}{2^k(2k-1)!!} (\iota_K^* \skb{\lambda}{\nu}) (\eta),
\]
see \eqref{eqn:iKk} and \eqref{eqn:iKkB}.
\end{proof}
\begin{remark}
\label{rem:residue}
In the \lq{F-method}\rq, 
the residue formula (2) in Theorem \ref{thm:reduction}
 can be explained by the fact
 that the Taylor series expansion
 of the Gauss hypergeometric function 
 ${}_2F_1(a,b;c;z)$ terminates
 if $a \in -{\mathbb{N}}$
 (or $b \in -{\mathbb{N}}$)
 (\cite{xkeastwood60}), 
 see Proposition \ref{prop:4.3}.  
The idea of the F-method
 will be used in Chapter \ref{sec:applbl}
 to construct explicitly
 the discrete summand
 of the branching law
 of the complementary series representation
 $I(\lambda)$
 ($0<\lambda<n$)
 of $G$
when we restrict it to the subgroup $G'$.  
\end{remark}

%%%%%%%%%%%%%%%%%%%%%%%%%%%%%%%%%%%%
\subsection{Functional equations satisfied by singular symmetry breaking operators
}
\label{subsec:TABC}

We set
\begin{alignat*}{2}
\index{sbon}{pATA@$p_{A}^{TA}$|textbf}
p_{A}^{TA}(\lambda,\nu)
:=& \frac{\pi^{\frac m 2}}{\Gamma(\nu)}, 
&&
\\
\index{sbon}{pCTB@$p_{C}^{TB}$|textbf}
p_{C}^{TB}(\lambda,\nu)
:=& \frac{(2k)! \pi^{m}}{2^{2k}\Gamma(\nu)\Gamma(m-\nu)}
\quad
&&\text{for }
\nulambda \in \backslash \!\backslash, 
\\
\index{sbon}{pBTC@$p_{B}^{TC}$|textbf}
p_{B}^{TC}(\lambda,\nu)
:=& \frac{2^{2l}}{(2l)!}
\quad
&&\text{for }
\nulambda \in /\!/, 
\\
\index{sbon}{pCTC@$p_{C}^{TC}$|textbf}
p_{C}^{TC}(\lambda,\nu)
:=& \frac{(-1)^{k+l}k! \pi^{\frac m 2}}
   {2^{2k-2l}l!}
\quad
&&\text{for }
\nulambda \in {\mathbb{X}}.  
\end{alignat*}

We have seen in Theorem \ref{thm:TAAT}
 that the functional identity
\index{sbon}{ttt4@$\T{\nu}{m-\nu}$}
\begin{equation}
\label{eqn:TApA}
  \T {\nu}{m-\nu} \circ \tA {\lambda}{\nu}
  =
  p_A^{TA} (\lambda,\nu) \tA{\lambda}{m-\nu}
\end{equation}
holds for all $\nulambda \in {\mathbb{C}}^2$, 
  where $\T{\nu}{m-\nu}:J(\nu) \to J(m-\nu)$
 is the normalized Knapp--Stein 
 intertwining operator
 for the subgroup $G'$.  
Combining \eqref{eqn:TApA} with the residue formulae
 in Theorem \ref{thm:reduction}, 
we obtain the following functional identities
 for (singular) symmetry breaking operators:
\begin{theorem}
[functional identities]
\label{thm:TABC}
~~~
\begin{enumerate}
\item[{\rm{1)}}]
$\T{\nu}{m-\nu} \circ \tB{\lambda}{\nu}
= p_{C}^{TB}(\lambda,\nu) \CC{\lambda}{m-\nu}$
\quad
 for $\nulambda \in \backslash\!\backslash$.  
\item[{\rm{2)}}]
$\T{\nu}{m-\nu} \circ \CC{\lambda}{\nu}
= p_{B}^{TC}(\lambda,\nu) \tB{\lambda}{m-\nu}$
\quad
 for $\nulambda \in /\!/$.  
\item[{\rm{3)}}]
\index{sbon}{XX@${{\mathbb{X}}}$}
$\T{\nu}{m-\nu} \circ \CC{\lambda}{\nu}
= p_{C}^{TC}(\lambda,\nu) \CC{\lambda}{m-\nu}$
\quad
 for $\nulambda \in {\mathbb{X}}$.  
\item[{\rm{4)}}]
$\T{\nu}{m-\nu} \circ \AAt_{\lambda,\nu}=0$
\hphantom{MMMMMMMM}
 for $\nu \in -{\mathbb{N}}$.  
\end{enumerate}
\end{theorem}

\begin{proof}
1)\enspace
First of all, 
 we observe 
\[
  \nulambda \in /\!/
  \,\,
  \Leftrightarrow
  \,\,
  (\lambda,m-\nu) \in \backslash \!\backslash.  
\]
Applying the residue formulae
 in Theorem \ref{thm:reduction} (1) and (2)
 to the identity \eqref{eqn:TApA}, 
 we have 
\index{sbon}{qBA@$q_B^A$}
\[
q_B^A(\lambda,\nu)\T{\nu}{m-\nu}\circ \tB\lambda\nu
=
p_A^{TA}(\lambda,\nu)q_C^A(\lambda,m-\nu)
\tC{\lambda}{m-\nu}
\quad
\text{for }(\lambda,\nu)\in /\!/.  
\]
Since $q_B^A(\lambda,\nu)\ne 0$, 
the first statement follows from the elementary identity
\index{sbon}{qCA@$q_C^A$}
\[
q_B^A(\lambda,\nu)
p_C^{TB}(\lambda,\nu)
=
p_A^{TA}(\lambda,\nu)
q_C^A(\lambda,m-\nu)
\quad
\text{ on } /\!/.  
\]

\par\noindent
2)\enspace
Apply Theorem \ref{thm:TABC} (1) 
 to the identity \eqref{eqn:TApA}
 with $(\lambda, m-\nu) \in \backslash\!\backslash$,
and compose
$\T{\nu}{m-\nu}$.
Now we have
\[
\T{\nu}{m-\nu}
\circ 
\T{m-\nu}{\nu}
\circ
\tB{\lambda}{m-\nu}
=
p_C^{TB}(\lambda,m-\nu)
\tC{\lambda}{\nu}
\quad
\text{for }
(\lambda,\nu) \in /\!/.  
\]
The second statement follows now from the identity \eqref{eqn:TTKS}
of the Knapp--Stein intertwining operator.
\par\noindent
3)\enspace
We apply Theorem \ref{thm:reduction} (3)
 to the second identity,
 and obtain 
\index{sbon}{qCB@$q_C^B$}
\[
\T{\nu}{m-\nu} \circ \tC{\lambda}{\nu}
=
p_B^{TC} (\lambda,\nu) q_C^B(\lambda,\nu) 
\tC{\lambda}{m-\nu}
\quad
\text{for }
(\lambda,\nu) \in {\mathbb{X}}.   
\]
Now the statement follows from the elementary identity
\index{sbon}{pCTC@$p_{C}^{TC}$}
\index{sbon}{pBTC@$p_{B}^{TC}$}
\[
p_C^{TC}(\lambda,\nu)
=
 p_B^{TC}(\lambda,\nu)
q_C^B(\lambda,\nu) 
\quad
\text{on }
{\mathbb{X}}.  
\]
\par\noindent
4)\enspace
By the definition of the renormalized operator
 $\AAt_{{\lambda},{\nu}}$
 (see \eqref{eqn:K2lmdnu}), 
 we have 
\index{sbon}{pATA@$p_{A}^{TA}$}
\[
\T{\nu}{m-\nu}
\circ
\AAt_{{\lambda},{\nu}}
=
\Gamma(\frac{\lambda-\nu}{2})
p_A^{TA}(\lambda,\nu)
\tA{\lambda}{m-\nu}
\quad
\text{for }
\nu \in -{\mathbb{N}}.  
\]
Therefore
\begin{equation}
\label{eqn:TA0}
\T{\nu}{m-\nu}\circ \AAt_{{\lambda},{\nu}}=0
\end{equation}
for $\nu \in -{\mathbb{N}}$
 and $\lambda-\nu \notin -2{\mathbb{N}}$.  
Since the left-hand side \eqref{eqn:TA0} depends
holomorphically on $\lambda \in {\mathbb{C}}$
 with fixed $\nu \in -{\mathbb{N}}$, 
 we proved the last statement.  
\end{proof}
If $m-2\nu \in 2 {\mathbb{N}}$ 
then $\T{\nu}{m-\nu}$
 is reduced to a differential operator
 (see \eqref{eqn:KSres}).  
In this case,
 Theorem \ref{thm:TABC} (1) and (3) reduce to:
\begin{corollary}
[functional identities
 with differential intertwining operators]
\label{cor:DCC}
~~~
\par\noindent
{\rm{1)}}\enspace
If $\nulambda \in \backslash\!\backslash$, 
 and 
 $m-2 \nu \in 2{\mathbb{N}}$, 
then 
\[
  \Delta_{{\mathbb{R}}^m}^{\frac m 2-\nu} \circ \B_{{\lambda},{\nu}}
  =
  \frac{(-1)^{\frac m 2-\nu}2^{\lambda-\nu}(2k)!\pi^{\frac m 2}}
  {\Gamma(\nu)}
  \CC{\lambda}{m-\nu}.  
\]
\par\noindent
{\rm{2)}}\enspace
For $\nulambda \in\mathbb{X}$
 such that $\nu < \frac{n-1}{2}$,
we set $k,l \in \mathbb{N}$ by \eqref{eqn:ldef} and \eqref{eqn:kdef}.
Then $k>l$ and
\[
\Delta_{\mathbb{R}^m}^{k-l} \circ \C_{\lambda,\nu}
= \frac{k!}{l!} \C_{\lambda,-\nu+n-1}.
\]
\end{corollary}
\begin{proof}
1)\enspace
By the residue formula \eqref{eqn:KSres} of the Knapp--Stein
intertwining operator,
 we have
\[
\T{\nu}{m-\nu}
=
\frac{(-1)^{\frac m2 -\nu}\pi^{\frac m 2}}{2^{m-2\nu}\Gamma(m-\nu)}
\Delta_{\mathbb{R}^m}^{\frac m 2-\nu}
\]
if $m-2\nu \in 2 {\mathbb{N}}$.  
Combining this
 with Theorem \ref{thm:TABC} (1), 
 we get the desired identity.  
\par\noindent
2)\enspace
By the residue formula \eqref{eqn:KSres} of the Knapp--Stein
intertwining operator,
we have
\[
\T{{\frac{m}{2}-(k-l)}}{\frac m 2+(k-l)}
= \frac{(-1)^{k-l} \pi^{\frac{m}{2}}}
          {2^{2k-2l} \Gamma(\frac{m}{2}+k-l)}
  \Delta_{\mathbb{R}^m}^{k-l}.
\]
On the other hand,
in view that $(\lambda,-\nu+m)\in\mathbb{X}$,
we have from Theorem \ref{thm:reduction} (3):
\index{sbon}{qCB@$q_C^B$}
\[
\B_{\lambda,-\nu+m}
= q_C^B(\lambda,-\nu+m) \C_{\lambda,-\nu+m}.
\]
Now Corollary follows from Theorem \ref{thm:TABC} (2).
\end{proof}

Next,
 we consider the Knapp--Stein intertwining operator
 $\T{n-\lambda}{\lambda}:I(n-\lambda) \to I(\lambda)$
 for the group $G$.  
We have seen in Theorem \ref{thm:TAAT}
 that the functional identity
\[
\tA{\lambda}{\nu} \circ \T{n-\lambda}{\lambda}
=\frac{\pi^{\frac n 2}}{\Gamma(\lambda)}\tA{n-\lambda}{\nu}
\]
holds for all $\nulambda \in {\mathbb{C}}^2$.  
By the residue formulae
 in Theorem \ref{thm:reduction}, 
we also obtain the functional identities
 among singular symmetry breaking operators
 $\B_{\lambda, \nu}$, $\CC{\lambda}{\nu}$, 
etc.  
However,
 we do not have formulae
 like Theorem \ref{thm:TABC}
 (2) or (3)
 that switch $\B$ and $\C$
 because the inversion 
 $(\lambda,\nu)\mapsto(n-\lambda,\nu)$
 does not exchange 
 $\backslash\!\backslash$ and $/\!/$
 whereas the inversion 
$(\lambda,\nu)\mapsto (\lambda,m-\nu)$
 did so.  
Thus we write down the reduction formulae
 only when $\T{n-\lambda}{\lambda}$
 reduces to a differential operator.  
\begin{corollary}
\label{cor:10.6}
{\rm{1)}}\enspace
If $\lambda= \frac n 2+j$ for some $j \in {\mathbb{N}}$, 
then 
\[
\A_{\lambda, \nu} \circ \Delta_{\mathbb{R}^n}^j
=(-1)^j2^{2j} \A_{n-\lambda, \nu}.  
\]
\par\noindent
{\rm{2)}}\enspace
If $(\lambda,\nu)=(\frac n 2+j,\frac n 2-1-j-2k)$
 for some $j, k \in {\mathbb{N}}$, 
then 
\[
\B_{\lambda, \nu} \circ \Delta_{\mathbb{R}^n}^j
=\frac{2^{j}(2k-1)!!}{(2j+2k-1)!!} \B_{n-\lambda, \nu}.  
\]
\par\noindent
{\rm{3)}}\enspace
If $(\lambda,\nu)=(\frac n 2+j,\frac n 2+j+2l)$
 for some $j, l \in {\mathbb{N}}$, 
then 
\[
\CC{\lambda}{\nu} \circ \Delta_{\mathbb{R}^n}^j
=\frac{(l+j)!}{l!} \CC{n-\lambda}{\nu}.  
\]
\end{corollary}

\begin{remark}\label{rem:DCC}
The identity in Corollary \ref{cor:DCC} (2)
 and Corollary \ref{cor:10.6} (3) is called
 {\it{factorization identities}} in
\cite[Chapter 6]{Juhl}.
\end{remark}

\section{Image of symmetry breaking operators}
\label{sec:Image}

The spherical principal series representation 
\index{sbon}{Jnu@${J(\nu)}$}
$J(\nu)$
of $G'=O(n,1)$ is not irreducible when
$\nu \in -\mathbb{N}$ or
$\nu - m \in \mathbb{N}$ where
$n = m+1$ as before.
In this chapter, 
 we determine the images of all of our symmetry breaking operators
$\A_{\lambda,\nu}$, $\AAt_{\lambda,\nu}$,
$\B_{\lambda,\nu}$, $\BB_{\lambda,\nu}$,
and $\C_{\lambda,\nu}$
 completely
 at the level of $({\mathfrak {g}}', K')$-modules,
 and obtain a partial information
 on their kernels
 when
$\nu \in (-\mathbb{N}) \cup (m+\mathbb{N})$.

\subsection{Finite-dimensional image for $\nu \in -\mathbb{N}$}
\label{subsec:ImageF}

For $\nu = -j \in -\mathbb{N}$,
we recall from Section \ref{subsec:matrix}
 that there is a non-splitting exact sequence
\index{sbon}{Tj@$T(j)$}
\index{sbon}{Fj@$F(j)$}
\[
0 \to F(j) \to J(-j) \to T(j) \to 0
\]
of $G'$-modules.
Therefore, 
the closure of the image
 of the symmetry breaking operators
 $\A_{\lambda,\nu}$, 
 $\AAt_{\lambda,\nu}$, 
 $\B_{\lambda,\nu}$, 
 $\BB_{\lambda,\nu}$, 
 and 
 $\C_{\lambda,\nu}$
 must be one of $\{0\}$, 
 $F(j)$ or $J(-j)$.   
\begin{theorem}
[images of symmetry breaking operators]
\label{thm:ImageF}
For $\nu = -j \in -\mathbb{N}$,
the images of the underlying 
$({\mathfrak {g}}, K)$-modules
 $I(\lambda)_K$ 
 under the symmetry breaking operators
 are given as follows:
\index{sbon}{Alnt@$\A_{\lambda,\nu}$}
\index{sbon}{Att@$\AAt_{\lambda,\nu}$}
\index{sbon}{Bt@$\B_{\lambda,\nu}$}
\index{sbon}{Btt@$\BB_{\lambda,\nu}$}
\index{sbon}{Ct@$\C_{\lambda,\nu}$}
\begin{alignat*}{3}
&\text{{\rm{1)}}}\qquad\qquad\qquad
&&\operatorname{Image} \A_{\lambda,\nu}
&&
= \begin{cases}
   F(j)   & \text{if\/ $\nulambda \notin L_{\operatorname{even}}$},
        \\
                       \{0\}   & \text{if\/ $\nulambda \in L_{\operatorname{even}}$}.  
      \end{cases}
\\
&\text{{\rm{2)}}}
&&\operatorname{Image} \AAt_{\lambda,\nu}       
&&= F(j)
\quad
\text{for any }\lambda \in {\mathbb{C}}.  
\\
&\text{{\rm{3)}}}
&&\operatorname{Image} \B_{\lambda,\nu}
&&=
\begin{cases}
 F(j) \quad &\text{if 
$\nulambda \in \backslash\!\backslash\setminus
 L_{\operatorname{even}},$
}
\\
 \{0\}      
&\text{if $\nulambda \in L_{\operatorname{even}}$ and $n$ is odd.}
\end{cases}
\\
&\text{{\rm{4)}}}
&&\operatorname{Image} \BB_{\lambda,\nu}       
&&= F(j)
    \quad\text{if $\nulambda \in L_{\operatorname{even}}$ and $n$ is odd}.
\\
&\text{{\rm{5)}}}
&&\operatorname{Image} \C_{\lambda,\nu}
&&= J(-j)_{K'}.
\end{alignat*}
We note that 
 $L_{\operatorname{even}} \subset \backslash\!\backslash$
 if $n$ is odd,
 and $\backslash\!\backslash\setminus L_{\operatorname{even}}
=\backslash\!\backslash$ if $n$ is even.  
\end{theorem}

\begin{proof}
We recall from \eqref{eqn:KAA}
 that the distribution kernel $\KA\lambda\nu$
 of $\A_{\lambda,\nu}$ is a polynomial of
$x_1,\dots,x_{n-1}$
of degree at most $j$ if
$\nu = -j \in \mathbb{N}$.  
Therefore, 
the image of $\A_{\lambda,\nu}$ must be contained in 
the space
 of polynomials
 of $x_1$, $\cdots$, $x_{n-1}$
 of degree $\le j$, 
 which is finite-dimensional.
Therefore the image of $\A_{\lambda,\nu}$ is either $F(j)$ or $\{0\}$
because the finite-dimensional representation 
 $F(j)$ is the unique proper $G'$-submodule of $J(\nu)$.
Since $\A_{\lambda,\nu}$ is nonzero if and only if
$\nulambda \in L_{\operatorname{even}}$
by Theorem \ref{thm:poleA},
we get the first statement.

Similar arguments work for $\AAt_{\lambda,\nu}$,
$\B_{\lambda,\nu}$, and $\BB_{\lambda,\nu}$, 
yielding the second, 
 third, 
 and fourth statements.
Here we recall 
that $\B_{\lambda,\nu}$
is defined for
$\nulambda \in \backslash\!\backslash$, 
 which is zero
 if and only if $n$ is odd and $\nulambda \in L_{\operatorname{even}}$.

Let us prove the last assertion.  
For any open subset $U$ in ${\mathbb{R}}^{n-1}$, 
 we can find a compactly supported function $h_U$
 such that $\operatorname{Supp}h_U \cap 
 \{x_n=0\} \subset U$
 and $\C_{\lambda,\nu}(h_U) \ne 0$
 (for example,
 we can take $h_U$
 such that $h_U(x)=x_n^{\nu-\lambda}$
 on some non-empty open subset in $U$).  
Taking countably many disjoint open subsets $U_i$ in ${\mathbb{R}}^{n-1}$, 
 we see that $\C_{\lambda,\nu}(h_{U_i})$
 are linearly independent
 because the support 
 of $\C_{\lambda,\nu}(h_{U_i})$ is contained in $U_i$.  
Thus the image of $\C_{\lambda,\nu}$ cannot be finite-dimensional.
\end{proof}

\subsection{Image for $\nu \in m+\mathbb{N}$}
\label{subsec:ImageT}

For $\nu = m+j$ ($j \in \mathbb{N}$),
we recall from Section \ref{subsec:matrix}
 that there is a non-splitting exact sequence
\[
0 \to T(j) \to J(m+j) \to F(j) \to 0
\]
of $G'$-modules.
Therefore, the closure of the image of 
 the symmetry breaking operators $\A_{\lambda,\nu}$,
$\B_{\lambda,\nu}$ and $\C_{\lambda,\nu}$ 
 must be one of $\{0\}$, 
 $T(j)$ or $J(m+j)$.  
We determine which case occurs precisely:
\begin{theorem}
[images of symmetry breaking operators]
\label{thm:ImageT}
Suppose $\nu = m+j$ \textup{(}$j \in \mathbb{N}$\textup{)}.
Then the images
 of the underlying $({\mathfrak {g}}, K)$-modules
 $I(\lambda)_K$
 of the symmetry breaking operators
 are given as follows:
\begin{enumerate}[\upshape 1)]
\item   
\[
\operatorname{Image}\A_{\lambda,\nu}
   = \begin{cases}
             T(j)_{K'}       &\text{if\/ $\lambda+j \in -2\mathbb{N}$}, \\
             J(m+j)_{K'}  &\text{if\/ $\lambda+j \notin -2\mathbb{N}$}. 
     \end{cases}
\]
\item  
For $\nulambda \in \backslash\!\backslash$,
or equivalently, $\lambda+j \in -2\mathbb{N}$,
\[
\operatorname{Image}\B_{\lambda,\nu} = T(j)_{K'}  .
\]
\item   
For $\nulambda \in /\!/$,
\[
\operatorname{Image}\C_{\lambda,\nu}
   = \begin{cases}
             T(j)_{K'}         &\text{if\/ $\lambda+j \in -2\mathbb{N}$ and $n$ is odd}, \\
             J(m+j)_{K'}    &\text{otherwise}. 
     \end{cases}
\]
\end{enumerate}
\end{theorem}

We remark that $\B_{\lambda,\nu}$ is defined for 
 $\nulambda \in \backslash\!\backslash$,
which is equivalent to the condition
$\lambda+j \in -2\mathbb{N}$
when $\nu = m+j$.

\begin{proof}
We know by Theorems \ref{thm:poleA} and  \ref{thm:BK}
that $\A_{\lambda,\nu}$ and $\B_{\lambda,\nu}$ do not vanish
if $\nu = m+j$ ($j \in \mathbb{N}$),
and therefore their image is either $T(j)_{K'}$ or $J(\nu)_{K'}$.
\par\noindent
1)\enspace
By Theorem \ref{thm:TAAT},
we have
\[
\T{m+j}{-j} \circ \A_{\lambda,m+j}
= \frac{\pi^{\frac{m}{2}}}{\Gamma(m+j)} \A_{\lambda,-j}.
\]
Since the kernel of the Knapp--Stein intertwining operator
$\T{m+j}{-j}$ is $T(j)$,
the image of $\A_{\lambda,m+j}$ is $T(j)_{K'}$
if and only if $\A_{\lambda,-j}$ is zero.
By Theorem \ref{thm:poleA},
this happens if and only if
$(\lambda,-j) \in L_{\operatorname{even}}$,
namely, $\lambda+j \in -2\mathbb{N}$.
Thus the first assertion is proved.
\par\noindent
2)\enspace
We recall from Theorem \ref{thm:TABC} (1)
 that 
for $(m+j,\lambda) \in \backslash\!\backslash$,
\index{sbon}{pCTB@$p_{C}^{TB}$}
\[
\T{m+j}{-j} \circ \B_{\lambda,m+j}
= p_C^{TB}(\lambda,m+j) \C_{\lambda,-j},
\]
and $p_C^{TB}(\lambda,m+j)=0$ by definition.  
Therefore
$\operatorname{Image}\B_{\lambda,m+j} = T(j)_{K'}$.

\par\noindent
3)\enspace
We apply Theorem \ref{thm:TABC} (2) with
$\nu = \lambda+j$.
Then we obtain
\[
\T{m+j}{-j} \circ \C_{\lambda,m+j}
= \frac{2^l}{(2l-1)!!} \, \B_{\lambda,-j}, 
\]
where $l \in \mathbb{N}$ is defined by
$m+j-\lambda = 2l$.
Hence
$\operatorname{Image}\C_{\lambda,m+j}$
is contained in $T(j)$ if and only if $\B_{\lambda,-j}$ vanishes,
which happens exactly when $n$ is odd and
$(\lambda,-j) \in L_{\operatorname{even}}$
by Theorem \ref{thm:BK}.
Since $\C_{\lambda,m+j}$
 is not zero,
the third statement is proved.  
\end{proof}

\subsection{Spherical vectors and symmetry breaking operators}
\label{subsec:KerminK}

Since the symmetry breaking operators are $G'$-homomorphisms
 by definition,
their kernels are just $G'$-submodules of the $G$-module $I(\lambda)$, 
 which are not of finite length.  
In this section
 we give a partial result 
 on their kernels,
 by determining a precise condition
 on the parameter 
 for which the spherical vector
\index{sbon}{1lmd@${\mathbf{1}}_{\lambda}$}
 $\mathbf{1}_\lambda$ lies in their kernels.
\begin{theorem}\label{thm:kernel}
\begin{enumerate}[\upshape 1)]
\item   
For $\nulambda \in \mathbb{C}^2$,
$\mathbf{1}_\lambda \in \operatorname{Ker} \A_{\lambda,\nu}$
if and only if\/ $\lambda \in -\mathbb{N}$.
\item   
For any $\nulambda \in L_{\operatorname{even}}$,
$\mathbf{1}_\lambda \notin \operatorname{Ker}\AAt_{\lambda,\nu}$.
\item   
For $\nulambda \in \backslash\!\backslash$,
$\mathbf{1}_\lambda \in \operatorname{Ker} \B_{\lambda,\nu}$
if and only if\/ $\lambda \in -\mathbb{N}$.
\item   
For any $\nulambda \in L_{\operatorname{even}}$ and $n$ odd,
$\mathbf{1}_\lambda \notin \operatorname{Ker}\BB_{\lambda,\nu}$.
\item   
For $\nulambda \in /\!/$,
$\mathbf{1}_\lambda \in \operatorname{Ker}\C_{\lambda,\nu}$
if and only if\/ $\nu > 0 \ge \lambda$.
\end{enumerate}
\end{theorem}

\begin{proof}
The first statement follows from Proposition \ref{prop:AminK},
the second one from Proposition \ref{prop:B1}, 
 the third one from Proposition \ref{prop:B1} and
the fifth one from Proposition \ref{prop:C11}.  
Finally,
 by Theorem \ref{thm:reduction} (4) and Remark \ref{rem:nuneg}
we have
\begin{align}
\BB_{\lambda,\nu}
({\bf{1}}_{\lambda})
=&\frac{2^k(2k-1)!!}{(-1)^k}
 \frac{\pi^{\frac{n-1}{2}}(-\lambda)!(-1)^{\lambda+l}}{l!}
 {\bf{1}}_{\nu}
\notag
\\
=&\frac{(-1)^{\frac{n-1}{2}}\pi^{\frac{n-1}{2}}(-\lambda)!(2k)!}{k!l!}
 {\bf{1}}_{\nu}
\label{eqn:BB1}
\end{align}
if $(\lambda,\nu) \in L_{\operatorname{even}}$
 and $n$ is even.  
Hence the fourth statement is shown.  
\end{proof}

\section{Application to analysis on anti-de Sitter space}
\label{sec:PGinv}

The last two chapters are devoted
 to applications
 of our results
 on symmetry breaking operators.  
In this chapter
 we discuss applications to harmonic analysis
 on the semisimple symmetric space
 $G/G'=O(n+1,1)/O(n,1)$.  
We begin with the scalar valued case
 in the first two sections, 
 and clarify 
 a close relationship
 between symmetry breaking operators
 for the restriction $G \downarrow G'$
 for the special parameter $\nu =0$
 and the Poisson transform 
 for $G/G'$.  
In particular,
 we shall see 
 that 
 the vanishing phenomenon of the symmetry breaking operators
 $\A_{\lambda,\nu}$
 at $L_{\operatorname{even}}$
 (Theorem \ref{thm:poleA}) explains
 a subtle difference 
 on the composition series
 of the eigenfunction
 of the Laplacian
 (Fact \ref{fact:Composition}) 
 and the principal series representation
 (\eqref{eqn:FIT} and \eqref{eqn:TIF}).  
More generally,
 we apply the results on symmetry breaking operators
 $\A_{\lambda,\nu}$
 with $\nu \in -{\mathbb{N}}$
 (Theorems \ref{thm:compo} and \ref{thm:mIF})
 to obtain some new results 
 of analysis on vector bundles
 over $G/G'$
 (Theorem \ref{thm:lise}).  

\subsection{Harmonic analysis on Lorentzian hyperbolic spaces
}
\label{subsec:hyp}
We define the indefinite hyperbolic space by 
\[
X(n+1,1) :=
\{ \xi \in \mathbb{R}^{n+2}: 
   \xi_0^2 + \dots + \xi_{n}^2 - \xi_{n+1}^2 = 1 \}.
\]
As a hypersurface of the Minkowski space
$$\mathbb{R}^{n+1,1} \equiv
(\mathbb{R}^{n+2}, d\xi_0^2 + \dots + d \xi_{n}^2 - d \xi_{n+1}^2),
$$
$X(n+1,1)$ carries a Lorentz metric for which the sectional curvature 
is constant $-1$
 (see \cite{Wo}),
and thus is a model space of anti-de Sitter manifolds.
The group $G=O(n+1,1)$
 acts transitively on $X(n+1,1)$.  
The isotropy subgroup
 at 
\[
e_n={}^{t\!} (0,\cdots,0,1,0)
 \in X(n+1,1)
\]
is nothing
 but $G'=O(n,1)$, 
 and therefore we have an isomorphism:
\[
   X(n+1,1)\simeq G/G', 
\]
which shows that $X(n+1,1)$ is a semisimple
 symmetric space
 of rank one.

Let $\Delta$ be the Laplace--Beltrami operator
 on $X(n+1,1)$.  
Since $X(n+1,1)$ is a Lorentzian manifold, 
 $\Delta$ is a hyperbolic operator.  
For $\lambda \in \mathbb{C}$,
we consider the eigenspace of the Laplacian
$\Delta$ on $X(n+1,1)$:
\[
\index{sbon}{SolGG@$\mathcal{S}ol (G/G';\lambda)$|textbf}
\mathcal{S}ol (G/G'; \lambda)
:= \{ f \in \mathcal{C}^\infty (G/G') : \Delta f = \lambda(\lambda-n) f \}.
\]

Since $G$ acts isometrically on the Lorentz manifold $G/G'$,
$G$ leaves 
$\mathcal{S}ol(G/G';\lambda)$ invariant for any
$\lambda\in\mathbb{C}$.  
Our parametrization is given in a way
 that 
\begin{equation}
\label{eqn:soleg}
\mathcal{S}ol (G/G';\lambda)
\simeq
\mathcal{S}ol (G/G';n-\lambda).  
\end{equation}
\begin{remark}
\label{rem:ACDB}
Since the Laplacian $\Delta$ is not
 an elliptic operator,
 the analytic regularity theorem 
 does not apply,
 and the eigenfunctions in $\mathcal{A}(G/G')$ (analytic functions),
$\mathcal{D}'(G/G')$ (distributions),
or $\mathcal{B}(G/G')$ (hyperfunctions)
 are not the same.  
However,
 the underlying $({\mathfrak {g}},K)$-module
 $\mathcal{S}ol (G/G';\lambda)_K$
 does not depend on the choice
 of the sheaves 
 ${\mathcal{A}} \subset C^{\infty}
 \subset {\mathcal{D}}' \subset {\mathcal{B}}$
 because $K$-finite hyperfunction solutions 
 are automatically real analytic
 by the elliptic regularity theorem
 (see \cite[Theorem 3.4.4]{KKK}).  
\end{remark}

Traditional questions of harmonic analysis on the symmetric space
 (see \cite[Chapter 1]{Hel}, 
 for example)
 are 
\begin{itemize}
\item
to expand functions on $G/G'$ by eigenfunctions of the Laplacian $\Delta$, 
\item
to find the $G$-module structure of
$\mathcal{S}ol(G/G';\lambda)$.  
\end{itemize}
For the anti-de Sitter space
$G/G' \simeq X(n+1,1)$,
a complete answer to these questions has been known.

Concerning the first question,
we decompose the regular representation on
the Hilbert space $L^2(G/G')$ 
into the discrete and continuous parts:
\[
L^2(G/G')
= L^2(G/G')_{\operatorname{disc}} \oplus L^2(G/G')_{\operatorname{cont}}.
\]
Then the discrete part is a multiplicity-free Hilbert direct sum of irreducible
unitary representations of $G$
({\it{discrete series representations}}
 for $G/G'$) as follows:
\begin{fact}
[{\cite{Fa,Str}, see also \cite[Fact 5.4]{KO2}}]
\label{fact:Disc}
\[
L^2 (G/G')_{\operatorname{disc}} 
= \bigoplus_{\substack{\frac{n}{2} < \lambda < n  \\
                               \lambda \equiv n+1\bmod 2 \mathbb{Z}}  }
   \overline{I(\lambda)} \oplus \sideset{}{^\oplus}\sum_{i=0}^\infty  
   \overline{T(2i+1)}.  
\]
\end{fact}
Here $\overline{I(\lambda)}$ denotes 
 the unitarization of 
 $I(\lambda)$, 
 and $\overline{T(2i+1)}$
 denotes that of $T(2i+1)$.  
We note that $\overline{I(\lambda)}$
 ($0 < \lambda <n$)
 is unitarily equivalent
 to a complementary series representation
\index{sbon}{Hlmd@${\mathcal{H}}_{\lambda}^G$}
 ${\mathcal{H}}_{\lambda}^G$
 of $G$ in the notation of Chapter \ref{sec:applbl}.

Concerning the second question,
the $G$-module
$\mathcal{S}ol(G/G';\lambda)$
is of finite length,
and therefore it is sufficient to determine a
Jordan--H\"{o}lder series 
at the level of underlying
$(\mathfrak{g},K)$-modules.
Here is a description
 when $\mathcal{S}ol(G/G';\lambda)$ is reducible:

\begin{fact}
[\cite{Sch}]
\label{fact:Composition}
For $\lambda = -i \in -\mathbb{N}$,
there is a non-splitting exact sequence
 of $({\mathfrak {g}},K)$-modules:
\index{sbon}{Fi@$F(i)$}
\[
\begin{array}{cccclccccl}
0 & \to  & F(i)  & \to & \mathcal{S}ol(G/G';\lambda)_K
   & \to & T(i)_K & \to & 0  & \text{\textup{(}$i$: even\textup{)}},
\\
0 & \to & T(i)_K & \to & \mathcal{S}ol(G/G';\lambda)_K
  & \to & F(i) & \to & 0  & \text{\textup{(}$i$: odd\textup{)}}.
\end{array}
\]
\end{fact}

\subsection{Application of symmetry breaking operators
 to anti-de Sitter spaces}
\label{subsec:Ahyp}
In this section,
 we discuss a relationship
 between the aforementioned results
 and our results on symmetry breaking operators
 with $\nu=0$.

We recall
 that 
\index{sbon}{In-lmdinf@${I(n-\lambda)^{-\infty}}$}
$I(n-\lambda)^{-\infty}$ is the space
 of distribution vectors
 of the representations $I(n-\lambda)$ of $G$.  
We begin with the key observation
 for $\nu=0$:
\begin{eqnarray}
\operatorname{Hom}_G (I(\lambda), C^{\infty}(G/G'))
\simeq
(I(n-\lambda)^{-\infty})^{G'}
\hphantom{MMMMMMMMM}
\notag
\\
\hphantom{MMMM}
\subset
(I(n-\lambda)^{-\infty})^{P'}
\simeq
H(\lambda,0)
=
\operatorname{Hom}_{G'}(I(\lambda), J(0)).  
\label{eqn:PGhom}
\end{eqnarray}
Here, 
 the first isomorphism is obtained
 by applying Proposition \ref{prop:Distker},
 to the case where $X=G/P$
 and $Y=G/G'$.  
The second isomorphism 
 is a special case
 of the main object 
of this article.  
The inclusive relation \eqref{eqn:PGhom}
 implies
 that finding irreducible $G$-submodules
 in $C^{\infty}(G/G')$
 is a subproblem 
 of the understanding
 of symmetry breaking operators
 $\operatorname{Hom}_{G'}(I(\lambda), J(\nu))$
 with $\nu=0$.

In light that $(\lambda,0)\in L_{\operatorname{even}}$
 if and only if 
 $\lambda \in -2 {\mathbb{N}}$, 
 we see from the classification of symmetry breaking operators
 (Theorem \ref{thm:9.5}):
\[
(I(n-\lambda)^{-\infty})^{P'}
\simeq
\begin{cases}
{\mathbb{C}} \KA{\lambda}{0}
&\lambda \not \in -2{\mathbb{N}}, 
\\
{\mathbb{C}}\KAA{\lambda}{0}
\oplus
{\mathbb{C}}\KC{\lambda}{0}
\qquad
&\lambda \in -2{\mathbb{N}}.  
\end{cases}
\]
Let us determine 
 when the $P'$-invariance implies
 $G'$-invariance:
\begin{lemma}
\label{lem:Gprimeinv}
\[
(I(n-\lambda)^{-\infty})^{G'}
\simeq
\begin{cases}
{\mathbb{C}} \KA{\lambda}{0}
&\lambda \not \in -2{\mathbb{N}}, 
\\
{\mathbb{C}}\KAA{\lambda}{0}
\qquad
&\lambda \in -2{\mathbb{N}}.  
\end{cases}
\]
\end{lemma}
\begin{proof}
First,
 we recall from \eqref{eqn:klmdmu}
\[
\ska{\lambda}{0}(\xi)
=2^{-\lambda+n}|\xi_n|^{\lambda-n}.  
\]
Since $G'$ fixes the $n$-th coordinate $\xi_n$,
 the distribution $\ska{\lambda}{0}$
 is $G'$-invariant,
 and so are the normalized distributions
$\KA{\lambda}{0}=\iota_N^{\ast} \tska{\lambda}{0}$
 and $\KAA{\lambda}{0}=\iota_N^{\ast} \ttska{\lambda}{0}$, 
 see \eqref{eqn:iNk}.

Second,
 for $(\lambda,\nu) \in L_{\operatorname{even}}$, 
the support 
\index{sbon}{KxttCnu@$\KC{\lambda}{\nu}$}
$\KC{\lambda}{\nu}$ 
 of the differential symmetry breaking operator
 $\C_{{\lambda},{\nu}}$
 is 
\index{sbon}{pplus@${p_+}$}
$\{[p_+]\}$ in $G/P$, 
which is not a $G'$-invariant subset
 (see Lemma \ref{lem:5.1}).  
In particular, 
 for $\lambda \in -2 {\mathbb{N}}$, 
 the distribution $\KC{\lambda}{0} \in (I(n-\lambda)^{-\infty})^{P'}$
 cannot be $G'$-invariant.  
Thus Lemma is proved.  
\end{proof}

\vskip 2pc

\begin{remark}
\label{rem:CLeven2}
We have seen in Remark \ref{rem:CLeven}
 that the differential symmetry breaking operator
 $\C_{\lambda,\nu}$ cannot be obtained
 as the residue of the meromorphic family
 ${\mathbb{A}}_{\lambda,\nu}$
 of symmetry breaking operators
 if $(\lambda,\nu)\in L_{\operatorname{even}}$.  
The above lemma gives an alternative proof
 of this fact 
 for $\nu=0$ 
because $\A_{\lambda,0}$
 is $G'$-invariant
 but $\C_{\lambda,0}$
 is not $G'$-invariant.  
\end{remark}

The distribution kernel 
 $\KA{\lambda}{0}$ (or $\KAA{\lambda}{0})
 \in (I(n-\lambda)^{-\infty})^{G'}$
 induces a $G$-intertwining operator
{} from $I(\lambda)$
 to $C^{\infty}(G/G')$.  
Let us give a concrete formula.  
For this,
 we write 
$[\,\, ,\,\,]: {\mathbb{R}}^{n+1,1}\times {\mathbb{R}}^{n+1,1}\to{\mathbb{R}}$
for the bilinear form 
 defined by 
\[
[x,\xi]:=x_{0}\xi_{0} + \cdots + x_{n}\xi_{n}-x_{n+1}\xi_{n+1}.  
\]
\begin{lemma}
\label{lem:nxxi}
For $g \in G$, 
we set 
 $x:=g e_n \in X(n+1,1)$.  
Then the $n$-th coordinate $(g^{-1}\xi)_n$
 of $g^{-1} \xi$ is given by
\[
  (g^{-1} \xi)_n=[x,\xi]
\qquad
 \text{for }\xi \in \Xi.  
\]
\end{lemma}
\begin{proof}
$
 (g^{-1} \xi)_n
=[e_n,g^{-1} \xi]
=[g e_n,\xi]
=[x,\xi].  
$
\end{proof}
We recall from \eqref{eqn:Gpair}
 that the $G$-invariant pairing
\index{sbon}{In-lmdinf@${I(n-\lambda)^{-\infty}}$}
\[
\langle \,\, , \,\, \rangle
: I(\lambda) \times I(n-\lambda)^{-\infty} \to {\mathbb{C}}
\]
induces a $G$-intertwining operator
\[
I(\lambda) \to C^{\infty}(G), 
 \quad
f \mapsto \langle f, \tska{\lambda}{\nu} (g\cdot)\rangle
=\langle f(g^{-1}\cdot), \tska{\lambda}{\nu} \rangle, 
\]
which in turn induces a $G$-intertwining operator
\[
\index{sbon}{Plmdt@$\widetilde{{\mathcal{P}}}_{\lambda}$|textbf}
\widetilde {\mathcal{P}}_{\lambda}:
I(\lambda) \to C^{\infty}(G/G'), 
\]
\[
(\widetilde {\mathcal{P}}_{\lambda} f)(x)
=
\frac{2^{-\lambda+n}}
     {\Gamma(\frac{\lambda+\nu-n+1}{2})\Gamma(\frac{\lambda-\nu}{2})}
\int_{S^n} f(b)|[x,b]|^{\lambda-n}db
\]
by Lemma \ref{lem:nxxi}.  
The image satisfies
 the differential equation
\[
\Delta (\widetilde {\mathcal{P}}_{\lambda} f)
=\lambda(n-\lambda)(\widetilde {\mathcal{P}}_{\lambda} f)
\]
 because the distribution kernel 
 $|[x,b]|^{\lambda-n}$
 satisfies
 the same equation.  
Thus $\widetilde {\mathcal{P}}_{\lambda}$ gives
 a $G$-intertwining operator
\begin{equation}
\label{eqn:Poisson}
\widetilde {\mathcal{P}}_{\lambda}:
I(\lambda) \to \mathcal{S}ol (G/G';\lambda).  
\end{equation}
The integral transform 
 $\widetilde {\mathcal{P}}_{\lambda}$ is called 
 the {\it{Poisson transform}}
 for the semisimple symmetric space $G/G'$
 as explained in Example \ref{ex:regular} (2).  
Similarly,
 we can define a renormalized Poisson transform 
\index{sbon}{Plmdtt@$\Tilde{\Tilde{{\mathcal{P}}}}_{\lambda}$|textbf}
\begin{equation}
\label{eqn:Poisson2}
\tilde{\tilde {\mathcal{P}}}_{\lambda}
:
 I(\lambda) \to \mathcal{S}ol (G/G'; \lambda)
 \subset C^{\infty}(G/G')
\quad
\text{for }
 \lambda \in -2 {\mathbb{N}}.  
\end{equation}
Since the Poisson transform $\widetilde {\mathcal{P}}_{\lambda}$
 and the symmetry breaking operator
 $\A_{{\lambda},{\nu}}$ are induced
 by the same distribution kernel 
 $\KA{\lambda}{\nu}$, 
 we get the evaluation 
 at the base point $e_n \in X(n+1,1)$
{}from Proposition \ref{prop:AminK}
 and Remark \ref{rem:nuneg}
as follows:
\begin{proposition}
\label{prop:P1}
Let 
\index{sbon}{1lmd@${\mathbf{1}}_{\lambda}$}
${\bf{1}}_{\lambda}$ be the normalized spherical vector
 in $I(\lambda)$.  
\begin{alignat*}{2}
\widetilde {\mathcal{P}}_{\lambda}
({\bf{1}}_{\lambda})(e_n)
=&\frac{\pi^{\frac{n-1}{2}}}{\Gamma(\lambda)} 
\qquad&&
\text{for }\,\, \lambda \in {\mathbb{C}}, 
\\
\tilde{\tilde {\mathcal{P}}}_{\lambda}
({\bf{1}}_{\lambda})(e_n)
=&\frac{\pi^{\frac{n-1}{2}}(2l)!(-1)^l}{l!}
\quad
&&\text{for }\,\,
\lambda=-2l\in -2 {\mathbb{N}}.   
\end{alignat*}
\end{proposition}

We note that the underlying $(\mathfrak{g},K)$-modules
 $I(\lambda)_K$ 
 and $\mathcal{S}ol(G/G';\lambda)_K$
 are isomorphic to each other 
in the Grothendieck group of $(\mathfrak{g},K)$-modules,
however, there is a subtle difference on
 the composition series:
For the principal series representation, 
we have nonsplitting exact sequences
 of $(\mathfrak{g},K)$-modules:
\[
\begin{array}{cccclcccccl}
0 & \to & F(i) & \to & I(-i)_K
   & \to & (T(i))_K & \to & 0, 
\\
0 & \to & T(i)_K & \to & I(n+i)
   & \to & F(i) & \to & 0.  
\end{array}
\]
for all $i \in {\mathbb{N}}$
 and there is no parity condition
 on $i$
 (see Section \ref{subsec:matrix})
 on the one had,
 whereas the parity condition
 on $i$ is crucial in Fact \ref{fact:Composition}.  
This difference has a close connection
 with the discrete subset 
\index{sbon}{Leven@$L_{\operatorname{even}}$}
 $L_{\operatorname{even}}$.  
To be more precise, 
 we determine the kernel
 and the image of the Poisson transforms
\begin{align*}
\widetilde {\mathcal{P}}_{\lambda}
:& I(\lambda) \to \mathcal{S}ol(G/G';\lambda), 
\\
\widetilde {\mathcal{P}}_{n-\lambda}
:& I(n-\lambda) \to \mathcal{S}ol(G/G';\lambda), 
\end{align*}
(or of the renormalized ones
 $\tilde{\tilde {\mathcal{P}}}_{\lambda}$
 and $\tilde{\tilde {\mathcal{P}}}_{n-\lambda}$) 
 at the reducible points of $I(\lambda)$.  
We note
 that both the images of $\widetilde {\mathcal{P}}_{\lambda}$
 and $\widetilde {\mathcal{P}}_{n-\lambda}$
 are contained in the same space $\mathcal{S}ol(G/G';\lambda)$, 
 see \eqref{eqn:soleg}.  

\vskip 2pc
\par\noindent
{\bf{Case I.}}\enspace
$\lambda=-i \in -{\mathbb{N}}$, 
 $i$ even.  
\begin{enumerate}
\item[$\bullet$]
$\KA{\lambda}{0}=0$, 
 and hence $\widetilde {\mathcal{P}}_{\lambda}$
 is a zero operator.  
\item[$\bullet$]
The renormalized Poisson transform 
 $\tilde{\tilde {\mathcal{P}}}_{\lambda}$ satisfies 
$\tilde{\tilde {\mathcal{P}}}_{\lambda}({\bf{1}}_{\lambda})\ne 0$
 by Proposition \ref{prop:P1}, 
and therefore 
 $\tilde{\tilde {\mathcal{P}}}_{\lambda}
 :I(\lambda) \to \mathcal{S}ol(G/G';\lambda)$
 is injective
 because the spherical vector 
 ${\bf{1}}_{\lambda}$
 belongs to the unique finite-dimensional subrepresentation
 of $I(\lambda)$.  
Indeed, 
 $\tilde{\tilde {\mathcal{P}}}_{\lambda}$
 induces a bijection from $I(\lambda)_K$
 onto $\mathcal{S}ol(G/G';\lambda)_K$
 in view of the Jordan--H{\"o}lder series 
 of $\mathcal{S}ol(G/G';\lambda)_K$ in Fact \ref{fact:Composition}.  
\end{enumerate}
\par\noindent
{\bf{Case II.}}\enspace
$\lambda=-i \in -{\mathbb{N}}$, 
 $i$ odd.  

Since ${\widetilde {\mathcal{P}}}_{\lambda}({\bf{1}}_{\lambda})= 0$
 by Proposition \ref{prop:P1}, 
and since $\KA{\lambda}{0} \ne 0$, 
 $\operatorname{Ker}{\widetilde {\mathcal{P}}}_{\lambda}$
 coincides with the unique finite-dimensional 
 subrepresentation $F(i)$
 of $I(\lambda)$ containing ${\bf{1}}_{\lambda}$.  
Further, 
 ${\widetilde {\mathcal{P}}}_{\lambda}$ induces 
 a surjective map from $I(\lambda)_K$
 to the underlying $({\mathfrak {g}}, K)$-module
 of a discrete series representation of $G/G'$
 which is isomorphic to $T(i)_K$
 by Facts 
\ref{fact:Disc} and \ref{fact:Composition}.  
\vskip 2pc
\par\noindent
{\bf{Case III.}}\enspace
$\lambda=n+i$, 
 $i \in {\mathbb{N}}$ even.  

If $\lambda -n \in 2 {\mathbb{N}}$, 
then the kernel $|[x,\xi]|^{\lambda-n}$
 of the Poisson transform is a homogeneous polynomial
 in $x=(x_0, \cdots, x_{n+1})$ of degree $i$.  
Therefore, 
 the image of the Poisson transform 
 ${\widetilde {\mathcal{P}}}_{\lambda}$ 
 is contained in the finite-dimensional vector space
 consisting of homogeneous polynomials
 of $x$ of degree $i$
 in the ambient space ${\mathbb{R}}^{n+1,1}$.  
Since $\KA{\lambda}{0}\ne 0$
 by Theorem \ref{thm:poleA}, 
 we conclude that 
\begin{align*}
\operatorname{Ker}{\widetilde {\mathcal{P}}}_{n+i}
\simeq & T(i), 
\\
\operatorname{Image}{\widetilde {\mathcal{P}}}_{n+i}
\simeq & F(i).  
\end{align*}

\par\noindent
{\bf{Case IV}}\enspace
$\lambda=n+i$, $i \in {\mathbb{N}}$ odd.  

We recall from the functional equation
 \eqref{eqn:ATTnew}
\[
\A_{n+i,0} \circ \T{-i}{n+i}
=
\frac{\pi^{\frac n 2}}{\Gamma(n+i)}
\A_{-i,0}
\]
 for the symmetry breaking operators
 $\A_{\lambda,\nu}$.  
This is regarded
 as the identity
 for the distribution kernels $\KA{\lambda}{\nu}$
 and the Riesz distribution 
 (the distribution kernel of the Knapp--Stein intertwining 
 operator of $G$).  
In turn, 
 we conclude
 that the following identity
 holds for the Poisson transforms:
\begin{equation}
\label{eqn:Pid}
 {\widetilde {\mathcal{P}}}_{n+i}
 \circ
 \T{-i}{n+i}
 =
  \frac{\pi^{\frac n 2}}{\Gamma(n+i)}
 {\widetilde {\mathcal{P}}}_{-i}.  
\end{equation}
In particular, 
 we have shown that 
\[
 {\widetilde {\mathcal{P}}}_{\lambda}
:I(\lambda) \to \mathcal{S}ol(G/G';\lambda)
\]
 is injective.  
Indeed, 
 ${\widetilde {\mathcal{P}}}_{\lambda}$ induces
 a bijection from $I(\lambda)_K$ 
 to $\mathcal{S}ol(G/G';\lambda)_K$
 in view of the Jordan--H{\"o}lder series
 of $\mathcal{S}ol(G/G';\lambda)_K$
 in Fact \ref{fact:Composition}.  

\begin{remark}
\label{rem:Plancherel}
The anti-de Sitter space
 $G/G' =X(n+1,1)$ has a compactification
 $G/G' \cup G/P$, 
and the disintegration
 of the regular representation on $L^2(G/G')$
 (Plancherel formula)
 is given by the boundary data
 (cf. \cite{SaVe}
 for the $p$-adic spherical variety):
\par\noindent
continuous spectrum
\[
\widetilde{\mathcal{P}}_{\lambda}
:I(\lambda) \to C^{\infty}(G/G')
 \quad
 (\lambda \in \frac n 2 + \sqrt{-1} {\mathbb{R}}), 
\]
discrete spectrum
\[
\widetilde{\mathcal{P}}_{-i}
:I(-i)/F(i) \to C^{\infty}(G/G')
 \quad
 (i \in \mathbb{N}).   
\]
\end{remark}

\vskip 1pc

\subsection{Analysis on vector bundles
 over anti-de Sitter spaces}
\label{subsec:vec}

For a finite-dimensional representation $F$ of $G'$, 
 we define a $G$-equivariant vector bundle
 ${\mathcal{F}} := G \times_{G'}F$
 over the homogeneous space $G/G'$, 
 and write $C^{\infty}(G/G', {\mathcal{F}})$
 for the space
 of smooth sections for ${\mathcal{F}}$
 endowed with the natural Fr{\'e}chet topology.  
In this section
 we consider the vector bundle ${\mathcal{F}}_j$
 associated to the finite-dimensional representation
\index{sbon}{Fj@$F(j)$}
 $F(j)$ of $G'$, 
 and determine 
 when irreducible representations
\index{sbon}{Ti@$T(i)$}
 $T(i)$ 
 and spherical principal series representations $I(\lambda)$
 of $G$
 occur in $C^{\infty}(G/G', {\mathcal{F}}_j)$
 as subrepresentations.

The main result of this section
 is stated as follows:
\begin{theorem}
\label{thm:lise}
Suppose $j \in {\mathbb{N}}$.  
\begin{enumerate}
\item[{\rm{1)}}]
For any $i \in {\mathbb{N}}$
 such that $i< j$, 
\[
\dim \operatorname{Hom}_G(T(i), C^{\infty}(G/G', {\mathcal{F}}_j))=1.  
\]
\item[{\rm{2)}}]
For any $i \in {\mathbb{N}}$
 such that $i\ge j$, 
\[
\dim \operatorname{Hom}_G(T(i), C^{\infty}(G/G', {\mathcal{F}}_j))
=
\begin{cases}
0
\quad
&(i \equiv j \mod 2), 
\\
1
&(i \not\equiv j \mod 2).  
\end{cases}
\]
\item[{\rm{3)}}]
For any $\lambda \in {\mathbb{C}}$
 and $j \in {\mathbb{N}}$, 
\[
  \dim \operatorname{Hom}_G(I(\lambda), C^{\infty}(G/G', {\mathcal{F}}_j))=1.  
\]
\end{enumerate}
\end{theorem}
As in the scalar valued case
 treated in Sections \ref{subsec:hyp} and \ref{subsec:Ahyp}, 
 the images of $T(i)$ and $I(\lambda)$
 in Theorem \ref{thm:lise}
 are contained in the eigenspace
 of a second-order differential operator.  
In fact,
 let $C_G \in U({\mathfrak {g}})$
 be the Casimir element 
 of the Lie algebra
 ${\mathfrak {g}}={\mathfrak {o}}(n+1,1)$.  
Via the left regular representation,
 $C_G$ acts on $C^{\infty}(G/G',{\mathcal{F}}_j)$
 as a second order hyperbolic differential operator.  
Let $\Delta_j$ be the differential operator
 induced by $2n C_G$  
 ($2n$ is a constant coming from the Killing form), 
 and set 
\[
  \mathcal{S}ol (G/G',{\mathcal{F}}_j;\lambda)
  :=
  \{f \in C^{\infty}(G/G',{\mathcal{F}}_j)
   : \Delta_j f = \lambda(n-\lambda)f\}.  
\]
Then 
\index{sbon}{SolGG@$\mathcal{S}ol (G/G';\lambda)$}
$
   \mathcal{S}ol (G/G',{\mathcal{F}}_0;\lambda)
   =
   \mathcal{S}ol (G/G';\lambda) 
$
 in the trivial line bundle case.  
Since the Casimir operator $C_G$ acts
 on $T(i)$ and $I(\lambda)$
 by the scalar 
 $-\frac {1}{2n}i(n+i)$ and $\frac{1}{2n} \lambda(n-\lambda)$, 
 respectively, 
 the image of $T(i)$ in Theorem \ref{thm:lise}
 is contained in $\mathcal{S}ol (G/G',{\mathcal{F}}_j;-i)$, 
 and that of $I(\lambda)$ is
 in $\mathcal{S}ol (G/G',{\mathcal{F}}_j;\lambda)$.  
This gives a generalization 
 of \eqref{eqn:Poisson}
 and \eqref{eqn:Poisson2}.

For the proof of Theorem \ref{thm:lise}, 
 we use a smooth version 
 of the Frobenius reciprocity theorem:

\begin{lemma}
\label{lem:Fro}
Suppose
 that $(\pi,{\mathcal{H}})$ is a continuous representation
 of $G$
 on a Banach space ${\mathcal{H}}$, 
 and we denote by ${\mathcal{H}}^{\infty}$
 the Fr{\'e}chet space
 of smooth vectors
 of $(\pi,{\mathcal{H}})$.  
Then we have a natural bijection:
\begin{equation}
\label{eqn:Fro}  
\operatorname{Hom}_{G'}({\mathcal{H}}^{\infty}|_{G'}, F)
\simeq
\operatorname{Hom}_{G}({\mathcal{H}}^{\infty}, 
  C^{\infty}(G/G', {\mathcal{F}}), 
\quad
\psi \leftrightarrow f.  
\end{equation}
\end{lemma}
\begin{proof}
We identify $C^{\infty}(G/G', {\mathcal{F}})$
 with the closed $G$-invariant subspace of $C^{\infty}(G,F)$
 defined by 
\[
   C^{\infty}(G,F)^{G'}
   :=
   \{f \in C^{\infty}(G,F)
     : f (g l)= l^{-1} f(g)
     \quad
     \text{for } g \in G, l \in G'\}.  
\]
Then the bijection \eqref{eqn:Fro} is given by 
\begin{alignat*}{4}
\psi &\mapsto f_{\psi}, 
\quad
&&f_{\psi}(u)(g)&&:=&& \psi(g^{-1}u), 
\\
f &\mapsto {\psi}_f, 
\quad
&&\psi_f(u)&&:=&& f(u)(e).  
\end{alignat*}
For a continuous $G'$-homomorphism 
 $\psi:{\mathcal{H}}^{\infty} \to F$, 
the map 
\[
  {\mathcal{H}}^{\infty} \to C^{\infty}(G,F), 
  \qquad
 u \mapsto f_{\psi}(u)
\]
is well-defined
 and continuous
 because it is a composition of continuous linear maps
\[
    {\mathcal{H}}^{\infty} 
\to C^{\infty}(G,{\mathcal{H}}^{\infty} )
\to C^{\infty}(G,F),
  \qquad
   u \mapsto (g \mapsto g^{-1}u)
     \mapsto (g \mapsto \psi(g^{-1}u)).  
\]
Other verifications for well-definedness are easy.  
Clearly,
 $\psi \mapsto f_{\psi}$
 and $f \mapsto \psi_f$
 give their inverses.  
Hence Lemma is proved.  
\end{proof}

\begin{proof}
[Proof of Theorem \ref{thm:lise}]
By Lemma \ref{lem:Fro}, 
 the statement follows immediately from Theorems \ref{thm:compo}
 and \ref{thm:mIF}.  
\end{proof}

\section{Application to branching laws
 of complementary series}
\label{sec:applbl}

The indefinite orthogonal group $G=O(n+1,1)$
 has a \lq{long}\rq\ complementary 
series \cite{Kos}.  
To be precise with our normalization,
 the spherical principal series representation 
\index{sbon}{Ilmd@${I(\lambda)}$}
 $I(\lambda)$ admits a $G$-invariant inner product 
 on the unitary axis $\lambda \in \frac n 2 + \sqrt{-1}{\mathbb{R}}$
 and on the real interval $\lambda \in (0,n)$.  
Taking the Hilbert completion,
 we obtain irreducible unitary representations,
 called the unitary principal series
 representation
 and the complementary series representation, 
 to be denoted by 
\index{sbon}{Hlmd@${\mathcal{H}}_{\lambda}^G$|textbf}
${\mathcal{H}}_{\lambda}^G$.  
In this chapter,
 we consider the restriction 
 of the complementary series representation
 ${\mathcal{H}}_{\lambda}^G$ of $G$
 to the subgroup $G'=O(n,1)$.  
As an application of our results
 on differential symmetry breaking operators
 $\widetilde{\mathbb{C}}_{\lambda,\nu}$
 (Chapter \ref{sec:10})
 combined with the idea
 of the \lq{F-method}\rq\
 (cf. \cite{xkeastwood60}), 
 we construct explicitly
 complementary series representations
 ${\mathcal{H}}_{\nu}^{G'}$
 of the subgroup $G'$
 as discrete summands
 in the restriction of ${\mathcal{H}}_{\lambda}^G|_{G'}$.  

\subsection{Discrete spectrum
 in complementary series}
\label{subsec:compThm}

For $\lambda \in {\mathbb{R}}$, 
 we set 
\index{sbon}{Dlmd@$D(\lambda)$|textbf}
\[
D(\lambda)
:=
\{\nu\in \lambda-1 + 2 {\mathbb{Z}}:
\frac {n-1}{2}< \nu \le \lambda-1\}.  
\]
Then $D(\lambda)$
 is a finite set,
 and $D(\lambda)$ is non-empty
 if and only if $\lambda> \frac{n+1}{2}$.  

We notice 
 that any continuous $G'$-homomorphism 
$
T
:I(n-\lambda) \to J(m-\nu)
$
induces a continuous 
$G'$-homomorphism
\index{sbon}{Ilmdinf@${I(\lambda)^{-\infty}}$}
$
T^{\vee}
:
I(\nu)^{-\infty}\to I(\lambda)^{-\infty}$
between the space
 of distribution vectors.  

\begin{theorem}
\label{thm:rest}
Suppose that $\frac {n+1}{2} < \lambda < n$.  
Then the $G'$-intertwining differential operator
 $\CC{n-\lambda}{m-\nu}^{\vee}
:J(\nu)^{-\infty} \to I(\lambda)^{-\infty}$
 induces an isometric embedding 
 (up to scalar), 
 ${\mathcal{H}}_{\nu}^{G'} \hookrightarrow {\mathcal{H}}_{\lambda}^{G}|_{G'}$
 if $\nu \in D(\lambda)$.  
In particular,
 the restriction $\pi_{\lambda}^{G}|_{G'}$
 contains $\bigoplus_{\nu \in D(\lambda)} {\mathcal{H}}_{\nu}^{G'}$
 as discrete summands.  
\end{theorem}
\begin{remark}
\label{rem:rest}
 If $\lambda \in (\frac {n+1}{2}, n)$, 
 then any $\nu \in D(\lambda)$
 satisfies
 $\frac{n-1}{2}<\nu<n-1$
and therefore ${\mathcal{H}}_{\nu}^{G'}$
 is a complementary series representation
 of $G'$.  
\end{remark}
\begin{remark}
\label{rem:SVcomp}
If $n-\lambda=m-\nu$, 
 namely,
 if $\lambda-\nu=1$, 
 then clearly 
 $\nu \in D(\lambda)$
 and Theorem \ref{thm:rest}
 implies
 $\operatorname{Hom}_{G'}
 ({\mathcal{H}}_{\lambda-1}^{G'}, {\mathcal{H}}_{\lambda}^{G}|_{G'})
 \ne \{0\}$.  
In this case
 this result was earlier proved in \cite{SV}.  
We note 
 that $\C_{n-\lambda,m-\nu}$
 is just the restriction operator 
 when $n-\lambda=m-\nu$.  
\end{remark}

\subsection{$L^2$-model of complementary series representations}
\label{subsec:Ltwomodel}

The proof uses an $L^2$-model
 of complementary series representations
 (\lq{Lagrangian model}\rq\ in \cite{HKM}, 
 or equivalently, 
\lq{commutative model}\rq\ in Vershik--Graev\cite{VG}).  

We recall
 that the Knapp--Stein intertwining operator
$
\ntT{\lambda}{n-\lambda}:
I(\lambda) \to I(n-\lambda)
$
 is a real operator
 ({\it{i.e.}},  
$
  \overline{\ntT{\lambda}{n-\lambda} \overline{f}}
  =\ntT{\lambda}{n-\lambda} f
$
)
if $\lambda \in {\mathbb{R}}$, 
 and consequently,
 the Hermitian form on $I(\lambda)$ defined by 
\[
(f_1, f_2)
:=(f_1, \ntT{\lambda}{n-\lambda} f_2)_{L^2({\mathbb{R}}^n)}
\quad
\text{for }
f_1, f_2 \in I(\lambda)
\]
is $G$-invariant.  
Furthermore,
 it is positive definite
 if $0 < \lambda< n$.  
We denote by ${\mathcal{H}}_{\lambda}$
 the Hilbert completion of $I(\lambda)$, 
 and use the same letter to denote the resulting unitary representation,
 which is called a 
 {\it{complementary series representation}}
 of $G$.

We define a family of Hilbert spaces
 $L^2({\mathbb{R}}^n)_s$
 with parameter $s \in {\mathbb{R}}$
 by 
\[
L^2(\mathbb{R}^n)_s
:=
L^2(\mathbb{R}^n, (\xi_1^2+\cdots+\xi_n^2)^{\frac s 2}
d\xi_1 \cdots d\xi_n).  
\]
By definition, 
$
L^2(\mathbb{R}^n)_0
=
L^2(\mathbb{R}^n)  
$, 
and 
\[
{\mathcal{S}}(\mathbb{R}^n)
\subset
\bigcap_{s>-n} L^2(\mathbb{R}^n)_s.  
\]

The space $I(\lambda)^{-\infty}$ 
 of distribution vectors
 of $I(\lambda)$
 is identified with 
 ${\mathcal{D}}'(X, {\mathcal{L}}_{\lambda})
 \simeq
 {\mathcal{D}}_{-\lambda}'(\Xi)$
 (see \eqref{eqn:IDist}).  
As in \eqref{eqn:restDn}, 
 we consider the restriction of distributions
 on $X \simeq S^n$ to ${\mathbb{R}}^n$, 
 and obtain a morphism
\[
  {\mathcal{H}}_{\lambda} \subset I(\lambda)^{-\infty}
  \to {\mathcal{S}}'({\mathbb{R}}^n).  
\] 

We then get an $L^2$-model 
 of the unitary representation 
 $(\pi_{\lambda}, {\mathcal{H}}_{\lambda})$
 by \eqref{eqn:Fxpower}:

\begin{proposition}
\label{prop:L2model}
If $\lambda \in (0,n)$, 
then the Euclidean Fourier transform 
 ${\mathcal{F}}_{\mathbb{R}^n}$
 gives a unitary isomorphism
 (up to scalar)
\[
{\mathcal{F}}_{\mathbb{R}^n}:
{\mathcal{H}}_{\lambda}
 \overset \sim \to 
 L^2(\mathbb{R}^n)_{n-2\lambda}.  
\]
\end{proposition}
\begin{proof}
See \cite{HKM} or \cite{VG}.  
\end{proof}

Suppose $\lambda-\nu-1 \in 2 {\mathbb{N}}$.  
We recall from \eqref{eqn:Cmul2}
that 
$
   \widetilde{C}_{2l}^{\mu}(s, t)
$
 is a homogeneous polynomial of two variables $s$, $t$
 of degree $2l$, 
 and that $\KC{\lambda}{\nu}$ is
 the distribution kernel 
 of the differential symmetry breaking operator
 $\C_{\lambda,\nu}$
 given in \eqref{eqn:KC}.  
Then it is immediate from the definition 
 of $\CC{\lambda}{\mu}$: 
\begin{lemma}
\label{lem:FKC}
\[
\index{sbon}{KxttCnu@$\KC{\lambda}{\nu}$}
({\mathcal{F}}_{\mathbb{R}^n} \KC{\lambda}{\nu})(\xi,\xi_n)
=(-1)^l \tilde C_{2l}^{\lambda-\frac{n-1}{2}}
 (-|\xi|^2, \xi_n).  
\]
\end{lemma}

We define a linear operator
 $\CC{n-\lambda}{m-\nu}^{\wedge}:
{\mathcal {S}}'({\mathbb{R}}^m)
\to 
{\mathcal {S}}'({\mathbb{R}}^n)
$
by a multiplication
 of the polynomial:
\index{sbon}{Cte@$\widetilde{\mathbb C}_{\lambda,\nu}^{\wedge}(\xi,\xi_n)$
|textbf}
\begin{equation}
\label{eqn:Chat}
(\CC{n-\lambda}{m-\nu}^{\wedge}v)(\xi, \xi_n):=
\widetilde{C}_{\lambda-\nu-1}^{\frac{n+1}{2}-\lambda}(|\xi|^2, \xi_n)v(\xi)
\quad
\text{for}
\quad
v \in {\mathcal{S}}'({\mathbb{R}}^m).  
\end{equation}

\begin{proposition}
\label{prop:Cdual}
Suppose $\lambda - \nu -1 \in 2{\mathbb{N}}$.  
\begin{enumerate}
\item[{\rm{1)}}]
$(n-\lambda, m-\nu) \in /\!/$
 and the dual map 
$
\CC{n-\lambda}{m-\nu}^{\vee}
:
J(\nu)^{-\infty} \to I(\lambda)^{-\infty}
$
 of the differential operator $\CC{n-\lambda}{m-\nu}:I(\lambda)\to J(\nu)$
 is a continuous $G'$-homomorphism.  
\item[{\rm{2)}}]
The diagram
$$
\begin{CD}
J(\nu)^{-\infty} 
@>{\CC{n-\lambda}{m-\nu}^{\vee}} >> 
I(\lambda)^{-\infty}
\\
@V{{\mathcal {F}}_{{\mathbb{R}^m}}\circ \operatorname{Rest}}VV 
@VV{{\mathcal {F}}_{{\mathbb{R}^n}}\circ \operatorname{Rest}}V 
\\
{\mathcal {S}}'({\mathbb{R}}^m)
@>>{\CC{n-\lambda}{m-\nu}^{\wedge}}>
{\mathcal {S}}'({\mathbb{R}}^n)
\end{CD}
$$
commutes.  
Here $\operatorname{Rest}$ denotes
 the restriction of distributions
 to the Bruhat cell.  
\item[{\rm{3)}}]
If $\nu >\frac{n-1}{2}$, 
 then the linear map $\CC{n-\lambda}{m-\nu}^{\wedge}$
 induces an isometry
 (up to a scalar)
 of Hilbert spaces:
\[
L^2(\mathbb{R}^m)_{m-2\nu}
\hookrightarrow
L^2(\mathbb{R}^n)_{n-2\lambda}. 
\]
\end{enumerate}
\end{proposition}

For the proof of Proposition \ref{prop:Cdual}, 
 we use the following formula,
 which is immediate from the integral expression
 of the Beta function
 by a change of variables.  
\begin{lemma}
\label{lem:Bfunct}
If $c, d \in {\mathbb{R}}$ satisfy
 $-\frac 1 2 < c < -d -\frac 1 2$, 
then $|\xi_n|^{2c}(|\xi|^2+\xi_n^2)^d$ is integrable
 as a function of $\xi_n$
 with parameter $\xi$, 
 and we have the identity
\[
\int_{-\infty}^{\infty}|\xi_n|^{2c}(|\xi|^2+\xi_n^2)^d d \xi_n
=
|\xi|^{2c+2d+1}B(c+\frac 1 2, -c-d-\frac 1 2).  
\]
\end{lemma}
\begin{proof}
[Proof of Proposition \ref{prop:Cdual}]
1)\enspace Clear.  
\par\noindent
2)\enspace
This follows from Lemma \ref{lem:FKC}.  
\par\noindent
3) \enspace
$\widetilde C_{2L}^{\frac {n+1}{2}-\lambda}(|\xi|^2, \xi_n)$
 is a linear combination
 of homogeneous polynomials
 $|\xi|^{2j}\xi_n^{2L-2j}$
 ($0 \le j \le L$). 
For each $j$, 
 we can apply Lemma \ref{lem:Bfunct}
 with $c=2L -2j$
 and $d=\frac n 2 - \lambda$
 if $\nu > \frac {n-1}2$, 
and get 
\begin{equation*}
\|
|\xi|^j |\xi_n|^{L-j}
v(\xi)
\|_{L^2({\mathbb{R}}^n)_{n-2\lambda}}^2
=
 B(2L-2j+\frac 1 2, \nu-\frac {n-1}{2}+2j)
\|v\|_{L^2({\mathbb{R}}^m)_{m-2\nu}}^2
\end{equation*}
by the Fubini theorem.  
Therefore, 
the map 
$
   \CC{n-\lambda}{m-\nu}^{\wedge}:
   L^2({\mathbb{R}}^m)_{m-2\nu}
   \to 
   L^2({\mathbb{R}}^n)_{n-2\lambda}
$
 is well-defined
 and continuous
 if $\nu > \frac{n-1}{2}$.  
Since the unitary representation of $G'$
 on $L^2({\mathbb{R}}^m)_{m-2\nu}$
 is irreducible,
 the continuous $G'$-intertwining operator
 is automatically isometric up to scalar.  
\end{proof}
\begin{proof}[Proof of Theorem \ref{thm:rest}]
By Proposition \ref{prop:Cdual}, 
 for every $\nu \in D(\lambda)$, 
 we have a $G'$-intertwining 
 and isometric (up to scalar) operator
\[
  \CC{n-\lambda}{m-\nu}^{\wedge}:L^2({\mathbb{R}}^m)_{m-2\nu}
   \to 
    L^2({\mathbb{R}}^n)_{n-2\lambda}, 
\]
which in turn induces 
 a $G'$-intertwining 
 and isometric (up to scalar) operator
\[
  \CC{n-\lambda}{m-\nu}^{\vee}:{\mathcal{H}}_{\nu}^{G'}
   \to 
    {\mathcal{H}}_{\lambda}^{G}, 
\]
by Propositions \ref{prop:L2model} and \ref{prop:Cdual} (2). 
\end{proof}
To end this chapter,
 we discuss the Fourier transform
 of the distribution kernel $\KA{\lambda}{\nu}$
 of the (generically) regular symmetry breaking operators
 $\A_{\lambda,\nu}$
 and compare 
 that of $\CC\lambda\nu$
 in Lemma \ref{lem:FKC}
 by using the hypergeometric function.  
The Gauss hypergeometric function
 has the following series expansion.  
\[
 {}_2F_1(a,b;c;z)
  =
  \sum_{j=0}^{\infty} \frac{(a)_j(b)_j}{(c)_j}\frac{z^j}{j!}, 
\]
where $(a)_j=a(a+1)\cdots(a+j-1)$.  
The series terminates
 if $a \in -{\mathbb{N}}$
 or $b \in -{\mathbb{N}}$, 
 and reduces to a polynomial.  
In particular,
 ${}_2F_1(\frac {\lambda-\nu}{2}, \frac{\lambda+\nu+1-n}{2};c;z)$
 reduces to a polynomial 
 if 
\index{sbon}{Xr@${/\!/}$}
\index{sbon}{Xl@${\backslash\!\backslash}$}
$(\lambda, \nu) \in \backslash\!\backslash \cup /\!/$.

Since the Gegenbauer polynomial
 of even degree is given as
\[
  C_{2l}^{\mu}(x)
  =
  \frac{(-1)^l \Gamma(l+\mu)}{l! \Gamma(\mu)}
  {}_2 F_1 (-l, l+\mu; \frac 1 2;x^2),
\]
 the following proposition
 gives a direct proof of Juhl's conformally covariant 
differential operators
 $\C_{\lambda,\nu}$
 (see \eqref{eqn:Clmdnu}), 
 and also 
 explains
 the residue formula
\index{sbon}{qCA@$q_C^A$}
 $\tA{\lambda}{\nu} =q_C^A(\lambda,\nu) 
 \CC{\lambda}{\nu}$
 of the (generically) regular symmetry breaking operators
 $\A_{\lambda,\nu}$
 in Theorem \ref{thm:reduction} from the view point of the F-method.  
\begin{proposition}
[{\cite[Proposition 5.3]{xkeastwood60}}]
\label{prop:4.3}
~~
\begin{enumerate}
\item[{\rm{1)}}]
The tempered distribution 
 ${\mathcal{F}}_{\mathbb{R}^n} \KA{\lambda}{\nu}
 \in {\mathcal{S}}'(\mathbb{R}^n)$
 is a real analytic function 
 in the open subset
 $\{(\xi,\xi_n)\in \mathbb{R}^{n-1}\oplus{\mathbb{R}}
  :
  |\xi|>|\xi_n|\}$, 
and takes the following form:
\index{sbon}{Kxt@$\KA{\lambda}{\nu} (x, x_n)$}
\begin{equation}
\label{eqn:FcA}
({\mathcal{F}}_{\mathbb{R}^n} \KA{\lambda}{\nu})(\xi, \xi_n)
=\frac{\pi^{\frac{n-1}{2}}|\xi|^{\nu-\lambda}}{\Gamma(\nu)2^{\nu-\lambda}}
{}_2F_1
(\frac{\lambda-\nu}{2}, \frac{\lambda+\nu+1-n}{2};
\frac 1 2;-\frac{\xi_n^2}{|\xi|^2}).  
\end{equation}
\item[{\rm{2)}}]
Suppose $\nu -\lambda = 2l$
 ($l \in {\mathbb{N}}$).  
Then 
\[
  ({\mathcal{F}}_{\mathbb{R}^n} \KA{\lambda}{\nu})(\xi,\xi_n)
=\frac{l! \pi^{\frac {n-1} 2}}{2^{2l}\Gamma(\nu)}
\tilde C_{2l}^{\lambda-\frac{n-1}{2}}
 (-|\xi|^2, \xi_n).  
\]
\end{enumerate}
\end{proposition}

\section{Appendix}

\subsection{Gegenbauer polynomials}

The Gegenbauer polynomials $C_N^\mu (t)$ are polynomials of degree $N$ given by
\[
C_N^\mu (t)
:= \frac{\Gamma (2\mu+N)}{\Gamma (N+1) \Gamma (2\mu)}   \,
   {}_2 F_1 (2\mu+N, -N; \mu+\frac{1}{2}; \frac{1-t}{2} ) 
\]

\[
= \sum_{j=0}^{[\frac{N}{2}]} (-1)^j
   \frac{\Gamma(N-j+\mu)}
          {\Gamma(\mu)\Gamma(j+1)\Gamma(N-2j+1)}
   (2t)^{N-2j}.
\]
We inflate $C_N^\mu(t)$ to a polynomial of two variables by
\begin{equation}\label{eqn:Cst}
C_N^\mu (s,t)
:= s^{\frac{N}{2}} C_N^\mu \Bigl( \frac{t}{\sqrt{s}} \Bigr).
\end{equation}
For instance,
$C_0^\mu (s,t) = 1$,
$C_1^\mu (s,t) = 2 \mu t$,
$C_2^\mu (s,t) = 2\mu (\mu+1) t^2 - \mu s$.
For even $N$, 
 we write 
\[
C_{2l}^\mu (s,t)
= \frac{\Gamma(\mu+l)}{\Gamma(\mu)}
   \sum_{j=0}^l a_j (l; \mu) s^j t^{2l-2j}
\]
where
\begin{equation}\label{eqn:ajlmu}
a_j (l; \mu)
\index{sbon}{ajlmu@$a_j (l; \mu)$|textbf}
:= \frac{(-1)^j 2^{2l-2j}}{j!(2l-2j)!}
   \prod_{i=1}^{l-j} (\mu + l + i - 1).  
\end{equation}

We set
\index{sbon}{Ctst@$\widetilde C_{2l}^\mu (s,t)$}
\begin{equation}\label{eqn:Cmul2}
\widetilde{C}_{2l}^\mu (s,t)
:= \frac{\Gamma(\mu)}{\Gamma(\mu+l)} C_{2l}^\mu (s,t)
= \sum_{j=0}^l a_j (l;\mu) s^j t^{2l-2j} .
\end{equation}

Slightly different from the usual notation in the literature,
we adopt the following normalization of the Gegenbauer polynomial:
\index{sbon}{Ctt@$\tilde{\tilde{C}}_N^\mu (t)$|textbf}
\begin{equation}\label{eqn:nGeg}
\Tilde{\Tilde{C}}_N^\mu (t)
:= (\mu+\frac{N}{2}) \Gamma(\mu) C_k^\mu (t)
\end{equation}
which implies
\begin{equation}\label{eqn:C0}
\Tilde{\Tilde{C}}_N^0 (t) =  \cos N t
\quad
\text{ and }\Tilde{\Tilde{C}}_0^\mu = \Gamma(\mu+1), 
\end{equation}
see {\cite[8934.4]{GR}}.

We recall from  (\cite[vol.~II, 16.3 (2)]{EMOT} or \cite[7311.2]{GR}):
\begin{equation}\label{eqn:Ct int}
\int_0^1 t^{N+2\rho} (1-t^2)^{\mu-\frac{1}{2}} C_N^\mu (t) dt
=
  \frac{\Gamma (2\mu+N) \Gamma (2\rho+N+1) \Gamma (\mu+\frac{1}{2}) \Gamma (\rho+\frac{1}{2})}
         {2^{N+1} \Gamma (2\mu) \Gamma (2\rho+1) N! \, \Gamma (N+\mu+\rho+1)}
\end{equation}
for $\rho > -\frac{1}{2}$.

By using twice the duplication formula of the Gamma function
\begin{equation}\label{eqn:dupl}
\Gamma (2\mu) = 2^{2\mu-1} \pi^{-\frac{1}{2}} \Gamma (\mu) \Gamma (\mu+\frac12),
\end{equation}
we get from \eqref{eqn:Ct int}
\begin{equation}\label{eqn:Ct int2}
\int_0^1 
t^a (1-t^2)^{\frac{n-3}{2}} 
\Tilde{\Tilde{C}}_N^{\frac{n}{2}-1} (t) dt
=
\frac{\pi \Gamma (n+N-1)}{2^{a+n-1} \Gamma (N+1)}  \,
\frac{\Gamma (a+1)}
{\Gamma (\frac{a-N+2}{2}) \Gamma (\frac{a+N+n}{2}) }.  
\end{equation}

\subsection{$K$-Bessel function and its renormalization}

We recall the definition
 of the I-Bessel function
 and the K-Bessel function:
\begin{align*}
  I_{\nu}(z):= & e^{-\frac{\sqrt{-1} \nu \pi}{2}}
                 J_{\nu}(e^{\frac{\sqrt{-1}\pi}{2}}z)
\\
             =& (\frac z 2)^{\nu}\sum_{j=0}^{\infty} 
                \frac{(\frac z 2)^{2j}}{j ! \Gamma(j+ \nu +1)}, 
\\
  K_{\nu}(z) := & \frac {\pi}{2 \sin \nu \pi}
                  (I_{-\nu}(z)-I_{\nu}(z)).  
\end{align*}

We renormalize the $K$-Bessel function as 
\index{sbon}{Knut@$\widetilde K_{\nu}(z)$|textbf}
\begin{equation}
\label{eqn:KBessel}
  \widetilde K_{\nu}(z)
  :=(\frac z 2)^{-\nu}K_{\nu}(z).  
\end{equation}
Since $K_{\nu}(z)=K_{-\nu}(z)$, 
 we have
\begin{equation}\label{eqn:Kdual}
  \widetilde K_{\nu}(z)=(\frac z 2)^{-2\nu}\widetilde K_{-\nu}(z).  
\end{equation}
For example, 
\[
K_{\frac1 2}(z)= \sqrt{\frac{\pi}{2}}e^{-z} z^{-\frac 12},
\quad
\widetilde K_{\frac 1 2}(z)=\frac{\sqrt{\pi} e^{-z}}{z},
\quad
\widetilde K_{-\frac 1 2}(z)=\frac{\sqrt{\pi}}{2} e^{-z}.  
\]

The Fourier transform
 of the distribution $(|x|^2+t^2)^{\lambda}$
 is given by the K-Bessel function:
\begin{equation}
\label{eqn:Fone}
\int_{\mathbb{R}^m}(|x|^2+t^2)^{\lambda}e^{-i \langle x, \xi \rangle} d x
=
\frac {2 |t|^{m+2 \lambda} \pi^{\frac m 2}}{\Gamma(-\lambda)}
  \widetilde K_{\frac m 2 + \lambda}(|t \xi|).  
\end{equation}

\subsection{Zuckerman derived functor modules
 $A_{\mathfrak {q}}(\lambda)$}
\label{subsec:16.3}

In algebraic representation theory,
 cohomological parabolic induction
 is a powerful tool
 in capturing isolated irreducible unitary
 representations
 of real reductive groups
({\it{e.g.}},  \cite{Vred, VZ}).  
For a convenience
 of the reader,
 we give a description
of the underlying $({\mathfrak {g}}, K)$-module
 of the infinite-dimensional irreducible subquotient 
\index{sbon}{Ti@$T(i)$}
 $T(i)$
 of the spherical principal series representation
 $I(\lambda)$
 ($\lambda=-i$ or $n+i$), 
 even though the proof of our main results 
 in this article
 is logically independent
 of this section.  

We take a maximal abelian subalgebra
 ${\mathfrak {t}}$
 in the Lie algebra 
 ${\mathfrak {k}} \simeq {\mathfrak {o}}(n+1)$
 of the maximal compact subgroup $K=O(n+1) \times O(1)$, 
 and extend it
 to a Cartan subalgebra ${\mathfrak {h}}$
 of ${\mathfrak {g}} = {\mathfrak {o}}(n+1,1)$.  
If $n$ is even
 then $\dim {\mathfrak {h}} = \dim {\mathfrak {t}}+1$, 
 and ${\mathfrak {h}}={\mathfrak {t}}$, 
 otherwise.

Fix a basis $\{f_i: 1 \le i \le [\frac n 2]+1\}$
 of ${\mathfrak {h}}_{\mathbb{C}}^*$
 in a way
 that the root system 
 is given as 
\begin{multline*}
  \Delta({\mathfrak {g}}_{\mathbb{C}}, {\mathfrak {h}}_{\mathbb{C}})
  =
  \{ \pm (f_i \pm f_j) : 1 \le i \le j \le [\frac n 2]+1\}
\\
  (\cup
  \{ \pm f_l: 1 \le l \le [\frac n 2]+1 \quad(\text{$n$: odd})\}).  
\end{multline*}

Let $\{H_i\} \subset {\mathfrak {h}}_{\mathbb{C}}$
 be the dual basis
 for $\{f_i\} \subset {\mathfrak {h}}_{\mathbb{C}}^{\ast}$.  
We define a subgroup of $L$
 to be the centralizer of $H_1$,
 and thus $L \simeq SO(2) \times O(n-1,1)$.

Let ${\mathfrak {q}}={\mathfrak {l}}_{\mathbb{C}} + {\mathfrak {u}}$
 be a $\theta$-stable parabolic subalgebra
 of ${\mathfrak {g}}_{\mathbb{C}}$
 where the nilpotent radical ${\mathfrak {u}}$
 is an ${\mathfrak {h}}_{\mathbb{C}}$-stable subspace
 with 
\[
\Delta({\mathfrak {u}}, {\mathfrak {h}}_{\mathbb{C}})
=\{
   f_1 \pm f_j:
   2 \le j \le [\frac n 2]+1
\}
  \,(\cup\, \{f_1\} \quad \text{($n$: odd)}).  
\]
Then $L$ is the normalizer of ${\mathfrak {q}}$ in $G$.

For $\mu \in {\mathbb{C}}$, 
 we write ${\mathbb{C}}_{\mu f_1}$
 for the one-dimensional representation
of the Lie algebra ${\mathfrak {l}}$
 with trivial action of the second factor ${\mathfrak {o}}(n-1,1)$.  
If $\mu \in {\mathbb{Z}}$, 
 it lifts to $L$ with trivial action
 of $O(n-1,1)$,
 for which we use the same notation ${\mathbb{C}}_{\mu f_1}$.  
The homogeneous space $G/L$
 carries a $G$-invariant complex structure
 with complex cotangent space
 ${\mathfrak {u}}$
 at the origin.

Let denote by ${\mathcal{L}}_{\mu} := G \times_L {\mathbb{C}}_{\mu f_1}$
 the holomorphic vector bundle over $G/L$
 associated to the one-dimensional representation
 ${\mathbb{C}}_{\mu f_1}$.  
With this notation,
 the canonical bundle $\Omega_{G/L} = \bigwedge^{\dim_{\mathbb{C}}
 \frak{u}} T^{\ast}(G/L)$
 is isomorphic to ${\mathcal{L}}_{n f_1} ={\mathcal{L}}_{2 \rho({\mathfrak {u}})}$.

As an analogue of the Dolbeault cohomology
 of a $G$-equivariant holomorphic vector bundle
 over a complex manifold $G/L$, 
 Zuckerman introduced 
 the cohomological parabolic induction 
 ${\mathcal{R}}_{\mathfrak {q}}^j\equiv 
  ({\mathcal{R}}_{\mathfrak {q}}^{\mathfrak {g}})^j$
 ($j \in {\mathbb{N}}$), 
 which is a covariant functor from the category 
 of $({\mathfrak {l}}, L \cap K)$-modules
 to the category of $({\mathfrak {g}}, K)$-modules.  
We follow the normalization 
 in \cite[Definition 6.20]{Vred}
 for ${\mathcal{R}}_{\mathfrak {q}}^j$, 
 and Vogan--Zuckerman \cite{VZ} for $A_{\mathfrak {q}}(\lambda)$, 
 which differs  from the usual normalization by the \lq{$\rho({\mathfrak {u}})$}\rq\
 shift.

The one-dimensional representation ${\mathbb{C}}_{\mu f_1}$
 is 
\begin{align*}
\text{in the good range}\Leftrightarrow & \mu > \frac n 2 -1,
\\
\text{in the weakly fair range}\Leftrightarrow & \mu \ge 0,
\end{align*}
with respect to the $\theta$-stable parabolic subalgebra.  
In our normalization,
${\mathcal{R}}_{\mathfrak {q}}^j({\mathbb{C}}_{\mu f_1})=0$
 if $j \ne n-1$, 
and ${\mathcal{R}}_{\mathfrak{q}}^{n-1}({\mathbb{C}}_{\mu f_1})$
 is nonzero and irreducible
 if $\mu \in {\mathbb{Z}} + \frac n2$
 and $\mu >-1$, 
which is a slightly sharper
 than the results applied by the general theory.

Then we have 
(see \cite[Fact 5.4]{KO2} 
 and the references therein):
\begin{proposition}
\label{prop:Aq}
\begin{enumerate}
\item[{\rm{1)}}]
For $i \in {\mathbb{N}}$, 
we have the following isomorphisms
 of $({\mathfrak {g}},K)$-modules:
\[
  T(i)_K \simeq A_{\mathfrak {q}}(i f_1)
         \simeq {\mathcal{R}}_{\mathfrak {q}}^{n-1}
                ({\mathbb{C}}_{(\frac n2 +i) f_1})
         \simeq H_{\bar \partial}^{n-1}(G/L, {\mathcal{L}}_{(n+i)f_1})_K.  
\]
\item[{\rm{2)}}]
For $i \in {\mathbb{Z}}$
 with $-\frac n 2 \le i <0$, 
\[
  I(n+i)_K \simeq A_{\mathfrak {q}}(i f_1)
         \simeq {\mathcal{R}}_{\mathfrak {q}}^{n-1}
                ({\mathbb{C}}_{(\frac n2 +i) f_1})
         \simeq H_{\bar \partial}^{n-1}(G/L, {\mathcal{L}}_{(n+i)f_1})_K.  
\]
\end{enumerate}
\end{proposition}
The latter module $I(n+i)_K$ 
 $(-\frac n 2 \le i <0)$
 is the underlying $({\mathfrak {g}}, K)$-module
 of the complementary series representation
\index{sbon}{Hlmd@${\mathcal{H}}_{\lambda}^G$}
 ${\mathcal{H}}_{\lambda}^G$
 with $\lambda=n+i$
 (see Chapter \ref{sec:applbl}).  

\begin{remark}
\label{rem:GL}
The homogeneous space $G/L$
 is connected
 if $n \ge 2$.  
If $n=1$, 
then $G/L$ splits into two disconnected components
 which are biholomorphic to the Poincar{\'e} upper and lower half plane.  
This explains geometrically 
the reason 
 why $T(i)$ remains irreducible
 as a representation
 of the identity component group
 $G_0=SO_0(n+1,1)$
 for $n \ge 2$, 
 and splits into a direct sum
 of holomorphic and anti-holomorphic discrete series 
 representations
 of $G_0$
 for $n=1$
 (see Section \ref{subsec:matrix}).  
\end{remark}
\vskip 2pc
\subsection*{Acknowledgments}
The authors were partially supported by JSPS 
KAKENHI (25247006), 
NSF grant DMS-0901024, 
the GCOE program of the University of Tokyo, 
and Mathematisches Forschungsinstitut Oberwolfach.  

%----------------------------------
%		References
%----------------------------------
%\tolerance=2000

\printindex{sbon}{List of Symbols}
\end{document}